\documentclass[english]{article}
\usepackage[T1]{fontenc}
\usepackage[latin9]{inputenc}
\usepackage{geometry}
\usepackage[active]{srcltx}
\usepackage{babel}
\usepackage{verbatim}
\usepackage{mathtools}
\usepackage{bm}
\usepackage{amsmath}
\usepackage{amssymb}
\usepackage[unicode=true,
 bookmarks=false,
 breaklinks=false,pdfborder={0 0 1},backref=false,colorlinks=false]
 {hyperref}
\hypersetup{
 colorlinks,linkcolor=red,anchorcolor=blue,citecolor=blue}

\makeatletter
\usepackage{babel}

\usepackage[T1]{fontenc}
\usepackage{cite}\usepackage{amsthm}\usepackage{dsfont}\usepackage{array}\usepackage{mathrsfs}\usepackage{comment}\onecolumn
\usepackage{natbib}
\usepackage{color}\usepackage{babel}

\allowdisplaybreaks

\usepackage{enumitem}
\setlist[itemize]{leftmargin=1.5em}
\setlist[enumerate]{leftmargin=1.5em}

\usepackage{babel}
\usepackage{algorithm}
\usepackage{algorithmic}
\usepackage{arydshln}

\newcommand{\argmax}{\mathop{\mathrm{argmax}}}

\DeclareMathOperator{\ind}{\mathds{1}}  

\numberwithin{equation}{section}

\definecolor{yxc}{RGB}{255,0,0}
\definecolor{yjc}{RGB}{125,0,0}
\definecolor{cm}{RGB}{0,0,200}
\definecolor{kzw}{RGB}{0,150,0}
\definecolor{byw}{RGB}{0,0,200}
\definecolor{cxc}{RGB}{255,150,0}
\definecolor{gen}{RGB}{180,0,180}

\newcommand{\yxc}[1]{\textcolor{yxc}{[YXC: #1]}}

\makeatother

\begin{document}
\theoremstyle{plain} \newtheorem{lemma}{\textbf{Lemma}} \newtheorem{proposition}{\textbf{Proposition}}\newtheorem{theorem}{\textbf{Theorem}}\setcounter{theorem}{0}
\newtheorem{corollary}{\textbf{Corollary}} \newtheorem{assumption}{\textbf{Assumption}}
\newtheorem{example}{\textbf{Example}} \newtheorem{definition}{\textbf{Definition}}
\newtheorem{fact}{\textbf{Fact}}\newtheorem{property}{Property}
\theoremstyle{definition}

\theoremstyle{remark}\newtheorem{remark}{\textbf{Remark}}\newtheorem{condition}{Condition}\newtheorem{claim}{Claim}\newtheorem{conjecture}{Conjecture}
\title{Minimax Estimation of Linear Functions of Eigenvectors \\ in the
Face of Small Eigen-Gaps}
\author{Gen Li\thanks{Department of Statistics and Data Science, The Wharton School, University of Pennsylvania, Philadelphia, PA 19104, USA; email: \texttt{\{ligen,yuxinc\}@wharton.upenn.edu.}}
\and 
Changxiao Cai\thanks{Department of Biostatistics, University of Pennsylvania, Philadelphia, PA 19104, USA; email: \texttt{changxiao.cai@pennmedicine.upenn.edu}.}
\and  
H.~Vincent Poor\thanks{Department of Electrical and Computer Engineering, Princeton University,
Princeton, NJ 08544, USA; email: \texttt{poor@princeton.edu}.}
\and 
Yuxin Chen\footnotemark[1]}

\maketitle
Eigenvector perturbation analysis plays a vital role in various 
data science applications. A large body of prior works, however, focused
on establishing $\ell_{2}$ eigenvector perturbation bounds, which
are often highly inadequate in addressing tasks that rely on fine-grained
behavior of an eigenvector. This paper makes progress on this by studying
the perturbation of linear functions of an unknown eigenvector. Focusing
on two fundamental problems --- matrix denoising and principal component
analysis --- in the presence of Gaussian noise, we develop a suite
of statistical theory that characterizes the perturbation of arbitrary
linear functions of an unknown eigenvector. In order to mitigate a
non-negligible bias issue inherent to the natural ``plug-in'' estimator,
we develop de-biased estimators that (1) achieve minimax lower bounds
for a family of scenarios (modulo some logarithmic factor), and (2)
can be computed in a data-driven manner without sample splitting.
Noteworthily, the proposed estimators are nearly minimax optimal even
when the associated eigen-gap is {\em substantially smaller} than what is
required in prior statistical theory.

\medskip


\tableofcontents{}

\section{Introduction}

A variety of large-scale data science applications involve extracting
actionable knowledge from the eigenvectors of a certain low-rank matrix.
Representative examples include principal component analysis (PCA)
\citep{johnstone2001distribution}, phase synchronization \citep{singer2011angular},
clustering in mixture models \citep{loffler2019optimality}, community
recovery \citep{abbe2017entrywise, lei2015consistency}, to name just a few. In reality,
it is often the case that one only observes a randomly corrupted version
of the low-rank matrix of interest, and has to retrieve information from the
``empirical'' eigenvectors (i.e., the eigenvectors of the observed
noisy matrix). This motivates the studies of eigenvector perturbation
theory from statistical viewpoints, with particular emphasis on high-dimensional
scenarios \citep{chen2020spectral}. In the current paper, we seek
to further expand such a statistical theory, focusing on the following
two concrete models. 
\begin{itemize}
\item \emph{Matrix denoising under i.i.d.~Gaussian noise}. Let $\bm{M}^{\star}\in\mathbb{R}^{n\times n}$
be an unknown rank-$r$ symmetric matrix whose $l$-th eigenvector
(resp.~eigenvalue) is $\bm{u}_{l}^{\star}$ (resp.~$\lambda_{l}^{\star}$).
What we have observed is a corrupted version $\bm{M}=\bm{M}^{\star}+\bm{H}$
of $\bm{M}^{\star}$, where $\bm{H}=[H_{i,j}]_{1\leq i,j\leq n}$
represents a symmetric Gaussian random matrix with $H_{i,j}\overset{\mathsf{i.i.d.}}{\sim}\mathcal{N}(0,\sigma^{2}),i > j$ and $H_{i,i}\overset{\mathsf{i.i.d.}}{\sim}\mathcal{N}(0,2\sigma^{2})$.
The aim is to estimate $\bm{u}_{l}^{\star}$ based on the $l$-th
eigenvector of the data matrix $\bm{M}$. 
\item \emph{Principal component analysis (PCA) and covariance estimation}.
Imagine that we have collected $n$ independent $p$-dimensional sample vectors $\bm{s}_{i}\overset{\mathsf{ind.}}{\sim}\mathcal{N}(\bm{0},\bm{\Sigma})$,
$1\leq i\leq n$. Suppose that the underlying covariance matrix enjoys
a ``spiked'' structure $\bm{\Sigma}=\bm{\Sigma}^{\star}+\sigma^{2}\bm{I}_{p}$,
where $\bm{\Sigma}^{\star}\succeq\bm{0}$ is an unknown rank-$r$
matrix whose $l$-th eigenvector (resp.~eigenvalue) is given by $\bm{u}_{l}^{\star}$
(resp.~$\lambda_{l}^{\star}$). We seek to estimate $\bm{u}_{l}^{\star}$
by examining the $l$-th eigenvector of the sample covariance matrix
$\frac{1}{n}\sum_{i=1}^{n}\bm{s}_{i}\bm{s}_{i}^{\top}$. 
\end{itemize}

While a large body of prior literature has investigated eigenvector
perturbation theory for the aforementioned two models, the majority
of past works focused on $\ell_{2}$ statistical analysis, namely,
quantifying the $\ell_{2}$ estimation error of $\bm{u}_{l}$ when
it is employed to estimate $\bm{u}_{l}^{\star}$. Such $\ell_{2}$
perturbation theory, however, is often too coarse if the ultimate
goal is to retrieve fine-grained information from the eigenvector
of interest, say, some linear function of the eigenvector $\bm{u}_{l}^{\star}$
(e.g., the Fourier transform of or any given entry of $\bm{u}_{l}^{\star}$).
Motivated by the inadequacy of existing $\ell_{2}$ theory, we seek
to investigate how to faithfully estimate linear functionals of the
eigenvectors --- that is, $\bm{a}^{\top}\bm{u}_{l}^{\star}$ for
some vector $\bm{a}\in\mathbb{R}^{n}$ given \emph{a priori}. Towards
achieving this goal, two challenges stand out, which merit careful
thinking. 
\begin{itemize}
\item \emph{The need of bias correction}. A natural strategy towards estimating
the linear form $\bm{a}^{\top}\bm{u}_{l}^{\star}$ is to invoke the
naive ``plug-in'' estimator $\bm{a}^{\top}\bm{u}_{l}$. However,
it has already been pointed out in the literature (e.g., \citet{koltchinskii2016asymptotics,koltchinskii2016perturbation})
that the plug-in estimator might suffer from a non-negligible bias.
This calls for careful designs of algorithms that allow for proper
bias correction in a data-driven yet efficient manner. 
\item \emph{How to cope with small eigen-gaps}.
When estimating the eigenvector $\bm{u}_{l}^{\star}$, most prior
works require the associated eigen-gap $\min_{i\neq l}|\lambda_{i}^{\star}-\lambda_{l}^{\star}|$
to exceed the spectral norm of the perturbation matrix (i.e., $\bm{H}$
in the matrix denoising case and $\frac{1}{n}\sum_{i}\bm{s}_{i}\bm{s}_{i}^{\top}-\bm{\Sigma}$
		in the PCA setting) \citep{davis1970rotation,chen2020spectral}. However, there is no lower bound in the literature
that precludes us from achieving faithful estimation when the eigen-gap
violates such requirements. It would thus be of great interest to
understand the statistical limits when the eigen-gap of interest is
particularly small. 
\end{itemize}

\paragraph{Main contributions. }

This paper investigates estimating the linear form $\bm{a}^{\top}\bm{u}_{i}^{\star}$
for the aforementioned two statistical models under Gaussian noise,
with particular emphasis on those scenarios with {\em small eigen-gaps}.
Our main contributions are summarized below.
\begin{enumerate}
\item We develop fine-grained perturbation analysis for linear forms of
eigenvectors, which is valid even when the eigen-gap $\min_{i\neq l}|\lambda_{i}^{\star}-\lambda_{l}^{\star}|$
is substantially smaller than the spectral norm of the perturbation
matrix. This eigen-gap condition significantly improves upon what
is required in prior theory. 
\item The natural ``plug-in'' estimator suffers from a non-negligible
bias issue, which is particularly severe when the associated eigen-gap
is small. To address this issue, we put forward a de-biased estimator
for $\bm{a}^{\top}\bm{u}_{l}^{\star}$ by multiplying the plug-in
estimator by a correction factor, which can be computed in a data-driven
manner without the need of sample splitting. The proposed estimator
provably achieves enhanced estimation accuracy compared to the plug-in
estimator, and is shown to be minimax optimal (up to some logarithmic
factor) for a broad class of scenarios. 
\end{enumerate}

\paragraph{Organization.}

The rest of this paper is organized as follows. In Section~\ref{sec:Problem-formulation},
we formulate the problem precisely and introduce basic definitions.
Section~\ref{sec:Main-results} presents our main theoretical findings,
whereas Section~\ref{sec:Related-work} provides a non-exhaustive
overview of prior works. The analysis strategy of our main theorems
is outlined in Section~\ref{sec:Analysis}. The detailed proofs and
auxiliary lemmas are postponed to the appendix. We conclude this paper
with a discussion of future directions in Section~\ref{sec:Discussion}. 

\paragraph{Notation.}

For any vector $\bm{v}$, we denote by $\|\bm{v}\|_{2}$ and $\|\bm{v}\|_{\infty}$
its $\ell_{2}$ norm and $\ell_{\infty}$ norm, respectively; for
any vectors $\bm{v}$ and $\bm{u}$, we use $\langle\bm{v},\,\bm{u}\rangle$
to represent their inner product. For any matrix $\bm{M}$, we let
$\left\Vert \bm{M}\right\Vert $ and $\left\Vert \bm{M}\right\Vert _{\mathrm{F}}$
denote the spectral norm and the Frobenius norm of $\bm{M}$, respectively.
For any matrix $\bm{U}$ whose columns are orthonormal, we use $\bm{U}^{\perp}$
to denote a matrix whose columns form an orthonormal basis of the
orthogonal complement of the column space of $\bm{U}$, and let $\mathcal{P}_{\bm{U}}(\bm{M})=\bm{U}\bm{U}^{\top}\bm{M}$
be the Euclidean projection of a matrix $\bm{M}$ onto the column
space of $\bm{U}$. For any two random matrices $\bm{Z}$ and $\bm{X}$,
the notation $\bm{Z}\overset{\mathrm{d}}{=}\bm{X}$ means $\bm{Z}$
and $\bm{X}$ are identical in distribution. For notational simplicity,
we write $[n]$ for the set $\{1,\cdots,n\}$. For any $a,b\in\mathbb{R}$,
we introduce the notation $a\land b=\min\{a,b\}$, $a\lor b=\max\{a,b\}$,
and $\min|a\pm b|=\min\{|a-b|,|a+b|\}$. We denote by $\mathcal{B}_{r}(\bm{z})\coloneqq\{\bm{x}\,\mid\,\|\bm{x}-\bm{z}\|_{2}\leq r\}$
the ball of radius $r$ centered at $\bm{z}$. Throughout the paper,
we denote by $f(n)\lesssim g(n)$ or $f(n)=O(g(n))$ the condition
$|f(n)|\leq Cg(n)$ for some universal constant $C>0$ when $n$ is
sufficiently large; we use $f(n)\gtrsim g(n)$ or $f(n)=\Omega(g(n))$
to indicate that $f(n)\geq C|g(n)|$ for some universal constant $C>0$
when $n$ is sufficiently large; and we also use $f(n)\asymp g(n)$
or $f(n)=\Theta(g(n))$ to indicate that $f(n)\lesssim g(n)$ and
$f(n)\gtrsim g(n)$ hold simultaneously. In addition, the standard notation $\widetilde{O}(g(n))$
(resp.~$\widetilde{\Omega}(g(n))$) is similar to $O(g(n))$ (resp.~$\Omega(g(n))$)
except that it hides the logarithmic dependency. The notation $f(n)=o(g(n))$
means that $\lim_{n\rightarrow\infty}f(n)/g(n)=0$, and $f(n)\gg g(n)$
(resp.~$f(n)\ll g(n)$) means that there exists some large
(resp.~small) constant $c_{1}>0$ (resp.~$c_{2}>0$) such that $f(n)\geq c_{1}g(n)$
(resp.~$f(n)\leq c_{2}g(n)$). Finally, for any $1\leq l\leq r$,
we set the expression $\sum_{k\neq l,1\leq k\leq r}g(k)$ to be zero
for any $g(\cdot)$ if $r=1$ (that is, the case where no $k$ satisfies
the requirement in the summation).

\section{Problem formulation}

\label{sec:Problem-formulation}

\subsection{Matrix denoising \label{subsec:Matrix-denoising-model}}

Suppose that we are interested in a symmetric matrix $\bm{M}^{\star}=[M_{i,j}^{\star}]_{1\leq i,j\leq n}\in\mathbb{R}^{n\times n}$
with eigen-decomposition 
\begin{equation}
\bm{M}^{\star}=\sum_{i=1}^{r}\lambda_{i}^{\star}\bm{u}_{i}^{\star}\bm{u}_{i}^{\star\top}=:\bm{U}^{\star}\bm{\Lambda}^{\star}\bm{U}^{\star\top},\label{eq:defn-true-matrix}
\end{equation}
where the $\bm{u}_i^{\star}$'s are orthonormal. 
Here, $\{\lambda_{i}^{\star}\}$ denotes the set of non-zero eigenvalues
of $\bm{M}^{\star}$, and $\bm{u}_{i}^{\star}$ indicates the (normalized)
eigenvector associated with $\lambda_{i}^{\star}$. It is assumed
throughout that
\begin{equation}
\lambda_{\min}^{\star}=|\lambda_{r}^{\star}|\leq\cdots\leq|\lambda_{1}^{\star}|=\lambda_{\max}^{\star},\label{eq:defn-lambda-min-max-kappa}
\end{equation}
and the condition number of $\bm{M}^{\star}$ is defined as
\begin{equation}
	\kappa := \frac{\lambda_{\max}^{\star}}{\lambda_{\min}^{\star}}.
\end{equation}
In addition, for any $1\leq l\leq r$, we introduce an eigen-gap (or eigenvalue separation)
metric that quantifies the distance between the eigenvalue $\lambda_{l}^{\star}$
and the remaining spectrum:
\begin{equation}
\Delta_{l}^{\star}\coloneqq\begin{cases}
\min_{k:\,k\neq l,\,1\leq k\leq r}\big|\lambda_{l}^{\star}-\lambda_{k}^{\star}\big|, & \text{if}\quad r>1,\\
\lambda_{\max}^{\star}, & \text{if}\quad r=1,
\end{cases}\label{eq:eigen-separation-Delta-l-denoising}
\end{equation}
which plays a crucial role in our perturbation theory. 

What we have observed is a randomly corrupted data matrix $\bm{M}=[M_{i,j}]_{1\leq i,j\leq n}$
as follows
\begin{equation}
\bm{M}=\bm{M}^{\star}+\bm{H},\label{eq:observed-matrix}
\end{equation}
where $\bm{H}=[H_{i,j}]_{1\leq i,j\leq n}$ represents a symmetric
noise matrix with independent random entries
\begin{equation}
H_{i,j}\overset{\mathrm{ind.}}{\sim}
\begin{cases}
\mathcal{N}(0,2\sigma^{2}), &i = j, \\
\mathcal{N}(0,\sigma^{2}), &i > j.
\end{cases}
\label{eq:Gaussian-noise-model}
\end{equation}
Throughout this paper, we denote by $\lambda_{l}$ the $l$-th largest
eigenvalue (in magnitude) of $\bm{M}$, and let $\bm{u}_{l}$ represent
the associated eigenvector of $\bm{M}$. Our goal is to estimate linear
functionals of an eigenvector $\bm{u}_{l}^{\star}$ --- that is,
$\bm{a}^{\top}\bm{u}_{l}^{\star}\,(1\leq l\leq r)$ for some fixed
vector $\bm{a}\in\mathbb{R}^{n}$ --- based on the observed noisy
data $\bm{M}$.

\subsection{Principal component analysis and covariance estimation\label{subsec:Principal-component-analysis-model}}

Turning to principal component analysis (PCA) or covariance estimation,
we concentrate on the following spiked covariance model. Imagine that
we have collected a sequence of $n$ i.i.d.~zero-mean Gaussian sample
vectors in $\mathbb{R}^{p}$ as follows
\[
\bm{s}_{i}\overset{\mathrm{ind.}}{\sim}\mathcal{N}\left(\bm{0},\bm{\Sigma}\right),\qquad1\leq i\leq n,
\]
where
\[
\bm{\Sigma}=\bm{\Sigma}^{\star}+\sigma^{2}\bm{I}_{p}\in\mathbb{R}^{p\times p}
\]
denotes the covariance matrix. Here and throughout, we assume that
the ``spiked component'' $\bm{\Sigma}^{\star}$ of $\bm{\Sigma}$
is an unknown rank-$r$ matrix with eigen-decomposition 
\begin{align*}
\bm{\Sigma}^{\star} & =\bm{U}^{\star}\bm{\Lambda}^{\star}\bm{U}^{\star\top}=\sum_{i=1}^{r}\lambda_{i}^{\star}\bm{u}_{i}^{\star}\bm{u}_{i}^{\star\top}\succeq\bm{0},
\end{align*}
where $\lambda_{i}^{\star}$ denotes the $i$-th largest eigenvalue
of $\bm{\Sigma}^{\star}$, with $\bm{u}_{i}^{\star}$ representing
the associated eigenvector. Akin to the matrix denoising case, we
assume
\[
0<\lambda_{\min}^{\star}=\lambda_{r}^{\star}\leq\cdots\leq\lambda_{1}^{\star}=\lambda_{\max}^{\star},
\]
and introduce the condition number $\kappa:=\lambda_{\max}^{\star}/\lambda_{\min}^{\star}$
and the eigen-separation metric
\begin{equation}
\Delta_{l}^{\star}\coloneqq\begin{cases}
\min_{k:\,k\neq l,\,1\leq k\leq r}\big|\lambda_{l}^{\star}-\lambda_{k}^{\star}\big|, & \text{if}\quad r>1,\\
\lambda_{\max}^{\star}, & \text{if}\quad r=1.
\end{cases}\label{eq:defn-kappa-Delta-PCA}
\end{equation}
Given a fixed vector $\bm{a}\in\mathbb{R}^{p}$, our aim is to develop
a reliable estimate of the linear functional $\bm{a}^{\top}\bm{u}_{l}^{\star}$
of an eigenvector $\bm{u}_{l}^{\star}$ ($1\leq l\leq r$), on the
basis of the sample vectors $\{\bm{s}_{i}\}_{1\leq i\leq n}$ (or
the sample covariance matrix $\frac{1}{n}\sum_{i=1}^{n}\bm{s}_{i}\bm{s}_{i}^{\top}$).

\section{Main results}

\label{sec:Main-results}

With the above description of the problem settings in place, we are
ready to present our findings concerning eigenvector perturbation.
Given that we cannot distinguish $\bm{u}_{l}^{\star}$ and $-\bm{u}_{l}^{\star}$
based on the observed matrix, the error of an estimator $u_{\bm{a}}$
for estimating $\bm{a}^{\top}\bm{u}_{l}^{\star}$ shall be measured
via the following metric that accounts for such a global ambiguity
issue:
\begin{equation}
\mathsf{dist}\left(u_{\bm{a}},\bm{a}^{\top}\bm{u}_{l}^{\star}\right)\coloneqq\min\left\{ \big|u_{\bm{a}}\pm\bm{a}^{\top}\bm{u}_{l}^{\star}\big|\right\} =\min\left\{ \big|u_{\bm{a}}-\bm{a}^{\top}\bm{u}_{l}^{\star}\big|,\big|u_{\bm{a}}+\bm{a}^{\top}\bm{u}_{l}^{\star}\big|\right\} .\label{eq:distance-metric}
\end{equation}

\subsection{Matrix denoising}

We begin with the matrix denoising problem introduced in Section~\ref{subsec:Matrix-denoising-model}.
Recalling that $\bm{u}_{l}$ is the eigenvector of $\bm{M}$ associated
with $\lambda_{l}$ ($1\leq l\leq n$), we investigate the following
two estimators when estimating the linear function $\bm{a}^{\top}\bm{u}_{l}^{\star}$. 

\begin{subequations}\label{eq:defn-estimators-denoising}
\begin{itemize}
\item A plug-in estimator: 
\begin{equation}
u_{\bm{a}}^{\mathsf{plugin}}\coloneqq\bm{a}^{\top}\bm{u}_{l};\label{eq:def-plug-in-matrix-denoising}
\end{equation}
\item A modified estimator that we propose (which we shall refer to as a
\emph{de-biased estimator} from now on):
\begin{align}
u_{\bm{a}}^{\mathsf{debiased}}\coloneqq\sqrt{1+b_{l}}\,\bm{a}^{\top}\bm{u}_{l}\qquad\text{with}\quad b_{l} & \coloneqq\sum_{i:\,r<i\le n}\frac{\sigma^{2}}{(\lambda_{l}-\lambda_{i})^{2}},\label{eq:def:bl-debiased-matrix-denoising}
\end{align}
where $b_{l}$ can be computed directly using the eigenvalues of $\bm{M}$
without the need of sample splitting.
As we shall see shortly, this
new estimator is put forward in order to remedy a non-negligible bias
issue underlying the naive plug-in estimator. 
\end{itemize}
\end{subequations}The following theorem quantifies the estimation
errors for both of these estimators. 

\begin{theorem}[Eigenvector perturbation]\label{thm:evector-pertur-sym-iid}Consider
any $1\leq l\leq r$, and suppose that 
\begin{equation}
\sigma\sqrt{n}\leq c_{0}\lambda_{\min}^{\star},\qquad r\leq c_{1}n/\log^{2}n\qquad\text{and}\qquad\Delta_{l}^{\star}>C_{0}\sigma\sqrt{r}\log n\label{eq:eigengap-condition-iid}
\end{equation}
for some sufficiently small (resp.~large) constants $c_{0},c_{1}>0$
(resp.~$C_{0}>0$). Let $\bm{a}\in\mathbb{R}^{n}$ be any fixed vector
with $\|\bm{a}\|_{2}=1$. With probability at least $1-O(n^{-10}),$
the estimators in (\ref{eq:defn-estimators-denoising}) satisfy\begin{subequations}\label{eq:eigenvec-perturbation-bound-iid}
\begin{align}
\mathsf{dist}\left(u_{\bm{a}}^{\mathsf{plugin}},\bm{a}^{\top}\bm{u}_{l}^{\star}\right) & \lesssim E_{\mathsf{md},l}+\frac{\sigma^{2}n}{\lambda_{l}^{\star2}}\left|\bm{a}^{\top}\bm{u}_{l}^{\star}\right|,\label{eq:eigenvec-perturbation-bound-iid-plugin}\\
\mathsf{dist}\left(u_{\bm{a}}^{\mathsf{debiased}},\bm{a}^{\top}\bm{u}_{l}^{\star}\right) & \lesssim E_{\mathsf{md},l},\label{eq:eigenvec-perturbation-bound-iid-debias}
\end{align}
\end{subequations}where $E_{\mathsf{md},l}$ is defined as
\begin{align}
E_{\mathsf{md},l} & \coloneqq\frac{\sigma^{2}r\log n}{(\Delta_{l}^{\star})^{2}}\left|\bm{a}^{\top}\bm{u}_{l}^{\star}\right|+\sigma\sqrt{r\log\Big(\frac{n\kappa\lambda_{\max}}{\Delta_{l}^{\star}}\Big)}\sum_{k:\,k\neq l,1\leq k\leq r}\frac{\left|\bm{a}^{\top}\bm{u}_{k}^{\star}\right|}{\left|\lambda_{l}^{\star}-\lambda_{k}^{\star}\right|}+\frac{\sigma\sqrt{\log\big(\frac{n\kappa\lambda_{\max}}{\Delta_{l}^{\star}}\big)}}{\left|\lambda_{l}^{\star}\right|}.\label{eq:def:E-md}
\end{align}
\end{theorem}


\begin{remark}
While the rank $r$ of the true matrix $\bm{M}^\star$ might be unknown \emph{a prior} in practice, it can often be estimated accurately in a data-driven manner. For instance, under the model assumed herein, 
	one might simply choose $r$ by identifying the smallest (in magnitude) eigenvalue $\lambda_l$ that is larger than $\lambda_{l+1}$ by an order of $\sigma\sqrt{n}$; see also 
	\citet{han2019universal} for a different approach. 
\end{remark}
\begin{remark}
While the quantity $b_l$ is provided in a data-driven manner (cf.~\eqref{eq:def:bl-debiased-matrix-denoising}),   
we find it helpful to also make note of another expression derived from its asymptotic limit. 
Specifically, the random matrix theory tells us that the eigenvalues $\{\lambda_i\}_{r<i\le n}$ of $\bm{M}$ obey the celebrated semi-circular law asymptotically (see, e.g., \citet{knowles2013isotropic}), and therefore the de-biased term $b_l$ satisfies (as $n$ grows):
\begin{align}
b_{l}=\sum_{i:\,r<i\le n}\frac{\sigma^{2}}{(\lambda_{l}-\lambda_{i})^{2}} & \,\approx\,\int_{-2}^{2}\frac{\sqrt{4-\lambda^{2}}}{2\pi\Big(\frac{\lambda_{l}}{\sigma\sqrt{n}}-\lambda\Big)^{2}}\mathrm{d}\lambda.
\end{align}
\end{remark}

\paragraph{Implications.}

Theorem \ref{thm:evector-pertur-sym-iid} develops statistical performance
guarantees for the aforementioned two estimators when estimating the
linear form $\bm{a}^{\top}\bm{u}_{l}^{\star}$ for a prescribed vector
$\bm{a}\in\mathbb{R}^{n}$. We now single out several main implications
of our results. 
\begin{itemize}
\item \emph{Estimation guarantees in the face of a small eigen-gap}. In
view of (\ref{eq:eigengap-condition-iid}), the eigen-gap $\Delta_{l}^{\star}$
is allowed to be substantially smaller than the spectral norm $\|\bm{H}\|$
of the perturbation matrix. This stands in stark contrast to, and significantly
improves upon, the celebrated Davis-Kahan $\sin\bm{\Theta}$ theorem
that requires $\Delta_{l}^{\star}\gtrsim\|\bm{H}\|$ \citep{davis1970rotation,chen2020spectral}.
To be more precise, recalling from standard random matrix theory \citep{Tao2012RMT}
that $\|\bm{H}\|\asymp\sigma\sqrt{n}$ with high probability, one
can compare our result with classical matrix perturbation theory as
follows\begin{subequations}\label{eq:eigen-gap-comparison-denoising}
\begin{align*}
\text{our eigen-gap requirement:}\quad & \Delta_{l}^{\star}=\widetilde{\Omega}\left(\sigma\sqrt{r}\right);\\
\text{eigen-gap requirement in classical theory:}\quad & \Delta_{l}^{\star}=\widetilde{\Omega}\left(\sigma\sqrt{n}\right).
\end{align*}
\end{subequations}
As a comparison, the prior work \citet{bao2021singular} studied the distributions of the singular vectors under the matrix denoising setting with $\sigma \asymp n^{-1/2}$, provided that the eigen-gap exceeds $\Omega(1)$;
our theory improves their eigen-gap condition by a factor on the order of $\sqrt{n/r}$.

\item \emph{Near minimaxity}. In order to assess the effectiveness of our
proposed estimator, it is helpful to compare the statistical guarantees
in Theorem \ref{thm:evector-pertur-sym-iid} with minimax lower bounds.
Consider, for simplicity, the scenario where $r=O(1)$ and $|\bm{a}^{\top}\bm{u}_{l}^{\star}|\le(1-\epsilon)\|\bm{a}\|_{2}$
for any small non-zero constant $\epsilon>0$ (so that $\bm{a}$ is
not perfectly aligned with $\bm{u}_{l}^{\star}$), and 
an instance-dependent minimax lower bound has been established in \citet[Theorem 3]{cheng2020tackling} for this scenario.  
Specifically, if we define the following two sets
\begin{align*}
\mathcal{M}_{0}(\bm{M}^{\star}) & :=\Big\{\bm{A}\,\mid\,\mathsf{rank}(\bm{A})=r,\,\lambda_{i}(\bm{A})=\lambda_{i}^{\star}\,(1\leq i\leq r),\,\|\bm{A}-\bm{M}^{\star}\|_{\mathrm{F}}\leq\frac{\sigma}{2}\Big\},\\
\mathcal{M}_{1}(\bm{M}^{\star}) & :=\Big\{\bm{A}\,\mid\,\mathsf{rank}(\bm{A})=r,\,\lambda_{i}(\bm{A})=\lambda_{i}^{\star}\,(1\leq i\leq r),\,\|\bm{u}_{l}(\bm{A})-\bm{u}_{l}^{\star}\|_{\mathrm{2}}\leq\frac{\sigma}{4\,|\lambda_{l}^{\star}|}\Big\},
\end{align*}
then we necessarily have
\begin{subequations}
\label{eq:minimax-lower-bound-denoising}
\begin{align}
\label{eq:minimax-lower-bound-denoising-E}	
\inf_{u_{\bm{a},l}}\sup_{\bm{A}\in\mathcal{M}_{0}(\bm{M}^{\star}) \cup \mathcal{M}_{1}(\bm{M}^{\star})}\mathbb{E}\Big[\mathsf{dist}\left(u_{\bm{a},l},\,\bm{a}^{\top}\bm{u}_{l}(\bm{A})\right)\Big] \gtrsim\frac{\sigma^{2}}{(\Delta_{l}^{\star})^{2}}\left|\bm{a}^{\top}\bm{u}_{l}^{\star}\right|+\sigma\max_{k:\,k\neq l}\frac{\left|\bm{a}^{\top}\bm{u}_{k}^{\star}\right|}{\left|\lambda_{l}^{\star}-\lambda_{k}^{\star}\right|} + \frac{\sigma}{|\lambda_{l}^{\star}|},
\end{align}
where the infimum is over all estimators $u_{\bm{a},l}$ 
based on the observed matrix $\bm{M}=\bm{A}+\bm{H}$, and  $\bm{u}_{l}(\bm{A})$
denotes the $l$-th eigenvector of the matrix $\bm{A}$. 
In addition, the analysis for \citet[Theorem 3]{cheng2020tackling} directly implies that 
\begin{align}\label{eq:minimax-lower-bound-denoising-prob}
\inf_{u_{\bm{a},l}}\sup_{\bm{A}\in\mathcal{M}_{0}(\bm{M}^{\star}) \cup \mathcal{M}_{1}(\bm{M}^{\star})}\mathbb{P}\bigg\{ \mathsf{dist}\left(u_{\bm{a},l},\,\bm{a}^{\top}\bm{u}_{l}(\bm{A})\right) \gtrsim\frac{\sigma^{2}}{(\Delta_{l}^{\star})^{2}}\left|\bm{a}^{\top}\bm{u}_{l}^{\star}\right|+\sigma\max_{k:\,k\neq l}\frac{\left|\bm{a}^{\top}\bm{u}_{k}^{\star}\right|}{\left|\lambda_{l}^{\star}-\lambda_{k}^{\star}\right|} + \frac{\sigma}{|\lambda_{l}^{\star}|}\bigg \} \ge \frac{1}{5}.
\end{align}
\end{subequations}
In comparison, our statistical guarantee (\ref{eq:eigenvec-perturbation-bound-iid-debias})
for the proposed de-biased estimator obeys
\[
\mathsf{dist}\left(u_{\bm{a}}^{\mathsf{debiased}},\bm{a}^{\top}\bm{u}_{l}^{\star}\right)\leq\widetilde{O}\Bigg(\frac{\sigma^{2}}{(\Delta_{l}^{\star})^{2}}\left|\bm{a}^{\top}\bm{u}_{l}^{\star}\right|+\sigma\sum_{k:\,k\neq l}\frac{\left|\bm{a}^{\top}\bm{u}_{k}^{\star}\right|}{\left|\lambda_{l}^{\star}-\lambda_{k}^{\star}\right|}+\frac{\sigma}{\left|\lambda_{l}^{\star}\right|}\Bigg)
\]
with high probability in this scenario, thereby matching the minimax lower bound (\ref{eq:minimax-lower-bound-denoising}) (modulo some logarithmic factor). 
This confirms the near optimality of our de-biased estimator when $r \asymp 1$. 

\end{itemize}

\paragraph{Sub-optimality of the vanilla plug-in estimator.}

Furthermore, Theorem~\ref{thm:evector-pertur-sym-iid} suggests that the statistical error (\ref{eq:eigenvec-perturbation-bound-iid-plugin})
of the vanilla plug-in estimator $\bm{a}^{\top}\bm{u}_{l}$ might contain an additional ``bias'' term 
\begin{equation}
E_{\mathsf{md},l}^{\mathsf{bias}}\coloneqq\frac{\sigma^{2}n}{\lambda_{l}^{\star2}}\left|\bm{a}^{\top}\bm{u}_{l}^{\star}\right|\label{eq:defn-E-extra-denoising}
\end{equation}
when compared to that of the de-biased estimator (cf.~(\ref{eq:eigenvec-perturbation-bound-iid-debias})). It is natural to wonder if the theoretical guarantee of the plug-in estimator in \eqref{eq:eigenvec-perturbation-bound-iid-plugin} is tight or not. To answer the question, we develop the following lower bound on the estimation error of the plug-in estimator $u_{\bm{a}}^{\mathsf{plugin}}$; the proof is deferred to Appendix~\ref{sec:Proof-for-lower-bound-plugin}.

\begin{theorem}
\label{thm:lower-bound-plugin}
Instate the assumptions of Theorem~\ref{thm:evector-pertur-sym-iid}. Let $\bm{a}\in\mathbb{R}^{n}$ be any fixed vector
with $\|\bm{a}\|_{2}=1$. With probability at least $1/3$,
the plug-in estimator in (\ref{eq:defn-estimators-denoising}) satisfies
\begin{align}
\mathsf{dist}\left(u_{\bm{a}}^{\mathsf{plugin}},\bm{a}^{\top}\bm{u}_{l}^{\star}\right) \gtrsim \frac{\sigma^{2}n}{\lambda_{l}^{\star2}}\left|\bm{a}^{\top}\bm{u}_{l}^{\star}\right|.
\end{align}
\end{theorem}

 In short, Theorem~\ref{thm:lower-bound-plugin} demonstrates that it is impossible for the plug-in estimator to get rid of this ``bias'' term \eqref{eq:defn-E-extra-denoising}. 
The influence of this extra term becomes increasingly large and non-negligible
as the correlation of $\bm{a}$ and $\bm{u}_{l}^{\star}$ increases.
To demonstrate the possibly severe impact incurred by this additional
term, let us examine a simple case as follows.
\begin{itemize}
\item \emph{Example}. Suppose that $r=2$, $\lambda_{1}^{\star}=2\lambda_{2}^{\star}$
(so that $\lambda_{1}^{\star}-\lambda_{2}^{\star}\asymp\lambda_{1}^{\star}$),
$|\bm{a}^{\top}\bm{u}_{l}^{\star}|\asymp1$ and $\sigma\sqrt{n}\asymp|\lambda_{1}^{\star}|$.
As can be straightforwardly verified, the main term (\ref{eq:def:E-md})
and the addition term (\ref{eq:defn-E-extra-denoising}) in this example
satisfy
\begin{align*}
E_{\mathsf{md},1} & =\widetilde{O}\bigg(\frac{\sigma^{2}}{\lambda_{1}^{\star2}}\left|\bm{a}^{\top}\bm{u}_{1}^{\star}\right|+\frac{\sigma}{|\lambda_{1}^{\star}|}\left|\bm{a}^{\top}\bm{u}_{2}^{\star}\right|+\frac{\sigma}{|\lambda_{1}^{\star}|}\bigg)=\widetilde{O}\bigg(\frac{\sigma^{2}}{\lambda_{1}^{\star2}}+\frac{\sigma}{|\lambda_{1}^{\star}|}\bigg)=\widetilde{O}\bigg(\frac{1}{\sqrt{n}}\bigg);\\
E_{\mathsf{md},1}^{\mathsf{bias}} & \asymp\frac{\sigma^{2}n}{\lambda_{1}^{\star2}}\left|\bm{a}^{\top}\bm{u}_{1}^{\star}\right|\asymp1.
\end{align*}
In other words, the additional bias term $E_{\mathsf{md},1}^{\mathsf{bias}}$
could be a factor of $\widetilde{O}(\sqrt{n})$ times larger than the main term $E_{\mathsf{md},1}$ in this case, and cannot be neglected.  
\end{itemize}
The above discussion reveals the necessity of proper bias correction
in order to mitigate the undesired effect of the bias term $E_{\mathsf{md},l}^{\mathsf{bias}}$.
Aimed at addressing this issue, our de-biased estimator $u_{\bm{a}}^{\mathsf{debiased}}$
compensates for the bias term $E_{\mathsf{md},l}^{\mathsf{bias}}$
by properly rescaling the plug-in estimator by a data-driven correction
factor $\sqrt{1+b_{l}}$. 
Note that when the signal-to-noise ratio is sufficiently large such that $|\lambda_l^{\star}| \gtrsim \sigma n$, then $E_{\mathsf{md},l}$ becomes the dominant term in the error bound; in such a case, there is no need for bias correction.

\paragraph{Comparisons with prior works. }

While estimation of linear forms of eigenvectors remains largely under-explored
in the literature, a small number of prior works have studied this
problem or its variants. Among them, perhaps the one that is the closest
to the current paper is \citet{koltchinskii2016perturbation}, which
considered estimating linear forms of singular vectors under i.i.d.~Gaussian
noise. In what follows, we briefly compare our result with \citet{koltchinskii2016perturbation},
focusing on the setting where the ground-truth matrix is symmetric
(so that the eigenvectors and the singular vectors become identical
up to global signs). 
\begin{itemize}
\item To begin with, the theory in \citet[Theorem~1.3]{koltchinskii2016perturbation}
operates under the assumption 
\[
\Delta_{l}^{\star}=\Omega\big(\mathbb{E}[\|\bm{H}\|]\big)=\Omega(\sigma\sqrt{n}),
\]
which is $\widetilde{O}\big(\sqrt{n/r}\big)$ times more stringent
than the eigen-gap condition imposed in our theory (see (\ref{eq:eigen-gap-comparison-denoising})). 
\item The estimation bias of the plug-in estimator was already pointed out
in \citet{koltchinskii2016perturbation}. However, the approach proposed
in \citet{koltchinskii2016perturbation} required additional independent
copies of $\bm{M}$ in order to estimate --- and hence correct ---
the bias effect (see \citet[Section 1]{koltchinskii2016perturbation}).
By contrast, our de-biased estimator does not require an additional
set of data samples and allows one to use all available information
fully. 
\item Next, we compare our theoretical guarantee with the one developed
for the de-biased estimator $u_{\bm{a}}^{\mathsf{debiased,KD}}$ proposed
in \citet{koltchinskii2016perturbation}. When $r\asymp1$, \citet[Theorem~1.3]{koltchinskii2016perturbation}
asserts that
\[
\mathsf{dist}\left(u_{\bm{a}}^{\mathsf{debiased,KD}},\bm{a}^{\top}\bm{u}_{l}^{\star}\right)\leq\widetilde{O}\bigg(\frac{\sigma}{\Delta_{l}^{\star}}\bigg)\eqqcolon E_{\mathsf{md},l}^{\mathsf{KD}},
\]
provided that $\Delta_{l}^{\star}\gtrsim\sigma\sqrt{n}$. This result,
however, might fall short of attaining minimax optimality. More specifically,
comparing our error bound $E_{\mathsf{md},l}$ (cf.~(\ref{eq:def:E-md}))
with $E_{\mathsf{md},l}^{\mathsf{KD}}$ makes clear that the theoretical
gain is on the order of
\[
\frac{E_{\mathsf{md},l}^{\mathsf{KD}}}{E_{\mathsf{md},l}}=\widetilde{O}\left(\frac{\Delta_{l}^{\star}}{\sigma\left|\bm{a}^{\top}\bm{u}_{l}^{\star}\right|}\,\land\,\frac{1}{\sum_{k:k\neq l}\left|\bm{a}^{\top}\bm{u}_{k}^{\star}\right|}\,\land\,\frac{\left|\lambda_{l}^{\star}\right|}{\Delta_{l}^{\star}}\right).
\]
For concreteness, consider the case with $r\asymp1$, $|\bm{a}^{\top}\bm{u}_{k}^{\star}|\asymp1/\sqrt{n}$
for all $k\neq l$, $\Delta_{l}^{\star}\asymp\left|\lambda_{l}^{\star}\right|/\sqrt{n}$,
and $\Delta_{l}^{\star}\asymp\sigma\sqrt{n}$, thus leading to the
gain
\[
\frac{E_{\mathsf{md},l}^{\mathsf{KD}}}{E_{\mathsf{md},l}}=\widetilde{O}\left(\sqrt{n}\right).
\]
In other words, our results might lead to considerable theoretical
improvement over \citet{koltchinskii2016perturbation} in the presence
of a small eigen-gap. 
\begin{remark}
	Note that the case $|\bm{a}^{\top}\bm{u}_{k}^{\star}|\asymp1/\sqrt{n}$ is of particular interest if one studies entrywise statistical performance; namely, when $\bm{a}$ is taken to be the standard basis (i.e.~$\bm{a} = \bm{e}_i$ for some $i \in [n]$) and when the energy of $\bm{u}_{k}^{\star}$ is more or less spread out across all entries (i.e.~$\|\bm{u}_{k}^{\star}\|_\infty \asymp \| \bm{u}_{k}^{\star} \|_2 / \sqrt{n}$).
\end{remark}
\end{itemize}

\subsection{Principal component analysis}

Next, we turn attention to the problem of principal component analysis
as formulated in Section~\ref{subsec:Principal-component-analysis-model}.
Denote by 
\[
\bm{S}=[\bm{s}_{1},\cdots,\bm{s}_{n}]\in\mathbb{R}^{p\times n}
\]
 the data matrix whose columns consist of i.i.d.~samples $\bm{s}_{i}\overset{\mathrm{i.i.d.}}{\sim}\mathcal{N}(\bm{0},\bm{\Sigma})$,
and let $\lambda_{l}$ represent the $l$-th largest eigenvalue of
$\frac{1}{n}\bm{S}\bm{S}^{\top}$ with associated eigenvector $\bm{u}_{l}$.
Our focus is the following two estimators aimed at estimating the
linear form $\bm{a}^{\top}\bm{u}_{l}^{\star}$ $(1\leq l\leq r)$. 
\begin{itemize}
\item A plug-in estimator: \begin{subequations}\label{subsec:estimators-PCA}
\begin{equation}
u_{\bm{a}}^{\mathsf{plugin}}\coloneqq\bm{a}^{\top}\bm{u}_{l};\label{eq:def-plug-in-PCA}
\end{equation}
\item A ``de-biased'' estimator:
\begin{equation}
u_{\bm{a}}^{\mathsf{debiased}}\coloneqq\sqrt{1+c_{l}}\,\bm{a}^{\top}\bm{u}_{l}.\label{eq:def-debiased-PCA}
\end{equation}
\end{subequations}Here, $c_{l}$ is a quantity that can be directly
computed using the spectrum of $\frac{1}{n}\bm{S}\bm{S}^{\top}$ as
follows:
\begin{equation}
c_{l}\coloneqq\begin{dcases}
\frac{\lambda_{l}}{n+\sum_{i:\,r<i\le n}\frac{\lambda_{i}}{\lambda_{l}-\lambda_{i}}}\sum\limits _{i:\,r<i\le n}\frac{\lambda_{i}}{(\lambda_{l}-\lambda_{i})^{2}}, & \text{if }n\geq p,\\
\frac{\frac{\sigma^{2}p}{n}}{\lambda_{l}-\frac{\sigma^{2}p}{n}}+\frac{\lambda_{l}}{\lambda_{l}-\frac{\sigma^{2}p}{n}}\frac{\lambda_{l}}{n+\sum_{i:\,r<i\le n}\frac{\lambda_{i}}{\lambda_{l}-\lambda_{i}}}\sum\limits _{i:\,r<i\le n}\frac{\lambda_{i}-\frac{\sigma^{2}p}{n}}{(\lambda_{l}-\lambda_{i})^{2}},\qquad & \text{if }n<p,
\end{dcases}\label{eq:def:bl-pca}
\end{equation}
without any need of using sample splitting. 
\end{itemize}
Akin to the matrix denoising counterpart, the plug-in estimator (\ref{eq:def-plug-in-PCA})
often incurs some non-negligible estimation bias, which motivates
the design of the adjusted estimator (\ref{eq:def-debiased-PCA})
to compensate for the bias. 

We are now ready to present our statistical guarantees for the two
estimators introduced in (\ref{subsec:estimators-PCA}). 

\begin{theorem}\label{thm:evector-pertur-sym-iid-pca}Consider any
$1\leq l\leq r$, and assume that \begin{subequations}
\begin{equation}
\lambda_{\max}^{\star}\sqrt{\frac{r}{n}}+\sqrt{\lambda_{\max}^{\star}\sigma^{2}\frac{p}{n}}+\sigma^{2}\bigg(\frac{p}{n}+\sqrt{\frac{p}{n}}\bigg)\le C_{0}\frac{\lambda_{\min}^{\star}}{\log^{2}n}\label{eq:noise-condition-iid-pca}
\end{equation}
\begin{equation}
\text{and}\qquad\Delta_{l}^{\star}>C_{1}(\lambda_{\max}^{\star}+\sigma^{2})\sqrt{\frac{r}{n}}\log n\label{eq:eigengap-condition-iid-pca}
\end{equation}
\end{subequations} hold for some sufficiently small (resp.~large)
constant $C_{0}>0$ (resp.~$C_{1}>0$). Consider any fixed vector
$\bm{a}\in\mathbb{R}^{p}$ with $\|\bm{a}\|_{2}=1$. Then with probability
at least $1-O(n^{-10})$, the estimators in (\ref{subsec:estimators-PCA})
satisfy \begin{subequations}\label{eq:eigenvec-perturbation-bound-pca}
\begin{align}
\mathsf{dist}\left(u_{\bm{a}}^{\mathsf{plugin}},\bm{a}^{\top}\bm{u}_{l}^{\star}\right) & \lesssim E_{\mathsf{PCA},l}+\frac{(\lambda_{l}^{\star}+\sigma^{2})\sigma^{2}p}{\lambda_{l}^{\star2}n}\left|\bm{a}^{\top}\bm{u}_{l}^{\star}\right|,\label{eq:eigenvec-perturbation-bound-pca-plugin}\\
\mathsf{dist}\left(u_{\bm{a}}^{\mathsf{debiased}},\bm{a}^{\top}\bm{u}_{l}^{\star}\right) & \lesssim E_{\mathsf{PCA},l},\label{eq:eigenvec-perturbation-bound-pca-debias}
\end{align}
\end{subequations} where the quantity $E_{\mathsf{PCA},l}$ is defined
as
\begin{align}
E_{\mathsf{PCA},l} & \coloneqq\frac{(\lambda_{\max}^{\star}+\sigma^{2})(\lambda_{l}^{\star}+\sigma^{2})\,r\log n}{(\Delta_{l}^{\star})^{2}n}\left|\bm{a}^{\top}\bm{u}_{l}^{\star}\right|+\sqrt{\frac{(\lambda_{\max}^{\star}+\sigma^{2})\,\sigma^{2}\kappa^2r}{\lambda_{l}^{\star2}n}} \log^2 n\nonumber \\
 & \qquad+\sum_{k:k\ne l}\frac{\left|\bm{a}^{\top}\bm{u}_{k}^{\star}\right|}{\left|\lambda_{l}^{\star}-\lambda_{k}^{\star}\right|\sqrt{n}}\sqrt{(\lambda_{l}^{\star}+\sigma^{2})(\lambda_{\max}^{\star}+\sigma^{2})(\kappa^{2}+r)\log\bigg(\frac{n\kappa\lambda_{\max}}{\Delta_{l}^{\star}}\bigg)} .\label{eq:def:E-pca}
\end{align}
\end{theorem}
\begin{remark}
Akin to the matrix denoising case, 
while the expression of the de-biasing term $c_l$ (cf.~\eqref{eq:def:bl-pca}) is fully data-driven and preferable in practice, 
we remark that the Marchenko-Pastur law allows us to approximate  the de-biasing term $c_l$ as follows (as $n$ grows)
\begin{equation}
c_{l}~\approx~\begin{dcases}
\frac{\lambda_{l}}{1+\int\frac{\lambda}{\lambda_{l}-\lambda}\mu(\mathrm{d}\lambda)}\int\frac{\lambda}{(\lambda_{l}-\lambda)^{2}}\mu(\mathrm{d}\lambda), & \text{if }n\geq p,\\
\frac{\frac{\sigma^{2}p}{n}}{\lambda_{l}-\frac{\sigma^{2}p}{n}}+\frac{\lambda_{l}}{\lambda_{l}-\frac{\sigma^{2}p}{n}}\frac{\lambda_{l}}{1+\int \frac{\lambda}{\lambda_{l}-\lambda}\mu(\mathrm{d}\lambda)}\int\frac{\lambda-\frac{\sigma^{2}p}{n}}{(\lambda_{l}-\lambda)^{2}}\mu(\mathrm{d}\lambda),\qquad & \text{if }n<p,
\end{dcases}
\end{equation}
where 
\begin{align}
\mu(\mathrm{d}\lambda) = \frac{n\sqrt{(\lambda_+-\lambda)(\lambda-\lambda_-)}}{2\pi\sigma^2p\lambda}\mathds{1}\{\lambda_-\le \lambda\le \lambda_+ \}\mathrm{d} \lambda\quad \text{with} \quad \lambda_{\pm} = \sigma^2(1\pm\sqrt{p/n})^2.
\end{align}
\end{remark}

\paragraph{Implications.}

In short, Theorem~\ref{thm:evector-pertur-sym-iid-pca} characterizes
the statistical accuracy of both the plug-in estimator and the modified
de-biased estimator, the latter of which enjoys improved statistical
guarantees. In the sequel, we single out a few implications of this
result.
\begin{itemize}
\item \emph{Estimation guarantees.} Let us first assess the statistical
error bound of the de-biased estimator (namely, $E_{\mathsf{PCA},l}$
in (\ref{eq:def:E-pca})). For simplicity of presentation, we shall
focus on the case with $r,\kappa\asymp1$, where the error term $E_{\mathsf{PCA},l}$
admits the following simpler expression
\begin{align}
E_{\mathsf{PCA},l} & =\widetilde{\Theta}\bigg(\frac{(\lambda_{l}^{\star}+\sigma^{2})^{2}}{(\Delta_{l}^{\star})^{2}n}\big|\bm{a}^{\top}\bm{u}_{l}^{\star}\big|+(\lambda_{l}^{\star}+\sigma^{2})\max_{k:k\ne l}\frac{\left|\bm{a}^{\top}\bm{u}_{k}^{\star}\right|}{\left|\lambda_{l}^{\star}-\lambda_{k}^{\star}\right|\sqrt{n}}+\frac{\sigma}{\lambda_{l}^{\star}\sqrt{n}}\sqrt{\lambda_{l}^{\star}+\sigma^{2}}\bigg).\label{eq:defn-E-pca-simpler}
\end{align}
In particular, the first term on the right-hand side of (\ref{eq:defn-E-pca-simpler})
quantifies the role of the ground truth $\bm{a}^{\top}\bm{u}_{l}^{\star}$
on the estimation error, which scales inverse quadratically in the
eigen-gap $\Delta_{l}^{\star}$; the second term on the right-hand
side of (\ref{eq:defn-E-pca-simpler}) can be understood as the additional
interference resulting from the linear form of other eigenvectors
(namely, $\bm{a}^{\top}\bm{u}_{k}^{\star}$ for $k\neq l$), which
is inversely proportional to the corresponding eigen-gap $\left|\lambda_{l}^{\star}-\lambda_{k}^{\star}\right|$. 
\item \emph{Relaxed eigen-gap condition}. To simplify discussions, let us
again focus on the case with $r,\kappa\asymp1$ and omit logarithmic
factors. Classical matrix perturbation theory (e.g., the Davis-Kahan
$\sin\bm{\Theta}$ theorem \citep{davis1970rotation}) requires the
eigen-gap to exceed the size of perturbation, namely, 
\[
\Delta_{l}^{\star}\gtrsim\Big\|\frac{1}{n}\bm{S}\bm{S}^{\top}-\bm{\Sigma}\Big\|.
\]
As it turns out, the eigen-gap requirement above leads to the following
condition (by invoking the high-probability bound to be presented
shortly in Lemma~\ref{lemma:spectral-M-Sigma})
\[
\Delta_{l}^{\star}\gtrsim\frac{\lambda_{l}^{\star}}{\sqrt{n}}+\sqrt{\frac{\lambda_{l}^{\star}\sigma^{2}p}{n}}+\sigma^{2}\bigg(\sqrt{\frac{p}{n}}+\frac{p}{n}\bigg)=:\mathsf{gap}_{\mathsf{DK}}.
\]
In comparison, the eigen-gap condition (\ref{eq:eigengap-condition-iid-pca})
in Theorem \ref{thm:evector-pertur-sym-iid-pca} reads
\[
\Delta_{l}^{\star}\gtrsim\frac{\lambda_{l}^{\star}+\sigma^{2}}{\sqrt{n}}=:\mathsf{gap}.
\]
To better understand and compare these two eigen-gap requirements,
we shall discuss them for a couple of distinct scenarios.
\begin{itemize}
\item If $\sigma^{2}\big(\sqrt{\frac{p}{n}}+\frac{p}{n}\big)\lesssim\lambda_{l}^{\star}\lesssim\sigma^{2}$
(the sample size needs to satisfy $n\geq p$ by the assumption (\ref{eq:noise-condition-iid-pca})),
the eigen-gap conditions above simplify to
\[
\mathsf{gap}\asymp\frac{\sigma^{2}}{\sqrt{n}}\qquad\text{and}\qquad\mathsf{gap}_{\mathsf{DK}}\asymp\sigma^{2}\sqrt{\frac{p}{n}}.
\]
\[
\Longrightarrow\qquad\frac{\mathsf{gap}_{\mathsf{DK}}}{\mathsf{gap}}\asymp\sqrt{p}.
\]
\item If $\sigma^{2}\lesssim\lambda_{l}^{\star}\lesssim\sigma^{2}p$, then
one has
\[
\mathsf{gap}\asymp\frac{\lambda_{l}^{\star}}{\sqrt{n}}\qquad\text{and}\qquad\mathsf{gap}_{\mathsf{DK}}\asymp\sqrt{\frac{\lambda_{l}^{\star}\sigma^{2}p}{n}}+\frac{\sigma^{2}p}{n}.
\]
Comparing these two terms reveals that
\[
\frac{\mathsf{gap}_{\mathsf{DK}}}{\mathsf{gap}}\asymp\sqrt{\frac{\sigma^{2}p}{\lambda_{l}^{\star}}}\bigg(1+\sqrt{\frac{\sigma^{2}p}{\lambda_{l}^{\star}n}}\bigg)\overset{(\mathrm{i})}{\asymp}\sqrt{\frac{\sigma^{2}p}{\lambda_{l}^{\star}}}\overset{(\mathrm{ii})}{\gtrsim}1,
\]
where (i) holds due to the assumption (\ref{eq:noise-condition-iid-pca})
and (ii) follows from the condition $\lambda_{l}^{\star}\lesssim\sigma^{2}p$.
\item If $\lambda_{l}^{\star}\gtrsim\sigma^{2}p$, then it is straightforward
to see that 
\[
\mathsf{gap}\asymp\mathsf{gap}_{\mathsf{DK}}\asymp\frac{\lambda_{l}^{\star}}{\sqrt{n}}.
\]
\end{itemize}
To sum up, our eigen-gap requirement (\ref{eq:eigengap-condition-iid-pca})
is 
\[
\Omega\left(\sqrt{p\Big(1\wedge\frac{\sigma^{2}}{\lambda_{l}^{\star}}\Big)}\vee1\right)
\]
 times less stringent than the one demanded in classical matrix perturbation
theory, thereby justifying the improvement of our results upon prior
art.
In addition, we note that \citep{bao2022statistical} also considered statistical inference for principal components of spike covariance matrices; when $\sigma = 1$, the eigen-gap therein needs to satisfy $\Delta_l^\star \gtrsim n^{-1/2 + \epsilon}$ for an arbitrary small fixed constant $\epsilon > 0$, thereby leading to a more stringent condition than ours. 

\item \emph{Bias reduction}. Similar to the matrix denoising case, the plug-in
estimator $\bm{a}^{\top}\bm{u}_{l}$ suffers from the following extra
``bias'' term in comparison to the de-biased estimator (\ref{eq:eigenvec-perturbation-bound-pca-debias}):
\begin{equation}
E_{\mathsf{pca},l}^{\mathsf{bias}}:=\frac{(\lambda_{l}^{\star}+\sigma^{2})\sigma^{2}p}{\lambda_{l}^{\star2}n}\left|\bm{a}^{\top}\bm{u}_{l}^{\star}\right|.\label{eq:defn-E-extra-pca}
\end{equation}
If $\bm{a}$ and $\bm{u}_{l}^{\star}$ are fairly correlated, then
this additional term becomes non-negligible and might affect the estimation
accuracy negatively. To see this, let us consider the following simple
case.
\begin{itemize}
\item \emph{Example}. Assume that $r=2$, $\lambda_{1}^{\star}=2\lambda_{2}^{\star}>0$,
$|\bm{a}^{\top}\bm{u}_{1}^{\star}|\asymp1$, $\sigma^{2}\asymp\lambda_{1}^{\star}$
and $p\asymp n$. As can be straightforwardly verified, the error
terms (\ref{eq:def:E-pca}) and (\ref{eq:defn-E-extra-pca}) in this
case become
\begin{align*}
E_{\mathsf{pca},1} & =\widetilde{O}\bigg(\frac{1}{n}\left|\bm{a}^{\top}\bm{u}_{1}^{\star}\right|+\frac{1}{\sqrt{n}}\left|\bm{a}^{\top}\bm{u}_{2}^{\star}\right|+\frac{1}{\sqrt{n}}\bigg)=\widetilde{O}\Big(\frac{1}{\sqrt{n}}\Big);\\
E_{\mathsf{pca},1}^{\mathsf{bias}} & \asymp\frac{p}{n}\left|\bm{a}^{\top}\bm{u}_{1}^{\star}\right|\asymp1.
\end{align*}
In other words, the bias term $E_{\mathsf{pca},1}^{\mathsf{bias}}$
could be $\sqrt{n}$ times larger than the error term $E_{\mathsf{pca},1}^{\mathsf{bias}}$
(up to some logarithmic factor). 
\end{itemize}
As a takeaway message from the above example, it is crucial to reduce
the bias incurred by $E_{\mathsf{pca},1}^{\mathsf{bias}}$. The proposed
de-biased estimator $u_{\bm{a}}^{\mathsf{debiased}}$ achieves bias
reduction by enlarging the plug-in estimator by a factor of $\sqrt{1+c_{l}}$,
where $c_{l}$ is computable in a data-driven manner. It is worth
noting that the factor $c_{l}$ (cf.~(\ref{eq:def:bl-pca})) takes
two different forms, depending on the relative ratio between the sample
size $n$ and the dimension $p$.

\end{itemize}

\paragraph{Minimax lower bounds and optimality. }

In order to evaluate the tightness of our statistical guarantees,
we develop minimax lower bounds for PCA. Here and below, we denote
by $\bm{u}_{l}(\bm{\Sigma})\in\mathbb{R}^{p}$ the eigenvector associated
with the $l$-th largest eigenvalue of a matrix $\bm{\Sigma}$, and
we define two sets of covariance matrices as follows:
\begin{align*}
\mathcal{M}_{1}(\bm{\Sigma}^{\star}) & :=\left\{ \bm{\Sigma}\in\mathbb{R}^{p\times p}\,\colon\,\mathsf{rank}(\bm{\Sigma})=r,\,\lambda_{i}(\bm{\Sigma})=\lambda_{i}^{\star}\,(1\leq i\leq r),\,\|\bm{\Sigma}-\bm{\Sigma}^{\star}\|_{\mathrm{F}}\leq\max_{k:\,k\neq l}\sqrt{\frac{(\lambda_{l}^{\star}+\sigma^{2})(\lambda_{k}^{\star}+\sigma^{2})}{n}}\right\} .\\
\mathcal{M}_{2}(\bm{\Sigma}^{\star}) & :=\left\{ \bm{\Sigma}\in\mathbb{R}^{p\times p}\,\colon\,\mathsf{rank}(\bm{\Sigma})=r,\,\lambda_{i}(\bm{\Sigma})=\lambda_{i}^{\star}\,(1\leq i\leq r),\,\|\bm{u}_{l}(\bm{\Sigma})-\bm{u}_{l}^{\star}\|_{\mathrm{2}}\leq\sqrt{\frac{(\lambda_{l}^{\star}+\sigma^{2})\sigma^{2}}{\lambda_{l}^{\star2}n}}\right\} .
\end{align*}
\begin{theorem}\label{thm:minimax-evector-pertur-sym-iid-pca}Consider
any fixed vector $\bm{a}\in\mathbb{R}^{p}$. For any given $\bm{\Sigma}$,
let $\{\bm{s}_{i}\}_{i=1}^{n}$ be independent samples satisfying
$\bm{s}_{i}\overset{\mathrm{i.i.d.}}{\sim}\mathcal{N}(\bm{0},\bm{\Sigma}+\sigma^{2}\bm{I}_{p})$.
Assume that the sample size obeys
\begin{equation}
n\geq\left\{ \max_{k:\,k\neq l}\frac{(\lambda_{k}^{\star}+\sigma^{2})(\lambda_{l}^{\star}+\sigma^{2})}{|\lambda_{l}^{\star}-\lambda_{k}^{\star}|^{2}}\right\} \vee\frac{(\lambda_{l}^{\star}+\sigma^{2})\sigma^{2}}{\lambda_{l}^{\star2}}.\label{eq:minimax-sample-size}
\end{equation}
Then one has\begin{subequations}\label{eq:minimax-lb}
\begin{align*}
 & \inf_{u_{\bm{a},l}}\sup_{\bm{\Sigma}\in\mathcal{M}_{1}(\bm{\Sigma}^{\star})}\mathbb{E}\Big[\min\big|u_{\bm{a},l}\pm\bm{a}^{\top}\bm{u}_{l}(\bm{\Sigma})\big|\Big]\\
 & \qquad\gtrsim\max_{k:\,k\neq l,\,1\leq k\leq r}\frac{(\lambda_{k}^{\star}+\sigma^{2})(\lambda_{l}^{\star}+\sigma^{2})}{|\lambda_{l}^{\star}-\lambda_{k}^{\star}|^{2}\,n}\big|\bm{a}^{\top}\bm{u}_{l}^{\star}\big|+\max_{k:\,k\neq l,\,1\leq k\leq r}\frac{\sqrt{(\lambda_{k}^{\star}+\sigma^{2})(\lambda_{l}^{\star}+\sigma^{2})}}{|\lambda_{l}^{\star}-\lambda_{k}^{\star}|\sqrt{n}}\big|\bm{a}^{\top}\bm{u}_{k}^{\star}\big|\eqqcolon E_{\mathsf{lb}1,l};\\
 & \inf_{u_{\bm{a},l}}\sup_{\bm{\Sigma}\in\mathcal{M}_{2}(\bm{\Sigma}^{\star})}\mathbb{E}\Big[\min\big|u_{\bm{a},l}\pm\bm{a}^{\top}\bm{u}_{l}(\bm{\Sigma})\big|\Big]\\
 & \qquad\gtrsim\sqrt{\frac{(\lambda_{l}^{\star}+\sigma^{2})\sigma^{2}}{\lambda_{l}^{\star2}n}}\|\bm{P}_{\bm{U}^{\star\perp}}\bm{a}\|_{2}\eqqcolon E_{\mathsf{lb}2,l}.
\end{align*}
\end{subequations}Here, the infimum is taken over all estimator $u_{\bm{a},l}$
for the linear form of the $l$-th eigenvector.\end{theorem}

The proof of this theorem can be found in Appendix~\ref{sec:Proof-for-minimax}.
To interpret this lower bound, let us consider, for simplicity, the
scenario where 
\begin{equation}
r,\kappa\asymp1\qquad\text{and}\qquad|\bm{a}^{\top}\bm{u}_{l}^{\star}|\le(1-\epsilon)\|\bm{a}\|_{2}\label{eq:assumption-lower-bound-match}
\end{equation}
for some arbitrarily small constant $\epsilon>0$. In this scenario,
the statistical error bound (\ref{eq:eigenvec-perturbation-bound-pca-debias})
derived in Theorem~\ref{thm:evector-pertur-sym-iid-pca} matches
the preceding minimax lower bounds in the sense that
\[
E_{\mathsf{PCA},l}\asymp E_{\mathsf{lb}1,l}+E_{\mathsf{lb}2,l}.
\]
To verify this relation under the conditions (\ref{eq:assumption-lower-bound-match}),
it is sufficient to see that
\begin{align*}
 & \max_{k:k\neq l}\frac{\sqrt{(\lambda_{k}^{\star}+\sigma^{2})(\lambda_{l}^{\star}+\sigma^{2})}}{|\lambda_{l}^{\star}-\lambda_{k}^{\star}|\sqrt{n}}\big|\bm{a}^{\top}\bm{u}_{k}^{\star}\big|+\sqrt{\frac{(\lambda_{l}^{\star}+\sigma^{2})\sigma^{2}}{\lambda_{l}^{\star2}n}}\|\bm{P}_{\bm{U}^{\star\perp}}\bm{a}\|_{2}\\
 & \qquad\asymp\sum_{k:\,k\neq l}\frac{\sqrt{(\lambda_{k}^{\star}+\sigma^{2})(\lambda_{l}^{\star}+\sigma^{2})}}{|\lambda_{l}^{\star}-\lambda_{k}^{\star}|\sqrt{n}}\big|\bm{a}^{\top}\bm{u}_{k}^{\star}\big|+\sqrt{\frac{(\lambda_{l}^{\star}+\sigma^{2})\sigma^{2}}{\lambda_{l}^{\star2}n}}\|\bm{P}_{\bm{U}^{\star\perp}}\bm{a}\|_{2}\\
 & \qquad\overset{(\mathrm{i})}{\asymp}\sum_{k:\,k\neq l}\frac{\sqrt{(\lambda_{k}^{\star}+\sigma^{2})(\lambda_{l}^{\star}+\sigma^{2})}}{|\lambda_{l}^{\star}-\lambda_{k}^{\star}|\sqrt{n}}\big|\bm{a}^{\top}\bm{u}_{k}^{\star}\big|+\sqrt{\frac{(\lambda_{l}^{\star}+\sigma^{2})\sigma^{2}}{\lambda_{l}^{\star2}n}}\bigg(\sum_{k:\,k\neq l}\big|\bm{a}^{\top}\bm{u}_{k}^{\star}\big|+\|\bm{P}_{\bm{U}^{\star\perp}}\bm{a}\|_{2}\bigg)\\
 & \qquad\overset{(\mathrm{ii})}{\asymp}\sum_{k:\,k\neq l}\frac{\sqrt{(\lambda_{k}^{\star}+\sigma^{2})(\lambda_{l}^{\star}+\sigma^{2})}}{|\lambda_{l}^{\star}-\lambda_{k}^{\star}|\sqrt{n}}\big|\bm{a}^{\top}\bm{u}_{k}^{\star}\big|+\sqrt{\frac{(\lambda_{l}^{\star}+\sigma^{2})\sigma^{2}}{\lambda_{l}^{\star2}n}}\|\bm{a}\|_{2},
\end{align*}
where (i) holds true since $\max_{k:k\neq l}\frac{\sqrt{(\lambda_{k}^{\star}+\sigma^{2})(\lambda_{l}^{\star}+\sigma^{2})}}{|\lambda_{l}^{\star}-\lambda_{k}^{\star}|\sqrt{n}}\gtrsim\sqrt{\frac{(\lambda_{l}^{\star}+\sigma^{2})\sigma^{2}}{\lambda_{l}^{\star2}n}}$,
and (ii) holds true as long as $|\bm{a}^{\top}\bm{u}_{l}^{\star}|\le(1-\epsilon)\|\bm{a}\|_{2}$
for some constant $\epsilon>0$. In conclusion, the above calculation
unveils the statistical optimality of the proposed de-biased estimator
for the scenario specified in (\ref{eq:assumption-lower-bound-match}). 

\paragraph{Comparison with past works.}

Estimation for linear forms of eigenvectors in the context of PCA
has been investigated in several recent works \citep{koltchinskii2016asymptotics,koltchinskii2017normal,koltchinskii2020efficient},
with the bias issue of plug-in estimators first recognized in \citet{koltchinskii2016asymptotics}.
Among these works, the state-of-the-art result was due to \citet{koltchinskii2020efficient},
which proposed an efficient de-biased estimator and established its
asymptotic normality. To better understand our contributions, it is
helpful to compare Theorem~\ref{thm:evector-pertur-sym-iid-pca}
with the theoretical guarantees in \citet{koltchinskii2020efficient}
under the spiked covariance model with $\bm{\Sigma}=\bm{\Sigma}^{\star}+\sigma^{2}\bm{I}_{p}$.
The theoretical guarantees developed in \citet[Theorem 3.3]{koltchinskii2020efficient}
operate under the following conditions (when translated to our setting
using our notation)
\begin{align}
 & \Delta_{l}^{\star}=\Omega(\lambda_{\max}^{\star}+\sigma^{2}),\quad\sigma^{2}=o(\lambda_{\min}^{\star}),\quad r,\kappa\asymp1,\quad\sum_{k:\,k\neq l}\left|\bm{a}^{\top}\bm{u}_{k}^{\star}\right|^{2}+\frac{\sigma^{2}}{\lambda_{l}^{\star}}\|\bm{P}_{\bm{U}^{\star\perp}}\bm{a}\|_{2}^{2}\asymp\|\bm{a}\|_{2}^{2}.\label{eq:condition-koltchinskii-gap}
\end{align}
In comparison, our results make improvements in the following aspects:
\begin{itemize}
\item \emph{Eigen-gap requirement}: our eigen-gap requirement (\ref{eq:eigengap-condition-iid-pca})
is $\widetilde{O}(\sqrt{r/n})$ times less stringent than the one
in (\ref{eq:condition-koltchinskii-gap});
\item \emph{Requirement on noise variance}: our result (i.e., Theorem~\ref{thm:evector-pertur-sym-iid-pca})
allows the noise variance $\sigma^{2}$ to be larger than $\lambda_{\min}^{\star}$;
\item \emph{Requirement on condition number and rank:} our theory permits
both $\kappa$ and $r$ to grow with the dimension. 
\end{itemize}

It is worth noting that \citet{koltchinskii2020efficient} accommodates
a more general class of covariance matrices than the aforementioned
spiked covariance. The main purpose of our discussion above is to
make clear the inadequacy of prior theories when the eigen-gap is
small. 

\section{Related works}

\label{sec:Related-work}

Spectral methods have served as an effective paradigm for a variety
of statistical data science problems, examples including matrix completion
\citep{KesMonSew2010,keshavan2010matrix,sun2016guaranteed,ma2017implicit}, tensor completion
\citep{xia2017statistically,montanari2018spectral,cai2019nonconvex,cai2020uncertainty},
community detection \citep{lei2019unified,abbe2017entrywise}, ranking
from pairwise comparisons \citep{negahban2017rank,chen2015spectral},
and so on. The mainstream analysis framework for spectral methods
is largely built upon classical matrix perturbation theory \citep{stewart1990matrix,chen2020spectral}.
This set of classical theory typically focuses on deriving $\ell_{2}$
eigenspace or singular subspace perturbation bounds (e.g., the Davis-Kahan
theorem \citep{davis1970rotation} and the Wedin theorem \citep{wedin1972perturbation}),
which has been derived for general purposes without incorporating
statistical properties of the specific problems of interest. Several
useful extensions have been developed tailored to high-dimensional
statistical applications, particularly when the perturbation matrix
of interest enjoys certain random structure \citep{vu2011singular,wang2015singular,yu2015useful,xia2019confidence,cai2018rate,o2018random}.
In particular, the $\ell_{2}$ perturbation bounds for the eigenvector
(or eigenspace) of the sample covariance matrix has been extensively
studied in the PCA literature, e.g., \citep{nadler2008finite,johnstone2009consistency,vu2012minimax,lounici2013sparse,lounici2014high,zhang2018heteroskedastic,zhu2019high,xia2021normal}. 
Another line of works \citep{vu2011singular,o2018random} improved Davis-Kahan's and Wedin's theorems in the matrix denoising setting with small eigen-gaps, which, however, is not tight unless the spectral norm $\|\bm{H}\|$ of the noise matrix is extremely small.

Moving beyond $\ell_{2}$ perturbation theory, more fine-grained eigenvector
perturbation bounds --- particularly entrywise eigenvector perturbation
or $\ell_{2,\infty}$ eigenspace perturbation --- has garnered growing
attention over the past few years \citep{abbe2017entrywise,ma2017implicit,cape2019two,cai2019subspace,chen2019noisy,chen2018asymmetry,cai2019nonconvex,lei2019unified,fan2018eigenvector,zhong2017near,abbe2020ell_p}.
Among these $\ell_{\infty}$ or $\ell_{2,\infty}$ theoretical guarantees,
the results in \citet{abbe2017entrywise,ma2017implicit,cai2019nonconvex,cai2019subspace,chen2017spectral,chen2019noisy,chen2019inference,chen2020partial}
were established via a powerful leave-one-out analysis framework,
while the works \citep{pmlr-v83-eldridge18a,chen2018asymmetry} invoked
a Neumann expansion trick paired with proper control of moments. 

In contrast to the rich literature on $\ell_{2}$, $\ell_{\infty}$
and/or $\ell_{2,\infty}$ perturbation theory, estimation theory concerning
linear functionals of eigenvectors (or singular vectors) are rather
scarce and under-explored. While entrywise perturbation can be regarded
as a special type of linear functionals of eigenvectors, the analysis
techniques mentioned above are typically incapable of analyzing an
arbitrary linear form. Only until recently, progress has been made
towards addressing this problem. In the matrix denoising setting,
effective concentration bounds have been established in \citet{koltchinskii2016perturbation}
for estimating linear forms of singular vectors under i.i.d.~Gaussian
noise, while \citet{bao2021singular} established the limiting distributions
of the angle between the singular vectors of the noisy matrix and
the corresponding ground-truth singular vectors. In \citet{koltchinskii2016asymptotics,koltchinskii2017normal,koltchinskii2020efficient},
several bias reduction procedures were developed for the problem of
PCA and covariance estimation, which established the asymptotic normality
and statistical efficiency of the proposed estimator. The eigen-gap
conditions required therein, however, are considerably more stringent
than the ones required in our theory. Another line of recent works
has studied linear form of eigenvectors was \citet{chen2018asymmetry,cheng2020tackling},
which, however, tackled a different setting of the matrix denoising
problem. Specifically, \citet{chen2018asymmetry,cheng2020tackling}
focused on the case where the noise matrix $\bm{H}$ is asymmetric
and contains independent entries (so that $H_{i,j}$ and $H_{j,i}$
are two independent copies of noise); in this case, a carefully de-biased
estimator proposed based on the eigenvector of the asymmetric data
matrix $\bm{M}$ is shown to be minimax-optimal. Additionally, \citet{fan2019asymptotic}
pinned down the asymptotic distribution for bilinear forms of eigenvectors
for large spiked random matrices, while \citet{xia2019statistical}
proposed a de-biasing method to estimate linear forms of the matrix
for noisy matrix completion. These are beyond the reach of the current
paper. 

\section{Analysis}

\label{sec:Analysis}

In this section, we discuss the analysis ideas for establishing Theorem
\ref{thm:evector-pertur-sym-iid} and Theorem \ref{thm:evector-pertur-sym-iid-pca}.
One of the main tools lies in the master theorems stated below, which
characterize the principal angle between the perturbed eigenvector
and an arbitrary subspace of interest. We shall see momentarily the
effectiveness of these master theorems when applied to matrix denoising
and PCA. 

\subsection{Master theorems}

\label{subsec:Master-theorems}

For any matrix $\bm{Q}\in\mathbb{R}^{n\times k}$ obeying $\bm{Q}^{\top}\bm{Q}=\bm{I}_{k}$
($1\leq k\leq n$), let $\bm{Q}^{\perp}\in\mathbb{R}^{n\times(n-k)}$
be an arbitrary matrix whose columns form an orthonormal basis of
the complement to the subspace spanned by the columns of $\bm{Q}$, namely
\begin{equation}
\big[\bm{Q},\bm{Q}^{\perp}\big]^{\top}\big[\bm{Q},\bm{Q}^{\perp}\big]=\bm{I}_{n}.\label{eq:defn-projection-matrices}
\end{equation}
Our results concern the decomposition of an eigenvector $\bm{u}_{l}$ of matrix $\bm{M}$
taking the following form: 
\begin{equation}
\bm{u}_{l}=\bm{u}_{l,\|}\cos\theta+\bm{u}_{l,\perp}\sin\theta.\label{eq:uk-decomposition}
\end{equation}
Here, $\theta$ denotes the principal angle between $\bm{u}_{l}$
and the subspace spanned by $\bm{Q}$, whereas $\bm{u}_{l,\|}$ and
$\bm{u}_{l,\perp}$ are two unit vectors (i.e.~$\|\bm{u}_{l,\|}\|_{2}=\|\bm{u}_{l,\perp}\|_{2}=1$)
such that 
\begin{itemize}
\item $\bm{u}_{l,\|}$ lies in the subspace spanned by $\bm{Q}$; this means
that $\bm{Q}\bm{Q}^{\top}\bm{u}_{l,\|}=\bm{u}_{l,\|}$, where $\bm{Q}\bm{Q}^{\top}$
is the projection matrix onto the subspace spanned by $\bm{Q}$;
\item $\bm{u}_{l,\perp}$ is perpendicular to the subspace spanned by $\bm{Q}$,
so that $\bm{Q}^{\perp}(\bm{Q}^{\perp})^{\top}\bm{u}_{l,\perp}=\bm{u}_{l,\perp}$. 
\end{itemize}

\paragraph{When $\bm{Q}$ is a unit vector.}

We shall begin with the case when $\bm{Q}$ is a unit vector. For
notational simplicity, let us write $\bm{q}$ for $\bm{Q}$ in this
case to emphasize that this is a vector, and let $\bm{q}^{\perp}\in\mathbb{R}^{n\times(n-1)}$
indicate $\bm{Q}^{\perp}$. In this case, we can take $\bm{u}_{l,\|}$
to be equal to $\bm{q}$. Our result is this:

\begin{theorem}\label{thm:master-thm-vector}Consider any vector
$\bm{q}\in\mathbb{R}^{n}$ with $\|\bm{q}\|_{2}=1$. Write 
\begin{equation}
\bm{u}_{l}=\bm{q}\cos\theta+\bm{u}_{l,\perp}\sin\theta\label{eq:ul-decomposition-rank1}
\end{equation}
for some $\theta$ as well as some vector $\bm{u}_{l,\perp}$ obeying
$\|\bm{u}_{l,\perp}\|_{2}=1$ and $\bm{q}^{\top}\bm{u}_{l,\perp}=0$.
Suppose that $\lambda_{l}\bm{I}_{n-1}-(\bm{q}^{\perp})^{\top}\bm{M}\bm{q}^{\perp}$
is invertible. Then one has\begin{subequations}\label{claim:master-thm-vector}
\begin{align}
\cos^{2}\theta & =\frac{1}{1+\big\|\big(\lambda_{l}\bm{I}_{n-1}-(\bm{q}^{\perp})^{\top}\bm{M}\bm{q}^{\perp}\big)^{-1}(\bm{q}^{\perp})^{\top}\bm{M}\bm{q}\big\|_{2}^{2}},\label{eq:cos-theta-rank1}\\
\lambda_{l} & =\bm{q}^{\top}\bm{M}\bm{q}+\bm{q}^{\top}\bm{M}\bm{q}^{\perp}\big(\lambda_{l}\bm{I}_{n-1}-(\bm{q}^{\perp})^{\top}\bm{M}\bm{q}^{\perp}\big)^{-1}(\bm{q}^{\perp})^{\top}\bm{M}\bm{q}.\label{eq:lambda-rank1}
\end{align}
In addition, when $\sin\theta\ne0$, the vector $\bm{u}_{l,\perp}$
satisfies
\begin{equation}
\bm{u}_{l,\perp}=\pm\frac{\bm{q}^{\perp}\big(\lambda_{l}\bm{I}_{n-1}-(\bm{q}^{\perp})^{\top}\bm{M}\bm{q}^{\perp}\big)^{-1}(\bm{q}^{\perp})^{\top}\bm{M}\bm{q}}{\big\|\big(\lambda_{l}\bm{I}_{n-1}-(\bm{q}^{\perp})^{\top}\bm{M}\bm{q}^{\perp}\big)^{-1}(\bm{q}^{\perp})^{\top}\bm{M}\bm{q}\big\|_{2}}.\label{eq:u-perp-rank1-1}
\end{equation}
\end{subequations}\end{theorem}\begin{proof}See Appendix~\ref{sec:Proof-of-Theorem-master-vector}.\end{proof}

In words, Theorem~\ref{thm:master-thm-vector} derives closed-form
expressions for both $\cos\theta$ and $\bm{u}_{l,\perp}$ (up to
global signs), in terms of simple and direct manipulation of the data
matrix $\bm{M}$ as well as the associated eigenvalue $\lambda_{l}$.
While the identities (\ref{eq:cos-theta-rank1}) and (\ref{eq:u-perp-rank1-1})
might seem somewhat complicated at first glance, they often allow
for convenient decomposition of the noise into independent components,
thus streamlining the analysis. Similarly, while the relation (\ref{eq:lambda-rank1})
takes the form of a nonlinear equation about $\lambda_{l}$, it often
enables convenient decoupling of complicated statistical dependency,
as we shall demonstrate momentarily.

\paragraph{When $\bm{Q}$ is a more general orthonormal matrix. }

The next theorem extends the relation (\ref{eq:lambda-rank1}) to
the case when $\bm{Q}$ is a general orthonormal matrix (beyond the
vector case), which proves useful in eigenvalue analysis for more
general low-rank problems.

\begin{theorem}\label{thm:master-theorem-general}Assume that $k<n$.
Consider the corresponding decomposition \eqref{eq:uk-decomposition}
for any matrix $\bm{Q}\in\mathbb{R}^{n\times k}$ obeying $\bm{Q}^{\top}\bm{Q}=\bm{I}_{k}$.
Suppose that $\lambda_{l}\bm{\bm{I}}_{n-k}-(\bm{Q}^{\perp})^{\top}\bm{M}\bm{Q}^{\perp}$
and $\lambda_{l}\bm{I}_{k}-\bm{{Q}}^{\top}\bm{M}\bm{{Q}}$ are both
invertible. Then one has \begin{subequations}
\begin{align}
\cos^{2}\theta & =\frac{1}{1+\big\|\big(\lambda_{l}\bm{\bm{I}}_{n-k}-(\bm{Q}^{\perp})^{\top}\bm{M}\bm{Q}^{\perp}\big)^{-1}(\bm{Q}^{\perp})^{\top}\bm{M}\bm{u}_{l,\|}\big\|_{2}^{2}},\label{eq:cos-theta-rank-r}\\
\big(\lambda_{l}\bm{I}_{k}-\bm{Q}^{\top}\bm{M}\bm{Q}\big)\bm{Q}^{\top}\bm{u}_{l,\|} & =\bm{Q}^{\top}\bm{M}\bm{Q}^{\perp}\big(\lambda_{l}\bm{\bm{I}}_{n-k}-(\bm{Q}^{\perp})^{\top}\bm{M}\bm{Q}^{\perp}\big)^{-1}(\bm{Q}^{\perp})^{\top}\bm{M}\bm{u}_{l,\|}.\label{eq:lambda-l-long-identity-1}
\end{align}
\end{subequations}\end{theorem}\begin{proof}See Appendix \ref{sec:Proof-of-Theorem-master-general}.\end{proof}

\subsection{Analysis for matrix denoising}

\label{sec:Analysis-for-matrix-denoising}

Armed with the preceding master theorems, we are now positioned to
develop consequences for matrix denoising. As a crucial first step
of the analysis, we need to establish an eigenvalue perturbation theory
that is tightly connected to the eigenvector perturbation bounds.
Recalling that $\lambda_{l}$ is the $l$-th largest eigenvalue (in
magnitude) of $\bm{M}$, we present a theorem that reveals the proximity
of $\lambda_{l}$ and the ground truth $\lambda_{l}^{\star}$.

\begin{theorem}[Eigenvalue perturbation for matrix denoising]\label{thm:eigval-pertur-sym-iid}Consider
the model in Section~\ref{subsec:Matrix-denoising-model}. Fix any
$1\leq l\leq r$, and instate the assumptions of Theorem~\ref{thm:evector-pertur-sym-iid}.
With probability at least $1-O(n^{-10}),$ one has 
\begin{equation}
|\lambda_{l}-\gamma(\lambda_{l})-\lambda_{l}^{\star}|\leq C_{1}\sigma\sqrt{r}\log n\label{eq:eigenvalue-perturbation-bound-iid}
\end{equation}
for some sufficiently large constant $C_{1}>0$, where $\gamma(\cdot)$
is defined as
\begin{equation}
\gamma(\lambda)\coloneqq\sigma^{2}\mathsf{tr}\Big[\Big(\lambda\bm{I}_{n-r}-(\bm{U}^{\star\perp})^{\top}\bm{H}\bm{U}^{\star\perp}\Big)^{-1}\Big].\label{eq:definition-gamma-lambda-iid}
\end{equation}
\end{theorem}\begin{remark}Here, we recall that the columns of $\bm{U}^{\star\perp}\in\mathbb{R}^{n\times(n-r)}$
form an orthonormal basis of the complement to the subspace spanned
by $\bm{U}^{\star}$.\end{remark}

\begin{remark} The error bound (\ref{eq:eigenvalue-perturbation-bound-iid})
concerning the empirical eigenvalue $\lambda_{l}$ contains a systematic
non-negligible term $\gamma(\lambda_{l})$. This makes clear the presence
of a bias effect, which needs to be properly subtracted if one desires
a near-optimal estimate of $\lambda_{l}^{\star}$. 
	It is also worth noting that the importance of bias correction in eigenvalue estimation has been recognized in prior literature as well (e.g.~\citep{paul2007asymptotics}). 

 \end{remark}

\subsubsection{Proof of eigenvalue perturbation theory (Theorem \ref{thm:eigval-pertur-sym-iid})}

\label{subsec:Proof-of-Theorem-matrix-denoising}

We start by demonstrating how to prove the eigenvalue perturbation
bound in Theorem \ref{thm:eigval-pertur-sym-iid}. Let us fix an arbitrary
$1\leq l\leq r$. The key ingredient of the analysis is to invoke
our master theorem (namely, Theorem~\ref{thm:master-theorem-general}). 

Before proceeding, we first verify a few useful facts. It is well
known that if $\sigma\sqrt{n}\leq c_{0}\lambda_{\min}^{\star}$ for
some sufficiently small constant $c_{0}>0$, then with probability
exceeding $1-O(n^{-20})$ one has (see, e.g., \citet[Theorem 3.1.4]{chen2020spectral})
\begin{equation}
\|\bm{H}\|\leq\lambda_{\min}^{\star}/3.\label{eq:H-norm-1-3}
\end{equation}
Recall that
\[
(\bm{U}^{\star\perp})^{\top}\bm{M}\bm{U}^{\star\perp}=(\bm{U}^{\star\perp})^{\top}\bm{M}^{\star}\bm{U}^{\star\perp}+(\bm{U}^{\star\perp})^{\top}\bm{H}\bm{U}^{\star\perp}=(\bm{U}^{\star\perp})^{\top}\bm{H}\bm{U}^{\star\perp},
\]
which together with (\ref{eq:H-norm-1-3}) implies that 
\begin{equation}
\big\|(\bm{U}^{\star\perp})^{\top}\bm{M}\bm{U}^{\star\perp}\big\|=\big\|(\bm{U}^{\star\perp})^{\top}\bm{H}\bm{U}^{\star\perp}\big\|\leq\|\bm{H}\|\leq\lambda_{\min}^{\star}/3\label{eq:H-norm-bound}
\end{equation}
with probability exceeding $1-O(n^{-20})$. This means that with high
probability: (i) the Weyl inequality yields
\begin{equation}
|\lambda_{l}|\geq|\lambda_{l}^{\star}|-\|\bm{H}\|\geq2\,|\lambda_{l}^{\star}|/3\qquad\text{and}\qquad|\lambda_{l}|\leq|\lambda_{l}^{\star}|+\|\bm{H}\|\leq4\,|\lambda_{l}^{\star}|/3;\label{eq:lambdal-range}
\end{equation}
(2) it holds true that $\lambda_{\min}^{\star}/3\ge\|\bm{H}\|\geq\|(\bm{U}^{\star\perp})^{\top}\bm{M}\bm{U}^{\star\perp}\|$,
and hence
\[
\lambda\bm{I}_{n-r}-(\bm{U}^{\star\perp})^{\top}\bm{M}\bm{U}^{\star\perp}\quad\text{is invertible}
\]
for any $\lambda\in\mathbb{R}$ obeying $|\lambda|\ge2\lambda_{\min}^{\star}/3$. 

With the above two observations in mind, take $\bm{Q}=\bm{U}^{\star}$
in Theorem~\ref{thm:master-theorem-general} to show that 
\begin{align}
\big(\lambda_{l}\bm{I}_{r}-\bm{U}^{\star\top}\bm{M}\bm{U}^{\star}\big)\bm{U}^{\star\top}\bm{u}_{l,\|} & =\bm{G}(\lambda_{l})\bm{U}^{\star\top}\bm{u}_{l,\|}\label{eq:lambda-I-decomposition-123}
\end{align}
with probability exceeding $1-O(n^{-20})$, where for any given $\lambda$
with $2\lambda_{\min}^{\star}/3\leq|\lambda|\leq4\lambda_{\max}^{\star}/3$,
we define 
\begin{equation}
\bm{G}(\lambda)\coloneqq\bm{U}^{\star\top}\bm{M}\bm{U}^{\star\perp}\Big(\lambda\bm{I}_{n-r}-(\bm{U}^{\star\perp})^{\top}\bm{M}\bm{U}^{\star\perp}\Big)^{-1}(\bm{U}^{\star\perp})^{\top}\bm{M}\bm{U}^{\star}.\label{eq:defn-Glambda}
\end{equation}
Note that $\bm{U}^{\star\top}$ and $\bm{u}_{l}^{\star\perp}$ are
not uniquely defined. To avoid ambiguity, here and throughout, we
let $\bm{U}^{\star\perp}\in\mathbb{R}^{n\times(n-r)}$ denote an arbitrary
matrix whose columns form an orthonormal basis of the complement to
the subspace spanned by $\bm{U}^{\star}$, and define 
\begin{equation}
\bm{u}_{l}^{\star\perp}=[\bm{u}_{1}^{\star},\bm{u}_{2}^{\star},\ldots,\bm{u}_{l-1}^{\star},\bm{u}_{l+1}^{\star},\ldots,\bm{u}_{r}^{\star},\bm{U}^{\star\perp}]\in\mathbb{R}^{n\times(n-1)}\label{eq:def:u-l-star-perp}
\end{equation}
for each $1\leq l\leq r$.

Recognizing that $\bm{U}^{\star\top}\bm{M}^{\star}\bm{U}^{\star}=\bm{\Lambda}^{\star}$,
$\bm{M}^{\star}\bm{U}^{\star\perp}=\bm{0}$ and $(\bm{U}^{\star\perp})^{\top}\bm{M}^{\star}=\bm{0}$,
we can rewrite (\ref{eq:lambda-I-decomposition-123}) as\begin{subequations}
\begin{align}
\big(\lambda_{l}\bm{I}_{r}-\bm{\Lambda}^{\star}-\bm{U}^{\star\top}\bm{H}\bm{U}^{\star}\big)\bm{U}^{\star\top}\bm{u}_{l,\|} & =\bm{G}(\lambda_{l})\bm{U}^{\star\top}\bm{u}_{l,\|}\label{eq:lambda-I-decomposition-123-1}
\end{align}
\begin{equation}
\text{with}\quad\bm{G}(\lambda)=\bm{U}^{\star\top}\bm{H}\bm{U}^{\star\perp}\Big(\lambda\bm{I}_{n-r}-(\bm{U}^{\star\perp})^{\top}\bm{H}\bm{U}^{\star\perp}\Big)^{-1}(\bm{U}^{\star\perp})^{\top}\bm{H}\bm{U}^{\star}.\label{eq:defn-Glambda-1}
\end{equation}
\end{subequations}Rearranging terms further gives 
\begin{align}
\big(\lambda_{l}\bm{I}_{r}-\bm{\Lambda}^{\star}-\bm{G}^{\perp}(\lambda_{l})\big)\bm{U}^{\star\top}\bm{u}_{l,\|} & =\bm{U}^{\star\top}\bm{H}\bm{U}^{\star}\bm{U}^{\star\top}\bm{u}_{l,\|}+\big(\bm{G}(\lambda_{l})-\bm{G}^{\perp}(\lambda_{l})\big)\bm{U}^{\star\top}\bm{u}_{l,\|},\label{eq:lambda-I-decomposition}
\end{align}
where we define 
\begin{equation}
\bm{G}^{\perp}(\lambda)\coloneqq\mathbb{E}\left[\bm{G}(\lambda)\mid(\bm{U}^{\star\perp})^{\top}\bm{H}\bm{U}^{\star\perp}\right],\label{eq:defn-Glambda-perp}
\end{equation}
with the (conditional) expectation taken assuming that $\lambda$
is independent of $\bm{H}$. Here, we single out the component $\bm{G}^{\perp}(\lambda)$
since --- as will be seen momentarily --- it often contains some
non-negligible bias term. Combining (\ref{eq:lambda-I-decomposition})
with the triangle inequality and the fact $\|\bm{U}^{\star\top}\bm{u}_{l,\|}\|_{2}=1$
then yields 
\begin{align}
\left\Vert \big(\lambda_{l}\bm{I}_{r}-\bm{\Lambda}^{\star}-\bm{G}^{\perp}(\lambda_{l})\big)\bm{U}^{\star\top}\bm{u}_{l,\|}\right\Vert _{2} & \leq\left\Vert \bm{U}^{\star\top}\bm{H}\bm{U}^{\star}\right\Vert +\left\Vert \bm{G}(\lambda_{l})-\bm{G}^{\perp}(\lambda_{l})\right\Vert \label{eq:lambda-l-long-identity-1-1}\\
 & \leq\big\|\bm{U}^{\star\top}\bm{H}\bm{U}^{\star}\big\|+\sup_{\lambda:\,|\lambda|\in\big[2|\lambda_{l}^{\star}|/3,\,4|\lambda_{l}^{\star}|/3\big]}\big\|\bm{G}(\lambda)-\bm{G}^{\perp}(\lambda)\big\|\label{eq:lambda-l-decomposition-126}
\end{align}
with probability at least $1-O(n^{-20})$, where the last line arises
from (\ref{eq:lambdal-range}).

In order to justify that $\big(\lambda_{l}\bm{I}_{r}-\bm{\Lambda}^{\star}-\bm{G}^{\perp}(\lambda_{l})\big)\bm{U}^{\star\top}\bm{u}_{l,\|}\approx\bm{0}$,
it remains to show that the two terms on the right-hand side of (\ref{eq:lambda-l-decomposition-126})
are both fairly small, which we accomplish through the following lemma.

\begin{lemma}\label{lemma:bound-Glambda-uniform}Assume that $\bm{H}\in\mathbb{R}^{n\times n}$
is a symmetric matrix with $H_{ij}\overset{\mathrm{i.i.d.}}{\sim}\mathcal{N}(0,\sigma^{2}),i\geq j$
and $\sigma\sqrt{n}\leq c_{0}\lambda_{\min}^{\star}$ for some sufficiently
small constant $c_{0}>0$. Then for any $1\leq l\leq r$, with probability
at least $1-O(n^{-11}),$ one has 
\begin{align}
\big\|\bm{U}^{\star\top}\bm{H}\bm{U}^{\star}\big\| & \lesssim\sigma\big(\sqrt{r}+\sqrt{\log n}\big),\nonumber \\
\sup_{\lambda:\,|\lambda|\in\big[2|\lambda_{l}^{\star}|/3,\,4|\lambda_{l}^{\star}|/3\big]}\big\|\bm{G}(\lambda)-\bm{G}^{\perp}(\lambda)\big\| & \lesssim\frac{\sigma^{2}}{\lambda_{\min}^{\star}}\big(\sqrt{rn\log n}+r\log n\big).\label{eq:Glambda-gap}
\end{align}
In addition, one has
\begin{equation}
\bm{G}^{\perp}(\lambda)=\Big\{\sigma^{2}\mathsf{tr}\Big[\big(\lambda\bm{I}_{n-r}-(\bm{U}^{\star\perp})^{\top}\bm{H}\bm{U}^{\star\perp}\big)^{-1}\Big]\Big\}\bm{I}_{r}.\label{eq:expression-Gperp-lambda}
\end{equation}
\end{lemma}\begin{proof}See Appendix \ref{subsec:Proof-of-Lemma-bound-Glambda-uniform}.\end{proof}

With the above lemma in place, by introducing\begin{subequations}\label{eq:defn-Mlambda-xi-lambda}
\begin{align}
\bm{M}_{\lambda} & :=\bm{\Lambda}^{\star}+\bm{G}^{\perp}(\lambda)=\bm{\Lambda}^{\star}+\gamma(\lambda)\bm{I}_{r}\label{eq:defn-Mlambda}\\
\gamma(\lambda) & :=\sigma^{2}\mathsf{tr}\Big[\big(\lambda\bm{I}_{n-r}-(\bm{U}^{\star\perp})^{\top}\bm{H}\bm{U}^{\star\perp}\big)^{-1}\Big]\label{eq:defn-gamma-lambda}
\end{align}
\end{subequations}for any $\lambda$ with $2\lambda_{\min}^{\star}/3\leq|\lambda|\leq4\lambda_{\max}^{\star}/3$,
we can invoke the union bound to show that with probability at least
$1-O(n^{-10})$,
\begin{align}
\big\|(\lambda_{l}\bm{I}_{r}-\bm{M}_{\lambda_{l}})\bm{U}^{\star\top}\bm{u}_{l,\|}\big\|_{2} & \leq\big\|\bm{U}^{\star\top}\bm{H}\bm{U}^{\star}\big\|+\sup_{\lambda:\,|\lambda|\in\big[2|\lambda_{l}^{\star}|/3,\,4|\lambda_{l}^{\star}|/3\big]}\big\|\bm{G}(\lambda)-\bm{G}^{\perp}(\lambda)\big\|\nonumber \\
 & \lesssim\sigma\big(\sqrt{r}+\sqrt{\log n}\big)+\frac{\sigma^{2}}{\lambda_{\min}^{\star}}\big(\sqrt{rn\log n}+r\log n\big)\nonumber \\
 & \leq C_{1}\sigma\sqrt{r}\log n=:\mathcal{E}_{\mathsf{MD}}\label{eq:lambdal-Ml-Epsilon}
\end{align}
holds for all $1\leq l\leq r$, where $C_{1}>0$ is some sufficiently
large constant. Intuitively, this means that $(\lambda_{l}\bm{I}_{r}-\bm{M}_{\lambda_{l}})\bm{U}^{\star\top}\bm{u}_{l,\|}\approx\bm{0}$,
and hence $\lambda_{l}$ is expected to be close to an eigenvalue
of $\bm{M}_{\lambda_{l}}$ --- which is $\lambda_{i}^{\star}+\gamma(\lambda_{l})$
for some $1\leq i\leq r$. 

With the above bound in place, the only possible range of $\lambda_{l}$
is characterized by the following lemma, which in turn establishes
Theorem \ref{thm:eigval-pertur-sym-iid}.

\begin{lemma}\label{lemma:lambda-l-mapping}Under the condition (\ref{eq:lambdal-Ml-Epsilon})
and the eigen-gap assumption (\ref{eq:eigengap-condition-iid}), with
probability at least $1-O(n^{-10})$ one has 
\begin{equation}
\big|\lambda_{l}-\lambda_{l}^{\star}-\gamma(\lambda_{l})\big|\leq\mathcal{E}_{\mathsf{MD}},\qquad1\leq l\leq r.\label{eq:lambda-l-mapping-iid}
\end{equation}
\end{lemma}\begin{proof}See Appendix \ref{subsec:Proof-of-Claim-lambdal-correspondence-iid}.\end{proof}

\subsubsection{Proof of eigenvector perturbation theory (Theorem \ref{thm:evector-pertur-sym-iid})}

\label{subsec:Proof-of-thm:evector-pertur-sym-iid}

Let us begin by decomposing $\bm{u}_{l}$ along the ground-truth direction
$\bm{u}_{l}^{\star}$ and its complement subspace as follows 
\begin{equation}
\bm{u}_{l}=\bm{u}_{l}^{\star}\cos\theta+\bm{u}_{l,\perp}\sin\theta,\label{eq:decomposition-ul-123}
\end{equation}
where the vector $\bm{u}_{l,\perp}$ obeys $\|\bm{u}_{l,\perp}\|_{2}=1$
and $\bm{u}_{l,\perp}^{\top}\bm{u}_{l}^{\star}=0$. Writing $\bm{a}=\bm{P}_{\bm{U}^{\star}}\bm{a}+\bm{P}_{\bm{U}^{\star\perp}}\bm{a}$
with $\bm{P}_{\bm{U}^{\star}}=\sum_{1\leq k\leq r}\bm{u}_{k}^{\star}\bm{u}_{k}^{\star\top}$
and $\bm{P}_{\bm{U}^{\star\perp}}=\bm{I}-\bm{P}_{\bm{U}^{\star}}$,
we obtain
\begin{align*}
\bm{a}^{\top}\bm{u}_{l}= & (\bm{P}_{\bm{U}^{\star}}\bm{a})^{\top}\bm{u}_{l}+(\bm{P}_{\bm{U}^{\star\perp}}\bm{a})^{\top}\bm{u}_{l}\\
= & (\bm{P}_{\bm{U}^{\star}}\bm{a})^{\top}(\bm{u}_{l}^{\star}\cos\theta+\bm{u}_{l,\perp}\sin\theta)+\langle\bm{P}_{\bm{U}^{\star\perp}}\bm{a},\,\bm{u}_{l}\rangle\\
= & \sum_{k=1}^{r}\bm{a}^{\top}\bm{u}_{k}^{\star}\bm{u}_{k}^{\star\top}(\bm{u}_{l}^{\star}\cos\theta+\bm{u}_{l,\perp}\sin\theta)+\langle\bm{P}_{\bm{U}^{\star\perp}}\bm{a},\,\bm{P}_{\bm{U}^{\star\perp}}\bm{u}_{l}\rangle\\
= & \bm{a}^{\top}\bm{u}_{l}^{\star}\cos\theta+\sum_{k:k\ne l}\bm{a}^{\top}\bm{u}_{k}^{\star}\bm{u}_{k}^{\star\top}\bm{u}_{l,\perp}\sin\theta+\langle\bm{P}_{\bm{U}^{\star\perp}}\bm{a},\,\bm{P}_{\bm{U}^{\star\perp}}\bm{u}_{l}\rangle,
\end{align*}
where the third line relies on the fact $\bm{P}_{\bm{U}^{\star\perp}}\bm{P}_{\bm{U}^{\star\perp}}=\bm{P}_{\bm{U}^{\star\perp}}$.
It then follows that
\[
\bm{a}^{\top}\bm{u}_{l}\pm\bm{a}^{\top}\bm{u}_{l}^{\star}=\bm{a}^{\top}\bm{u}_{l}^{\star}(\cos\theta\pm1)+\sum_{k:k\ne l}\bm{a}^{\top}\bm{u}_{k}^{\star}\bm{u}_{k}^{\star\top}\bm{u}_{l,\perp}\sin\theta+\langle\bm{P}_{\bm{U}^{\star\perp}}\bm{a},\,\bm{P}_{\bm{U}^{\star\perp}}\bm{u}_{l}\rangle,
\]
allowing us to deduce that
\begin{align}
\min\big|\bm{a}^{\top}\bm{u}_{l}\pm\bm{a}^{\top}\bm{u}_{l}^{\star}\big|\leq & \big|\bm{a}^{\top}\bm{u}_{l}^{\star}\big|(1-|\cos\theta|)+\Big|\sum_{k:k\ne l}\bm{a}^{\top}\bm{u}_{k}^{\star}\bm{u}_{k}^{\star\top}\bm{u}_{l,\perp}\sin\theta\Big|+\big|\langle\bm{P}_{\bm{U}^{\star\perp}}\bm{a},\,\bm{P}_{\bm{U}^{\star\perp}}\bm{u}_{l}\rangle\big|\nonumber \\
\leq & \big|\bm{a}^{\top}\bm{u}_{l}^{\star}\big|(1-\cos^{2}\theta)+\Big|\sum_{k:k\ne l}\bm{a}^{\top}\bm{u}_{k}^{\star}\bm{u}_{k}^{\star\top}\bm{u}_{l,\perp}\sin\theta\Big|+\big|\langle\bm{P}_{\bm{U}^{\star\perp}}\bm{a},\,\bm{P}_{\bm{U}^{\star\perp}}\bm{u}_{l}\rangle\big|.\label{eq:au-min-UB-iid-sym-1}
\end{align}
and
\begin{align}
\min\big|\bm{a}^{\top}\bm{u}_{l}\sqrt{1+b_{l}}\pm\bm{a}^{\top}\bm{u}_{l}^{\star}\big| & \leq\big|\bm{a}^{\top}\bm{u}_{l}^{\star}\big|\cdot\big|1-\sqrt{1+b_{l}}|\cos\theta|\big|+\sqrt{1+b_{l}}\big|\langle\bm{P}_{\bm{U}^{\star\perp}}\bm{a},\,\bm{P}_{\bm{U}^{\star\perp}}\bm{u}_{l}\rangle\big|\nonumber \\
 & \quad+\sqrt{1+b_{l}}\Big|\sum_{k:k\ne l}\bm{a}^{\top}\bm{u}_{k}^{\star}\bm{u}_{k}^{\star\top}\bm{u}_{l,\perp}\sin\theta\Big|.\label{eq:au-min-UB-iid-sym-1-bl}
\end{align}
As a result, it boils down to bounding the terms
\[
1-\cos^{2}\theta,\quad1-\sqrt{1+b_{l}}|\cos\theta|,\quad\sqrt{1+b_{l}},\quad\sum_{k:k\ne l}\bm{a}^{\top}\bm{u}_{k}^{\star}\bm{u}_{k}^{\star\top}\bm{u}_{l,\perp}\sin\theta,\quad\text{and}\quad\langle\bm{P}_{\bm{U}^{\star\perp}}\bm{a},\,\bm{P}_{\bm{U}^{\star\perp}}\bm{u}_{l}\rangle.
\]

We claim that $\lambda_{l}\bm{I}_{n-1}-(\bm{u}_{l}^{\star\perp})^{\top}\bm{M}\bm{u}_{l}^{\star\perp}$
is invertible. This can be seen from (\ref{eq:lambdal-tilde-lambdak-LB})
stated in Lemma \ref{lemma:lambda-M-minus-spectrum} directly, whose
validation is independent with this claim. The invertibility taken
together with Theorem \ref{thm:master-thm-vector} reveals that $\cos\theta\neq0$.
If $\sin\theta=0$, then we have $\bm{u}_{l}=\pm\bm{u}_{l}^{\star}$
and the conclusion is obvious since $\min\big|\bm{a}^{\top}\bm{u}_{l}\pm\bm{a}^{\top}\bm{u}_{l}^{\star}\big|=0$.
Therefore, we shall assume $\cos\theta\neq0$ and $\sin\theta\neq0$
in the remainder of the proof. Invoking Theorem \ref{thm:master-thm-vector}
yields\begin{subequations}\label{eq:costheta-au-iid-1} 
\begin{align}
\cos^{2}\theta & =\frac{1}{1+\big\|\big(\lambda_{l}\bm{I}_{n-1}-(\bm{u}_{l}^{\star\perp})^{\top}\bm{M}\bm{u}_{l}^{\star\perp}\big)^{-1}(\bm{u}_{l}^{\star\perp})^{\top}\bm{M}\bm{u}_{l}^{\star}\big\|_{2}^{2}},\\
\bm{u}_{k}^{\star\top}\bm{u}_{l,\perp} & =\frac{\bm{u}_{k}^{\star\top}\bm{u}_{l}^{\star\perp}\big(\lambda_{l}\bm{I}_{n-1}-(\bm{u}_{l}^{\star\perp})^{\top}\bm{M}\bm{u}_{l}^{\star\perp}\big)^{-1}(\bm{u}_{l}^{\star\perp})^{\top}\bm{M}\bm{u}_{l}^{\star}}{\big\|\bm{u}_{l}^{\star\perp}\big(\lambda_{l}\bm{I}_{n-1}-(\bm{u}_{l}^{\star\perp})^{\top}\bm{M}\bm{u}_{l}^{\star\perp}\big)^{-1}(\bm{u}_{l}^{\star\perp})^{\top}\bm{M}\bm{u}_{l}^{\star}\big\|_{2}}.
\end{align}
\end{subequations} Recognizing that $\bm{M}^{\star}\bm{u}_{l}^{\star}=\lambda_{l}^{\star}\bm{u}_{l}^{\star}$,
we can alternatively write (\ref{eq:costheta-au-iid-1}) as follows\begin{subequations}\label{eq:costheta-au-expression}
\begin{align}
\cos^{2}\theta & =\frac{1}{1+\big\|\big(\lambda_{l}\bm{I}_{n-1}-\bm{M}^{(l)}\big)^{-1}(\bm{u}_{l}^{\star\perp})^{\top}\bm{H}\bm{u}_{l}^{\star}\big\|_{2}^{2}},\label{eq:cos-eig-linear-form}\\
\bm{u}_{k}^{\star\top}\bm{u}_{l,\perp} & =\frac{\bm{u}_{k}^{\star\top}\bm{u}_{l}^{\star\perp}\big(\lambda_{l}\bm{I}_{n-1}-\bm{M}^{(l)}\big)^{-1}(\bm{u}_{l}^{\star\perp})^{\top}\bm{H}\bm{u}_{l}^{\star}}{\big\|\bm{u}_{l}^{\star\perp}\big(\lambda_{l}\bm{I}_{n-1}-\bm{M}^{(l)}\big)^{-1}(\bm{u}_{l}^{\star\perp})^{\top}\bm{H}\bm{u}_{l}^{\star}\big\|_{2}}.\label{eq:a-u-eig-linear-form}
\end{align}
\end{subequations}Here, we define 
\begin{align}
\bm{M}^{(l)} & \coloneqq(\bm{u}_{l}^{\star\perp})^{\top}\bm{M}\bm{u}_{l}^{\star\perp}=(\bm{u}_{l}^{\star\perp})^{\top}\bm{M}^{\star}\bm{u}_{l}^{\star\perp}+(\bm{u}_{l}^{\star\perp})^{\top}\bm{H}\bm{u}_{l}^{\star\perp}.\label{eq:definition-M-slash-l}
\end{align}
With the above relations in mind, we can demonstrate that 
\begin{align*}
 & \big|\sum_{k:k\ne l}\bm{a}^{\top}\bm{u}_{k}^{\star}\bm{u}_{k}^{\star\top}\bm{u}_{l,\perp}\sin\theta\big|=\big|\sum_{k:k\ne l}\bm{a}^{\top}\bm{u}_{k}^{\star}\bm{u}_{k}^{\star\top}\bm{u}_{l,\perp}\big|\sqrt{1-\cos^{2}\theta}\\
 & =\frac{\big|\sum_{k:k\ne l}\bm{a}^{\top}\bm{u}_{k}^{\star}\bm{u}_{k}^{\star\top}\bm{u}_{l}^{\star\perp}\big(\lambda_{l}\bm{I}-\bm{M}^{(l)}\big)^{-1}(\bm{u}_{l}^{\star\perp})^{\top}\bm{H}\bm{u}_{l}^{\star}\big|}{\big\|\bm{u}_{l}^{\star\perp}\big(\lambda_{l}\bm{I}_{n-1}-\bm{M}^{(l)}\big)^{-1}(\bm{u}_{l}^{\star\perp})^{\top}\bm{H}\bm{u}_{l}^{\star}\big\|_{2}}\cdot\sqrt{\frac{\big\|\big(\lambda_{l}\bm{I}_{n-1}-\bm{M}^{(l)}\big)^{-1}(\bm{u}_{l}^{\star\perp})^{\top}\bm{H}\bm{u}_{l}^{\star}\big\|_{2}^{2}}{1+\big\|\big(\lambda_{l}\bm{I}_{n-1}-\bm{M}^{(l)}\big)^{-1}(\bm{u}_{l}^{\star\perp})^{\top}\bm{H}\bm{u}_{l}^{\star}\big\|_{2}^{2}}}\\
 & \leq\Big|\sum_{k:k\ne l}\bm{a}^{\top}\bm{u}_{k}^{\star}\bm{u}_{k}^{\star\top}\bm{u}_{l}^{\star\perp}\big(\lambda_{l}\bm{I}_{n-1}-\bm{M}^{(l)}\big)^{-1}(\bm{u}_{l}^{\star\perp})^{\top}\bm{H}\bm{u}_{l}^{\star}\Big|,
\end{align*}
where the last inequality comes from the fact $\big\|\bm{u}_{l}^{\star\perp}\big(\lambda_{l}\bm{I}-\bm{M}^{(l)}\big)^{-1}(\bm{u}_{l}^{\star\perp})^{\top}\bm{H}\bm{u}_{l}^{\star}\big\|_{2}=\big\|\big(\lambda_{l}\bm{I}-\bm{M}^{(l)}\big)^{-1}(\bm{u}_{l}^{\star\perp})^{\top}\bm{H}\bm{u}_{l}^{\star}\big\|_{2}$
(since the columns of $\bm{u}_{l}^{\star\perp}$ are orthonormal).
Substituting this into (\ref{eq:au-min-UB-iid-sym-1}) and (\ref{eq:au-min-UB-iid-sym-1-bl})
yields 
\begin{align}
\min\big|\bm{a}^{\top}\bm{u}_{l}\pm\bm{a}^{\top}\bm{u}_{l}^{\star}\big| & \leq\big|\bm{a}^{\top}\bm{u}_{l}^{\star}\big|\cdot(1-\cos^{2}\theta)+\big|\langle\bm{P}_{\bm{U}^{\star\perp}}\bm{a},\,\bm{P}_{\bm{U}^{\star\perp}}\bm{u}_{l}\rangle\big|\nonumber \\
 & \quad+\Big|\sum_{k:k\ne l}\bm{a}^{\top}\bm{u}_{k}^{\star}\cdot\bm{u}_{k}^{\star\top}\bm{u}_{l}^{\star\perp}\big(\lambda_{l}\bm{I}-\bm{M}^{(l)}\big)^{-1}(\bm{u}_{l}^{\star\perp})^{\top}\bm{H}\bm{u}_{l}^{\star}\Big|;\label{eq:au-min-UB-iid-sym}
\end{align}
and
\begin{align}
\min\big|\sqrt{1+b_{l}}\bm{a}^{\top}\bm{u}_{l}\pm\bm{a}^{\top}\bm{u}_{l}^{\star}\big| & \leq\big|\bm{a}^{\top}\bm{u}_{l}^{\star}\big|\cdot\big|1-\sqrt{1+b_{l}}|\cos\theta|\big|+\sqrt{1+b_{l}}\big|\langle\bm{P}_{\bm{U}^{\star\perp}}\bm{a},\,\bm{P}_{\bm{U}^{\star\perp}}\bm{u}_{l}\rangle\big|\nonumber \\
 & \quad+\sqrt{1+b_{l}}\Big|\sum_{k:k\ne l}\bm{a}^{\top}\bm{u}_{k}^{\star}\cdot\bm{u}_{k}^{\star\top}\bm{u}_{l}^{\star\perp}\big(\lambda_{l}\bm{I}-\bm{M}^{(l)}\big)^{-1}(\bm{u}_{l}^{\star\perp})^{\top}\bm{H}\bm{u}_{l}^{\star}\Big|.\label{eq:au-min-UB-iid-sym-bl}
\end{align}
In what follows, we shall control these quantities separately.

\paragraph{1. Controlling the spectrum of $\bm{M}^{(l)}$.} Before
proceeding, we find it helpful to first study the spectrum of $\bm{M}^{(l)}$.
Let $\{\lambda_{i}^{(l)}\}_{i=1}^{n-1}$ denote the eigenvalues of
$\bm{M}^{(l)}$ with $|\lambda_{1}^{(l)}|\geq|\lambda_{2}^{(l)}|\geq\cdots\geq|\lambda_{n-1}^{(l)}|$
with associate eigenvectors $\{\bm{u}_{i}^{(l)}\}_{i=1}^{n-1}$. In
addition, we define several matrices as follows\begin{subequations}\label{eq:def:U-Lambda-star-l}
\begin{align}
\bm{U}_{\smallsetminus l}^{\star} & \coloneqq[\bm{u}_{1}^{\star},\cdots,\bm{u}_{l-1}^{\star},\bm{u}_{l+1}^{\star},\cdots,\bm{u}_{r}^{\star}]\in\mathbb{R}^{n\times(r-1)},\\
\bm{U}^{\star(l)} & \coloneqq(\bm{u}_{l}^{\star\perp})^{\top}\bm{U}_{\smallsetminus l}^{\star}=\begin{bmatrix}\bm{I}_{r-1}\\
\bm{0}
\end{bmatrix}\in\mathbb{R}^{(n-1)\times(r-1)},\\
\bm{U}^{\star(l)\perp} & \coloneqq(\bm{u}_{l}^{\star\perp})^{\top}\bm{U}^{\star\perp}=\begin{bmatrix}\bm{0}\\
\bm{I}_{n-r}
\end{bmatrix}\in\mathbb{R}^{(n-1)\times(n-r)},\\
\bm{\Lambda}^{\star(l)} & \coloneqq\mathsf{diag}\big(\{\lambda_{i}^{\star}\}_{i\neq l}\big)\in\mathbb{R}^{(r-1)\times(r-1)},
\end{align}
\end{subequations}and define 
\[
\bm{u}_{k,\parallel}^{(l)}\coloneqq\frac{1}{\|\bm{P}_{\bm{U}^{\star(l)}}\bm{u}_{k}^{(l)}\|_{2}}\bm{P}_{\bm{U}^{\star(l)}}\bm{u}_{k}^{(l)}
\]
 for each $k\neq l$.

Armed with this set of notation, we are ready to present Lemma~\ref{lemma:lambda-M-minus-spectrum},
which studies the eigenvalues of $\bm{M}^{(l)}$.

\begin{lemma}\label{lemma:lambda-M-minus-spectrum}Instate the assumptions
of Theorem \ref{thm:evector-pertur-sym-iid}, and recall the definition
of $\mathcal{E}_{\mathsf{MD}}$ in Lemma~\ref{lemma:lambda-l-mapping}.
With probability at least $1-O(n^{-10})$, the following holds: 
\begin{enumerate}
\item For each $1\leq k<r$, one has $\lambda_{k}^{(l)}-\gamma(\lambda_{k}^{(l)})\in\mathcal{B}_{\mathcal{E}_{\mathsf{MD}}}(\lambda_{i}^{\star})$
for some $i\ne l$, and
\[
\big\|\big(\lambda_{k}^{(l)}\bm{I}_{r-1}-\gamma(\lambda_{k}^{(l)})\bm{I}_{r-1}-\bm{\Lambda}^{\star(l)}\big)\bm{U}^{\star(l)\top}\bm{u}_{k,\parallel}^{(l)}\big\|_{2}\lesssim\sigma\sqrt{r}\log n;
\]
\item For each $k\geq r$, one has $|\lambda_{k}^{(l)}|\lesssim\sigma\sqrt{n}$; 
\item Moreover, one has
\[
\big|\lambda-\lambda_{l}\big|\gtrsim\begin{cases}
\Delta_{l}^{\star}, & \text{if}\ \lambda-\gamma(\lambda)\in\mathcal{B}_{\mathcal{E}_{\mathsf{MD}}}(\lambda_{k}^{\star})\,\text{\text{for some }}k\neq l\text{ and }1\leq k\leq r;\\
|\lambda_{l}^{\star}|, & \text{\text{if}}\ |\lambda|\lesssim\sigma\sqrt{n}.
\end{cases}
\]
In particular, we have
\begin{equation}
\big|\lambda_{k}^{(l)}-\lambda_{l}\big|\gtrsim\begin{cases}
\Delta_{l}^{\star}, & 1\leq k<r;\\
|\lambda_{l}^{\star}|, & k\geq r.
\end{cases}\label{eq:lambdal-tilde-lambdak-LB}
\end{equation}
\end{enumerate}
\end{lemma}\begin{proof}See Appendix \ref{subsec:Proof-of-Lemma:lambda-M-minus-spectrum}.\end{proof}In
words, this lemma tells us that: 
\begin{itemize}
\item For any $1\leq k\leq r$, the properly corrected $\lambda_{k}^{(l)}$
(namely, $\lambda_{k}^{(l)}-\gamma(\lambda_{k}^{(l)})$) stays very
close to one of the true non-zero eigenvalues excluding $\lambda_{l}^{\star}$; 
\item For any $k\geq r$, the eigenvalue $\lambda_{k}^{(l)}$ is reasonably
small; 
\item Any eigenvalue of $\bm{M}^{(l)}$ is sufficiently separated from the
$l$-th eigenvalue $\lambda_{l}$ of $\bm{M}$, where the separation
is lower bounded by the order of the associated eigen-gap. 
\end{itemize}
\paragraph{2. Controlling $\cos^{2}\theta$.} We now turn to bounding
$\cos^{2}\theta$. In view of the expression of $\cos^{2}\theta$
in (\ref{eq:cos-eig-linear-form}), it suffices to look at $\big\|\big(\lambda_{l}\bm{I}-\bm{M}^{(l)}\big)^{-1}(\bm{u}_{l}^{\star\perp})^{\top}\bm{H}\bm{u}_{l}^{\star}\big\|_{2}$.
A simple yet crucial observation is that: the matrix $(\bm{u}_{l}^{\star\perp})^{\top}\bm{H}\bm{u}_{l}^{\star\perp}$
is independent of $(\bm{u}_{l}^{\star\perp})^{\top}\bm{H}\bm{u}_{l}^{\star}$
(which follows from the same argument as in the proof of Lemma \ref{lemma:bound-Glambda-uniform}
in Appendix~\ref{subsec:Proof-of-Lemma-bound-Glambda-uniform}).
Consequently, $\bm{M}^{(l)}$ (defined in (\ref{eq:definition-M-slash-l}))
is independent of $(\bm{u}_{l}^{\star\perp})^{\top}\bm{H}\bm{u}_{l}^{\star}\sim\mathcal{N}(\bm{0},\sigma^{2}\bm{I}_{n-1})$,
which is a Gaussian random vector in $\mathbb{R}^{n-1}$. In light
of this observation, we can bound $\big\|\big(\lambda_{l}\bm{I}-\bm{M}^{(l)}\big)^{-1}(\bm{u}_{l}^{\star\perp})^{\top}\bm{H}\bm{u}_{l}^{\star}\big\|_{2}$
as follows. 

\begin{lemma}\label{lemma:lambda-M-l-inv-u-perp-H-u-l2-norm}Instate
the assumptions of Theorem \ref{thm:evector-pertur-sym-iid}. The
following holds with probability at least $1-O(n^{-10})$:
\begin{equation}
\big\|\big(\lambda_{l}\bm{I}_{n-1}-\bm{M}^{(l)}\big)^{-1}(\bm{u}_{l}^{\star\perp})^{\top}\bm{H}\bm{u}_{l}^{\star}\big\|_{2}^{2}=\sum_{k:\,r<k\le n}\frac{\sigma^{2}}{(\lambda_{l}-\lambda_{k})^{2}}+O\bigg(\frac{\sigma^{2}r\log n}{\big(\Delta_{l}^{\star}\big)^{2}}+\frac{\sigma^{2}\sqrt{n\log n}}{\lambda_{l}^{\star2}}\bigg),\label{eq:lambda-M-124}
\end{equation}
which further indicates that
\begin{equation}
\big\|\big(\lambda_{l}\bm{I}_{n-1}-\bm{M}^{(l)}\big)^{-1}(\bm{u}_{l}^{\star\perp})^{\top}\bm{H}\bm{u}_{l}^{\star}\big\|_{2}^{2}\asymp \frac{\sigma^{2}n}{\lambda_{l}^{\star2}}+O\bigg(\frac{\sigma^{2}r\log n}{\big(\Delta_{l}^{\star}\big)^{2}}\bigg)\ll1.\label{eq:lambda-M-123}
\end{equation}
\end{lemma}\begin{proof}See Appendix \ref{subsec:Proof-of-lemma:lambda-M-l-inv-u-perp-H-u-l2-norm}.\end{proof}

Combining this lemma with (\ref{eq:cos-eig-linear-form}), we reach
\begin{equation}
1-\cos^{2}\theta=1-\frac{1}{1+\big\|\big(\lambda_{l}\bm{I}_{n-1}-\bm{M}^{(l)}\big)^{-1}(\bm{u}_{l}^{\star\perp})^{\top}\bm{H}\bm{u}_{l}^{\star}\big\|_{2}^{2}}\asymp \frac{\sigma^{2}n}{\lambda_{l}^{\star2}}+O\bigg(\frac{\sigma^{2}r\log n}{\big(\Delta_{l}^{\star}\big)^{2}}\bigg)\ll1,\label{eq:cos-bound}
\end{equation}
where the last step arises from the assumption (\ref{eq:eigengap-condition-iid}).
In addition, recalling the de-bias parameter $b_{l}$
\[
b_{l}=\sum_{k:\,r<k\le n}\frac{\sigma^{2}}{(\lambda_{l}-\lambda_{k})^{2}},
\]
one arrives at
\begin{align}
\big|(1+b_{l})\cos^{2}\theta-1\big| & =\Bigg|\frac{1+b_{l}}{1+\big\|\big(\lambda_{l}\bm{I}_{n-1}-\bm{M}^{(l)}\big)^{-1}(\bm{u}_{l}^{\star\perp})^{\top}\bm{H}\bm{u}_{l}^{\star}\big\|_{2}^{2}}-1\Bigg|\nonumber \\
 & =\frac{\Big|b_{l}-\big\|\big(\lambda_{l}\bm{I}_{n-1}-\bm{M}^{(l)}\big)^{-1}(\bm{u}_{l}^{\star\perp})^{\top}\bm{H}\bm{u}_{l}^{\star}\big\|_{2}^{2}\Big|}{1+\big\|\big(\lambda_{l}\bm{I}_{n-1}-\bm{M}^{(l)}\big)^{-1}(\bm{u}_{l}^{\star\perp})^{\top}\bm{H}\bm{u}_{l}^{\star}\big\|_{2}^{2}}\nonumber \\
 & \leq\Big|b_{l}-\big\|\big(\lambda_{l}\bm{I}_{n-1}-\bm{M}^{(l)}\big)^{-1}(\bm{u}_{l}^{\star\perp})^{\top}\bm{H}\bm{u}_{l}^{\star}\big\|_{2}^{2}\Big|\nonumber \\
 & \overset{(\mathrm{i})}{\lesssim}\frac{\sigma^{2}r\log n}{\big(\Delta_{l}^{\star}\big)^{2}}+\frac{\sigma^{2}\sqrt{n\log n}}{\lambda_{l}^{\star2}}\overset{(\mathrm{ii})}{\ll}1,\label{eq:cos-value}
\end{align}
where (i) follows from (\ref{eq:lambda-M-124}) and (ii) is due to
the assumption (\ref{eq:eigengap-condition-iid}). Combined with (\ref{eq:cos-bound}),
this further allows us to obtain $1+b_{l}\lesssim1$ and
\begin{equation}
\big|1-\sqrt{1+b_{l}}|\cos\theta|\big|=\bigg|\frac{1-(1+b_{l})\cos^{2}\theta}{1+\sqrt{1+b_{l}}|\cos\theta|}\bigg|\lesssim\big|1-(1+b_{l})\cos^{2}\theta\big|\lesssim\frac{\sigma^{2}r\log n}{\big(\Delta_{l}^{\star}\big)^{2}}+\frac{\sigma^{2}\sqrt{n\log n}}{\lambda_{l}^{\star2}}.\label{eq:cos-bound-bl}
\end{equation}

\paragraph{3. Controlling $\sum_{k:k\ne l}\bm{a}^{\top}\bm{u}_{k}^{\star}\cdot\bm{u}_{k}^{\star\top}\bm{u}_{l}^{\star\perp}\big(\lambda_{l}\bm{I}-\bm{M}^{(l)}\big)^{-1}(\bm{u}_{l}^{\star\perp})^{\top}\bm{H}\bm{u}_{l}^{\star}$.}
The key observation is that $(\bm{u}_{l}^{\star\perp})^{\top}\bm{H}\bm{u}_{l}^{\star}\sim\mathcal{N}(\bm{0},\bm{I}_{n-1})$
is independent of $\bm{M}^{(l)}$ (but dependent of $\lambda_{l}$).
This term can be bounded via the following lemma, which will be established
in Appendix \ref{subsec:Proof-of-lemma:a-top-P-Uk}.

\begin{lemma}\label{lemma:a-top-P-Uk}Instate the assumptions of
Theorem \ref{thm:evector-pertur-sym-iid}. With probability at least
$1-O(n^{-10})$, one has
\begin{align}
 & \Big|\sum_{k:k\ne l}\bm{a}^{\top}\bm{u}_{k}^{\star}\cdot\bm{u}_{k}^{\star\top}\bm{u}_{l}^{\star\perp}\big(\lambda_{l}\bm{I}_{n-1}-\bm{M}^{(l)}\big)^{-1}(\bm{u}_{l}^{\star\perp})^{\top}\bm{H}\bm{u}_{l}^{\star}\Big|\nonumber \\
 & \qquad\qquad\lesssim\frac{\sigma}{|\lambda_{l}^{\star}|}\sqrt{\log\bigg(\frac{n\kappa\lambda_{\max}}{\Delta_{l}^{\star}}\bigg)}+\sigma\sqrt{r\log\bigg(\frac{n\kappa\lambda_{\max}}{\Delta_{l}^{\star}}\bigg)}\sum_{k:k\neq l}\frac{\big|\bm{a}^{\top}\bm{u}_{k}^{\star}\big|}{|\lambda_{l}^{\star}-\lambda_{k}^{\star}|}.\label{eq:a-top-uk}
\end{align}
\end{lemma}

\paragraph{4. Controlling $\langle\bm{P}_{\bm{U}^{\star\perp}}\bm{a},\,\bm{P}_{\bm{U}^{\star\perp}}\bm{u}_{l}\rangle$.}
When it comes to the last term $\langle\bm{P}_{\bm{U}^{\star\perp}}\bm{a},\,\bm{P}_{\bm{U}^{\star\perp}}\bm{u}_{l}\rangle$,
one can take advantage of the rotational invariance of $\bm{P}_{\bm{U}^{\star\perp}}\bm{u}_{l}$
in the subspace spanned by $\bm{U}^{\star\perp}$ to upper bound it.
This is formalized in Lemma \ref{lemma:a-top-P-U-perp-u-perp}, with
the proof postponed to Appendix \ref{subsec:Proof-of-lemma:a-top-P-U-perp-u-perp}.

\begin{lemma}\label{lemma:a-top-P-U-perp-u-perp}Instate the assumptions
of Theorem \ref{thm:evector-pertur-sym-iid}. With probability at
least $1-O(n^{-10})$, 
\begin{equation}
\big|\langle\bm{P}_{\bm{U}^{\star\perp}}\bm{a},\,\bm{P}_{\bm{U}^{\star\perp}}\bm{u}_{l}\rangle\big|\lesssim\sqrt{\frac{\log n}{n}}\,\big\|\bm{P}_{\bm{U}^{\star\perp}}\bm{a}\big\|_{2}\big\|\bm{P}_{\bm{U}^{\star\perp}}\bm{u}_{l}\big\|_{2}.\label{eq:a-top-U-perp-u-UB}
\end{equation}
\end{lemma}Consequently, it remains to upper bound $\big\|\bm{P}_{\bm{U}^{\star\perp}}\bm{u}_{l}\big\|_{2}$.
Recall $\bm{u}_{l,\parallel}\coloneqq\bm{P}_{\bm{U}^{\star}}(\bm{u})/\|\bm{P}_{\bm{U}^{\star}}(\bm{u})\|_{2}$
defined in in Section~\ref{subsec:Master-theorems}. By virtue of
Theorem~\ref{thm:master-theorem-general}, one has
\begin{align}
\big\|\bm{P}_{\bm{U}^{\star\perp}}\bm{u}_{l}\big\|_{2}^{2} & =1-\frac{1}{1+\big\|\big(\lambda_{l}\bm{\bm{I}}_{n-r}-(\bm{U}^{\star\perp})^{\top}\bm{M}\bm{U}^{\star\perp}\big)^{-1}(\bm{U}^{\star\perp})^{\top}\bm{M}\bm{u}_{l,\parallel}\big\|_{2}^{2}}\nonumber \\
 & \le\big\|\big(\lambda_{l}\bm{\bm{I}}_{n-r}-(\bm{U}^{\star\perp})^{\top}\bm{M}\bm{U}^{\star\perp}\big)^{-1}(\bm{U}^{\star\perp})^{\top}\bm{M}\bm{u}_{l,\parallel}\big\|_{2}^{2}\nonumber \\
 & \leq\big\|\big(\lambda_{l}\bm{\bm{I}}_{n-r}-(\bm{U}^{\star\perp})^{\top}\bm{H}\bm{U}^{\star\perp}\big)^{-1}\big\|^{2}\|(\bm{U}^{\star\perp})^{\top}\bm{H}\bm{u}_{l,\parallel}\|_{2}^{2},\label{eq:P-Uu-bound}
\end{align}
where the last inequality makes use of the fact that
\[
(\bm{U}^{\star\perp})^{\top}\bm{M}=(\bm{U}^{\star\perp})^{\top}\bm{M}^{\star}+(\bm{U}^{\star\perp})^{\top}\bm{H}=(\bm{U}^{\star\perp})^{\top}\bm{H}.
\]
Additionally, it is easily seen that
\begin{align*}
\|\lambda_{l}\bm{\bm{I}}_{n-r}-(\bm{U}^{\star\perp})^{\top}\bm{H}\bm{U}^{\star\perp}\| & \ge|\lambda_{l}|-\|(\bm{U}^{\star\perp})^{\top}\bm{H}\bm{U}^{\star\perp}\|\gtrsim|\lambda_{l}^{\star}|\\
\|\bm{U}^{\star\perp\top}\bm{H}\bm{u}_{l,\parallel}\|_{2} & \le\|\bm{H}\|\lesssim\sigma\sqrt{n}
\end{align*}
with high probability. These combined with (\ref{eq:P-Uu-bound})
lead to
\begin{equation}
\big\|\bm{P}_{\bm{U}^{\star\perp}}\bm{u}_{l}\big\|_{2}\lesssim\frac{\sigma\sqrt{n}}{|\lambda_{l}^{\star}|}.
\end{equation}
Substitution into (\ref{eq:a-top-U-perp-u-UB}) reveals that
\begin{equation}
\big|\langle\bm{P}_{\bm{U}^{\star\perp}}\bm{a},\,\bm{P}_{\bm{U}^{\star\perp}}\bm{u}_{l}\rangle\big|\lesssim\sqrt{\frac{\log n}{n}}\,\big\|\bm{P}_{\bm{U}^{\star\perp}}\bm{a}\big\|_{2}\big\|\bm{P}_{\bm{U}^{\star\perp}}\bm{u}_{l}\big\|_{2}\lesssim\frac{\sigma\sqrt{\log n}}{|\lambda_{l}^{\star}|}\big\|\bm{P}_{\bm{U}^{\star\perp}}\bm{a}\big\|_{2}.\label{eq:a-top-U-perp-u}
\end{equation}

\paragraph{5. Combining bounds.}In view of (\ref{eq:au-min-UB-iid-sym}),
the bounds (\ref{eq:cos-bound}), (\ref{eq:a-top-uk}) and (\ref{eq:a-top-U-perp-u})
taken collectively lead to our advertised result
\begin{align*}
\min\big|\bm{a}^{\top}\bm{u}_{l}\pm\bm{a}^{\top}\bm{u}_{l}^{\star}\big| & \lesssim\bigg(\frac{\sigma^{2}n}{\lambda_{l}^{\star2}}+\frac{\sigma^{2}r\log n}{\big(\Delta_{l}^{\star}\big)^{2}}\bigg)\big|\bm{a}^{\top}\bm{u}_{l}^{\star}\big|+ \frac{\sigma\sqrt{\log n}}{|\lambda_{l}^{\star}|}\big\|\bm{P}_{\bm{U}^{\star\perp}}\bm{a}\big\|_{2} \\
 & \quad +\sigma\sqrt{r\log\bigg(\frac{n\kappa\lambda_{\max}}{\Delta_{l}^{\star}}\bigg)}\sum_{k:k\neq l}\frac{\big|\bm{a}^{\top}\bm{u}_{k}^{\star}\big|}{|\lambda_{l}^{\star}-\lambda_{k}^{\star}|} +\frac{\sigma}{|\lambda_{l}^{\star}|}\sqrt{\log\bigg(\frac{n\kappa\lambda_{\max}}{\Delta_{l}^{\star}}\bigg)}.
\end{align*}

Regarding the analysis for the de-biased estimate, one can substitute
(\ref{eq:cos-bound-bl}), (\ref{eq:a-top-uk}) and (\ref{eq:a-top-U-perp-u})
into (\ref{eq:au-min-UB-iid-sym-bl}) to obtain
\begin{align*}
\min\big|\bm{a}^{\top}\bm{u}_{l}\sqrt{1+b_{l}}\pm\bm{a}^{\top}\bm{u}_{l}^{\star}\big| & \lesssim\bigg(\frac{\sigma^{2}\sqrt{n\log n}}{\lambda_{l}^{\star2}}+\frac{\sigma^{2}r\log n}{\big(\Delta_{l}^{\star}\big)^{2}}\bigg)\big|\bm{a}^{\top}\bm{u}_{l}^{\star}\big| +\frac{\sigma\sqrt{\log n}}{|\lambda_{l}^{\star}|}\big\|\bm{P}_{\bm{U}^{\star\perp}}\bm{a}\big\|_{2} \\
 & \quad +\sigma\sqrt{r\log\bigg(\frac{n\kappa\lambda_{\max}}{\Delta_{l}^{\star}}\bigg)}\sum_{k:k\neq l}\frac{\big|\bm{a}^{\top}\bm{u}_{k}^{\star}\big|}{|\lambda_{l}^{\star}-\lambda_{k}^{\star}|} +\frac{\sigma}{|\lambda_{l}^{\star}|}\sqrt{\log\bigg(\frac{n\kappa\lambda_{\max}}{\Delta_{l}^{\star}}\bigg)} \\
 & \lesssim\frac{\sigma^{2}r\log n}{\big(\Delta_{l}^{\star}\big)^{2}}\big|\bm{a}^{\top}\bm{u}_{l}^{\star}\big|+\sigma\sqrt{r\log\bigg(\frac{n\kappa\lambda_{\max}}{\Delta_{l}^{\star}}\bigg)}\sum_{k:k\neq l}\frac{\big|\bm{a}^{\top}\bm{u}_{k}^{\star}\big|}{|\lambda_{l}^{\star}-\lambda_{k}^{\star}|}+\frac{\sigma}{|\lambda_{l}^{\star}|}\sqrt{\log\bigg(\frac{n\kappa\lambda_{\max}}{\Delta_{l}^{\star}}\bigg)}
\end{align*}
where the last step holds since $\sigma\sqrt{n}\lesssim\lambda_{\min}^{\star}$
and $\big|\bm{a}^{\top}\bm{u}_{l}^{\star}\big|\leq\|\bm{a}\|_{2}\|\bm{u}_{l}^{\star}\|_{2}=1.$
This concludes the proof.

\subsection{Analysis for principal component analysis}

Akin to the matrix denoising counterpart, the first step towards establishing
the desired eigenvector perturbation bounds lies in the development
of a fine-grained eigenvalue perturbation theory. Here and throughout,
we let $\bm{U}^{\star\perp}\in\mathbb{R}^{p\times(p-r)}$ represent
a matrix consisting of orthonormal columns perpendicular to the subspace
spanned by $\bm{U}^{\star}$.

\begin{theorem}[Eigenvalue perturbation for PCA]\label{thm:eigval-pertur-sym-iid-pca}Consider
the model in Section~\ref{subsec:Principal-component-analysis-model}.
Fix any $1\leq l\leq r$, and instate the assumptions of Theorem~\ref{thm:evector-pertur-sym-iid-pca}.
Then with probability at least $1-O(n^{-10}),$ one has
\begin{equation}
\bigg|\frac{\lambda_{l}}{1+\beta(\lambda_{l})}-\lambda_{l}^{\star}-\sigma^{2}\bigg|\leq C_{2}(\lambda_{\max}^{\star}+\sigma^{2})\sqrt{\frac{r}{n}}\log n\label{eq:eigenvalue-perturbation-bound-iid-pca}
\end{equation}
for some sufficiently large constant $C_{2}>0$, where we define
\begin{equation}
\beta(\lambda):=\frac{1}{n}\mathsf{tr}\Big[\frac{1}{n}\bm{S}_{\perp}^{\top}\big(\lambda\bm{I}_{p-r}-\frac{1}{n}\bm{S}_{\perp}\bm{S}_{\perp}^{\top}\big)^{-1}\bm{S}_{\perp}\Big]\qquad\text{with }\bm{S}_{\perp}\coloneqq(\bm{U}^{\star\perp})^{\top}\bm{S}.\label{eq:definition-gamma-lambda-iid-pca}
\end{equation}
\end{theorem}\begin{remark}As asserted by Theorem \ref{thm:eigval-pertur-sym-iid-pca},
the empirical eigenvalue $\lambda_{l}$ exhibits a form of ``inflation''
in comparison to the corresponding ground-truth value $\lambda_{l}^{\star}+\sigma^{2}$.
As a result, it is advisable to properly shrink $\lambda_{l}$ when
estimating $\lambda_{l}^{\star}+\sigma^{2}$. \end{remark}

In what follows, we shall first outline the proof for Theorem \ref{thm:eigval-pertur-sym-iid-pca}
(which is very similar to the analysis for Theorem~\ref{thm:eigval-pertur-sym-iid}),
followed by a proof sketch for the eigenvector perturbation theory
in Theorem~\ref{thm:evector-pertur-sym-iid-pca}. 

\subsubsection{Proof of eigenvalue perturbation theory (Theorem \ref{thm:eigval-pertur-sym-iid-pca})}

\label{subsec:Proof-of-eigenvalue-pca}

Before embarking on the proof, we shall define 
\begin{equation}
\bm{S}_{\parallel}:=\bm{U}^{\star\top}\bm{S}\in\mathbb{R}^{r\times n},\quad\bm{S}_{\perp}:=(\bm{U}^{\star\perp})^{\top}\bm{S}\in\mathbb{R}^{(p-r)\times n}\quad\text{and}\quad\bm{\Lambda}:=\bm{U}^{\star\top}\bm{\Sigma}\bm{U}^{\star}=\bm{\Lambda}^{\star}+\sigma^{2}\bm{I}_{r}\label{eq:def-Lambda-pca}
\end{equation}
for notional convenience, allowing one to express \begin{subequations}
\begin{align}
(\bm{U}^{\star\perp})^{\top}\bm{S}\bm{S}^{\top}\bm{U}^{\star\perp} & =\bm{S}_{\perp}\bm{S}_{\perp}^{\top},\\
(\bm{U}^{\star\perp})^{\top}\bm{S}\bm{S}^{\top}\bm{U}^{\star} & =\bm{S}_{\perp}\bm{S}_{\parallel}^{\top},\\
\bm{U}^{\star\top}\bm{S}\bm{S}^{\top}\bm{U}^{\star} & =\bm{S}_{\parallel}\bm{S}_{\parallel}^{\top}.
\end{align}
\end{subequations} As can be straightforwardly verified: 
\begin{itemize}
\item The columns of $\bm{S}_{\parallel}$ are independent zero-mean Gaussian
random vectors with covariance matrix $\bm{\Lambda}$;
\item The columns of $\bm{S}_{\perp}$ are i.i.d.~zero-mean Gaussian random
vectors with covariance matrix $\sigma^{2}\bm{I}_{p-r}$;
\item $\bm{S}_{\parallel}$ is statistically independent of $\bm{S}_{\perp}$
(from standard properties for Gaussian random vectors). 
\end{itemize}
In addition, the following lemma controls the distance between $\frac{1}{n}\bm{S}\bm{S}^{\top}$
and $\bm{\Sigma}$ when measured by the spectral norm.

\begin{lemma}\label{lemma:spectral-M-Sigma}Assume that $n\geq r$.
Then with probability at least $1-O(n^{-10})$, one has
\begin{equation}
\Big\|\frac{1}{n}\bm{S}\bm{S}^{\top}-\bm{\Sigma}\Big\|\lesssim\lambda_{\max}^{\star}\sqrt{\frac{r\log n}{n}}+\sqrt{(\lambda_{\max}^{\star}+\sigma^{2})\sigma^{2}\frac{p}{n}}\,\log n+\sigma^{2}\bigg(\sqrt{\frac{p}{n}}+\frac{p}{n}+\sqrt{\frac{\log n}{n}}\bigg).
\end{equation}
\end{lemma}\begin{proof}See Appendix \ref{subsec:Proof-of-lemma:lemma:spectral-M-Sigma}.\end{proof}\begin{remark}In
particular, under the noise assumption (\ref{eq:noise-condition-iid-pca}),
Lemma~\ref{lemma:spectral-M-Sigma} tells us that $\big\|\frac{1}{n}\bm{S}\bm{S}^{\top}-\bm{\Sigma}\big\|\ll\lambda_{\min}^{\star}$
with probability at least $1-O(n^{-10})$, which together with Weyl's
inequality gives
\begin{equation}
2\lambda_{l}^{\star}/3\leq\lambda_{l}\leq4\lambda_{l}^{\star}/3,\qquad1\leq l\leq r.\label{eq:lambda-l-range-PCA}
\end{equation}
\end{remark}

We now move on to present the proof of Theorem \ref{thm:eigval-pertur-sym-iid-pca}.
The key ingredient underlying the analysis is, once again, to invoke
our master theorem (namely, Theorem~\ref{thm:master-theorem-general}),
by treating $\frac{1}{n}\bm{S}\bm{S}^{\top}$, $\bm{\Sigma}^{\star}$
and $\bm{U}^{\star}$ as $\bm{M}$, $\bm{M}^{\star}$ and $\bm{Q}$,
respectively. Recalling the definition of
\begin{equation}
\bm{u}_{l,\parallel}\coloneqq\frac{1}{\|\bm{P}_{\bm{U}^{\star}}(\bm{u})\|_{2}}\bm{P}_{\bm{U}^{\star}}(\bm{u})
\end{equation}
as in Section~\ref{subsec:Master-theorems} (so that $\bm{U}^{\star}\bm{U}^{\star\top}\bm{u}_{l,\parallel}=\bm{u}_{l,\parallel}$),
one can invoke (\ref{eq:lambda-l-long-identity-1}) in Theorem~\ref{thm:master-theorem-general}
to derive 
\begin{equation}
\Big(\lambda_{l}\bm{I}_{r}-\frac{1}{n}\bm{S}_{\parallel}\bm{S}_{\parallel}^{\top}\Big)\bm{U}^{\star\top}\bm{u}_{l,\parallel}=\bm{K}(\lambda_{l})\,\bm{U}^{\star\top}\bm{u}_{l,\parallel},\label{eq:lambda-S-Ustar-PCA-identity}
\end{equation}
where we recall the definitions of $\bm{S}_{\parallel}$ and $\bm{S}_{\perp}$
in (\ref{eq:def-Lambda-pca}), and $\bm{K}(\lambda)$ is given by
\begin{equation}
\bm{K}(\lambda)\coloneqq\frac{1}{n}\bm{S}_{\parallel}\cdot\underbrace{\frac{1}{n}\bm{S}_{\perp}^{\top}\Big(\lambda\bm{I}_{p-r}-\frac{1}{n}\bm{S}_{\perp}\bm{S}_{\perp}^{\top}\Big)^{-1}\bm{S}_{\perp}}_{=:\,\bm{C}(\lambda)}\cdot\,\bm{S}_{\parallel}^{\top}.\label{eq:def:G-lambda-pca}
\end{equation}
It is also helpful to define
\begin{equation}
\bm{K}^{\perp}(\lambda)\coloneqq\mathbb{E}\big[\bm{K}(\lambda)\mid\bm{C}(\lambda)\big],\label{eq:def:G-lambda-exp-pca}
\end{equation}
with $\lambda$ regarded as a deterministic quantity independent of
the data samples. Then rearranging terms in (\ref{eq:lambda-S-Ustar-PCA-identity})
yields
\[
\big(\lambda_{l}\bm{I}_{r}-\bm{\Lambda}-\bm{K}^{\perp}(\lambda_{l})\big)\bm{U}^{\star\top}\bm{u}_{l,\parallel}=\Big(\frac{1}{n}\bm{S}_{\parallel}\bm{S}_{\parallel}^{\top}-\bm{\Lambda}+\bm{K}(\lambda_{l})-\bm{K}^{\perp}(\lambda_{l})\Big)\bm{U}^{\star\top}\bm{u}_{l,\parallel},
\]
which together with (\ref{eq:lambda-l-range-PCA}) results in the
following bound:
\begin{equation}
\big\|\big(\lambda_{l}\bm{I}_{r}-\bm{\Lambda}-\bm{K}^{\perp}(\lambda_{l})\big)\bm{U}^{\star\top}\bm{u}_{l,\parallel}\big\|_{2}\le\Big\|\frac{1}{n}\bm{S}_{\parallel}\bm{S}_{\parallel}^{\top}-\bm{\Lambda}\Big\|+\sup_{\lambda:\,\lambda\in[2\lambda_{l}^{\star}/3,4\lambda_{l}^{\star}/3]}\big\|\bm{K}(\lambda)-\bm{K}^{\perp}(\lambda)\big\|,\label{eq:pca-eig-val-UB}
\end{equation}
Akin to the proof of Theorem \ref{thm:eigval-pertur-sym-iid} in Section,
our goal is to show $\big(\lambda_{l}\bm{I}_{r}-\bm{\Lambda}-\bm{K}^{\perp}(\lambda_{l})\big)\bm{U}^{\star\top}\bm{u}_{l,\parallel}\approx\bm{0}$,
which would then imply that $\lambda_{l}$ is sufficiently close to
some eigenvalue of $\bm{\Lambda}+\bm{K}^{\perp}(\lambda_{l})$. In
light of this, we intend to upper bound the two terms on the right-hand
side of (\ref{eq:pca-eig-val-UB}) in the sequel.
\begin{itemize}
\item Let us first look at the first term on the right-hand side of (\ref{eq:pca-eig-val-UB}).
Since the columns of $\bm{S}_{\parallel}=\bm{U}^{\star\top}\bm{S}$
are independent~Gaussian random vectors with distribution $\mathcal{N}(\bm{0},\bm{\Lambda})$,
we can rewrite 
\begin{equation}
\bm{S}_{\parallel}=\bm{\Lambda}^{1/2}\bm{Z},\label{eq:S-parallel-Z-connection}
\end{equation}
 where $\bm{Z}=[Z_{i,j}]\in\mathbb{R}^{r\times n}$ is a Gaussian
random matrix with i.i.d.~entries $Z_{i,j}\overset{\mathrm{i.i.d.}}{\sim}\mathcal{N}(0,1)$.
Applying standard Gaussian concentration inequalities reveals that:
with probability at least $1-O(n^{-10})$,
\begin{align}
\Big\|\frac{1}{n}\bm{S}_{\parallel}\bm{S}_{\parallel}^{\top}-\bm{\Lambda}\Big\|\le\|\bm{\Lambda}\|\cdot\Big\|\frac{1}{n}\bm{Z}\bm{Z}^{\top}-\bm{I}_{r}\Big\| & \lesssim(\lambda_{\max}^{\star}+\sigma^{2})\sqrt{\frac{r\log n}{n}}.\label{eq:gauss_spectral}
\end{align}
\item As for the second term on the right-hand side of (\ref{eq:pca-eig-val-UB}),
we claim for the moment that
\begin{align}
\sup_{\lambda:\,\lambda\in[2\lambda_{l}^{\star}/3,4\lambda_{l}^{\star}/3]}\big\|\bm{K}(\lambda)-\bm{K}^{\perp}(\lambda)\big\| & \ll(\lambda_{\max}^{\star}+\sigma^{2})\sqrt{\frac{r}{n}}\log n,\label{eq:pca-eig-val-UB-term2}\\
\sup_{\lambda:\,\lambda\in[2\lambda_{l}^{\star}/3,4\lambda_{l}^{\star}/3]}\|\bm{C}(\lambda)\| & \lesssim\frac{\sigma^{2}}{\lambda_{l}^{\star}}\bigg(1+\frac{p}{n}\bigg),\label{eq:B-sp-norm-UB}
\end{align}
where $\bm{C}(\lambda)$ is defined in (\ref{eq:def:G-lambda-pca}).
The proof of this claim is postponed to the end of the section. 
\end{itemize}
Substituting (\ref{eq:gauss_spectral}) and (\ref{eq:pca-eig-val-UB-term2})
into (\ref{eq:pca-eig-val-UB}) reveals that with probability exceeding
$1-O(n^{-10})$,
\begin{equation}
\big\|\big(\lambda_{l}\bm{I}_{r}-\bm{\Lambda}-\bm{K}^{\perp}(\lambda_{l})\big)\bm{u}_{l,\parallel}\big\|_{2}\lesssim(\lambda_{\max}^{\star}+\sigma^{2})\sqrt{\frac{r}{n}}\log n=:\mathcal{E}_{\mathsf{PCA}}.\label{eq:def-E-PCA}
\end{equation}

With the preceding inequality in place, we are ready to study the
eigenvalues of $\bm{\Lambda}+\bm{K}^{\perp}(\lambda_{l})$. Similar
to the analysis in the proof of Lemma \ref{lemma:bound-Glambda-uniform}
in Appendix~\ref{subsec:Proof-of-Lemma-bound-Glambda-uniform}, it
is straightforward to verify that
\begin{align}
\bm{K}^{\perp}(\lambda) & =\beta(\lambda)\bm{\Lambda},\label{eq:K-perp-form}
\end{align}
where $\beta(\lambda)=\frac{1}{n}\mathsf{tr}\big(\bm{C}(\lambda)\big)$
has been defined in (\ref{eq:definition-gamma-lambda-iid-pca}). This
immediately demonstrates that the $l$-th eigenvalue of $\bm{\Lambda}+\bm{K}^{\perp}(\lambda_{l})$
is equal to 
\[
\big(1+\beta(\lambda_{l})\big)(\lambda_{l}^{\star}+\sigma^{2}).
\]
Moreover, it is readily seen from (\ref{eq:B-sp-norm-UB}) that $\beta(\lambda)$
satisfies
\begin{align}
\sup_{\lambda:\,\lambda\in[2\lambda_{l}^{\star}/3,4\lambda_{l}^{\star}/3]}\beta(\lambda) & \leq\sup_{\lambda:\,\lambda\in[2\lambda_{l}^{\star}/3,4\lambda_{l}^{\star}/3]}\frac{n\wedge p}{n}\|\bm{C}(\lambda)\|\lesssim\frac{n\wedge p}{n}\cdot\frac{\sigma^{2}}{\lambda_{l}^{\star}}\bigg(1+\frac{p}{n}\bigg)\asymp\frac{\sigma^{2}p}{\lambda_{l}^{\star}n}=o(1)\label{eq:beta-UB}
\end{align}
as long as the noise level obeys $\sigma^{2}p/n\ll\lambda_{\min}^{\star}/\log n$.
Finally, combining (\ref{eq:K-perp-form}) with (\ref{eq:def-Lambda-pca})
and (\ref{eq:def-E-PCA}), we can repeat the same argument as in the
proof for Lemma~\ref{lemma:lambda-l-mapping} in Section~\ref{subsec:Proof-of-Claim-lambdal-correspondence-iid}
to reach
\[
\big|\lambda_{l}-(\lambda_{l}^{\star}+\sigma^{2})\big(1+\beta(\lambda_{l})\big)\big|\lesssim(\lambda_{\max}^{\star}+\sigma^{2})\sqrt{\frac{r}{n}}\log n;
\]
for conciseness, we omit the details of proof. This inequality establishes
the proximity of $\lambda_{l}$ and $(\lambda_{l}^{\star}+\sigma^{2})\big(1+\beta(\lambda_{l})\big)$.
Taking this collectively with (\ref{eq:beta-UB}) (i.e., $1+\beta(\lambda_{l})\asymp1$),
we establish the advertised bound (\ref{eq:eigenvalue-perturbation-bound-iid-pca}).

\paragraph{Proof of the inequality (\ref{eq:pca-eig-val-UB-term2})}

Recall the definitions of $\bm{K}(\lambda)$, $\bm{C}(\lambda)$ as
well as $\bm{K}^{\perp}(\lambda)$ in (\ref{eq:def:G-lambda-pca})
and (\ref{eq:def:G-lambda-exp-pca}). Recognizing that one can express
$\bm{S}_{\parallel}=\bm{\Lambda}^{1/2}\bm{Z}$ with $\bm{Z}\in\mathbb{R}^{r\times n}$
being an i.i.d.~standard Gaussian matrix (see (\ref{eq:S-parallel-Z-connection})),
we can define
\begin{align*}
\overline{\bm{K}}(\lambda) & :=\frac{1}{n}\bm{Z}\bm{C}(\lambda)\bm{Z}^{\top}\qquad\text{and}\qquad\overline{\bm{K}}^{\perp}(\lambda):=\mathbb{E}\big[\,\overline{\bm{K}}(\lambda)\mid\bm{C}(\lambda)\big],
\end{align*}
which allow us to express
\begin{align*}
\bm{K}(\lambda) & \coloneqq\frac{1}{n}\bm{S}_{\parallel}\bm{C}(\lambda)\bm{S}_{\parallel}^{\top}=\frac{1}{n}\bm{\Lambda}^{1/2}\bm{Z}\bm{C}(\lambda)\bm{Z}^{\top}\bm{\Lambda}^{1/2}=\bm{\Lambda}^{1/2}\overline{\bm{K}}(\lambda)\bm{\Lambda}^{1/2},\\
\bm{K}^{\perp}(\lambda) & \coloneqq\mathbb{E}\big[\bm{K}(\lambda)\mid\bm{C}(\lambda)\big]=\bm{\Lambda}^{1/2}\mathbb{E}\big[\,\overline{\bm{K}}(\lambda)\mid\bm{C}(\lambda)\big]\bm{\Lambda}^{1/2}=\bm{\Lambda}^{1/2}\overline{\bm{K}}^{\perp}(\lambda)\bm{\Lambda}^{1/2}.
\end{align*}
One can then develop the following upper bound
\begin{align}
\big\|\bm{K}(\lambda)-\bm{K}^{\perp}(\lambda)\big\| & =\big\|\bm{\Lambda}^{1/2}\big(\overline{\bm{K}}(\lambda)-\overline{\bm{K}}^{\perp}(\lambda)\big)\bm{\Lambda}^{1/2}\big\|\leq\|\bm{\Lambda}\|\big\|\overline{\bm{K}}(\lambda)-\overline{\bm{K}}^{\perp}(\lambda)\big\|\nonumber \\
 & =(\lambda_{\max}^{\star}+\sigma^{2})\frac{1}{n}\big\|\bm{Z}\bm{C}(\lambda)\bm{Z}^{\top}-\mathbb{E}[\bm{Z}\bm{C}(\lambda)\bm{Z}^{\top}\mid\bm{C}(\lambda)]\big\|.\label{eq:pca-eig-val-UB-term2-temp}
\end{align}
By construction, $\bm{S}_{\parallel}\coloneqq\bm{U}^{\star\top}\bm{S}$
and $\bm{S}_{\perp}:=(\bm{U}^{\star\perp})^{\top}\bm{S}$ are mutually
statistically independent, thus implying that $\bm{Z}$ is also independent
of $\bm{C}(\lambda)$ with $\lambda$ treated as a deterministic quantity.

The remainder of the proof thus comes down to controlling
\[
\big\|\bm{Z}\bm{C}(\lambda)\bm{Z}^{\top}-\mathbb{E}[\bm{Z}\bm{C}(\lambda)\bm{Z}^{\top}\mid\bm{C}(\lambda)]\big\|.
\]
By virtue of the rotational invariance of Gaussian random matrices,
we can replace $\bm{C}(\lambda)$ in the quantity above by a diagonal
matrix comprised of the eigenvalues of $\bm{C}(\lambda)$. To see
this, we denote by $\bm{V}\bm{D}\bm{V}^{\top}$ the eigen-decomposition
of $\bm{C}(\lambda)$ and find that
\[
\bm{Z}\bm{C}(\lambda)\bm{Z}^{\top}=\bm{Z}\bm{V}\bm{D}\bm{V}^{\top}\bm{Z}^{\top}\overset{\mathrm{d}}{=}\bm{Z}\bm{D}\bm{Z}^{\top},
\]
where the last step arises from the rotational invariance of the Gaussian
random matrix, namely $\bm{Z}\bm{V}\overset{\mathrm{d}}{=}\bm{Z}$.
In view of Lemma~\ref{lemma:covariance-Gaussian-concentration},
it suffices to control the eigenvalues of $\bm{C}(\lambda)$.

As can be straightforwardly verified, the rank of $\bm{C}(\lambda)$
is upper bounded by $(p-r)\wedge n$ and the $i$-th largest eigenvalue
of $\bm{C}(\lambda)$ (cf.~(\ref{eq:def:G-lambda-pca})) satisfies
\begin{align*}
\lambda_{i}\big(\bm{C}(\lambda)\big) & =\lambda_{i}\Big(\frac{1}{n}\bm{S}_{\perp}^{\top}\Big(\lambda\bm{I}_{p-r}-\frac{1}{n}\bm{S}_{\perp}\bm{S}_{\perp}^{\top}\Big)^{-1}\bm{S}_{\perp}\Big)=\frac{\lambda_{i}(\frac{1}{n}\bm{S}_{\perp}\bm{S}_{\perp}^{\top})}{\lambda-\lambda_{i}(\frac{1}{n}\bm{S}_{\perp}\bm{S}_{\perp}^{\top})},\qquad1\leq i\leq(p-r)\wedge n.
\end{align*}
In addition, (\ref{eq:S-perp-op-UB}) demonstrates that with probability
at least $1-O(n^{-10})$,
\[
0\leq\lambda_{i}\Big(\frac{1}{n}\bm{S}_{\perp}\bm{S}_{\perp}^{\top}\Big)\lesssim\sigma^{2}\bigg(1+\sqrt{\frac{p}{n}}+\frac{p}{n}+\sqrt{\frac{\log n}{n}}\bigg)\ll\lambda_{\min}^{\star},\qquad1\leq i\leq(p-r)\wedge n,
\]
where the last step holds due to the noise assumption (\ref{eq:noise-condition-iid-pca}).
Combining these two observations establishes the claim bound (\ref{eq:B-sp-norm-UB}):
\begin{align*}
\sup_{\lambda:\,\lambda\in[2\lambda_{l}^{\star}/3,4\lambda_{l}^{\star}/3]}\|\bm{C}(\lambda)\| & \lesssim\frac{\sigma^{2}}{\lambda_{l}^{\star}}\bigg(1+\sqrt{\frac{p}{n}}+\frac{p}{n}+\sqrt{\frac{\log n}{n}}\bigg)\asymp\frac{\sigma^{2}}{\lambda_{l}^{\star}}\bigg(1+\frac{p}{n}\bigg),
\end{align*}
where the last step arises from the Cauchy-Schwarz inequality. Consequently,
one can invoke Lemma~\ref{lemma:covariance-Gaussian-concentration}
and apply the standard epsilon-net argument (similar to the proof
of Lemma \ref{lemma:bound-Glambda-uniform} in Appendix~\ref{subsec:Proof-of-Lemma-bound-Glambda-uniform}
and hence omitted here) to demonstrate that
\begin{align*}
 & \sup_{\lambda:\,\lambda\in[2\lambda_{l}^{\star}/3,4\lambda_{l}^{\star}/3]}\frac{1}{n}\big\|\bm{Z}\bm{C}(\lambda)\bm{Z}^{\top}-\mathbb{E}[\bm{Z}\bm{C}(\lambda)\bm{Z}^{\top}\mid\bm{C}(\lambda)]\big\|\\
 & \qquad\lesssim\sup_{\lambda:\,\lambda\in[2\lambda_{l}^{\star}/3,4\lambda_{l}^{\star}/3]}\frac{1}{n}\|\bm{C}(\lambda)\|_{\mathrm{F}}\sqrt{r\log n}+\sup_{\lambda:\,\lambda\in[2\lambda_{l}^{\star}/3,4\lambda_{l}^{\star}/3]}\frac{1}{n}\|\bm{C}(\lambda)\|\big(r\log n+\log^{2}n\big)\\
 & \qquad\lesssim\sup_{\lambda:\,\lambda\in[2\lambda_{l}^{\star}/3,4\lambda_{l}^{\star}/3]}\frac{1}{n}\|\bm{C}(\lambda)\|\sqrt{r(n\wedge p)}\,\log^{2}n\\
 & \qquad\lesssim\frac{\sigma^{2}}{\lambda_{l}^{\star}}\bigg(\frac{p}{n}+\sqrt{\frac{p}{n}}\bigg)\sqrt{\frac{r}{n}}\,\log^{2}n\ll\sqrt{\frac{r}{n}}\log n
\end{align*}
with probability at least $1-O(n^{-10})$. Here, the last line follows
from (\ref{eq:B-sp-norm-UB}) and the noise assumption that $\sigma^{2}(p/n+\sqrt{p/n})\ll\lambda_{\min}^{\star}/\log n$.
Combining this with (\ref{eq:pca-eig-val-UB-term2-temp}), we arrive
at
\[
\sup_{\lambda:\,\lambda\in[2\lambda_{l}^{\star}/3,4\lambda_{l}^{\star}/3]}\big\|\bm{K}(\lambda)-\bm{K}^{\perp}(\lambda)\big\|\ll(\lambda_{\max}^{\star}+\sigma^{2})\sqrt{\frac{r}{n}}\log n
\]
as claimed.

\subsubsection{Proof of eigenvector perturbation theory (Theorem \ref{thm:evector-pertur-sym-iid-pca})}

We now turn to our eigenvector perturbation theory. As before, we
find it convenient to decompose the $l$-th eigenvector $\bm{u}_{l}$
of $\frac{1}{n}\bm{S}\bm{S}^{\top}$ as follows
\begin{equation}
\bm{u}_{l}=\bm{u}_{l}^{\star}\cos\theta+\bm{u}_{l,\perp}\sin\theta,\label{eq:ul-expansion-PCA-1}
\end{equation}
where the vector $\bm{u}_{l,\perp}$ obeys $\|\bm{u}_{l,\perp}\|_{2}=1$
and $\bm{u}_{l,\perp}^{\top}\bm{u}_{l}^{\star}=0$. We shall employ
this decomposition to identify several key quantities that we'd like
to control. Specifically, armed with this decomposition, we can derive
\begin{align*}
\bm{a}^{\top}\bm{u}_{l}= & (\bm{P}_{\bm{U}^{\star}}\bm{a})^{\top}\bm{u}_{l}+(\bm{P}_{\bm{U}^{\star\perp}}\bm{a})^{\top}\bm{u}_{l}\\
= & (\bm{P}_{\bm{U}^{\star}}\bm{a})^{\top}(\bm{u}_{l}^{\star}\cos\theta+\bm{u}_{l,\perp}\sin\theta)+(\bm{P}_{\bm{U}^{\star\perp}}\bm{a})^{\top}\bm{u}_{l}\\
= & \sum_{1\leq k\leq r}\bm{a}^{\top}\bm{u}_{k}^{\star}\bm{u}_{k}^{\star\top}(\bm{u}_{l}^{\star}\cos\theta+\bm{u}_{l,\perp}\sin\theta)+(\bm{P}_{\bm{U}^{\star\perp}}\bm{a})^{\top}(\bm{P}_{\bm{U}^{\star\perp}}\bm{u}_{l})\\
= & \bm{a}^{\top}\bm{u}_{l}^{\star}\cos\theta+\sum_{k:k\ne l}\bm{a}^{\top}\bm{u}_{k}^{\star}\bm{u}_{k}^{\star\top}\bm{u}_{l,\perp}\sin\theta+(\bm{P}_{\bm{U}^{\star\perp}}\bm{a})^{\top}(\bm{P}_{\bm{U}^{\star\perp}}\bm{u}_{l}),
\end{align*}
where we use the fact that $\bm{a}=\bm{P}_{\bm{U}^{\star}}\bm{a}+\bm{P}_{\bm{U}^{\star\perp}}\bm{a}$
with
\[
\bm{P}_{\bm{U}^{\star}}=\Sigma_{1\leq k\leq r}\bm{u}_{k}^{\star}\bm{u}_{k}^{\star\top}\qquad\text{and}\qquad\bm{P}_{\bm{U}^{\star\perp}}=\bm{I}-\bm{P}_{\bm{U}^{\star}}.
\]
As a result, we arrive at
\[
\bm{a}^{\top}\bm{u}_{l}\pm\bm{a}^{\top}\bm{u}_{l}^{\star}=\bm{a}^{\top}\bm{u}_{l}^{\star}(\cos\theta\pm1)+\sum_{k:k\ne l}\bm{a}^{\top}\bm{u}_{k}^{\star}\bm{u}_{k}^{\star\top}\bm{u}_{l,\perp}\sin\theta+(\bm{P}_{\bm{U}^{\star\perp}}\bm{a})^{\top}(\bm{P}_{\bm{U}^{\star\perp}}\bm{u}_{l}),
\]
which further implies 
\begin{align}
\min\big|\bm{a}^{\top}\bm{u}_{l}\pm\bm{a}^{\top}\bm{u}_{l}^{\star}\big|\leq & \big|\bm{a}^{\top}\bm{u}_{l}^{\star}\big|(1-|\cos\theta|)+\Big|\sum_{k:k\ne l}\bm{a}^{\top}\bm{u}_{k}^{\star}\bm{u}_{k}^{\star\top}\bm{u}_{l,\perp}\sin\theta\Big|+\big|(\bm{P}_{\bm{U}^{\star\perp}}\bm{a})^{\top}(\bm{P}_{\bm{U}^{\star\perp}}\bm{u}_{l})\big|\nonumber \\
\leq & \big|\bm{a}^{\top}\bm{u}_{l}^{\star}\big|(1-\cos^{2}\theta)+\Big|\sum_{k:k\ne l}\bm{a}^{\top}\bm{u}_{k}^{\star}\bm{u}_{k}^{\star\top}\bm{u}_{l,\perp}\sin\theta\Big|+\big|(\bm{P}_{\bm{U}^{\star\perp}}\bm{a})^{\top}(\bm{P}_{\bm{U}^{\star\perp}}\bm{u}_{l})\big|.\label{eq:au-min-UB-iid-sym-1-pca}
\end{align}
and
\begin{align}
\min\big|\bm{a}^{\top}\bm{u}_{l}\sqrt{1+c_{l}}\pm\bm{a}^{\top}\bm{u}_{l}^{\star}\big| & \leq\big|\bm{a}^{\top}\bm{u}_{l}^{\star}\big|\cdot\big|1-\sqrt{1+c_{l}}|\cos\theta|\big|+\sqrt{1+c_{l}}\big|\langle\bm{P}_{\bm{U}^{\star\perp}}\bm{a},\,\bm{P}_{\bm{U}^{\star\perp}}\bm{u}_{l}\rangle\big|\nonumber \\
 & \quad+\sqrt{1+c_{l}}\Big|\sum_{k:k\ne l}\bm{a}^{\top}\bm{u}_{k}^{\star}\bm{u}_{k}^{\star\top}\bm{u}_{l,\perp}\sin\theta\Big|.\label{eq:au-min-UB-iid-sym-1-pca-cl}
\end{align}
Thus, it comes down to bounding the following terms
\[
1-\cos^{2}\theta,\quad1-\sqrt{1+c_{l}}|\cos\theta|,\quad\sqrt{1+c_{l}},\quad\sum_{k:k\ne l}\bm{a}^{\top}\bm{u}_{k}^{\star}\bm{u}_{k}^{\star\top}\bm{u}_{l,\perp}\sin\theta,\quad\text{and}\quad\langle\bm{P}_{\bm{U}^{\star\perp}}\bm{a},\,\bm{P}_{\bm{U}^{\star\perp}}\bm{u}_{l}\rangle.
\]
separately, which forms the main content of the remainder of the proof. 

We claim that $\lambda_{l}\bm{I}_{p-1}-(\bm{u}_{l}^{\star\perp})^{\top}\frac{1}{n}\bm{S}\bm{S}^{\top}\bm{u}_{l}^{\star\perp}$
is invertible. This will be seen from (\ref{eq:lambdal-tilde-lambdak-pca-LB})
stated in Lemma~ \ref{lemma:lambda-S-minus-spectrum-1-1} directly.
The invertibility taken together with Theorem~\ref{thm:master-thm-vector}
reveals that $\cos\theta\neq0$. If $\sin\theta=0$, then we have
$\bm{u}_{l}=\pm\bm{u}_{l}^{\star}$ and the conclusion is obvious
since $\min\big|\bm{a}^{\top}\bm{u}_{l}\pm\bm{a}^{\top}\bm{u}_{l}^{\star}\big|=0$.
Therefore, it suffices to focus on the case where $\cos\theta\neq0$
and $\sin\theta\neq0$ in the sequel. 

\paragraph{1. Identifying several key quantities.}

Invoke Theorem \ref{thm:master-thm-vector} to show that\begin{subequations}\label{eq:costheta-au-iid-1-pca}
\begin{align}
\cos^{2}\theta & =\frac{1}{1+\big\|\big(\lambda_{l}\bm{I}_{p-1}-(\bm{u}_{l}^{\star\perp})^{\top}\frac{1}{n}\bm{S}\bm{S}^{\top}\bm{u}_{l}^{\star\perp}\big)^{-1}(\bm{u}_{l}^{\star\perp})^{\top}\frac{1}{n}\bm{S}\bm{S}^{\top}\bm{u}_{l}^{\star}\big\|_{2}^{2}},\\
\bm{u}_{k}^{\star\top}\bm{u}_{l,\perp} & =\frac{\bm{u}_{k}^{\star\top}\bm{u}_{l}^{\star\perp}\big(\lambda_{l}\bm{I}_{p-1}-(\bm{u}_{l}^{\star\perp})^{\top}\frac{1}{n}\bm{S}\bm{S}^{\top}\bm{u}_{l}^{\star\perp}\big)^{-1}(\bm{u}_{l}^{\star\perp})^{\top}\frac{1}{n}\bm{S}\bm{S}^{\top}\bm{u}_{l}^{\star}}{\big\|\bm{u}_{l}^{\star\perp}\big(\lambda_{l}\bm{I}_{p-1}-(\bm{u}_{l}^{\star\perp})^{\top}\frac{1}{n}\bm{S}\bm{S}^{\top}\bm{u}_{l}^{\star\perp}\big)^{-1}(\bm{u}_{l}^{\star\perp})^{\top}\frac{1}{n}\bm{S}\bm{S}^{\top}\bm{u}_{l}^{\star}\big\|_{2}}.
\end{align}
\end{subequations}For notational convenience, we shall define 
\begin{align}
\bm{s}_{l,\parallel} & :=\bm{u}_{l}^{\star\top}\bm{S}\in\mathbb{R}^{1\times n}\qquad\text{and}\qquad\bm{S}_{l,\perp}:=(\bm{u}_{l}^{\star\perp})^{\top}\bm{S}\in\mathbb{R}^{(p-1)\times n},\label{eq:def:s-l-para-perp}
\end{align}
allowing us to write (\ref{eq:costheta-au-iid-1-pca}) more succinctly
as follows\begin{subequations}\label{eq:costheta-au-expression-pca}
\begin{align}
\cos^{2}\theta & =\frac{1}{1+\big\|\big(\lambda_{l}\bm{I}_{p-1}-\frac{1}{n}\bm{S}_{l,\perp}\bm{S}_{l,\perp}^{\top}\big)^{-1}\frac{1}{n}\bm{S}_{l,\perp}\bm{s}_{l,\parallel}^{\top}\big\|_{2}^{2}},\label{eq:cos-eig-linear-form-pca}\\
\bm{u}_{k}^{\star\top}\bm{u}_{l,\perp} & =\frac{\bm{u}_{k}^{\star\top}\bm{u}_{l}^{\star\perp}\big(\lambda_{l}\bm{I}_{p-1}-\frac{1}{n}\bm{S}_{l,\perp}\bm{S}_{l,\perp}^{\top}\big)^{-1}\frac{1}{n}\bm{S}_{l,\perp}\bm{s}_{l,\parallel}^{\top}}{\big\|\bm{u}_{l}^{\star\perp}\big(\lambda_{l}\bm{I}_{p-1}-\frac{1}{n}\bm{S}_{l,\perp}\bm{S}_{l,\perp}^{\top}\big)^{-1}\frac{1}{n}\bm{S}_{l,\perp}\bm{s}_{l,\parallel}^{\top}\big\|_{2}}.\label{eq:a-u-eig-linear-form-pca}
\end{align}
\end{subequations} With the above relations in mind, we can demonstrate
that 
\begin{align*}
 & \Big|\sum_{k:k\ne l}\bm{a}^{\top}\bm{u}_{k}^{\star}\bm{u}_{k}^{\star\top}\bm{u}_{l,\perp}\sin\theta\Big|=\Big|\sum_{k:k\ne l}\bm{a}^{\top}\bm{u}_{k}^{\star}\bm{u}_{k}^{\star\top}\bm{u}_{l,\perp}\Big|\sqrt{1-\cos^{2}\theta}\\
 & \quad=\frac{\Big|\sum_{k:k\ne l}\bm{a}^{\top}\bm{u}_{k}^{\star}\bm{u}_{k}^{\star\top}\bm{u}_{l}^{\star\perp}\big(\lambda_{l}\bm{I}_{p-1}-\frac{1}{n}\bm{S}_{l,\perp}\bm{S}_{l,\perp}^{\top}\big)^{-1}\frac{1}{n}\bm{S}_{l,\perp}\bm{s}_{l,\parallel}^{\top}\Big|}{\big\|\bm{u}_{l}^{\star\perp}\big(\lambda_{l}\bm{I}_{p-1}-\frac{1}{n}\bm{S}_{l,\perp}\bm{S}_{l,\perp}^{\top}\big)^{-1}\frac{1}{n}\bm{S}_{l,\perp}\bm{s}_{l,\parallel}^{\top}\big\|_{2}}\cdot\sqrt{\frac{\big\|\big(\lambda_{l}\bm{I}_{p-1}-\frac{1}{n}\bm{S}_{l,\perp}\bm{S}_{l,\perp}^{\top}\big)^{-1}\frac{1}{n}\bm{S}_{l,\perp}\bm{s}_{l,\parallel}^{\top}\big\|_{2}^{2}}{1+\big\|\big(\lambda_{l}\bm{I}_{p-1}-\frac{1}{n}\bm{S}_{l,\perp}\bm{S}_{l,\perp}^{\top}\big)^{-1}\frac{1}{n}\bm{S}_{l,\perp}\bm{s}_{l,\parallel}^{\top}\big\|_{2}^{2}}}\\
 & \quad\leq\left|\sum_{k:k\ne l}\bm{a}^{\top}\bm{u}_{k}^{\star}\bm{u}_{k}^{\star\top}\bm{u}_{l}^{\star\perp}\Big(\lambda_{l}\bm{I}_{p-1}-\frac{1}{n}\bm{S}_{l,\perp}\bm{S}_{l,\perp}^{\top}\Big)^{-1}\frac{1}{n}\bm{S}_{l,\perp}\bm{s}_{l,\parallel}^{\top}\right|,
\end{align*}
where the last step follows since the columns of $\bm{u}_{l}^{\star\perp}$
are orthonormal. Substitution into (\ref{eq:au-min-UB-iid-sym-1-pca})
then yields
\begin{align}
\min\big|\bm{a}^{\top}\bm{u}_{l}\pm\bm{a}^{\top}\bm{u}_{l}^{\star}\big| & \leq\big|\bm{a}^{\top}\bm{u}_{l}^{\star}\big|\cdot(1-\cos^{2}\theta)+\big|(\bm{P}_{\bm{U}^{\star\perp}}\bm{a})^{\top}(\bm{P}_{\bm{U}^{\star\perp}}\bm{u}_{l})\big| \nonumber \\
 & \quad+ \Big|\sum_{k:k\ne l}\bm{a}^{\top}\bm{u}_{k}^{\star}\bm{u}_{k}^{\star\top}\bm{u}_{l}^{\star\perp}\Big(\lambda_{l}\bm{I}_{p-1}-\frac{1}{n}\bm{S}_{l,\perp}\bm{S}_{l,\perp}^{\top}\Big)^{-1}\frac{1}{n}\bm{S}_{l,\perp}\bm{s}_{l,\parallel}^{\top}\Big|;\label{eq:au-min-UB-iid-sym-pca}
 \end{align}
and
\begin{align}
\min\big|\sqrt{1+c_{l}}\bm{a}^{\top}\bm{u}_{l}\pm\bm{a}^{\top}\bm{u}_{l}^{\star}\big| & \leq\big|\bm{a}^{\top}\bm{u}_{l}^{\star}\big|\cdot\big|1-\sqrt{1+c_{l}}|\cos\theta|\big|+\sqrt{1+c_{l}}\big|\langle\bm{P}_{\bm{U}^{\star\perp}}\bm{a},\,\bm{P}_{\bm{U}^{\star\perp}}\bm{u}_{l}\rangle\big|\nonumber \\
 & \quad+\sqrt{1+c_{l}}\Big|\sum_{k:k\ne l}\bm{a}^{\top}\bm{u}_{k}^{\star}\bm{u}_{k}^{\star\top}\bm{u}_{l}^{\star\perp}\Big(\lambda_{l}\bm{I}_{p-1}-\frac{1}{n}\bm{S}_{l,\perp}\bm{S}_{l,\perp}^{\top}\Big)^{-1}\frac{1}{n}\bm{S}_{l,\perp}\bm{s}_{l,\parallel}^{\top}\Big|.\label{eq:au-min-UB-iid-sym-pca-cl}
\end{align}
In what follows, we shall control these quantities separately.

\paragraph{2. Controlling the spectrum of $\frac{1}{n}\bm{S}_{l,\perp}\bm{S}_{l,\perp}^{\top}$.}

Before moving forward to bound the terms mentioned above, we take
a moment to first look at the eigenvalues of $\frac{1}{n}\bm{S}_{l,\perp}\bm{S}_{l,\perp}^{\top}$.
We first introduce some useful notation as follows: 
\begin{itemize}
\item Let $\{\gamma_{i}^{(l)}\}_{i=1}^{p-1}$ denote the eigenvalues of
$\frac{1}{n}\bm{S}_{l,\perp}\bm{S}_{l,\perp}^{\top}$ (see the definition
of $\bm{S}_{l,\perp}$ in (\ref{eq:def:s-l-para-perp})), and we assume
that
\begin{equation}
\gamma_{1}^{(l)}\geq\cdots\geq\gamma_{p-1}^{(l)}.\label{eq:order-gamma-i-p}
\end{equation}
\item Let $\bm{u}_{i}^{(l)}$ be the eigenvector of $\frac{1}{n}\bm{S}_{l,\perp}\bm{S}_{l,\perp}^{\top}$
associated with the eigenvalue $\gamma_{i}^{(l)}$. 
\end{itemize}
Similar to (\ref{eq:def:U-Lambda-star-l}), we find it helpful to
introduce the following matrices\begin{subequations}\label{eq:def:U-Lambda-star-l-1-pca}
\begin{align}
\bm{U}_{\smallsetminus l}^{\star} & :=[\bm{u}_{1}^{\star},\cdots,\bm{u}_{l-1}^{\star},\bm{u}_{l+1}^{\star},\cdots,\bm{u}_{r}^{\star}]\in\mathbb{R}^{p\times(r-1)};\\
\bm{U}^{\star(l)} & :=(\bm{u}_{l}^{\star\perp})^{\top}\bm{U}_{\smallsetminus l}^{\star}=\begin{bmatrix}\bm{I}_{r-1}\\
\bm{0}
\end{bmatrix}\in\mathbb{R}^{(p-1)\times(r-1)};\\
\bm{U}^{\star(l)\perp} & :=(\bm{u}_{l}^{\star\perp})^{\top}\bm{U}^{\star\perp}=\begin{bmatrix}\bm{0}\\
\bm{I}_{p-r}
\end{bmatrix}\in\mathbb{R}^{(p-1)\times(p-r)};\\
\bm{\Lambda}^{\star(l)} & :=\mathsf{diag}\big(\{\lambda_{i}^{\star}\}_{i:i\neq l}\big)\in\mathbb{R}^{(r-1)\times(r-1)}.
\end{align}
In addition, we define 
\begin{equation}
\bm{u}_{i,\parallel}^{(l)}:=\frac{1}{\big\|\bm{P}_{\bm{U}^{\star(l)}}\bm{u}_{i}^{(l)}\big\|_{2}}\bm{P}_{\bm{U}^{\star(l)}}\bm{u}_{i}^{(l)},\qquad i\neq l,
\end{equation}
\end{subequations}where $\bm{P}_{\bm{U}^{\star(l)}}=\bm{U}^{\star(l)}\big(\bm{U}^{\star(l)}\big)^{\top}$.
Equipped with this set of notation, we are ready to present a lemma
that characterizes the eigenvalues of the matrix $\frac{1}{n}\bm{S}_{l,\perp}\bm{S}_{l,\perp}^{\top}$. 

\begin{lemma}\label{lemma:lambda-S-minus-spectrum-1-1}Instate the
assumptions of Theorem \ref{thm:evector-pertur-sym-iid-pca}, and
recall the definition of $\beta(\cdot)$ in (\ref{eq:definition-gamma-lambda-iid-pca}).
With probability at least $1-O(n^{-10})$, the eigenvalues $\{\gamma_{i}^{(l)}\}_{i=1}^{p-1}$
of $\frac{1}{n}\bm{S}_{l,\perp}\bm{S}_{l,\perp}^{\top}$ (see (\ref{eq:order-gamma-i-p}))
satisfy the following properties. 
\begin{enumerate}
\item For each $1\leq i<r$, one has
\[
\frac{\gamma_{i}^{(l)}}{1+\beta(\gamma_{i}^{(l)})}\in\mathcal{B}_{\mathcal{E}_{\mathsf{PCA}}}(\lambda_{k}^{\star}+\sigma^{2})\qquad\text{for some }k\neq l\text{ and }1\leq k\leq r
\]
and
\[
\Big\|\Big(\gamma_{i}^{(l)}\bm{I}_{r-1}-\big(1+\beta(\gamma_{i}^{(l)})\big)\big(\bm{\Lambda}^{\star(l)}+\sigma^{2}\bm{I}_{r-1}\big)\Big)\bm{U}^{\star(l)\top}\bm{u}_{i,\parallel}^{(l)}\Big\|_{2}\lesssim\mathcal{E}_{\mathsf{PCA}},
\]
where $\mathcal{E}_{\mathsf{PCA}}$ is defined in (\ref{eq:def-E-PCA}).
\item For each $r\le i\le n\wedge(p-1)$, one has
\begin{equation}
\left|\gamma_{i}^{(l)}-\sigma^{2}\frac{p\vee n}{n}\right|\lesssim\sigma^{2}\sqrt{\frac{p+\log n}{n}}.\label{eq:lambdal-tilde-lambdak-pca-UB-i>r}
\end{equation}
\item For each $n\wedge(p-1)<i\le p-1$, we have $\gamma_{i}^{(l)}=0$.
\item Furthermore, one has
\begin{equation}
|\lambda-\lambda_{l}|\gtrsim\begin{cases}
\Delta_{l}^{\star}, & \text{if}\ \frac{\lambda}{1+\beta(\lambda)}\in\mathcal{B}_{\mathcal{E}_{\mathsf{PCA}}}(\lambda_{i}^{\star}+\sigma^{2})\quad\text{\text{for some }}i\neq l\text{ and }1\leq i\leq r;\\
\lambda_{l}^{\star}, & \text{if}\ \big|\lambda-\sigma^{2}\frac{p\vee n}{n}\big|\lesssim\sigma^{2}\sqrt{\frac{p+\log n}{n}}.
\end{cases}\label{eq:lambda-lambdal-pca-LB}
\end{equation}
In particular, one has
\begin{equation}
\big|\gamma_{i}^{(l)}-\lambda_{l}\big|\gtrsim\begin{cases}
\Delta_{l}^{\star}, & 1\leq i<r;\\
\lambda_{l}^{\star}, & i\geq r.
\end{cases}\label{eq:lambdal-tilde-lambdak-pca-LB}
\end{equation}
\end{enumerate}
\end{lemma}\begin{proof}See Appendix \ref{subsec:Proof-of-lemma:lambda-S-minus-spectrum-1-1}.\end{proof}

\paragraph{3. Controlling $\cos^{2}\theta$.}

In view of the expression of $\cos^{2}\theta$ in (\ref{eq:cos-eig-linear-form-pca}),
it suffices to control $\|(\lambda_{l}\bm{I}_{p-1}-\frac{1}{n}\bm{S}_{l,\perp}\bm{S}_{l,\perp}^{\top})^{-1}\frac{1}{n}\bm{S}_{l,\perp}\bm{s}_{l,\parallel}^{\top}\|_{2}^{2}$,
which is accomplished in the following lemma. 

\begin{lemma}\label{lemma:lambda-M-l-inv-u-perp-M-u-l2-norm-1-pca}Consider
any $1\leq l\leq r$. Instate the assumptions of Theorem \ref{thm:evector-pertur-sym-iid-pca},
and recall the definition of $c_{l}$ in (\ref{eq:def:bl-pca}). The
following holds with probability at least $1-O(n^{-10})$:
\begin{equation}
\Big\|\Big(\lambda_{l}\bm{I}_{p-1}-\frac{1}{n}\bm{S}_{l,\perp}\bm{S}_{l,\perp}^{\top}\Big)^{-1}\frac{1}{n}\bm{S}_{l,\perp}\bm{s}_{l,\parallel}^{\top}\Big\|_{2}^{2}\lesssim\frac{(\lambda_{\max}^{\star}+\sigma^{2})(\lambda_{l}^{\star}+\sigma^{2})r\log n}{\big(\Delta_{l}^{\star}\big)^{2}n}+\frac{(\lambda_{l}^{\star}+\sigma^{2})\sigma^{2}p\log^{2}n}{\lambda_{l}^{\star2}n}\ll1.\label{eq:claim:lambda-M-l-inv-u-perp-M-u-l2-norm-1-pca-2}
\end{equation}
Moreover, for the case with $n\geq p$, one has
\begin{align}
 & \bigg|\Big\|\Big(\lambda_{l}\bm{I}_{p-1}-\frac{1}{n}\bm{S}_{l,\perp}\bm{S}_{l,\perp}^{\top}\Big)^{-1}\frac{1}{n}\bm{S}_{l,\perp}\bm{s}_{l,\parallel}^{\top}\Big\|_{2}^{2}-c_{l}\bigg|\nonumber \\
 & \qquad\lesssim\frac{(\lambda_{\max}^{\star}+\sigma^{2})(\lambda_{l}^{\star}+\sigma^{2})r\log n}{\big(\Delta_{l}^{\star}\big)^{2}n}+\frac{\sigma^{2}p}{\lambda_{l}^{\star2}n}\bigg((\lambda_{l}^{\star}+\sigma^{2})\sqrt{\frac{\log n}{p}}+(\lambda_{\max}^{\star}+\sigma^{2})\sqrt{\frac{r\log n}{n}}\bigg),\label{eq:claim:lambda-M-l-inv-u-perp-M-u-l2-norm-1-pca-1}
\end{align}
and for the case with $p>n$, we have
\begin{equation}
\bigg|\Big\|\Big(\lambda_{l}\bm{I}_{p-1}-\frac{1}{n}\bm{S}_{l,\perp}\bm{S}_{l,\perp}^{\top}\Big)^{-1}\frac{1}{n}\bm{S}_{l,\perp}\bm{s}_{l,\parallel}^{\top}\Big\|_{2}^{2}-c_{l}\bigg|\lesssim\frac{(\lambda_{\max}^{\star}+\sigma^{2})(\lambda_{l}^{\star}+\sigma^{2})r\log n}{\big(\Delta_{l}^{\star}\big)^{2}n}+\frac{\sigma^{2}\kappa\sqrt{pr\log n}}{\lambda_{l}^{\star}n}.\label{eq:claim:lambda-M-l-inv-u-perp-M-u-l2-norm-1-pca-3}
\end{equation}
\end{lemma}\begin{proof}See Appendix \ref{subsec:Proof-of-lemma:lambda-M-l-inv-u-perp-M-u-l2-norm-1-pca-1}.\end{proof}

This lemma taken collectively with (\ref{eq:cos-eig-linear-form-pca})
leads to
\begin{align}
|\cos^{2}\theta-1| & =\Bigg|\frac{1}{1+\big\|(\lambda_{l}\bm{I}_{p-1}-\frac{1}{n}\bm{S}_{l,\perp}\bm{S}_{l,\perp}^{\top})^{-1}\frac{1}{n}\bm{S}_{l,\perp}\bm{s}_{l,\parallel}^{\top}\big\|_{2}^{2}}-1\Bigg|\nonumber \\
 & =\frac{\big\|(\lambda_{l}\bm{I}_{p-1}-\frac{1}{n}\bm{S}_{l,\perp}\bm{S}_{l,\perp}^{\top})^{-1}\frac{1}{n}\bm{S}_{l,\perp}\bm{s}_{l,\parallel}^{\top}\big\|_{2}^{2}}{1+\big\|(\lambda_{l}\bm{I}_{p-1}-\frac{1}{n}\bm{S}_{l,\perp}\bm{S}_{l,\perp}^{\top})^{-1}\frac{1}{n}\bm{S}_{l,\perp}\bm{s}_{l,\parallel}^{\top}\big\|_{2}^{2}}\nonumber \\
 & \lesssim\frac{(\lambda_{\max}^{\star}+\sigma^{2})(\lambda_{l}^{\star}+\sigma^{2})r\log n}{\big(\Delta_{l}^{\star}\big)^{2}n}+\frac{(\lambda_{l}^{\star}+\sigma^{2})\sigma^{2}p\log^{2}n}{\lambda_{l}^{\star2}n}\ll1.\label{eq:cos-bound-pca}
\end{align}
where the last step follows from the assumptions (\ref{eq:noise-condition-iid-pca})
and (\ref{eq:eigengap-condition-iid-pca}). In addition, when $n\geq p$,
one can combine (\ref{eq:cos-eig-linear-form-pca}) and (\ref{eq:claim:lambda-M-l-inv-u-perp-M-u-l2-norm-1-pca-1})
to demonstrate that
\begin{align}
 & \big|(1+c_{l})\cos^{2}\theta-1\big|=\Bigg|\frac{1+c_{l}}{1+\big\|(\lambda_{l}\bm{I}_{p-1}-\frac{1}{n}\bm{S}_{l,\perp}\bm{S}_{l,\perp}^{\top})^{-1}\frac{1}{n}\bm{S}_{l,\perp}\bm{s}_{l,\parallel}^{\top}\big\|_{2}^{2}}-1\Bigg|\nonumber \\
 & \qquad=\frac{\Big|c_{l}-\big\|(\lambda_{l}\bm{I}_{p-1}-\frac{1}{n}\bm{S}_{l,\perp}\bm{S}_{l,\perp}^{\top})^{-1}\frac{1}{n}\bm{S}_{l,\perp}\bm{s}_{l,\parallel}^{\top}\big\|_{2}^{2}\Big|}{1+\big\|(\lambda_{l}\bm{I}_{p-1}-\frac{1}{n}\bm{S}_{l,\perp}\bm{S}_{l,\perp}^{\top})^{-1}\frac{1}{n}\bm{S}_{l,\perp}\bm{s}_{l,\parallel}^{\top}\big\|_{2}^{2}}\nonumber \\
 & \qquad\leq\Big|c_{l}-\Big\|\Big(\lambda_{l}\bm{I}_{p-1}-\frac{1}{n}\bm{S}_{l,\perp}\bm{S}_{l,\perp}^{\top}\Big)^{-1}\frac{1}{n}\bm{S}_{l,\perp}\bm{s}_{l,\parallel}^{\top}\Big\|_{2}^{2}\Big|\nonumber \\
 & \qquad\lesssim\frac{(\lambda_{\max}^{\star}+\sigma^{2})(\lambda_{l}^{\star}+\sigma^{2})r\log n}{\big(\Delta_{l}^{\star}\big)^{2}n}+\frac{\sigma^{2}p}{\lambda_{l}^{\star2}n}\bigg((\lambda_{\max}^{\star}+\sigma^{2})\sqrt{\frac{r\log n}{n}}+(\lambda_{l}^{\star}+\sigma^{2})\frac{\log^{2}n}{\sqrt{p}}\bigg),\label{eq:cos-value-pca}
\end{align}
where the first line comes from the definition of $\cos^{2}\theta$
in (\ref{eq:cos-eig-linear-form-pca}), and the last inequality holds
due to (\ref{eq:claim:lambda-M-l-inv-u-perp-M-u-l2-norm-1-pca-1}).
Moreover, if $p>n$, putting (\ref{eq:cos-eig-linear-form-pca}) and
(\ref{eq:claim:lambda-M-l-inv-u-perp-M-u-l2-norm-1-pca-3}) together
reveals that
\begin{align}
\big|(1+c_{l})\cos^{2}\theta-1\big| & =\Bigg|\frac{1+c_{l}}{1+\big\|(\lambda_{l}\bm{I}_{p-1}-\frac{1}{n}\bm{S}_{l,\perp}\bm{S}_{l,\perp}^{\top})^{-1}\frac{1}{n}\bm{S}_{l,\perp}\bm{s}_{l,\parallel}^{\top}\big\|_{2}^{2}}-1\Bigg|\nonumber \\
 & \leq\Big|c_{l}-\Big\|\Big(\lambda_{l}\bm{I}_{p-1}-\frac{1}{n}\bm{S}_{l,\perp}\bm{S}_{l,\perp}^{\top}\Big)^{-1}\frac{1}{n}\bm{S}_{l,\perp}\bm{s}_{l,\parallel}^{\top}\Big\|_{2}^{2}\Big|\nonumber \\
 & \lesssim\frac{(\lambda_{\max}^{\star}+\sigma^{2})(\lambda_{l}^{\star}+\sigma^{2})r\log n}{\big(\Delta_{l}^{\star}\big)^{2}n}+\frac{\sigma^{2}\kappa\sqrt{pr\log n}}{\lambda_{l}^{\star}n}\ll1,\label{eq:cos-value-pca-p}
\end{align}
where the last inequality holds due to the conditions (\ref{eq:noise-condition-iid-pca})
and (\ref{eq:eigengap-condition-iid-pca}). Taken collectively with
(\ref{eq:cos-bound-pca}), this leads to $1+c_{l}\lesssim1$ and
\begin{align}
\big|1-\sqrt{1+c_{l}}|\cos\theta|\big| & =\bigg|\frac{1-(1+c_{l})\cos^{2}\theta}{1+\sqrt{1+c_{l}}|\cos\theta|}\bigg|\lesssim\big|1-(1+c_{l})\cos^{2}\theta\big| \nonumber \\
& \lesssim\frac{(\lambda_{\max}^{\star}+\sigma^{2})(\lambda_{l}^{\star}+\sigma^{2})r\log n}{\big(\Delta_{l}^{\star}\big)^{2}n}+\frac{\sigma^{2}\kappa\sqrt{pr\log n}}{\lambda_{l}^{\star}n}.\label{eq:cos-bound-cl}
\end{align}

\paragraph{4. Controlling $\sum_{k:k\ne l}\bm{a}^{\top}\bm{u}_{k}^{\star}\bm{u}_{k}^{\star\top}\bm{u}_{l}^{\star\perp}\big(\lambda_{l}\bm{I}_{p-1}-\frac{1}{n}\bm{S}_{l,\perp}\bm{S}_{l,\perp}^{\top}\big)^{-1}\frac{1}{n}\bm{S}_{l,\perp}\bm{s}_{l,\parallel}^{\top}$.}

Recognizing that the vector $\bm{s}_{l,\parallel}$ (see (\ref{eq:def:s-l-para-perp}))
obeys 
\[
\bm{s}_{l,\parallel}\sim\mathcal{N}\big(\bm{0},(\lambda_{l}^{\star}+\sigma^{2})\bm{I}_{n}\big)
\]
and is independent of $\bm{S}_{l,\perp}$ (see (\ref{eq:def:s-l-para-perp})),
we can control this quantity through the lemma below.

\begin{lemma}\label{lemma:a-top-P-Uk-1-pca}Instate the assumptions
of Theorem \ref{thm:evector-pertur-sym-iid-pca}. The following holds
with probability at least $1-O(n^{-10})$:
\begin{align}
 & \bigg|\sum_{k:k\ne l}\bm{a}^{\top}\bm{u}_{k}^{\star}\bm{u}_{k}^{\star\top}\bm{u}_{l}^{\star\perp}\Big(\lambda_{l}\bm{I}_{p-1}-\frac{1}{n}\bm{S}_{l,\perp}\bm{S}_{l,\perp}^{\top}\Big)^{-1}\frac{1}{n}\bm{S}_{l,\perp}\bm{s}_{l,\parallel}^{\top}\bigg|\nonumber \\
 & \qquad\lesssim\sum_{k:k\ne l}\frac{\left|\bm{a}^{\top}\bm{u}_{k}^{\star}\right|}{\left|\lambda_{l}^{\star}-\lambda_{k}^{\star}\right|\sqrt{n}}\sqrt{(\lambda_{l}^{\star}+\sigma^{2})(\lambda_{\max}^{\star}+\sigma^{2})(\kappa^{2}+r)\log\bigg(\frac{n\kappa\lambda_{\max}}{\Delta_{l}^{\star}}\bigg)}\label{eq:a-top-uk-pca}
\end{align}
\end{lemma}\begin{proof}See Appendix \ref{subsec:Proof-of-lemma:a-top-P-Uk-1-pca}.\end{proof}

\paragraph{5. Controlling $(\bm{P}_{\bm{U}^{\star\perp}}\bm{a})^{\top}(\bm{P}_{\bm{U}^{\star\perp}}\bm{u}_{l})$.}

When it comes to $(\bm{P}_{\bm{U}^{\star\perp}}\bm{a})^{\top}(\bm{P}_{\bm{U}^{\star\perp}}\bm{u}_{l})$,
we attempt to utilize certain rotational invariance property of $\bm{P}_{\bm{U}^{\star\perp}}\bm{u}_{l}$
in the subspace spanned by $\bm{U}^{\star\perp}$ to upper bound this
quantity. This is formalized in Lemma \ref{lemma:a-top-P-U-perp-u-perp-pca}.

\begin{lemma}\label{lemma:a-top-P-U-perp-u-perp-pca}Instate the
assumptions of Instate the assumptions of Theorem \ref{thm:evector-pertur-sym-iid-pca}.
With probability at least $1-O(n^{-10})$, 
\begin{equation}
\big|(\bm{P}_{\bm{U}^{\star\perp}}\bm{a})^{\top}(\bm{P}_{\bm{U}^{\star\perp}}\bm{u}_{l})\big|\lesssim\sqrt{\frac{\log n}{p-r}}\,\big\|\bm{P}_{\bm{U}^{\star\perp}}\bm{a}\big\|_{2}\big\|\bm{P}_{\bm{U}^{\star\perp}}\bm{u}_{l}\big\|_{2}.\label{eq:a-top-U-perp-u-UB-pca}
\end{equation}
\end{lemma}\begin{proof}The proof is almost identical to the proof
of Lemma \ref{lemma:a-top-P-U-perp-u-perp}, and is hence omitted
for conciseness of presentation.\end{proof} 

In view of Lemma \ref{lemma:a-top-P-U-perp-u-perp-pca}, it suffices
to control $\big\|\bm{P}_{\bm{U}^{\star\perp}}\bm{u}_{l}\big\|_{2}$.
To this end, it is seen from Theorem~\ref{thm:master-theorem-general}
that
\begin{align}
\big\|\bm{P}_{\bm{U}^{\star\perp}}\bm{u}_{l}\big\|_{2}^{2} & =1-\frac{1}{1+\big\|\big(\lambda_{l}\bm{I}-\frac{1}{n}\bm{S}_{\perp}\bm{S}_{\perp}^{\top}\big)^{-1}\frac{1}{n}\bm{S}_{\perp}\bm{S}_{\parallel}^{\top}\bm{U}^{\star\top}\bm{u}_{l,\parallel}\big\|_{2}^{2}}\nonumber \\
 & \le\Big\|\Big(\lambda_{l}\bm{I}-\frac{1}{n}\bm{S}_{\perp}\bm{S}_{\perp}^{\top}\Big)^{-1}\frac{1}{n}\bm{S}_{\perp}\bm{S}_{\parallel}^{\top}\bm{U}^{\star\top}\bm{u}_{l,\parallel}\Big\|_{2}^{2}\nonumber \\
 & \leq\Big\|\Big(\lambda_{l}\bm{I}-\frac{1}{n}\bm{S}_{\perp}\bm{S}_{\perp}^{\top}\Big)^{-1}\Big\|^{2}\cdot\Big\|\frac{1}{n}\bm{S}_{\perp}\bm{S}_{\parallel}^{\top}\Big\|^{2}\cdot\big\|\bm{U}^{\star\top}\bm{u}_{l,\parallel}\big\|_{2}^{2}\nonumber \\
 & =\Big\|\Big(\lambda_{l}\bm{I}-\frac{1}{n}\bm{S}_{\perp}\bm{S}_{\perp}^{\top}\Big)^{-1}\Big\|^{2}\cdot\Big\|\frac{1}{n}\bm{S}_{\perp}\bm{S}_{\parallel}^{\top}\Big\|^{2},\label{eq:P-Uperp-ul-norm-temp}
\end{align}
where we recall $\bm{u}_{l,\parallel}$ is defined to be a unit vector
$\bm{u}_{l,\parallel}\coloneqq\bm{P}_{\bm{U}^{\star}}(\bm{u})/\|\bm{P}_{\bm{U}^{\star}}(\bm{u})\|_{2}$
and satisfies $\bm{U}^{\star}\bm{U}^{\star\top}\bm{u}_{l,\parallel}=\bm{u}_{l,\parallel}$.

The preceding inequality then motivates us to control both $\big\|\big(\lambda_{l}\bm{I}-\frac{1}{n}\bm{S}_{\perp}\bm{S}_{\perp}^{\top}\big)^{-1}\big\|$
and $\big\|\frac{1}{n}\bm{S}_{\perp}\bm{S}_{\parallel}^{\top}\big\|$.
As shown in the proof of Lemma~\ref{lemma:spectral-M-Sigma} in Appendix
\ref{subsec:Proof-of-lemma:lemma:spectral-M-Sigma} (cf.~(\ref{eq:S-perp-op-UB}
and (\ref{eq:S-para-perp-op-UB})), we know that
\begin{align}
\bigg|\frac{1}{n}\big\|\bm{S}_{\perp}\bm{S}_{\perp}^{\top}\big\| & -\sigma^{2}\bigg|\lesssim\sigma^{2}\bigg(\sqrt{\frac{p}{n}}+\frac{p}{n}+\sqrt{\frac{\log n}{n}}\bigg)=o(\lambda_{\min}^{\star}),\label{eq:S-perp-op-UB-repeat}\\
\frac{1}{n}\big\|\bm{S}_{\perp}\bm{S}_{\parallel}^{\top}\big\| & \lesssim\sqrt{\frac{(\lambda_{\max}^{\star}+\sigma^{2})\sigma^{2}(p-r)}{n}}\log n,\label{eq:S-para-perp-op-UB-repeat}
\end{align}
where the relation in (\ref{eq:S-perp-op-UB-repeat}) arises from
the noise condition (\ref{eq:noise-condition-iid-pca}). Combining
these with Theorem~\ref{thm:eigval-pertur-sym-iid-pca}, we obtain
\begin{align*}
\lambda_{l}-\frac{1}{n}\big\|\bm{S}_{\perp}\bm{S}_{\perp}^{\top}\big\| & =\lambda_{l}-\sigma^{2}-o(\lambda_{\min}^{\star})\\
 & \overset{(\mathrm{i})}{\geq}\big(1+\beta(\lambda_{l})\big)(\lambda_{l}^{\star}+\sigma^{2})-\big(1+\beta(\lambda_{l})\big)\cdot O\Big((\lambda_{\max}^{\star}+\sigma^{2})\sqrt{\frac{r}{n}}\log n\Big)-\sigma^{2}-o(\lambda_{\min}^{\star})\\
 & =\big(1+\beta(\lambda_{l})\big)\lambda_{l}^{\star}+\beta(\lambda_{l})\sigma^{2}-O\Big(\big(1+\beta(\lambda_{l})\big)(\lambda_{\max}^{\star}+\sigma^{2})\sqrt{\frac{r}{n}}\log n\Big)-o(\lambda_{\min}^{\star})\\
 & \overset{(\mathrm{ii})}{\gtrsim}\lambda_{l}^{\star}+O\bigg(\frac{\sigma^{4}p}{\lambda_{l}^{\star}n}\bigg)\overset{(\mathrm{iii})}{\asymp}\lambda_{l}^{\star},
\end{align*}
where (i) is due to the bound developed for $\lambda_{l}$ in (\ref{eq:eigenvalue-perturbation-bound-iid-pca})
in Theorem~\ref{thm:eigval-pertur-sym-iid-pca}; (ii) arises from
the fact $\beta(\lambda_{l})\lesssim\frac{\sigma^{2}p}{\lambda_{l}^{\star}n}\ll1$
(as shown in (\ref{eq:beta-UB})) and the noise condition (\ref{eq:noise-condition-iid-pca})
that $(\lambda_{\max}^{\star}+\sigma^{2})\sqrt{r/n}\log n\ll\lambda_{\min}^{\star}$;
(iii) is legal as long as $\sigma^{2}\sqrt{p/n}\ll\lambda_{\min}^{\star}$.
As a result, we obtain
\begin{align}
\Big\|\Big(\lambda_{l}\bm{I}-\frac{1}{n}\bm{S}_{\perp}\bm{S}_{\perp}^{\top}\Big)^{-1}\Big\| & \le\frac{1}{\lambda_{l}-\|\frac{1}{n}\bm{S}_{\perp}\bm{S}_{\perp}^{\top}\|}\lesssim\frac{1}{\lambda_{l}^{\star}};\label{eq:lambdal-S-perp-diff}
\end{align}

Plugging (\ref{eq:S-para-perp-op-UB-repeat}) and (\ref{eq:lambdal-S-perp-diff})
into (\ref{eq:P-Uperp-ul-norm-temp}) immediately reveals that
\begin{equation}
\big\|\bm{P}_{\bm{U}^{\star\perp}}\bm{u}_{l}\big\|_{2}\le\Big\|\Big(\lambda_{l}\bm{I}-\frac{1}{n}\bm{S}_{\perp}\bm{S}_{\perp}^{\top}\Big)^{-1}\Big\|\cdot\frac{1}{n}\big\|\bm{S}_{\perp}\bm{S}_{\parallel}^{\top}\big\|\lesssim\frac{\sqrt{(\lambda_{\max}^{\star}+\sigma^{2})\sigma^{2}}}{\lambda_{l}^{\star}}\sqrt{\frac{p-r}{n}}\log n.
\end{equation}
Taken together with Lemma \ref{lemma:a-top-P-U-perp-u-perp-pca},
this leads to the bound
\begin{equation}
\big|(\bm{P}_{\bm{U}^{\star\perp}}\bm{a})^{\top}(\bm{P}_{\bm{U}^{\star\perp}}\bm{u}_{l})\big|\lesssim\sqrt{\frac{(\lambda_{\max}^{\star}+\sigma^{2})\sigma^{2}}{\lambda_{l}^{\star2}n}}\log^{2}n\,\big\|\bm{P}_{\bm{U}^{\star\perp}}\bm{a}\big\|_{2}.\label{eq:a-top-U-perp-u-pca}
\end{equation}

\paragraph{6. Combining bounds.}

Finally, we can combine (\ref{eq:cos-bound-pca}), (\ref{eq:a-top-uk-pca})
and (\ref{eq:a-top-U-perp-u-pca}) to arrive at the error bound for
the plug-in estimator:
\begin{align*}
\min\big|\bm{a}^{\top}\bm{u}_{l}\pm\bm{a}^{\top}\bm{u}_{l}^{\star}\big| & \lesssim\bigg(\frac{\left(\lambda_{\max}^{\star}+\sigma^{2}\right)\left(\lambda_{l}^{\star}+\sigma^{2}\right)r\log n}{\Delta_l^{^\star 2}n}+\frac{(\lambda_{l}^{\star}+\sigma^{2})\sigma^{2}p}{\lambda_{l}^{\star2}n}\bigg)\big|\bm{a}^{\top}\bm{u}_{l}^{\star}\big|\\
 & \quad+\sum_{k:k\ne l}\frac{\left|\bm{a}^{\top}\bm{u}_{k}^{\star}\right|}{\left|\lambda_{l}^{\star}-\lambda_{k}^{\star}\right|\sqrt{n}}\sqrt{(\lambda_{l}^{\star}+\sigma^{2})(\lambda_{\max}^{\star}+\sigma^{2})(\kappa^{2}+r)\log\bigg(\frac{n\kappa\lambda_{\max}}{\Delta_{l}^{\star}}\bigg)}\\
 & \quad+\sqrt{\frac{(\lambda_{\max}^{\star}+\sigma^{2})\sigma^{2}}{\lambda_{l}^{\star2}n}}\log^{2}n\,\big\|\bm{P}_{\bm{U}^{\star\perp}}\bm{a}\big\|_{2}
\end{align*}
as claimed.

\paragraph{7. Analyzing the de-biased estimator. } To finish up,
let us turn to the de-biased estimator. 
\begin{itemize}
\item Consider first the case with $n\geq p$. We can substitute (\ref{eq:cos-value-pca}),
(\ref{eq:cos-bound-cl}), (\ref{eq:a-top-uk-pca}), and (\ref{eq:a-top-U-perp-u-pca})
into (\ref{eq:au-min-UB-iid-sym-pca-cl}) to obtain
\begin{align*}
 & \min\big|\bm{a}^{\top}\bm{u}_{l}\sqrt{1+c_{l}}\pm\bm{a}^{\top}\bm{u}_{l}^{\star}\big|\\
 & \quad\lesssim\bigg(\frac{(\lambda_{\max}^{\star}+\sigma^{2})(\lambda_{l}^{\star}+\sigma^{2})r\log n}{\Delta_l^{^\star 2}n}+\frac{\sigma^{2}p}{\lambda_{l}^{\star2}n}\Big((\lambda_{\max}^{\star}+\sigma^{2})\sqrt{\frac{r\log n}{n}}+(\lambda_{l}^{\star}+\sigma^{2})\frac{\log n}{\sqrt{p}}\Big)\bigg)\big|\bm{a}^{\top}\bm{u}_{l}^{\star}\big|\\
 & \quad\quad+\sum_{k:k\ne l}\frac{\left|\bm{a}^{\top}\bm{u}_{k}^{\star}\right|}{\left|\lambda_{l}^{\star}-\lambda_{k}^{\star}\right|\sqrt{n}}\sqrt{(\lambda_{l}^{\star}+\sigma^{2})(\lambda_{\max}^{\star}+\sigma^{2})(\kappa^{2}+r)\log\bigg(\frac{n\kappa\lambda_{\max}}{\Delta_{l}^{\star}}\bigg)}+\sqrt{\frac{(\lambda_{\max}^{\star}+\sigma^{2})\sigma^{2}}{\lambda_{l}^{\star2}n}}\log^{2}n\,\big\|\bm{P}_{\bm{U}^{\star\perp}}\bm{a}\big\|_{2}\\
 & \quad\lesssim\frac{(\lambda_{\max}^{\star}+\sigma^{2})(\lambda_{l}^{\star}+\sigma^{2})\,r\log n}{\Delta_l^{^\star 2}n}\big|\bm{a}^{\top}\bm{u}_{l}^{\star}\big|+\sum_{k:k\ne l}\frac{\left|\bm{a}^{\top}\bm{u}_{k}^{\star}\right|}{\left|\lambda_{l}^{\star}-\lambda_{k}^{\star}\right|\sqrt{n}}\sqrt{(\lambda_{l}^{\star}+\sigma^{2})(\lambda_{\max}^{\star}+\sigma^{2})(\kappa^{2}+r)\log\bigg(\frac{n\kappa\lambda_{\max}}{\Delta_{l}^{\star}}\bigg)}\\
 & \quad\quad+\sqrt{\frac{(\lambda_{\max}^{\star}+\sigma^{2})\sigma^{2}r}{\lambda_{l}^{\star2}n}}\log^{2}n,
\end{align*}
where the last step holds due to the noise assumption (\ref{eq:noise-condition-iid-pca})
as well as the facts that $|\bm{a}^{\top}\bm{u}_{l}^{\star}|\leq\|\bm{a}\|_{2}\|\bm{u}_{l}^{\star}\|_{2}\leq1$
and $\big\|\bm{P}_{\bm{U}^{\star\perp}}\bm{a}\big\|_{2}\leq\|\bm{a}\|_{2}=1$. 
\item Consider instead the case with $n<p$ (which implies $\sigma^{2}\ll\lambda_{\min}^{\star}$).
Then we can substitute (\ref{eq:cos-value-pca-p}), (\ref{eq:cos-bound-cl}),
(\ref{eq:a-top-uk-pca}) and (\ref{eq:a-top-U-perp-u-pca}) into (\ref{eq:au-min-UB-iid-sym-pca-cl})
to derive
\begin{align*}
 & \min\big|\bm{a}^{\top}\bm{u}_{l}\sqrt{1+c_{l}}\pm\bm{a}^{\top}\bm{u}_{l}^{\star}\big|\\
 & \quad\lesssim\bigg(\frac{(\lambda_{\max}^{\star}+\sigma^{2})(\lambda_{l}^{\star}+\sigma^{2})r\log n}{\Delta_l^{^\star 2}n}+\frac{\sigma^{2}\sqrt{\kappa^{2}pr\log n}}{\lambda_{l}^{\star}n}\bigg)\big|\bm{a}^{\top}\bm{u}_{l}^{\star}\big|\\
 & \quad\quad+\sum_{k:k\ne l}\frac{\left|\bm{a}^{\top}\bm{u}_{k}^{\star}\right|}{\left|\lambda_{l}^{\star}-\lambda_{k}^{\star}\right|\sqrt{n}}\sqrt{(\lambda_{l}^{\star}+\sigma^{2})(\lambda_{\max}^{\star}+\sigma^{2})(\kappa^{2}+r)\log\bigg(\frac{n\kappa\lambda_{\max}}{\Delta_{l}^{\star}}\bigg)}+\sqrt{\frac{(\lambda_{\max}^{\star}+\sigma^{2})\sigma^{2}}{\lambda_{l}^{\star2}n}}\log^{2}n\,\big\|\bm{P}_{\bm{U}^{\star\perp}}\bm{a}\big\|_{2}\\
 & \quad\lesssim\frac{(\lambda_{\max}^{\star}+\sigma^{2})(\lambda_{l}^{\star}+\sigma^{2})r\log n}{\Delta_l^{^\star 2}n}\big|\bm{a}^{\top}\bm{u}_{l}^{\star}\big|+\sum_{k:k\ne l}\frac{\left|\bm{a}^{\top}\bm{u}_{k}^{\star}\right|}{\left|\lambda_{l}^{\star}-\lambda_{k}^{\star}\right|\sqrt{n}}\sqrt{(\lambda_{l}^{\star}+\sigma^{2})(\lambda_{\max}^{\star}+\sigma^{2})(\kappa^{2}+r)\log\bigg(\frac{n\kappa\lambda_{\max}}{\Delta_{l}^{\star}}\bigg)}\\
 & \quad\quad+\sqrt{\frac{(\lambda_{\max}^{\star}+\sigma^{2})\sigma^{2}\kappa^{2}r}{\lambda_{l}^{\star2}n}}\log^{2}n.
\end{align*}
Here, we use the noise assumption (\ref{eq:noise-condition-iid-pca}),
$|\bm{a}^{\top}\bm{u}_{l}^{\star}|\leq\|\bm{a}\|_{2}\|\bm{u}_{l}^{\star}\|_{2}\leq1$
and $\big\|\bm{P}_{\bm{U}^{\star\perp}}\bm{a}\big\|_{2}\leq\|\bm{a}\|_{2}=1$
again in the last step. 
\end{itemize}

\section{Discussion}

\label{sec:Discussion}

This paper has explored estimation of linear functionals of unknown
eigenvectors under i.i.d.~Gaussian noise, covering the contexts of
both matrix denoising and principal component analysis. We have demonstrated
a non-negligible bias issue inherent to the naive plug-in estimator,
and have proposed more effective estimators that allow for bias correction
in a minimax-optimal and data-driven manner. In comparison to prior
works, our theory accommodates the scenario in which the associated
eigen-gap is substantially smaller than the size of the perturbation,
thereby expanding on what generic matrix perturbation theory has to
offer in these statistical applications. 

Moving forward, there are numerous extensions that are worth pursuing.
For example, the present work is likely suboptimal with respect to
the dependence on the rank $r$ and the condition number $\kappa$,
which calls for a more refined analytical framework to achieve optimal
estimation for more general scenarios. In addition, our current theory
focuses on i.i.d.~Gaussian noise, and a natural question arises as
to how to accommodate sub-Gaussian noise and/or heteroscedastic data.
Furthermore, given the minimax estimation guarantees, an interesting
direction lies in developing statistical inference and uncertainty
quantification schemes for linear forms of the eigenvectors. Accomplishing
this task would require developing distributional guarantees for the
proposed de-biased estimators as well as accurate estimation of the
error variance, which we leave to future investigation.

\section*{Acknowledgements}

Y.~Chen is supported in part by the Alfred P.~Sloan Research Fellowship, the grants AFOSR YIP award FA9550-19-1-0030,
ONR N00014-19-1-2120, ARO YIP award W911NF-20-1-0097, ARO W911NF-18-1-0303, 
NSF CCF-1907661, DMS-2014279, IIS-2218713 and IIS-2218773. 
H.~V.~Poor is supported in part by NSF CCF-1908308,
and in part by a Princeton Schmidt Data-X Research Award. 

\appendix

\section{Proofs of master theorems}

\subsection{Proof of Theorem \ref{thm:master-thm-vector}\label{sec:Proof-of-Theorem-master-vector}}

Given that $\bm{u}_{l}$ is an eigenvector of $\bm{M}$, one has $\bm{M}\bm{u}_{l}=\lambda_{l}\bm{u}_{l}$,
which together with the decomposition (\ref{eq:ul-decomposition-rank1})
and the condition $\bm{u}_{l,\perp}=\bm{q}^{\perp}(\bm{q}^{\perp})^{\top}\bm{u}_{l,\perp}$
gives 
\[
\bm{M}(\bm{q}\cos\theta+\bm{u}_{l,\perp}\sin\theta)=\lambda_{l}(\bm{q}\cos\theta+\bm{u}_{l,\perp}\sin\theta)
\]
\begin{equation}
\Longleftrightarrow\qquad\bm{M}\bm{q}\cos\theta+\bm{M}\bm{q}^{\perp}(\bm{q}^{\perp})^{\top}\bm{u}_{l,\perp}\sin\theta=\lambda_{l}\bm{q}\cos\theta+\lambda_{l}\bm{u}_{l,\perp}\sin\theta.\label{eq:Mu-lambdau-decomposition-1}
\end{equation}
Left-multiplying both sides of this equation by $\bm{q}^{\top}$ (resp.~$(\bm{q}^{\perp})^{\top}$)
and using the assumptions of $\bm{u}_{l,\perp}$ (namely, $\bm{q}^{\top}\bm{u}_{l,\perp}=0$
and $\bm{q}^{\perp}(\bm{q}^{\perp})^{\top}\bm{u}_{l,\perp}=\bm{u}_{l,\perp}$)
give\begin{subequations} 
\begin{align}
\bm{q}^{\top}\bm{M}\bm{q}\cos\theta+\bm{q}^{\top}\bm{M}\bm{q}^{\perp}(\bm{q}^{\perp})^{\top}\bm{u}_{l,\perp}\sin\theta & =\lambda_{l}\cos\theta,\label{eq:ulparallel-equation-2}\\
(\bm{q}^{\perp})^{\top}\bm{M}\bm{q}\cos\theta+(\bm{q}^{\perp})^{\top}\bm{M}\bm{q}^{\perp}(\bm{q}^{\perp})^{\top}\bm{u}_{l,\perp}\sin\theta & =\lambda_{l}(\bm{q}^{\perp})^{\top}\bm{u}_{l,\perp}\sin\theta.\label{eq:ulperp-equation-1}
\end{align}
\end{subequations}

Rearrange terms in (\ref{eq:ulperp-equation-1}) to arrive at 
\begin{equation}
\big(\lambda_{l}\bm{I}_{n-1}-(\bm{q}^{\perp})^{\top}\bm{M}\bm{q}^{\perp}\big)(\bm{q}^{\perp})^{\top}\bm{u}_{l,\perp}\sin\theta=(\bm{q}^{\perp})^{\top}\bm{M}\bm{q}\cos\theta.\label{eq:ulperp-expression-0}
\end{equation}
Given the assumption that $\lambda_{l}\bm{I}_{n-1}-(\bm{q}^{\perp})^{\top}\bm{M}\bm{q}^{\perp}$
is invertible and the fact that $\|(\bm{q}^{\perp})^{\top}\bm{u}_{l,\perp}\|_{2}=\|\bm{q}^{\perp}(\bm{q}^{\perp})^{\top}\bm{u}_{l,\perp}\|_{2}=\|\bm{u}_{l,\perp}\|_{2}=1$,
we claim that it is straightforward to verify that $\cos\theta\ne0$.
To see this, suppose instead that $\cos\theta=0$, then the right-hand
side of (\ref{eq:ulperp-expression-0}) equals to $0$, whereas the
left-hand side of (\ref{eq:ulperp-expression-0}) is non-zero because
$\big(\lambda_{l}\bm{I}_{n-1}-(\bm{q}^{\perp})^{\top}\bm{M}\bm{q}^{\perp}\big)(\bm{q}^{\perp})^{\top}\bm{u}_{l,\perp}\neq\bm{0}$
and $\sin\theta=\sqrt{1-\cos^{2}\theta}=1$. This leads to contradiction,
which in turn reveals that $\cos\theta\ne0$. In addition, if $\sin\theta=0$
(or $\cos\theta=1$), then one has $\bm{q}=\bm{u}_{l}$ and $(\bm{q}^{\perp})^{\top}\bm{M}\bm{q}=0$
(see the relation (\ref{eq:ulperp-expression-0})), from which the
claims (\ref{claim:master-thm-vector}) immediately follow. Hence,
we shall focus on the cases where $\cos\theta\neq0$ and $\sin\theta\ne0$
in the sequel.

Notice that (\ref{eq:ulperp-expression-0}) can be rewritten as
\begin{equation}
(\bm{q}^{\perp})^{\top}\bm{u}_{l,\perp}=\frac{\cos\theta}{\sin\theta}\big(\lambda_{l}\bm{I}_{n-1}-(\bm{q}^{\perp})^{\top}\bm{M}\bm{q}^{\perp}\big)^{-1}(\bm{q}^{\perp})^{\top}\bm{M}\bm{q}.\label{eq:ulperp-expression-1}
\end{equation}
This together with the unit norm constraint of $\bm{u}_{l,\perp}$
and $\bm{u}_{l,\perp}=\bm{q}^{\perp}(\bm{q}^{\perp})^{\top}\bm{u}_{l,\perp}$
implies that 
\begin{equation}
\bm{u}_{l,\perp}=\pm\frac{\bm{q}^{\perp}\big(\lambda_{l}\bm{I}_{n-1}-(\bm{q}^{\perp})^{\top}\bm{M}\bm{q}^{\perp}\big)^{-1}(\bm{q}^{\perp})^{\top}\bm{M}\bm{q}}{\big\|\bm{q}^{\perp}\big(\lambda_{l}\bm{I}_{n-1}-(\bm{q}^{\perp})^{\top}\bm{M}\bm{q}^{\perp}\big)^{-1}(\bm{q}^{\perp})^{\top}\bm{M}\bm{q}\big\|_{2}}\label{eq:ul-perp-expression-2}
\end{equation}
as claimed in (\ref{eq:u-perp-rank1-1}). In addition, substitution
of (\ref{eq:ulperp-expression-1}) into (\ref{eq:ulparallel-equation-2})
with a little algebra yields 
\begin{align*}
\big(\lambda_{l}-\bm{q}^{\top}\bm{M}\bm{q}\big)\cos\theta & =\bm{q}^{\top}\bm{M}\bm{q}^{\perp}(\bm{q}^{\perp})^{\top}\bm{u}_{l,\perp}\sin\theta\\
 & =\bm{q}^{\top}\bm{M}\bm{q}^{\perp}\big(\lambda_{l}\bm{I}_{n-1}-(\bm{q}^{\perp})^{\top}\bm{M}\bm{q}^{\perp}\big)^{-1}(\bm{q}^{\perp})^{\top}\bm{M}\bm{q}\cos\theta,
\end{align*}
\begin{equation}
\Longrightarrow\qquad\lambda_{l}-\bm{q}^{\top}\bm{M}\bm{q}=\bm{q}^{\top}\bm{M}\bm{q}^{\perp}\big(\lambda_{l}\bm{I}_{n-1}-(\bm{q}^{\perp})^{\top}\bm{M}\bm{q}^{\perp}\big)^{-1}(\bm{q}^{\perp})^{\top}\bm{M}\bm{q},\label{eq:lambda-l-long-identity-2}
\end{equation}
thus establishing the claim (\ref{eq:lambda-rank1}).

Finally, rearranging terms in (\ref{eq:ulparallel-equation-2}) yields
\begin{align}
\frac{\sin\theta}{\cos\theta} & =\frac{\lambda_{l}-\bm{q}^{\top}\bm{M}\bm{q}}{\bm{q}^{\top}\bm{M}\bm{q}^{\perp}(\bm{q}^{\perp})^{\top}\bm{u}_{l,\perp}}.\label{eq:ulparallel-equation-1-1-1}
\end{align}
This taken collectively with the elementary identity $\cos^{2}\theta+\sin^{2}\theta=1$
immediately leads to 
\begin{align*}
\cos^{2}\theta & =\frac{1}{1+\frac{|\lambda_{l}-\bm{q}^{\top}\bm{M}\bm{q}|^{2}}{|\bm{q}^{\top}\bm{M}\bm{q}^{\perp}(\bm{q}^{\perp})^{\top}\bm{u}_{l,\perp}|^{2}}}\overset{(\mathrm{i})}{=}\frac{1}{1+\frac{|\lambda_{l}-\bm{q}^{\top}\bm{M}\bm{q}|^{2}\cdot\|\bm{q}^{\perp}(\lambda_{l}\bm{I}_{n-1}-(\bm{q}^{\perp})^{\top}\bm{M}\bm{q}^{\perp})^{-1}(\bm{q}^{\perp})^{\top}\bm{M}\bm{q}\|_{2}^{2}}{|\bm{q}^{\top}\bm{M}\bm{q}^{\perp}(\lambda_{l}\bm{I}_{n-1}-(\bm{q}^{\perp})^{\top}\bm{M}\bm{q}^{\perp})^{-1}(\bm{q}^{\perp})^{\top}\bm{M}\bm{q}|^{2}}}\\
 & \overset{(\mathrm{ii})}{=}\frac{1}{1+\|\bm{q}^{\perp}\big(\lambda_{l}\bm{I}_{n-1}-(\bm{q}^{\perp})^{\top}\bm{M}\bm{q}^{\perp}\big)^{-1}(\bm{q}^{\perp})^{\top}\bm{M}\bm{q}\|_{2}^{2}}\\
 & \overset{(\mathrm{iii})}{=}\frac{1}{1+\|\big(\lambda_{l}\bm{I}_{n-1}-(\bm{q}^{\perp})^{\top}\bm{M}\bm{q}^{\perp}\big)^{-1}(\bm{q}^{\perp})^{\top}\bm{M}\bm{q}\|_{2}^{2}},
\end{align*}
where (i) relies on the expression (\ref{eq:ul-perp-expression-2}),
(ii) results from the identity (\ref{eq:lambda-l-long-identity-2}),
and (iii) follows since $(\bm{q}^{\perp})^{\top}\bm{q}^{\perp}=\bm{I}_{n-1}$.
This establishes the claimed relation (\ref{eq:cos-theta-rank1}).

\subsection{Proof of Theorem \ref{thm:master-theorem-general}\label{sec:Proof-of-Theorem-master-general}}

Given that $\bm{M}\bm{u}_{l}=\lambda_{l}\bm{u}_{l}$, one can invoke
the decomposition (\ref{eq:uk-decomposition}) to obtain 
\begin{equation}
\bm{M}\bm{u}_{l,\|}\cos\theta+\bm{M}\bm{u}_{l,\perp}\sin\theta=\lambda_{l}\bm{u}_{l,\|}\cos\theta+\lambda_{l}\bm{u}_{l,\perp}\sin\theta,
\end{equation}
which together with the conditions $\bm{u}_{l,\|}=\bm{Q}\bm{Q}^{\top}\bm{u}_{l,\|}$
and $\bm{u}_{l,\perp}=\bm{Q}^{\perp}(\bm{Q}^{\perp})^{\top}\bm{u}_{l,\perp}$
implies that 
\begin{equation}
\bm{M}\bm{Q}\bm{Q}^{\top}\bm{u}_{l,\|}\cos\theta+\bm{M}\bm{Q}^{\perp}(\bm{Q}^{\perp})^{\top}\bm{u}_{l,\perp}\sin\theta=\lambda_{l}\bm{u}_{l,\|}\cos\theta+\lambda_{l}\bm{u}_{l,\perp}\sin\theta.\label{eq:Mu-lambdau-decomposition}
\end{equation}
Left-multiplying both sides of this relation by $\bm{Q}^{\top}$ (resp.~$(\bm{Q}^{\perp})^{\top}$)
gives\begin{subequations} 
\begin{align}
\bm{Q}^{\top}\bm{M}\bm{Q}\bm{Q}^{\top}\bm{u}_{l,\|}\cos\theta+\bm{Q}^{\top}\bm{M}\bm{Q}^{\perp}(\bm{Q}^{\perp})^{\top}\bm{u}_{l,\perp}\sin\theta & =\lambda_{l}\bm{Q}^{\top}\bm{u}_{l,\|}\cos\theta,\label{eq:ulparallel-equation}\\
(\bm{Q}^{\perp})^{\top}\bm{M}\bm{Q}\bm{Q}^{\top}\bm{u}_{l,\|}\cos\theta+(\bm{Q}^{\perp})^{\top}\bm{M}\bm{Q}^{\perp}(\bm{Q}^{\perp})^{\top}\bm{u}_{l,\perp}\sin\theta & =\lambda_{l}(\bm{Q}^{\perp})^{\top}\bm{u}_{l,\perp}\sin\theta,\label{eq:ulperp-equation}
\end{align}
\end{subequations}thus indicating that 
\begin{align}
\big(\lambda_{l}\bm{I}_{k}-\bm{Q}^{\top}\bm{M}\bm{Q}\big)\bm{Q}^{\top}\bm{u}_{l,\|}\cos\theta & =\bm{Q}^{\perp}\bm{M}\bm{Q}^{\perp}(\bm{Q}^{\perp})^{\top}\bm{u}_{l,\perp}\sin\theta,\nonumber \\
\big(\lambda_{l}\bm{\bm{I}}_{n-k}-(\bm{Q}^{\perp})^{\top}\bm{M}\bm{Q}^{\perp}\big)(\bm{Q}^{\perp})^{\top}\bm{u}_{l,\perp}\sin\theta & =(\bm{Q}^{\perp})^{\top}\bm{M}\bm{Q}\bm{Q}^{\top}\bm{u}_{l,\|}\cos\theta=(\bm{Q}^{\perp})^{\top}\bm{M}\bm{u}_{l,\|}\cos\theta,\label{eq:Qperp-top-M-u-para-relation}
\end{align}
where the last identity follows since $\bm{Q}\bm{Q}^{\top}\bm{u}_{l,\|}=\bm{u}_{l,\|}$.
These two relations taken together demonstrate that 
\begin{align*}
\big(\lambda_{l}\bm{I}_{k}-\bm{Q}^{\top}\bm{M}\bm{Q}\big)\bm{Q}^{\top}\bm{u}_{l,\|}\cos\theta & =\bm{Q}^{\perp}\bm{M}\bm{Q}^{\perp}\big((\bm{Q}^{\perp})^{\top}\bm{u}_{l,\perp}\sin\theta\big)\\
 & =\bm{Q}^{\perp}\bm{M}\bm{Q}^{\perp}\big(\lambda_{l}\bm{\bm{I}}_{n-k}-(\bm{Q}^{\perp})^{\top}\bm{M}\bm{Q}^{\perp}\big)^{-1}(\bm{Q}^{\perp})^{\top}\bm{M}\bm{u}_{l,\|}\cos\theta.
\end{align*}
In addition, in view of the invertibility of $\lambda_{l}\bm{\bm{I}}_{n-k}-(\bm{Q}^{\perp})^{\top}\bm{M}\bm{Q}^{\perp}$
(due to the assumption) and $\|(\bm{Q}^{\perp})^{\top}\bm{u}_{l,\perp}\|_{2}=\|\bm{u}_{l,\perp}\|_{2}=1$,
one can deduce from (\ref{eq:Qperp-top-M-u-para-relation}) that $\cos\theta\neq0$.
To verify this, suppose $\cos\theta=0$ (or $\sin\theta=1$), then
the left-hand side of (\ref{eq:Qperp-top-M-u-para-relation}) is non-zero
while the right-hand side of (\ref{eq:Qperp-top-M-u-para-relation})
is zero. This results in contradiction, thus justifying that $\cos\theta\neq0$.
Consequently, dividing both sides of the above identity by $\cos\theta$
concludes the proof for the claim (\ref{eq:lambda-l-long-identity-1}).

\section{Proofs of auxiliary lemmas for Theorem \ref{thm:eigval-pertur-sym-iid}}

\subsection{Proof of Lemma \ref{lemma:bound-Glambda-uniform}}

\label{subsec:Proof-of-Lemma-bound-Glambda-uniform}

For notational convenience, divide the matrix $\bm{H}$ as follows
\begin{equation}
\bm{H}=\begin{bmatrix}\bm{H}_{\mathsf{ul}} & \bm{H}_{\mathsf{ur}}\\
\bm{H}_{\mathsf{ur}}^{\top} & \bm{H}_{\mathsf{\mathsf{lr}}}
\end{bmatrix},\quad\bm{H}_{\mathsf{ul}}\in\mathbb{R}^{r\times r},\quad\bm{H}_{\mathsf{ur}}\in\mathbb{R}^{r\times(n-r)},\quad\bm{H}_{\mathsf{lr}}\in\mathbb{R}^{(n-r)\times(n-r)}.\label{eq:defn:H-submatrix}
\end{equation}
In view of the rotational invariance of a symmetric Gaussian matrix,
we know that $\bm{R}\bm{H}\bm{R}^{\top}$ has the same distribution
as $\bm{H}$ for any fixed orthonormal matrix $\bm{R}\in\mathbb{R}^{n\times n}$
obeying $\bm{R}\bm{R}^{\top}=\bm{I}_{n}$. As a result, it is easily
seen that the triple
\begin{equation}
\Big(\bm{U}^{\star\top}\bm{H}\bm{U}^{\star},(\bm{U}^{\star\perp})^{\top}\bm{H}\bm{U}^{\star\perp},\bm{U}^{\star\top}\bm{H}\bm{U}^{\star\perp}\Big)\,\overset{\mathrm{d}}{=}\,(\bm{H}_{\mathsf{ul}},\bm{H}_{\mathsf{lr}},\bm{H}_{\mathsf{ur}}),\label{eq:distributional-equiv-U}
\end{equation}
where $\overset{\mathrm{d}}{=}$ denotes equivalence in distribution.
Equipped with this fact, we are ready to derive the advertised concentration
bounds.

\paragraph{Controlling $\big\|\bm{U}^{\star\top}\bm{H}\bm{U}^{\star}\big\|$.}
Apply the standard Gaussian concentration inequalities \citep{Vershynin2012}
and (\ref{eq:distributional-equiv-U}) to conclude that with probability
at least $1-O(n^{-10})$,
\[
\|\bm{U}^{\star\top}\bm{H}\bm{U}^{\star}\|=\|\bm{H}_{\mathsf{ul}}\|\lesssim\sigma(\sqrt{r}+\sqrt{\log n}).
\]

\paragraph{Controlling $\big\|\bm{G}(\lambda)-\bm{G}^{\perp}(\lambda)\big\|$.}
Consider any fixed $\lambda$ obeying $2\,|\lambda_{l}^{\star}|/3\leq|\lambda|\leq4\,|\lambda_{l}^{\star}|/3$.
Recalling the expression of $\bm{G}(\lambda)$ in (\ref{eq:defn-Glambda-1}),
we have
\begin{align}
\big\|\bm{G}(\lambda)\big\| & =\big\|\bm{U}^{\star\top}\bm{H}\bm{U}^{\star\perp}(\bm{U}^{\star\perp})^{\top}\big(\lambda\bm{I}_{n}-\bm{U}^{\star\perp}(\bm{U}^{\star\perp})^{\top}\bm{H}\bm{U}^{\star\perp}(\bm{U}^{\star\perp})^{\top}\big)^{-1}\bm{U}^{\star\perp}(\bm{U}^{\star\perp})^{\top}\bm{H}\bm{U}^{\star}\big\|\nonumber \\
 & =\big\|\bm{U}^{\star\top}\bm{H}\bm{U}^{\star\perp}\big(\lambda\bm{I}_{n-r}-(\bm{U}^{\star\perp})^{\top}\bm{H}\bm{U}^{\star\perp}\big)^{-1}(\bm{U}^{\star\perp})^{\top}\bm{H}\bm{U}^{\star}\big\|.\label{eq:Glambda-decomposition}
\end{align}
Combining this with the fact (\ref{eq:distributional-equiv-U}), we
see that $\|\bm{G}(\lambda)\|$ has the same distribution as $\big\|\bm{H}_{\mathsf{ur}}\big(\lambda\bm{I}_{n-r}-\bm{H}_{\mathsf{lr}}\big)^{-1}\bm{H}_{\mathsf{ur}}^{\top}\big\|$.
Repeating the same argument also indicates that $\big\|\bm{G}(\lambda)-\bm{G}^{\perp}(\lambda)\big\|$
has the same distribution as 
\begin{equation}
\Big\|\bm{H}_{\mathsf{ur}}\big(\lambda\bm{I}_{n-r}-\bm{H}_{\mathsf{lr}}\big)^{-1}\bm{H}_{\mathsf{ur}}^{\top}-\mathbb{E}\big[\bm{H}_{\mathsf{ur}}\big(\lambda\bm{I}_{n-r}-\bm{H}_{\mathsf{lr}}\big)^{-1}\bm{H}_{\mathsf{ur}}^{\top}\mid\bm{H}_{\mathsf{lr}}\big]\Big\|.\label{eq:distributional-equiv-Glambda-U}
\end{equation}
This allows us to turn attention to $\bm{H}_{\mathsf{ur}}(\lambda\bm{I}_{n-r}-\bm{H}_{\mathsf{lr}})^{-1}\bm{H}_{\mathsf{ur}}^{\top}$. 

As a key observation, $\bm{H}_{\mathsf{ur}}$ and $\bm{H}_{\mathsf{lr}}$
are statistically independent, thus enabling convenient decoupling
of the randomness. Let $\gamma_{1}\geq\cdots\geq\gamma_{n-r}$ represent
the eigenvalues of $\bm{H}_{\mathsf{lr}}$. Denote by $\{\bm{h}_{i}\}_{i=1}^{n-r}$
the columns of $\bm{H}_{\mathsf{ur}}$, i.e.~$\bm{H}_{\mathsf{ur}}=[\bm{h}_{1},\cdots,\bm{h}_{n-r}]$,
which are independent of $\bm{H}_{\mathsf{lr}}$ and $\{\gamma_{i}\}$.
Invoking the rotational invariance of Gaussian random matrices once
again, we see that 
\begin{equation}
\bm{H}_{\mathsf{ur}}(\lambda\bm{I}_{n-r}-\bm{H}_{\mathsf{lr}})^{-1}\bm{H}_{\mathsf{ur}}^{\top}\,\overset{\mathrm{d}}{=}\,\sum_{i=1}^{n-r}\frac{1}{\lambda-\gamma_{i}}\bm{h}_{i}\bm{h}_{i}^{\top},\label{eq:distributional-equiv-Glambda-U-1}
\end{equation}
which is a sum of independent random matrices when conditional on
$\bm{H}_{\mathsf{lr}}$. This can be controlled via Lemma \ref{lemma:covariance-Gaussian-concentration}.
Specifically, conditional on $\bm{H}_{\mathsf{lr}}$ and assuming
that $|\gamma_{i}|\leq\lambda_{\min}^{\star}/3$ for all $i$, we
have 
\begin{align*}
 & \Big\|\sum_{i}\bm{H}_{\mathsf{ur}}(\lambda\bm{I}_{n-r}-\bm{H}_{\mathsf{lr}})^{-1}\bm{H}_{\mathsf{ur}}^{\top}-\mathbb{E}\Big[\sum_{i}\bm{H}_{\mathsf{ur}}(\lambda\bm{I}_{n-r}-\bm{H}_{\mathsf{lr}})^{-1}\bm{H}_{\mathsf{ur}}^{\top}\mid\bm{H}_{\mathsf{lr}}\Big]\,\Big\|\\
 & \qquad=\Big\|\sum_{i}\frac{1}{\lambda-\gamma_{i}}\big(\bm{h}_{i}\bm{h}_{i}^{\top}-\mathbb{E}[\bm{h}_{i}\bm{h}_{i}^{\top}]\big)\Big\|\\
 & \qquad\lesssim\frac{\sigma^{2}}{\min_{i}|\lambda-\gamma_{i}|}\big(\sqrt{rn\log n}+r\log n\big)\\
 & \qquad\lesssim\frac{\sigma^{2}}{|\lambda_{l}^{\star}|}\big(\sqrt{rn\log n}+r\log n\big)
\end{align*}
with probability at least $1-O(n^{-20})$, where the penultimate line
relies on Lemma \ref{lemma:covariance-Gaussian-concentration}, and
the last step follows since $|\lambda-\gamma_{i}|\geq|\lambda|-\max_{i}|\gamma_{i}|\geq2\,|\lambda_{l}^{\star}|/3-\|\bm{H}\|\geq\lambda_{\min}^{\star}/3$
(see (\ref{eq:H-norm-1-3})). Consequently, we have established that,
with probability at least $1-O(n^{-11})$,
\begin{equation}
\big\|\bm{G}(\lambda)-\bm{G}^{\perp}(\lambda)\big\|\lesssim\frac{\sigma^{2}}{|\lambda_{l}^{\star}|}\big(\sqrt{rn\log n}+r\log n\big)\leq\frac{\sigma^{2}}{\lambda_{\min}}\big(\sqrt{rn\log n}+r\log n\big)\label{eq:Glambda-deviation-UB}
\end{equation}
for a given $\lambda$. 

Finally, we apply the standard epsilon-net argument to establish a
uniform bound that holds simultaneously over all $\lambda$ obeying
$2\,|\lambda_{l}^{\star}|/3\leq|\lambda|\leq4\,|\lambda_{l}^{\star}|/3$.
Set $\epsilon_{0}=c\,|\lambda_{l}^{\star}|/n$ for some sufficiently
small constant $c>0$, and let $\mathcal{N}_{\epsilon_{0}}$ denote
an $\epsilon_{0}$-net for $[-4|\lambda_{l}^{\star}|/3,\,-2|\lambda_{l}^{\star}|/3]\cup[2|\lambda_{l}^{\star}|/3,\,4|\lambda_{l}^{\star}|/3]$
with cardinality
\begin{equation}
|\mathcal{N}_{\epsilon_{0}}|\lesssim\lambda_{l}^{\star}/\epsilon_{0}\asymp n;\label{eq:lambda-eps-net}
\end{equation}
see \citet{vershynin2016high} for an introduction of the epsilon-net.
This means that for each $\lambda$ obeying $2\,|\lambda_{l}^{\star}|/3\leq|\lambda|\leq4\,|\lambda_{l}^{\star}|/3$,
one can find a point $\widehat{\lambda}\in\mathcal{N}_{\epsilon_{0}}$
such that $|\lambda-\widehat{\lambda}|\leq\epsilon_{0}$.
\begin{itemize}
\item Take the union bound to show that: with probability exceeding $1-O(n^{-11})$,
\begin{equation}
\big\|\bm{G}(\widehat{\lambda})-\bm{G}^{\perp}(\widehat{\lambda})\big\|\lesssim\frac{\sigma^{2}}{\lambda_{\min}^{\star}}\big(\sqrt{rn\log n}+r\log n\big),\qquad\forall\widehat{\lambda}\in\mathcal{N}.\label{eq:Glambda-deviation-UB-1}
\end{equation}
\item For any $\lambda$ of interest, let $\widehat{\lambda}$ be a point
in $\mathcal{N}_{\epsilon_{0}}$ obeying $|\lambda-\widehat{\lambda}|\leq\epsilon_{0}$.
Then conditioned on $\|\bm{H}\|\leq\lambda_{\min}^{\star}/3$, 
\begin{align}
\|\bm{G}(\lambda)-\bm{G}(\widehat{\lambda})\| & \leq\big\|\bm{H}_{\mathsf{ur}}\big(\lambda\bm{I}_{n-r}-\bm{H}_{\mathsf{lr}}\big)^{-1}\bm{H}_{\mathsf{ur}}^{\top}-\bm{H}_{\mathsf{ur}}\big(\widehat{\lambda}\bm{I}_{n-r}-\bm{H}_{\mathsf{lr}}\big)^{-1}\bm{H}_{\mathsf{ur}}^{\top}\big\|\nonumber \\
 & \leq\|\bm{H}_{\mathsf{ur}}\|^{2}\cdot\big\|\big(\lambda\bm{I}_{n-r}-\bm{H}_{\mathsf{lr}}\big)^{-1}-\big(\widehat{\lambda}\bm{I}_{n-r}-\bm{H}_{\mathsf{lr}}\big)^{-1}\big\|\nonumber \\
 & \leq\|\bm{H}_{\mathsf{ur}}\|^{2}\max_{i}\Big|\frac{1}{\lambda-\gamma_{i}}-\frac{1}{\widehat{\lambda}-\gamma_{i}}\Big|=\|\bm{H}_{\mathsf{ur}}\|^{2}\max_{i}\Big|\frac{\lambda-\widehat{\lambda}}{(\lambda-\gamma_{i})(\widehat{\lambda}-\gamma_{i})}\Big|\nonumber \\
 & \lesssim\sigma^{2}n\cdot\max_{i}\frac{|\lambda-\widehat{\lambda}|}{\lambda_{l}^{\star2}}\leq\sigma^{2}n\cdot\frac{\epsilon_{l}}{\lambda_{l}^{\star2}}\nonumber \\
 & \lesssim\frac{\sigma^{2}}{\lambda_{\min}^{\star}}\label{eq:Glambda-net-approx-error}
\end{align}
holds with probability $1-O(n^{-11})$. Here, the penultimate line
has made use of the Gaussian concentration bound $\|\bm{H}_{\mathsf{ur}}\|\lesssim\sigma\sqrt{n}$,
whereas the last inequality results from (\ref{eq:lambda-eps-net}). 
\item Combining the above two facts together, we arrive at
\begin{align}
 & \sup_{\lambda:\,|\lambda|\in[2|\lambda_{l}^{\star}|/3,\,4|\lambda_{l}^{\star}|/3]}\big\|\bm{G}(\lambda)-\bm{G}^{\perp}(\lambda)\big\|\nonumber \\
 & \qquad=\sup_{\lambda:\,|\lambda|\in[2|\lambda_{l}^{\star}|/3,\,4|\lambda_{l}^{\star}|/3]}\big\|\bm{G}(\lambda)-\bm{G}(\widehat{\lambda})+\bm{G}^{\perp}(\widehat{\lambda})-\bm{G}^{\perp}(\lambda)+\bm{G}(\widehat{\lambda})-\bm{G}^{\perp}(\widehat{\lambda})\big\|\nonumber \\
 & \qquad\leq\sup_{\lambda:\,|\lambda|\in[2|\lambda_{l}^{\star}|/3,\,4|\lambda_{l}^{\star}|/3]}\|\bm{G}(\lambda)-\bm{G}(\widehat{\lambda})\|+\sup_{\widehat{\lambda}:\,\widehat{\lambda}\in\mathcal{N}_{\epsilon_{0}}}\big\|\bm{G}(\widehat{\lambda})-\bm{G}^{\perp}(\widehat{\lambda})\big\|\nonumber \\
 & \qquad\quad\qquad+\sup_{\lambda:\,|\lambda|\in[2|\lambda_{l}^{\star}|/3,\,4|\lambda_{l}^{\star}|/3]}\|\bm{G}^{\perp}(\lambda)-\bm{G}^{\perp}(\widehat{\lambda})\|\nonumber \\
 & \qquad\lesssim\frac{\sigma^{2}}{\lambda_{\min}^{\star}}\big(\sqrt{rn\log n}+r\log n\big).\label{eq:sup-diff-Glambda-Gperp}
\end{align}
Here, the last inequality results from (\ref{eq:Glambda-deviation-UB-1}),
(\ref{eq:Glambda-net-approx-error}), and the following consequence
of Jensen's inequality
\[
\|\bm{G}^{\perp}(\lambda)-\bm{G}^{\perp}(\widehat{\lambda})\|=\|\mathbb{E}\big[\bm{G}(\lambda)\mid\bm{H}_{\mathsf{lr}}\big]-\mathbb{E}\big[\bm{G}(\widehat{\lambda})\mid\bm{H}_{\mathsf{lr}}\big]\|\leq\mathbb{E}\Big[\|\bm{G}(\lambda)-\bm{G}(\widehat{\lambda})\|\mid\bm{H}_{\mathsf{lr}}\Big]\lesssim\frac{\sigma^{2}}{\lambda_{\min}^{\star}},
\]
where we have used (\ref{eq:Glambda-net-approx-error}) again in the
last step. 
\end{itemize}
This concludes the proof of (\ref{eq:Glambda-gap}). 

Finally, the above argument also reveals that 
\begin{align*}
\bm{G}^{\perp}(\lambda) & =\mathbb{E}\Big[\underset{\eqqcolon\,\bm{A}}{\underbrace{\bm{U}^{\star\top}\bm{H}\bm{U}^{\star\perp}\big(\lambda\bm{I}_{n-r}-(\bm{U}^{\star\perp})^{\top}\bm{H}\bm{U}^{\star\perp}\big)^{-1}(\bm{U}^{\star\perp})^{\top}\bm{H}\bm{U}^{\star}}}\mid(\bm{U}^{\star\perp})^{\top}\bm{H}\bm{U}^{\star\perp}\Big]\\
 & =\sigma^{2}\mathsf{tr}\big[\big(\lambda\bm{I}_{n-r}-(\bm{U}^{\star\perp})^{\top}\bm{H}\bm{U}^{\star\perp}\big)^{-1}\big]\bm{I}_{r},
\end{align*}
which holds since the matrix $\bm{A}$ obeys $\bm{A}\overset{\mathrm{d}}{=}\bm{H}_{\mathsf{ur}}(\lambda\bm{I}_{n-r}-\bm{H}_{\mathsf{lr}})^{-1}\bm{H}_{\mathsf{ur}}^{\top}$,
which has been analyzed in (\ref{eq:distributional-equiv-Glambda-U-1}).

\subsection{Proof of Lemma \ref{lemma:lambda-l-mapping}}

\label{subsec:Proof-of-Claim-lambdal-correspondence-iid}

We first claim that: with (\ref{eq:lambdal-Ml-Epsilon}) in place,
one necessarily has\begin{subequations}\label{eq:lambda-l-vague-mapping}
\begin{align}
\big|\lambda_{l}-\lambda_{i}^{\star}-\gamma(\lambda_{l})\big| & \leq\mathcal{E}_{\mathsf{MD}},\qquad\text{for some }1\leq i\leq r\label{eq:defn-lambda-E1}\\
\text{or}\qquad\big|\lambda_{l}\big| & \leq\mathcal{E}_{\mathsf{MD}}\label{eq:defn-lambda-E2}
\end{align}
\end{subequations}for any $1\leq l\leq r$. To see this, we recall
that for any symmetric matrix $\bm{A}$, one has
\[
\min_{i}\big|\lambda_{i}(\bm{A})\big|=\sqrt{\lambda_{\min}(\bm{A}^{2})}=\sqrt{\min_{\bm{x}\in\mathbb{S}^{n-1}}\bm{x}^{\top}\bm{A}^{2}\bm{x}}=\min_{\bm{x}\in\mathbb{S}^{n-1}}\|\bm{A}\bm{x}\|_{2},
\]
where $\mathbb{S}^{n-1}:=\{\bm{z}\in\mathbb{R}^{n}\mid\|\bm{z}\|_{2}=1\}$
and $\lambda_{i}(\bm{A})$ denotes the $i$-th largest eigenvalue
of $\bm{A}$. Recall the definition of $\bm{M}_{\lambda}$ in (\ref{eq:defn-Mlambda}).
Given that the eigenvalues of $\lambda_{l}\bm{I}-\bm{M}_{\lambda_{l}}$
are exactly $\lambda_{l}-\lambda_{i}(\bm{M}_{\lambda_{l}})$ $(1\leq i\leq n)$
and that $\bm{u}_{l,\|}$ is a unit vector, we obtain
\begin{align*}
\min_{1\leq i\leq n}\big|\lambda_{l}-\lambda_{i}(\bm{M}_{\lambda_{l}})\big|= & \min_{1\leq i\leq n}\big|\lambda_{i}\big(\lambda_{l}\bm{I}-\bm{M}_{\lambda_{l}}\big)\big|\leq\big\|(\lambda_{l}\bm{I}_{n}-\bm{M}_{\lambda_{l}})\bm{u}_{l,\|}\big\|_{2}\leq\mathcal{E}_{\mathsf{MD}}.
\end{align*}
This immediately establishes (\ref{eq:lambda-l-vague-mapping}), since
the set of eigenvalues of $\bm{M}_{\lambda_{l}}$ is $\{\lambda_{i}^{\star}+\gamma(\lambda_{l})\mid1\leq i\leq r\}\cup\{0\}$
(in view of the definition (\ref{eq:defn-Mlambda})). 

It thus boils down to how to use (\ref{eq:lambda-l-vague-mapping})
to establish the advertised claim (\ref{eq:lambda-l-mapping-iid}).
Towards this, we find it helpful to define
\begin{align}
\bm{M}(t) & :=\bm{M}^{\star}+t\bm{H},\label{eq:defn-Mt}\\
\gamma(\lambda,t) & :=t^{2}\sigma^{2}\mathsf{tr}\Big(\big(\lambda\bm{I}_{n-r}-t(\bm{U}^{\star\perp})^{\top}\bm{H}\bm{U}^{\star\perp}\big)^{-1}\Big).
\end{align}
We denote by $\{\lambda_{i,t}\}_{i=1}^{n}$ the eigenvalues of $\bm{M}(t)$
obeying $|\lambda_{1,t}|\geq\cdots\geq|\lambda_{n,t}|$; in other
words, $\lambda_{1,t},\cdots,\lambda_{r,t}$ correspond to the $r$
eigenvalues of $\bm{M}(t)$ with the largest magnitudes. Armed with
this notation, we clearly have
\[
\lambda_{l,1}=\lambda_{l},\qquad1\leq l\leq r.
\]
The subsequent analysis consists of three steps.
\begin{itemize}
\item First, we establish the correspondence between $\{\lambda_{l,t}\mid1\leq l\leq r\}$
and $\{\lambda_{i}^{\star}\mid1\leq i\leq r\}$ through the following
lemma; the proof is postponed to Appendix \ref{subsec:Proof-of-Lemma-general-t-lambda-correspondence}.
\begin{lemma}\label{lemma:general-t-lambda-correspondence}Instate
the assumptions of Lemma~\ref{lemma:bound-Glambda-uniform}. Then
with probability exceeding $1-O(n^{-10})$, for any $1\leq l\leq r$,
one can find $1\leq i\leq r$ such that
\[
\sup_{t\in[1/\sqrt{n},\,1]}\big|\lambda_{l,t}-\gamma(\lambda_{l,t},t)-\lambda_{i}^{\star}\big|\leq\mathcal{E}_{t},
\]
where $\mathcal{E}_{t}:=C_{1}t\sigma\sqrt{r}\log n$ for some constant
$C_{1}>0$ large enough. \end{lemma}In other words, this lemma reveals
that for all $1/\sqrt{n}\leq t\leq1$, one has
\[
\lambda_{l,t}-\gamma(\lambda_{l,t},t)\in\cup_{i=1}^{r}\mathcal{B}_{\mathcal{E}_{t}}(\lambda_{i}^{\star}),
\]
where $\mathcal{B}_{\tau}(\lambda):=\left\{ z\mid|z-\lambda|\leq\tau\right\} $
denotes the ball of radius $\tau$ centered at $\lambda$.
\item Secondly, when $0\leq t\leq1/\sqrt{n}$, one has $\|t\bm{H}\|\lesssim c_{0}/\sqrt{n}\cdot(\sigma\sqrt{n})\leq\sigma/2$,
where $c_{0}>0$ is some sufficiently small constant. In this scenario,
Weyl's inequality tells us that $|\lambda_{l,t}-\lambda_{l}^{\star}|\leq\|t\bm{H}\|\leq\sigma/2$.
Further, the definition of $\gamma(\cdot,\cdot)$ indicates that
\[
\big|\gamma(\lambda_{l,t},t)\big|\leq t^{2}\sigma^{2}\frac{n-r}{|\lambda_{l,t}|-\|t\bm{H}\|}\lesssim\frac{1}{n}\cdot\sigma^{2}\frac{n}{\lambda_{\min}^{\star}}=\frac{\sigma^{2}}{\lambda_{\min}^{\star}}\leq\sigma/2,
\]
where the last inequality holds due to the assumption $\sigma\sqrt{n}\lesssim\lambda_{\min}^{\star}$.
As a result, 
\[
|\lambda_{l,t}-\gamma(\lambda_{l,t},t)-\lambda_{l}^{\star}|\leq|\lambda_{l,t}-\lambda_{l}^{\star}|+|\gamma(\lambda_{l,t},t)|\leq\sigma\lesssim\mathcal{E}_{1/\sqrt{n}}
\]
\begin{equation}
\Longrightarrow\qquad\lambda_{l,t}-\gamma(\lambda_{l,t},t)\in\mathcal{B}(\lambda_{l}^{\star},\mathcal{E}_{1/\sqrt{n}}),\qquad0\leq t\leq1/\sqrt{n}.\label{eq:continuous-lambda-l-0}
\end{equation}
\item Recognizing that the set of eigenvalues $\lambda_{l,t}$ ($1\leq l\leq r$)
depends continuously on $t$ \cite[Theorem 6]{embree2001generalizing},
we know that $\lambda_{l,t}-\gamma(\lambda_{l,t},t)$ is also a continuous
function in $t$. In addition, for any $1\leq l\leq r$, if $\min_{k:k\neq l}\big|\lambda_{l}^{\star}-\lambda_{k}^{\star}\big|>2\mathcal{E}_{1}\geq2\mathcal{E}_{t}$
($1/\sqrt{n}\leq t\leq1$), then one necessarily has
\[
\mathcal{B}_{\mathcal{E}_{t}}(\lambda_{l}^{\star})\,\cap\ \left\{ \cup_{k:k>l}\mathcal{B}_{\mathcal{E}_{t}}(\lambda_{k}^{\star})\right\} =\emptyset\qquad\text{and}\qquad\mathcal{B}_{\mathcal{E}_{t}}(\lambda_{l}^{\star})\,\cap\ \left\{ \cup_{k:k<l}\mathcal{B}_{\mathcal{E}_{t}}(\lambda_{k}^{\star})\right\} =\emptyset.
\]
In other words, $\mathcal{B}_{\mathcal{E}_{t}}(\lambda_{l}^{\star})$
remains an isolated region within the set $\cup_{i=1}^{r}\mathcal{B}_{\mathcal{E}_{t}}(\lambda_{i}^{\star})$
when we increase $t$ from $1/\sqrt{n}$ to 1. This together with
the above two facts (namely, the continuity of $\lambda_{l,t}-\gamma(\lambda_{l,t},t)$
in $t$ and (\ref{eq:continuous-lambda-l-0})) requires that 
\[
\lambda_{l,t}-\gamma(\lambda_{l,t},t)\in\mathcal{B}_{\mathcal{E}_{t}}(\lambda_{l}^{\star}),\qquad1/\sqrt{n}\leq t\leq1,
\]
provided that $\min_{k:k\neq l}\big|\lambda_{l}^{\star}-\lambda_{k}^{\star}\big|>2\mathcal{E}_{1}$.
\end{itemize}
Given that our notation satisfies $\lambda_{l,1}=\lambda_{l}$, $\gamma(\lambda_{l,1},1)=\gamma(\lambda_{l})$,
and $\mathcal{E}_{1}=\mathcal{E}_{\mathsf{MD}}$, we conclude that
with probability at least $1-O(n^{-10})$,
\begin{equation}
\big|\lambda_{l}-\gamma(\lambda_{l})-\lambda_{l}^{\star}\big|\leq\mathcal{E}_{\mathsf{MD}},\qquad1\leq l\leq r.\label{eq:lambda-l-bound-correspondence}
\end{equation}

\subsection{Proof of Lemma \ref{lemma:general-t-lambda-correspondence}}

\label{subsec:Proof-of-Lemma-general-t-lambda-correspondence}

Fix an arbitrary $1\leq l\leq r$. We have already shown in (\ref{eq:lambda-l-vague-mapping})
that the claim holds when $t=1$. An inspection of the proof of (\ref{eq:lambda-l-vague-mapping})
reveals that: Lemma \ref{lemma:general-t-lambda-correspondence} can
be established using the same argument, except that we need to generalize
the bound (\ref{eq:Glambda-gap}) into a uniform bound on $\|\bm{G}(\lambda,t)-\bm{G}^{\perp}(\lambda,t)\|$,
namely,
\[
\|\bm{G}(\lambda,t)-\bm{G}^{\perp}(\lambda,t)\|\lesssim\frac{t\sigma^{2}}{\lambda_{\min}^{\star}}\sqrt{rn}\log n
\]
holds simultaneously for all $1/\sqrt{n}\leq t\leq1$ and $\lambda$
with $|\lambda|\in[2|\lambda_{l}^{\star}|/3,\,4|\lambda_{l}^{\star}|/3]$.
Towards this end, we shall resort to the epsilon-net argument once
again. Choose $\epsilon_{1}=c/\sqrt{n}$ for some sufficiently small
constant $c>0$, and let $\mathcal{N}_{\epsilon_{1}}$ be an $\epsilon$-net
for $[1/\sqrt{n},1]$ such that (1) it has cardinality $|\mathcal{N}_{\epsilon_{1}}|\lesssim\sqrt{n}$;
(2) for any $t\in[1/\sqrt{n},1]$, there exists some point $\widehat{t}\in\mathcal{N}_{\epsilon_{1}}$
obeying $|\widehat{t}-t|\leq\epsilon_{1}$. 
\begin{itemize}
\item Applying Lemma \ref{lemma:bound-Glambda-uniform} with the noise matrix
chosen as $t\bm{H}$ and applying the union bound, we see that with
probability exceeding $1-O\left(n^{-11}\right)$, one has
\begin{align*}
\sup_{\lambda:\,|\lambda|\in[2|\lambda_{l}^{\star}|/3,\,4|\lambda_{l}^{\star}|/3]}\|\bm{G}(\lambda,\widehat{t})-\bm{G}^{\perp}(\lambda,\widehat{t})\| & \lesssim\frac{\widehat{t}^{2}\sigma^{2}}{\lambda_{\min}^{\star}}\big(\sqrt{rn\log n}+r\log n\big)\leq\frac{\widehat{t}\sigma^{2}}{\lambda_{\min}^{\star}}\sqrt{rn}\log n
\end{align*}
simultaneously for all $\widehat{t}\in\mathcal{N}_{\epsilon_{1}}$,
where in the second line we have used $\widehat{t}^{2}\leq\widehat{t}$
since $\widehat{t}\in[0,1]$.
\item For any $t\in[1/\sqrt{n},1]$, let $\widehat{t}\in\mathcal{N}_{\epsilon_{1}}$
be a point obeying $|\widehat{t}-t|\leq\epsilon_{1}$. Recognizing
that $\bm{G}(\lambda,\widehat{t})-\bm{G}(\lambda,t)\overset{\mathrm{d}}{=}t^{2}\bm{H}_{\mathsf{ur}}\big(\lambda\bm{I}_{n-r}-t\bm{H}_{\mathsf{lr}}\big)^{-1}\bm{H}_{\mathsf{ur}}^{\top}-\widehat{t}^{2}\bm{H}_{\mathsf{ur}}\big(\lambda\bm{I}_{n-r}-\widehat{t}\bm{H}_{\mathsf{lr}}\big)^{-1}\bm{H}_{\mathsf{ur}}^{\top}$,
one can bound
\begin{align*}
 & \|\bm{G}(\lambda,\widehat{t})-\bm{G}(\lambda,t)\|\leq\big\|\widehat{t}^{2}\bm{H}_{\mathsf{ur}}\big(\lambda\bm{I}_{n-r}-\widehat{t}\bm{H}_{\mathsf{lr}}\big)^{-1}\bm{H}_{\mathsf{ur}}^{\top}-t^{2}\bm{H}_{\mathsf{ur}}\big(\lambda\bm{I}_{n-r}-t\bm{H}_{\mathsf{lr}}\big)^{-1}\bm{H}_{\mathsf{ur}}^{\top}\big\|\\
 & \quad\leq\big\|\widehat{t}^{2}\bm{H}_{\mathsf{ur}}\big(\lambda\bm{I}_{n-r}-\widehat{t}\bm{H}_{\mathsf{lr}}\big)^{-1}\bm{H}_{\mathsf{ur}}^{\top}-t^{2}\bm{H}_{\mathsf{ur}}\big(\lambda\bm{I}_{n-r}-\widehat{t}\bm{H}_{\mathsf{lr}}\big)^{-1}\bm{H}_{\mathsf{ur}}^{\top}\big\|\\
 & \quad\quad\quad\quad+\big\| t^{2}\bm{H}_{\mathsf{ur}}\big(\lambda\bm{I}_{n-r}-\widehat{t}\bm{H}_{\mathsf{lr}}\big)^{-1}\bm{H}_{\mathsf{ur}}^{\top}-t^{2}\bm{H}_{\mathsf{ur}}\big(\lambda\bm{I}_{n-r}-t\bm{H}_{\mathsf{lr}}\big)^{-1}\bm{H}_{\mathsf{ur}}^{\top}\big\|\\
 & \quad\leq|t-\widehat{t}|\cdot|t+\widehat{t}|\cdot\|\bm{H}_{\mathsf{ur}}\|^{2}\big\|\big(\lambda\bm{I}_{n-r}-\widehat{t}\bm{H}_{\mathsf{lr}}\big)^{-1}\big\|+t^{2}\|\bm{H}_{\mathsf{ur}}\|^{2}\big\|\big(\lambda\bm{I}_{n-r}-t\bm{H}_{\mathsf{lr}}\big)^{-1}-\big(\lambda\bm{I}_{n-r}-\widehat{t}\bm{H}_{\mathsf{lr}}\big)^{-1}\big\|.
\end{align*}
Recalling the notation that $\gamma_{1}\geq\cdots\geq\gamma_{n-r}$
represent the eigenvalues of $\bm{H}_{\mathsf{lr}}$, we have
\[
\big\|\big(\lambda\bm{I}_{n-r}-\widehat{t}\bm{H}_{\mathsf{lr}}\big)^{-1}\big\|=\max_{i}\Big|\frac{1}{\lambda-\widehat{t}\gamma_{i}}\Big|\lesssim\frac{1}{\lambda_{\min}^{\star}}\qquad\text{and}
\]
\begin{align*}
\big\|\big(\lambda\bm{I}_{n-r}-t\bm{H}_{\mathsf{lr}}\big)^{-1}-\big(\lambda\bm{I}_{n-r}-\widehat{t}\bm{H}_{\mathsf{lr}}\big)^{-1}\big\| & =\max_{i}\Big|\frac{1}{\lambda-t\gamma_{i}}-\frac{1}{\lambda-\widehat{t}\gamma_{i}}\Big|=\max_{i}\Big|\frac{(t-\widehat{t})\gamma_{i}}{(\lambda-t\gamma_{i})(\lambda-\widehat{t}\gamma_{i})}\Big|\\
 & \lesssim\frac{|t-\widehat{t}|}{\lambda_{\min}^{\star}}\leq\frac{\epsilon_{1}}{\lambda_{\min}^{\star}},
\end{align*}
where we have used the bounds $2\,|\lambda_{l}^{\star}|/3\leq|\lambda|\leq4\,|\lambda_{l}^{\star}|/3$,
$|\gamma_{i}|\leq\lambda_{\min}^{\star}/3$ and $\widehat{t}\leq t+\epsilon_{1}\leq1.1$.
Combining these with the high-probability bound $\|\bm{H}_{\mathsf{ur}}\|\lesssim\sigma\sqrt{n}$,
we arrive at
\[
\|\bm{G}(\lambda,\widehat{t})-\bm{G}(\lambda,t)\|\lesssim\epsilon_{1}\cdot t\cdot\sigma^{2}n\cdot\frac{1}{\lambda_{\min}^{\star}}+t^{2}\cdot\sigma^{2}n\cdot\frac{\epsilon_{1}}{\lambda_{\min}^{\star}}\lesssim\frac{t\sigma^{2}}{\lambda_{\min}^{\star}}\sqrt{n},
\]
where the last step arises since $t^{2}\leq t$ for any $t\in[0,1]$.
Similarly, this bound holds for $\|\bm{G}^{\perp}(\lambda,\widehat{t})-\bm{G}^{\perp}(\lambda,t)\|$
as well.
\end{itemize}
Putting these two upper bounds together, we conclude that with probability
at least $1-O(n^{-11})$,
\begin{align*}
\|\bm{G}(\lambda,t)-\bm{G}^{\perp}(\lambda,t)\| & \leq\|\bm{G}(\lambda,\widehat{t})-\bm{G}^{\perp}(\lambda,\widehat{t})\|+\|\bm{G}(\lambda,\widehat{t})-\bm{G}(\lambda,t)\|+\|\bm{G}^{\perp}(\lambda,\widehat{t})-\bm{G}^{\perp}(\lambda,t)\|\\
 & \lesssim\frac{\widehat{t}\sigma^{2}}{\lambda_{\min}^{\star}}\sqrt{rn}\log n+\frac{t\sigma^{2}}{\lambda_{\min}^{\star}}\sqrt{n}\asymp\frac{t\sigma^{2}}{\lambda_{\min}^{\star}}\sqrt{rn}\log n
\end{align*}
holds simultaneously for all $\lambda$ with $|\lambda|\in[2|\lambda_{l}^{\star}|/3,\,4|\lambda_{l}^{\star}|/3]$
and all $t\in[1/\sqrt{n},1]$. Finally, taking a union bound over
$1\leq l\leq r$ concludes the proof.

\section{Proofs of auxiliary lemmas for Theorem \ref{thm:evector-pertur-sym-iid}}

\subsection{Proof of Lemma \ref{lemma:lambda-M-minus-spectrum}}

\label{subsec:Proof-of-Lemma:lambda-M-minus-spectrum}

To begin with, let us first analyze the eigenvalues of $\bm{M}^{(l)}$,
which is accomplished by the following lemma.

\begin{lemma}\label{lem:eigenvalues-M-slash-l}Instate the assumptions
of Theorem \ref{thm:evector-pertur-sym-iid}. With probability at
least $1-O(n^{-10}),$ one has\begin{subequations}\label{eq:tilde-lambda-UB}
\begin{align}
\big|\lambda_{k}^{(l)}-\gamma(\lambda_{k}^{(l)})-\lambda_{k}^{\star}\big| & \leq\mathcal{E}_{\mathsf{MD}},\qquad1\leq k<l,\label{eq:tilde-lambda-UB1}\\
\big|\lambda_{k}^{(l)}-\gamma(\lambda_{k}^{(l)})-\lambda_{k+1}^{\star}\big| & \leq\mathcal{E}_{\mathsf{MD}},\qquad l\leq k<r,\label{eq:tilde-lambda-UB2}\\
\big|\lambda_{k}^{(l)}\big| & \leq\|\bm{H}\|\lesssim\sigma\sqrt{n},\qquad k\geq r,\label{eq:tilde-lambda-UB3}
\end{align}
\end{subequations}where $\mathcal{E}_{\mathsf{MD}}=C_{1}\sigma\sqrt{r}\log n$
for some sufficiently large constant $C_{1}>0$ and $\gamma(\cdot)$
is defined in (\ref{eq:definition-gamma-lambda-iid-pca}).\end{lemma}\begin{proof}See
Appendix \ref{subsec:Proof-of-Lemma-eigenvalues-M-slash-l}.\end{proof}

Lemma \ref{lem:eigenvalues-M-slash-l} can then be invoked to study
Lemma \ref{lemma:lambda-M-minus-spectrum}. Recalling the fact
\begin{align}
\lambda_{k}-\gamma(\lambda_{k}) & \in\mathcal{B}_{\mathcal{E}_{\mathsf{MD}}}(\lambda_{k}^{\star}),\qquad\qquad1\leq k\leq r\label{eq:lambda-k-LB-123}\\
|\lambda_{k}| & \leq\|\bm{H}\|\lesssim\sigma\sqrt{n},\qquad k>r\nonumber 
\end{align}
as shown in Theorem~\ref{thm:eigval-pertur-sym-iid}, we are positioned
to prove the claim (\ref{eq:lambdal-tilde-lambdak-LB}) as follows.
\begin{itemize}
\item For any $\lambda$ such that $|\lambda|\lesssim\sigma\sqrt{n}$, one
has
\begin{align*}
\big|\lambda_{l}-\lambda\big| & \geq\big|\lambda_{l}-\gamma(\lambda_{l})\big|-\big|\gamma(\lambda_{l})\big|-|\lambda|\overset{(\mathrm{i})}{\geq}\big|\lambda_{l}^{\star}\big|-\mathcal{E}_{\mathsf{MD}}-\big|\gamma(\lambda_{l})\big|-O(\sigma\sqrt{n})\\
 & \overset{(\mathrm{ii})}{\geq}\big|\lambda_{l}^{\star}\big|-O(\sigma\sqrt{r}\log n)-O\Big(\frac{\sigma^{2}n}{\lambda_{\min}^{\star}}\Big)-O(\sigma\sqrt{n})\overset{(\mathrm{iii})}{\gtrsim}\big|\lambda_{l}^{\star}\big|,
\end{align*}
where (i) arises from (\ref{eq:lambda-k-LB-123}) and (\ref{eq:tilde-lambda-UB3}),
(ii) follows since
\begin{equation}
\big|\gamma(\lambda_{l})\big|=\Bigg|\sum_{i}\frac{\sigma^{2}}{\lambda_{l}-\lambda_{i}\big((\bm{U}^{\star\perp})^{\top}\bm{H}\bm{U}^{\star\perp}\big)}\Bigg|\lesssim\frac{\sigma^{2}n}{|\lambda_{l}|-\|\bm{H}\|}\lesssim\frac{\sigma^{2}n}{\lambda_{\min}^{\star}},\label{eq:gamma-lambda-UB}
\end{equation}
and (iii) is valid as long as $\sigma\sqrt{r}\log n\leq c_{0}\lambda_{\min}^{\star}$
and $\sigma\sqrt{n}\leq c_{0}\lambda_{\min}^{\star}$ hold for some
small constant $c_{0}>0$. 
\item For any $\lambda$ satisfying $\lambda-\gamma(\lambda)\in\mathcal{B}_{\mathcal{E}_{\mathsf{MD}}}(\lambda_{k}^{\star})$
for some $1\leq k\leq r$, we define an auxiliary function $f:\pm[2\lambda_{\min}^{\star}/3,\,4\lambda_{\max}^{\star}/3]\rightarrow\mathbb{R}$
by
\begin{equation}
f(\lambda):=\lambda-\gamma(\lambda),\label{eq:def:function-f}
\end{equation}
where we denote $\pm[a,b]=[-b,-a]\cup[a,b]$ for $a<b$ and $\gamma(\cdot)$
is defined in (\ref{eq:definition-gamma-lambda-iid-pca}). To begin
with, for any $\lambda$ with $|\lambda|\in[\lambda_{\min}^{\star}/3,\,2\lambda_{\max}^{\star}]$
one has
\begin{align*}
f'(\lambda) & =1+\sum_{i}\frac{\sigma^{2}}{\big[\lambda-\lambda_{i}\big((\bm{U}^{\star\perp})^{\top}\bm{H}\bm{U}^{\star\perp}\big)\big]^{2}}\geq1-\frac{\sigma^{2}n}{(\tfrac{1}{3}\lambda_{\min}-\|\bm{H}\|)^{2}}\geq\frac{1}{2},\\
f'(\lambda) & \leq1+\frac{\sigma^{2}n}{(\tfrac{1}{3}\lambda_{\min}-\|\bm{H}\|)^{2}}\leq\frac{3}{2},
\end{align*}
with the proviso that $\sigma\sqrt{n}\leq c_{0}\lambda_{\min}^{\star}$
for some constant $c_{0}>0$ small enough. This means that within
the range $|\lambda|\in[2\lambda_{\min}^{\star}/3,\,4\lambda_{\max}^{\star}/3]$,
the function $f(\cdot)$ is monotonically increasing and continuous.
As a result, the inverse of $f(\cdot)$ exists, which is also monotonically
increasing and obeys
\begin{equation}
\frac{2}{3}\leq\frac{\mathrm{d}f^{-1}(\tau)}{\mathrm{d}\tau}\leq2\qquad\forall\tau\text{ with }|\tau|\in[\lambda_{\min}^{\star}/2,3\lambda_{\max}^{\star}/2].\label{eq:f-inverse-derivative-bound}
\end{equation}
 In view of (\ref{eq:lambda-l-mapping-iid}) and the condition that
$\lambda-\gamma(\lambda)\in\mathcal{B}_{\mathcal{E}_{\mathsf{MD}}}(\lambda_{k}^{\star})$
for some $1\leq k\leq r$, one can invoke (\ref{eq:f-inverse-derivative-bound})
to reach
\begin{align*}
\big|\lambda_{l}-\lambda\big| & =\big|f^{-1}(\lambda_{l}^{\star}+O(\mathcal{E}_{\mathsf{MD}}))-f^{-1}(\lambda_{k}^{\star}+O(\mathcal{E}_{\mathsf{MD}}))\big|\\
 & \geq\inf_{\tau:\,|\tau|\in[\lambda_{\min}^{\star}/2,\,3\lambda_{\max}^{\star}/2]}\Big|\frac{\mathrm{d}f^{-1}(\tau)}{\mathrm{d}\tau}\Big|\big|\lambda_{l}^{\star}-\lambda_{k}^{\star}+O(\mathcal{E}_{\mathsf{MD}})\big|\\
 & \geq\frac{2}{3}\,\big|\lambda_{l}^{\star}-\lambda_{k}^{\star}+O(\mathcal{E}_{\mathsf{MD}})\big|\gtrsim\big|\lambda_{l}^{\star}-\lambda_{k}^{\star}\big|,
\end{align*}
where the last inequality holds due to our eigen-gap assumption (\ref{eq:eigengap-condition-iid})
and the fact that $\mathcal{E}_{\mathsf{MD}}\asymp\sigma\sqrt{r}\log n$. 
\item With the analysis above, we note that (\ref{eq:lambdal-tilde-lambdak-LB})
is an immediate consequence of (\ref{eq:tilde-lambda-UB}) in Lemma~\ref{lem:eigenvalues-M-slash-l}.
\end{itemize}

\subsubsection{Proof of Lemma \ref{lem:eigenvalues-M-slash-l}}

\label{subsec:Proof-of-Lemma-eigenvalues-M-slash-l}

The proof of this lemma follows from the same argument employed to
establish Theorem \ref{thm:eigval-pertur-sym-iid}. The idea is to
invoke Theorem~\ref{thm:master-theorem-general} to analyze the spectrum
of $\bm{M}^{(l)}$. Before proceeding, we introduce several notation
tailored to this setting as well as a few simple facts. To begin with,
we define
\begin{align*}
\bm{M}^{\star(l)} & :=(\bm{u}_{l}^{\star\perp})^{\top}\bm{M}^{\star}\bm{u}_{l}^{\star\perp}=(\bm{u}_{l}^{\star\perp})^{\top}\bm{U}^{\star}\bm{\Lambda}^{\star}\bm{U}^{\star\top}\bm{u}_{l}^{\star\perp}=\bm{U}^{\star(l)}\bm{\Lambda}^{\star(l)}\bm{U}^{\star(l)\top},\\
\bm{H}^{(l)} & :=(\bm{u}_{l}^{\star\perp})^{\top}\bm{H}\bm{u}_{l}^{\star\perp},
\end{align*}
where we recall the definitions of $\bm{u}_{l}^{\star\perp}$ (resp.~$\bm{U}^{\star(l)}$
and $\bm{\Lambda}^{\star(l)}$) in (\ref{eq:def:u-l-star-perp}) (resp.~(\ref{eq:def:U-Lambda-star-l})).
In addition, denote
\begin{align*}
\bm{G}^{(l)}(\lambda) & :=\bm{U}^{\star(l)\top}\bm{H}^{(l)}\bm{U}^{\star(l)\perp}\big(\lambda\bm{I}_{n-r}-(\bm{U}^{\star(l)\perp})^{\top}\bm{H}^{(l)}\bm{U}^{\star(l)\perp}\big)^{-1}(\bm{U}^{\star(l)\perp})^{\top}\bm{H}^{(l)}\bm{U}^{\star(l)},\\
\bm{G}^{(l)\perp}(\lambda) & :=\mathbb{E}\big[\bm{G}^{(l)}(\lambda)\mid(\bm{U}^{\star(l)\perp})^{\top}\bm{H}^{(l)}\bm{U}^{\star(l)\perp}\big],
\end{align*}
where $\bm{U}^{\star(l)\perp}$ is defined in (\ref{eq:def:U-Lambda-star-l})
and the expectation is taken assuming that $\lambda$ is independent
of $\bm{H}$. By construction, one has $\bm{u}_{l}^{\star\perp}\bm{U}^{\star(l)\perp}=\bm{U}^{\star\perp}$,
and consequently
\begin{align}
(\bm{U}^{\star(l)\perp})^{\top}\bm{H}^{(l)}\bm{U}^{\star(l)\perp} & =(\bm{U}^{\star(l)\perp})^{\top}(\bm{u}_{l}^{\star\perp})^{\top}\bm{H}\bm{u}_{l}^{\star\perp}\bm{U}^{\star(l)\perp}=(\bm{U}^{\star\perp})^{\top}\bm{H}\bm{U}^{\star\perp},\nonumber \\
(\bm{U}^{\star(l)\perp})^{\top}\bm{H}^{(l)}\bm{U}^{\star(l)} & =(\bm{U}^{\star(l)\perp})^{\top}(\bm{u}_{l}^{\star\perp})^{\top}\bm{H}\bm{u}_{l}^{\star\perp}\bm{U}^{\star(l)}=(\bm{U}^{\star\perp})^{\top}\bm{H}\bm{U}^{\star}\bm{P}^{(l)},\label{eq:def:P-l}
\end{align}
where $\bm{P}^{(l)}$ is obtained by removing the $l$-th column of
$\bm{I}_{r}$, namely, 
\[
\bm{P}^{(l)}:=\begin{bmatrix}\bm{I}_{l-1} & \bm{0}\\
0 & 0\\
\bm{0} & \bm{I}_{r-l}
\end{bmatrix}\in\mathbb{R}^{r\times(r-1)}.
\]
 Therefore, $\bm{G}^{(l)}(\lambda)$ and $\bm{G}^{(l)\perp}(\lambda)$
admit the following simplified expressions
\begin{align*}
\bm{G}^{(l)}(\lambda) & =\bm{U}^{\star\top}\bm{u}_{l}^{\star\perp}(\bm{u}_{l}^{\star\perp})^{\top}\bm{H}\bm{U}^{\star\perp}\big(\lambda\bm{I}_{n-r}-(\bm{U}^{\star\perp})^{\top}\bm{H}\bm{U}^{\star\perp}\big)^{-1}(\bm{U}^{\star\perp})^{\top}\bm{H}\bm{u}_{l}^{\star\perp}(\bm{u}_{l}^{\star\perp})^{\top}\bm{U}^{\star}\\
 & =\bm{P}^{(l)\top}\bm{U}^{\star\top}\bm{H}\bm{U}^{\star\perp}\big(\lambda\bm{I}_{n-r}-(\bm{U}^{\star\perp})^{\top}\bm{H}\bm{U}^{\star\perp}\big)^{-1}(\bm{U}^{\star\perp})^{\top}\bm{H}\bm{U}^{\star}\bm{P}^{(l)}\\
 & =\bm{P}^{(l)\top}\bm{G}(\lambda)\bm{P}^{(l)}\\
\bm{G}^{(l)\perp}(\lambda) & =\mathbb{E}\big[\bm{G}^{(l)}(\lambda)\mid(\bm{U}^{\star\perp})^{\top}\bm{H}\bm{U}^{\star\perp}\big],
\end{align*}
where $\bm{G}(\lambda)$ is defined in (\ref{eq:defn-Glambda-1}).

With the above preparation in place, we can repeat the proof of Theorem
\ref{thm:eigval-pertur-sym-iid} to obtain
\begin{align*}
\big\|\bm{U}^{\star(l)\top}\bm{H}^{(l)}\bm{U}^{\star(l)}\big\| & \lesssim\sigma\big(\sqrt{r}+\sqrt{\log n}\big)\\
\sup_{\lambda:\,|\lambda|\in[2\lambda_{\min}^{\star}/3,\,4\lambda_{\max}^{\star}/3]}\big\|\bm{G}^{(l)}(\lambda)-\bm{G}^{(l)\perp}(\lambda)\big\| & \lesssim\frac{\sigma^{2}}{\lambda_{\min}^{\star}}\big(\sqrt{rn\log n}+r\log n\big)\\
\bm{G}^{(l)\perp}(\lambda) & =\gamma(\lambda)\bm{P}^{(l)\top}\bm{U}^{\star}\bm{U}^{\star\top}\bm{P}^{(l)}
\end{align*}
with probability at least $1-O(n^{-10}),$ where $\gamma(\cdot)$
is defined in (\ref{eq:definition-gamma-lambda-iid-pca}). The above
observations reveal that the $k$-th eigenvalue of $\bm{M}^{\star(l)}+\bm{G}^{(l)\perp}(\lambda)$
is given by
\[
\lambda_{k}\big(\bm{M}^{\star(l)}+\bm{G}^{(l)\perp}(\lambda)\big)=\begin{cases}
\lambda_{k}^{\star}+\gamma(\lambda), & 1\leq k\leq l-1;\\
\lambda_{k+1}^{\star}+\gamma(\lambda), & l\leq k\leq r-1;\\
0, & r\leq k.
\end{cases}
\]
As a result, repeating the same arguments of Theorem \ref{thm:eigval-pertur-sym-iid}
(which we omit for brevity) immediately establishes the claim of this
lemma. 

\subsection{Proof of Lemma \ref{lemma:lambda-M-l-inv-u-perp-H-u-l2-norm}}

\label{subsec:Proof-of-lemma:lambda-M-l-inv-u-perp-H-u-l2-norm}

Let $\bm{U}^{(l)}\bm{\Lambda}^{(l)}\bm{U}^{(l)\top}$ represent the
eigen-decomposition of $\bm{M}^{(l)}$, where $\bm{U}^{(l)}=\big[\bm{u}_{1}^{(l)},\cdots,\bm{u}_{n-1}^{(l)}\big]$.
We can derive
\begin{align}
\big\|\big(\lambda_{l}\bm{I}_{n-1}-\bm{M}^{(l)}\big)^{-1}(\bm{u}_{l}^{\star\perp})^{\top}\bm{H}\bm{u}_{l}^{\star}\big\|_{2}^{2} & =\big\|\bm{U}^{(l)}\big(\lambda_{l}\bm{I}_{n-1}-\bm{\Lambda}^{(l)}\big)^{-1}\bm{U}^{(l)\top}(\bm{u}_{l}^{\star\perp})^{\top}\bm{H}\bm{u}_{l}^{\star}\big\|_{2}^{2}\nonumber \\
 & =\big\|\big(\lambda_{l}\bm{I}_{n-1}-\bm{\Lambda}^{(l)}\big)^{-1}\bm{U}^{(l)\top}(\bm{u}_{l}^{\star\perp})^{\top}\bm{H}\bm{u}_{l}^{\star}\big\|_{2}^{2}\nonumber \\
 & =\sum_{1\le k<n}\bigg(\frac{\bm{u}_{k}^{(l)\top}(\bm{u}_{l}^{\star\perp})^{\top}\bm{H}\bm{u}_{l}^{\star}}{\lambda_{l}-\lambda_{k}^{(l)}}\bigg)^{2}.\label{eq:cos0}
\end{align}
By construction (cf.~(\ref{eq:definition-M-slash-l})), the matrix
$\bm{M}^{(l)}$ (and hence $\bm{\Lambda}^{(l)}$ and $\bm{U}^{(l)}$)
is independent of $(\bm{u}_{l}^{\star\perp})^{\top}\bm{H}\bm{u}_{l}^{\star}$,
thus indicating that 
\[
\bm{U}^{(l)\top}(\bm{u}_{l}^{\star\perp})^{\top}\bm{H}\bm{u}_{l}^{\star}\sim\mathcal{N}(0,\sigma^{2}\bm{I}_{n-1}).
\]
In addition, notice that the distribution of $\bm{U}^{(l)\top}(\bm{u}_{l}^{\star\perp})^{\top}\bm{H}\bm{u}_{l}^{\star}$
is independent with $\bm{U}^{(l)}$ and $\bm{\Lambda}^{(l)}$. In
what follows, we shall look at (\ref{eq:cos0}) by controlling the
sum over $k<r$ and the sum over $k\ge r$ separately.
\begin{itemize}
\item To begin with, let us upper bound $\sum_{1\le k<r}\Big(\frac{\bm{u}_{k}^{(l)\top}(\bm{u}_{l}^{\star\perp})^{\top}\bm{H}\bm{u}_{l}^{\star}}{\lambda_{l}-\lambda_{k}^{(l)}}\Big)^{2}$.
Given a sequence of i.i.d.~standard Gaussian random variables $Z_{i}\overset{\mathrm{i.i.d.}}{\sim}\mathcal{N}(0,\sigma^{2})$,
one knows from the standard Gaussian concentration inequality that
the following holds with probability at least $1-O(n^{-20})$:\begin{subequations}
\begin{align}
\max_{1\le i\leq n}|Z_{i}| & \lesssim\sigma\sqrt{\log n};\label{eq:Gaussian-UB}\\
\max_{1\le i\leq n}|Z_{i}^{2}-\sigma^{2}| & \leq\max_{1\le i\leq n}|Z_{i}-\sigma|\cdot\max_{1\le i\leq n}|Z_{i}+\sigma|\lesssim\sigma^{2}\log n.\label{eq:Gaussian-sq-UB}
\end{align}
\end{subequations}In addition, Lemma \ref{lemma:lambda-M-minus-spectrum}
tells us that $\min_{1\le k<r}\big|\lambda_{l}-\lambda_{k}^{(l)}\big|^{2}\gtrsim\min_{i:i\neq l}|\lambda_{l}^{\star}-\lambda_{i}^{\star}|^{2}$.
These two bounds taken together give
\begin{equation}
\sum_{1\le k<r}\bigg(\frac{\bm{u}_{k}^{(l)\top}(\bm{u}_{l}^{\star\perp})^{\top}\bm{H}\bm{u}_{l}^{\star}}{\lambda_{l}-\lambda_{k}^{(l)}}\bigg)^{2}\lesssim\frac{\sigma^{2}r\log n}{\min_{i:i\neq l}|\lambda_{l}^{\star}-\lambda_{i}^{\star}|^{2}}=\frac{\sigma^{2}r\log n}{\big(\Delta_{l}^{\star}\big)^{2}}.\label{eq:lambda-M-r-bound}
\end{equation}
\item Next, we move on to the remaining term (the sum over $r\leq k<n$).
We claim for the moment that: with probability exceeding $1-O(n^{-10})$,
\begin{align}
\bigg|\sum_{r\le k<n}\bigg(\frac{\bm{u}_{k}^{(l)\top}(\bm{u}_{l}^{\star\perp})^{\top}\bm{H}\bm{u}_{l}^{\star}}{\lambda_{l}-\lambda_{k}^{(l)}}\bigg)^{2}-\sum_{r\le k<n}\frac{\sigma^{2}}{(\lambda_{l}-\lambda_{k}^{(l)})^{2}}\bigg|\lesssim\frac{\sigma^{2}}{\lambda_{l}^{\star2}}\sqrt{n\log n},\label{eq:lambda-M-r-n-value-bound}
\end{align}
The proof of this claim is deferred to the end of this section. It
then suffices to control the term $\sum_{r\leq k<n}\sigma^{2}/(\lambda_{l}-\lambda_{k}^{(l)})^{2}$,
which is established in the lemma below (with the proof postponed
to Appendix \ref{subsec:Proof-of-Lemma-lambda-M-r-n}).

\begin{lemma}\label{lemma:lambda-M-r-n} Instate the assumptions
of Theorem \ref{thm:evector-pertur-sym-iid}. With probability at
least $1-O(n^{-10})$, 
\begin{equation}
\bigg|\sum_{r\le k\le n-1}\frac{1}{(\lambda_{l}-\lambda_{k}^{(l)})^{2}}-\sum_{r+1\le k\le n}\frac{1}{(\lambda_{l}-\lambda_{k})^{2}}\bigg|\lesssim\frac{1}{\lambda_{l}^{\star2}}.\label{eq:lambda-M-r-n-value}
\end{equation}
As a consequence, we have
\begin{equation}
\bigg|\sum_{r\le k\le n-1}\frac{1}{(\lambda_{l}-\lambda_{k}^{(l)})^{2}}\bigg|\vee\bigg|\sum_{r+1\le k\le n}\frac{1}{(\lambda_{l}-\lambda_{k})^{2}}\bigg|\lesssim\frac{n}{\lambda_{l}^{\star2}}.\label{eq:lambda-M-r-n-bound}
\end{equation}
\end{lemma}

Therefore, combining (\ref{eq:lambda-M-r-n-value-bound}) and (\ref{eq:lambda-M-r-n-value})
gives
\begin{align}
 & \bigg|\sum_{r\le k\le n-1}\bigg(\frac{\bm{u}_{k}^{(l)\top}(\bm{u}_{l}^{\star\perp})^{\top}\bm{H}\bm{u}_{l}^{\star}}{\lambda_{l}-\lambda_{k}^{(l)}}\bigg)^{2}-\sum_{r+1\le k\le n}\frac{\sigma^{2}}{(\lambda_{l}-\lambda_{k})^{2}}\bigg|\nonumber \\
 & \qquad\lesssim\bigg|\sum_{r\le k\le n-1}\bigg(\frac{\bm{u}_{k}^{(l)\top}(\bm{u}_{l}^{\star\perp})^{\top}\bm{H}\bm{u}_{l}^{\star}}{\lambda_{l}-\lambda_{k}^{(l)}}\bigg)^{2}-\sum_{r\le k<n}\frac{\sigma^{2}}{(\lambda_{l}-\lambda_{k}^{(l)})^{2}}\bigg|+\bigg|\sum_{r\le k<n}\frac{\sigma^{2}}{(\lambda_{l}-\lambda_{k}^{(l)})^{2}}-\sum_{r+1\le k\le n}\frac{\sigma^{2}}{(\lambda_{l}-\lambda_{k})^{2}}\bigg|\nonumber \\
 & \qquad\lesssim\frac{\sigma^{2}}{\lambda_{l}^{\star2}}\sqrt{n\log n}+\frac{\sigma^{2}}{\lambda_{l}^{\star2}}\asymp\frac{\sigma^{2}}{\lambda_{l}^{\star2}}\sqrt{n\log n}.\label{eq:cos_r_n}
\end{align}

\item Inserting (\ref{eq:lambda-M-r-bound}) and (\ref{eq:cos_r_n}) into
(\ref{eq:cos0}), we arrive at the advertised bound:
\begin{align*}
\big\|\big(\lambda_{l}\bm{I}_{n-1}-\bm{M}^{(l)}\big)^{-1}(\bm{u}_{l}^{\star\perp})^{\top}\bm{H}\bm{u}_{l}^{\star}\big\|_{2}^{2} & =\sum_{r+1\le k\le n}\frac{\sigma^{2}}{(\lambda_{l}-\lambda_{k})^{2}}+O\bigg(\frac{\sigma^{2}r\log n}{\big(\Delta_{l}^{\star}\big)^{2}}+\frac{\sigma^{2}}{\lambda_{l}^{\star2}}\sqrt{n\log n}\bigg)\\
 &\asymp\frac{\sigma^{2}n}{\lambda_{l}^{\star2}}+O\bigg(\frac{\sigma^{2}r\log n}{\big(\Delta_{l}^{\star}\big)^{2}}\bigg)\ll1,
\end{align*}
where the last inequality holds due to our noise assumption (\ref{eq:eigengap-condition-iid}). 
\end{itemize}

\paragraph{Proof of the claim (\ref{eq:lambda-M-r-n-value-bound}). }

Note that $\bm{u}_{k}^{(l)\top}(\bm{u}_{l}^{\star\perp})^{\top}\bm{H}\bm{u}_{l}^{\star}\overset{\mathsf{i.i.d.}}{\sim}\mathcal{N}(0,\sigma^{2})$
is independent of $\bm{\Lambda}^{(l)}$ but depends on $\bm{\Lambda}$.
Therefore, we shall use the epsilon-net argument (i.e.~Lemma \ref{lemma:eps-net})
to bound it. Before proceeding, observe that from (\ref{eq:lambdal-tilde-lambdak-LB}),
(\ref{eq:gamma-lambda-UB}) and the condition $\sigma\sqrt{n}\ll\lambda_{\min}^{\star}$,
the following holds for any $\lambda$ obeying $\lambda-\gamma(\lambda)\in\mathcal{B}_{\mathcal{E}_{\mathsf{MD}}}(\lambda_{l}^{\star})$:
\begin{equation}
|\lambda-\lambda_{k}^{(l)}|\ge|\lambda_{l}-\lambda_{k}^{(l)}|-|\lambda-\lambda_{l}-\gamma(\lambda_{l})|-|\gamma(\lambda_{l})|\gtrsim|\lambda_{l}^{\star}|,\qquad k\geq r.\label{eq:lambda-lambda-kl-LB}
\end{equation}
Now we begin to check the conditions of Lemma \ref{lemma:eps-net}.
Since $|f(x)-f(y)|\leq\sup_{x}|f'(x)||x-y|$, the following holds
with probability at least $1-O(n^{-20})$ for all $\lambda$ with
$\lambda-\gamma(\lambda)\in\mathcal{B}_{\mathcal{E}_{\mathsf{MD}}}(\lambda_{l}^{\star})$:

\begin{align*}
\bigg|\frac{\mathrm{d}}{\mathrm{d}\lambda}\sum_{r\le k<n}\frac{\big(\bm{u}_{k}^{(l)\top}(\bm{u}_{l}^{\star\perp})^{\top}\bm{H}\bm{u}_{l}^{\star}\big)^{2}-\sigma^{2}}{(\lambda-\lambda_{k}^{(l)})^{2}}\bigg| & =\bigg|\sum_{r\le k<n}\frac{\big(\bm{u}_{k}^{(l)\top}(\bm{u}_{l}^{\star\perp})^{\top}\bm{H}\bm{u}_{l}^{\star}\big)^{2}-\sigma^{2}}{(\lambda-\lambda_{k}^{(l)})^{3}}\Big|\\
 & \leq n\cdot\max_{r\le k<n}\frac{1}{\big|\lambda-\lambda_{k}^{(l)}\big|^{3}}\cdot\max_{r\le k<n}\big|\bm{u}_{k}^{(l)\top}(\bm{u}_{l}^{\star\perp})^{\top}\bm{H}\bm{u}_{l}^{\star}\big)^{2}-\sigma^{2}\big|\\
 & \lesssim\frac{\sigma^{2}n\log n}{\lambda_{l}^{\star3}},
\end{align*}
where we use (\ref{eq:Gaussian-sq-UB}) and (\ref{eq:lambda-lambda-kl-LB}).
In addition, for any fixed $\lambda$ obeying $\lambda-\gamma(\lambda)\in\mathcal{B}_{\mathcal{E}_{\mathsf{MD}}}(\lambda_{l}^{\star})$,
one has
\begin{align*}
\max_{r\le k<n}\frac{1}{(\lambda-\lambda_{k}^{(l)})^{2}}\big\|\big(\bm{u}_{k}^{(l)\top}(\bm{u}_{l}^{\star\perp})^{\top}\bm{H}\bm{u}_{l}^{\star}\big)^{2}-\sigma^{2}\big\|_{\psi_{1}} & \lesssim\frac{\sigma^{2}}{\lambda_{l}^{\star2}}=:L;\\
\sum_{r\leq k<n}\frac{1}{(\lambda-\lambda_{k}^{(l)})^{4}}\mathbb{E}\Big[\big(\big(\bm{u}_{k}^{(l)\top}(\bm{u}_{l}^{\star\perp})^{\top}\bm{H}\bm{u}_{l}^{\star}\big)^{2}-\sigma^{2}\big)^{2} & \Big]\lesssim\frac{\sigma^{4}n}{\lambda_{l}^{\star4}}=:V
\end{align*}
where $\|\cdot\|_{\psi_{1}}$ denote the sub-exponential norm. We
can then apply the matrix Bernstein inequality \cite[Corollary 2.1]{Koltchinskii2011oracle}
to find: with probability exceeding $1-O(n^{-20})$, 
\begin{align*}
\bigg|\sum_{r\le k<n}\frac{\big(\bm{u}_{k}^{(l)\top}(\bm{u}_{l}^{\star\perp})^{\top}\bm{H}\bm{u}_{l}^{\star}\big)^{2}-\sigma^{2}}{(\lambda-\lambda_{k}^{(l)})^{2}} & \bigg|\lesssim L\log^{2}n+\sqrt{V\log n}\lesssim\frac{\sigma^{2}\sqrt{n\log n}}{\lambda_{l}^{\star2}}.
\end{align*}
Recognizing the fact that $\{\lambda\colon\lambda-\gamma(\lambda)\in\mathcal{B}_{\mathcal{E}_{\mathsf{MD}}}(\lambda_{l}^{\star})\}\subset\big[\lambda_{l}^{\star}-|\lambda_{l}^{\star}|/3,\,\lambda_{l}^{\star}+|\lambda_{l}^{\star}|/3\big]$,
the claim (\ref{eq:lambda-lambda-kl-LB}) immediately follows from
Lemma \ref{lemma:eps-net}.

\subsubsection{Proof of Lemma \ref{lemma:lambda-M-r-n}}

\label{subsec:Proof-of-Lemma-lambda-M-r-n}

Let us look at (\ref{eq:lambda-M-r-n-bound}) first. According to
Lemma \ref{lemma:lambda-M-minus-spectrum}, we have $|\lambda_{l}-\lambda_{k}^{(l)}|\gtrsim|\lambda_{l}^{\star}|$
for all $k\ge r$, thus leading to
\[
\bigg|\sum_{r\le k\le n-1}\frac{1}{(\lambda_{l}-\lambda_{k}^{(l)})^{2}}\bigg|\lesssim\frac{n}{\lambda_{l}^{\star2}}.
\]
In addition, the upper bound for $\sum_{r\le k\le n-1}1/(\lambda_{l}-\lambda_{k})^{2}$
is an immediate consequence of (\ref{eq:lambda-M-r-n-value}). Therefore,
the remainder of the proof amounts to establishing (\ref{eq:lambda-M-r-n-value}),
which requires us to characterize the relation between the spectrums
of $\bm{M}^{(l)}=(\bm{u}_{l}^{\star\perp})^{\top}\bm{M}\bm{u}_{l}^{\star\perp}$
and $\bm{M}$.

Without loss of generality, assume that $\lambda_{l}^{\star}>0$,
and that there are $m$ (resp.~$r-m$) eigenvalues of $\bm{M}^{\star}$
larger (resp.~smaller) than $0$. By Weyl's inequality (similar to
(\ref{eq:lambdal-range}) in the proof of Theorem \ref{thm:eigval-pertur-sym-iid}),
it is easily seen that there are $m$ eigenvalues of $\bm{M}$ larger
than $c\sigma\sqrt{n}$ and that there are $r-m$ eigenvalues of $\bm{M}$
smaller than $-c\sigma\sqrt{n}$, where $c>0$ is some constant. Recalling
that $\{\lambda_{k}\}_{k=1}^{n}$ are defined as the eigenvalues of
$\bm{M}$ satisfying $|\lambda_{1}|\geq\cdots\geq|\lambda_{n}|$,
we further denote by $\{\phi_{k}\}_{k=1}^{n}$ the eigenvalues of
$\bm{M}$ so that $\phi_{1}\ge\cdots\ge\phi_{n}$. Consider the set
of eigenvalues of $\bm{M}$ with magnitudes upper bounded by $c\sigma\sqrt{n}$.
We have the following relation:
\begin{equation}
\sum_{k:\,r+1\leq k\le n}\frac{1}{(\lambda_{l}-\lambda_{k})^{2}}=\sum_{k:\,m+1\leq k\le n-r+m}\frac{1}{(\lambda_{l}-\phi_{k})^{2}}.
\end{equation}
Similarly, for $\bm{M}^{(l)}$ we can write
\begin{equation}
\sum_{k:\,r\le k\le n-1}\frac{1}{(\lambda_{l}-\lambda_{k}^{(l)})^{2}}=\sum_{k:\,m\le k\le n-r+m-1}\frac{1}{(\lambda_{l}-\phi_{k}^{(l)})^{2}},
\end{equation}
where $\{\phi_{k}^{(l)}\}_{k=1}^{n-1}$ denote the eigenvalues of
$\bm{M}^{(l)}$ in descending order. As a result, in order to establish
(\ref{eq:lambda-M-r-n-value}), it is sufficient to show
\begin{equation}
\Big|\sum_{k:\,m\le k\le n-r+m-1}\frac{1}{(\lambda_{l}-\phi_{k}^{(l)})^{2}}-\sum_{k:\,m+1\le k\le n-r+m}\frac{1}{(\lambda_{l}-\phi_{k})^{2}}\Big|\lesssim\frac{1}{\lambda_{l}^{\star2}}.
\end{equation}

In view of an eigenvalue interlacing result stated in Lemma \ref{lemma:eigval-interlacing},
the definition $\bm{M}^{(l)}:=(\bm{u}_{l}^{\star\perp})^{\top}\bm{M}\bm{u}_{l}^{\star\perp}$
allows us to deduce that
\begin{equation}
\phi_{k+1}\le\phi_{k}^{(l)}\le\phi_{k}\qquad1\leq k<n.\label{eq:interlacing}
\end{equation}
By the assumption $\lambda_{l}^{\star}>0$, one has $\lambda_{l}\geq\lambda_{l}^{\star}-\|\bm{H}\|\gtrsim\lambda_{l}^{\star}$
and thus $\phi_{k}\leq\lambda_{l}$ for all $k\ge m$. Consequently,
we know from (\ref{eq:interlacing}) that for all $k\geq m$, 
\[
\phi_{k+1}\le\phi_{k}^{(l)}\leq\phi_{k}\leq\lambda_{l},
\]
which further implies that
\[
\frac{1}{(\lambda_{l}-\phi_{k+1})^{2}}\le\frac{1}{(\lambda_{l}-\phi_{k}^{(l)})^{2}}\le\frac{1}{(\lambda_{l}-\phi_{k})^{2}},\qquad k\geq m.
\]
This enables us to bound
\begin{align*}
\sum_{k:\,m\le k\le n-r+m-1}\frac{1}{(\lambda_{l}-\phi_{k}^{(l)})^{2}} & \leq\frac{1}{(\lambda_{l}-\phi_{m}^{(l)})^{2}}+\sum_{k:\,m+1\le k\le n-r+m-1}\frac{1}{(\lambda_{l}-\phi_{k})^{2}};\\
\sum_{k:\,m\le k\le n-r+m-1}\frac{1}{(\lambda_{l}-\phi_{k}^{(l)})^{2}} & \geq\sum_{k:\,m\le k\le n-r+m-1}\frac{1}{(\lambda_{l}-\phi_{k+1})^{2}}=\sum_{k:\,m+1\le k\le n-r+m}\frac{1}{(\lambda_{l}-\phi_{k})^{2}}.
\end{align*}
Consequently, we conclude that
\[
\Big|\sum_{k:\,m\le k\le n-r+m-1}\frac{1}{(\lambda_{l}-\phi_{k}^{(l)})^{2}}-\sum_{k:\,m+1\le k\le n-r+m}\frac{1}{(\lambda_{l}-\phi_{k})^{2}}\Big|\le\frac{1}{(\lambda_{l}-\phi_{m}^{(l)})^{2}}\asymp\frac{1}{\lambda_{l}^{\star2}},
\]
where the last relation holds since $|\phi_{m}^{(l)}|\le\|\bm{H}\|$
and hence $\lambda_{l}-\phi_{m}^{(l)}\asymp\lambda_{l}\asymp\lambda_{l}^{\star}$. 

The above analysis can be easily adopted to handle the case where
$\lambda_{l}^{\star}<0$ as well (which we omit here for brevity).
Therefore, we have finished the proof.

\subsection{Proof of Lemma \ref{lemma:a-top-P-Uk}}

\label{subsec:Proof-of-lemma:a-top-P-Uk}

Since $(\bm{u}_{l}^{\star\perp})^{\top}\bm{H}\bm{u}_{l}^{\star}\sim\mathcal{N}(\bm{0},\sigma^{2}\bm{I}_{n-1})$
is a Gaussian random vector independent from $\bm{M}^{(l)}$ but dependent
on $\lambda_{l}$, our proof strategy is to apply Lemma \ref{lemma:eps-net}. 

To this end, recall the definition 
\[
\bm{u}_{k}^{\star(l)}:=(\bm{u}_{l}^{\star\perp})^{\top}\bm{u}_{k}^{\star}
\]
 and the eigen-decomposition of 
\[
\bm{M}^{(l)}=\bm{U}^{(l)}\bm{\Lambda}^{(l)}\bm{U}^{(l)\top}=\sum_{1\leq i<n}\lambda_{i}^{(l)}\bm{u}_{i}^{(l)}\bm{u}_{i}^{(l)\top}.
\]
To begin with, we claim that with probability at least $1-O(n^{-10})$:
\begin{equation}
V:=\sup_{\lambda-\gamma(\lambda)\in\mathcal{B}_{\mathcal{E}_{\mathsf{MD}}}(\lambda_{l}^{\star})}\bigg\|\sum_{k:k\ne l}\bm{a}^{\top}\bm{u}_{k}^{\star}\bm{u}_{k}^{\star(l)\top}\big(\lambda\bm{I}-\bm{M}^{(l)}\big)^{-1}\bigg\|_{2}\lesssim\sum_{k:k\ne l}\frac{|\bm{a}^{\top}\bm{u}_{k}^{\star}|\sqrt{r}}{|\lambda_{l}^{\star}-\lambda_{k}^{\star}|}+\frac{1}{|\lambda_{l}^{\star}|}.\label{eq:claim:atukukt-lambda-inv-l2-norm}
\end{equation}
whose proof is postponed to the end of the section. Consequently,
we can invoke standard Gaussian concentration inequalities to obtain:
with probability at least $1-O\big(\kappa^{-10}(\lambda_{\max}/\Delta_{l}^{\star})^{-20}n^{-20}\big)$,
\begin{align*}
\bigg|(\bm{u}_{l}^{\star\perp})^{\top}\bm{H}\bm{u}_{l}^{\star}\cdot\sum_{k:k\ne l}\bm{a}^{\top}\bm{u}_{k}^{\star}\bm{u}_{k}^{\star\top}\bm{u}_{l}^{\star\perp}\big(\lambda\bm{I}-\bm{M}^{(l)}\big)^{-1}\bigg| & \lesssim\sigma\sqrt{\log\bigg(\frac{n\kappa\lambda_{\max}}{\Delta_{l}^{\star}}\bigg)}\cdot V.
\end{align*}

In addition, we collect a basic fact regarding the derivatives of
matrices: for any invertible matrix $\bm{A}$,
\[
\frac{\mathrm{d}\bm{A}^{-1}}{\mathrm{d}x}=-\bm{A}^{-1}\frac{\mathrm{d}\bm{A}}{\mathrm{d}x}\bm{A}^{-1}.
\]
With this identity in mind, one can derive: with probability at least
$1-O(n^{-10})$, for all $\lambda$ with $\lambda-\gamma(\lambda)\in\mathcal{B}_{\mathcal{E}_{\mathsf{MD}}}(\lambda_{l}^{\star})$,

\begin{align*}
 & \bigg|\frac{\mathrm{d}}{\mathrm{d}\lambda}\sum_{k:k\ne l}\bm{a}^{\top}\bm{u}_{k}^{\star}\bm{u}_{k}^{\star\top}\bm{u}_{l}^{\star\perp}\big(\lambda\bm{I}-\bm{M}^{(l)}\big)^{-1}(\bm{u}_{l}^{\star\perp})^{\top}\bm{H}\bm{u}_{l}^{\star}\bigg|\\
 & \qquad=\bigg|\sum_{k:k\ne l}\bm{a}^{\top}\bm{u}_{k}^{\star}\bm{u}_{k}^{\star\top}\bm{u}_{l}^{\star\perp}\big(\lambda\bm{I}-\bm{M}^{(l)}\big)^{-2}(\bm{u}_{l}^{\star\perp})^{\top}\bm{H}\bm{u}_{l}^{\star}\bigg|\\
 & \qquad\overset{}{\leq}n\cdot\max_{1\leq i<n}\frac{1}{(\lambda-\lambda_{i}^{(l)})^{2}}\cdot\bigg\|\sum_{k:k\ne l}\bm{a}^{\top}\bm{u}_{k}^{\star}\bm{u}_{k}^{\star\top}\bm{u}_{l}^{\star\perp}\bigg\|_{2}\cdot\big\|(\bm{u}_{l}^{\star\perp})^{\top}\bm{H}\bm{u}_{l}^{\star}\big\|_{2}\\
 & \qquad\overset{(\mathrm{i})}{\lesssim}n\cdot\frac{1}{\Delta_{l}^{\star2}\wedge\lambda_{l}^{\star2}}\cdot\sum_{k:k\ne l}|\bm{a}^{\top}\bm{u}_{k}^{\star}|\,\big\|\bm{u}_{k}^{\star\top}\bm{u}_{l}^{\star\perp}\big\|_{2}\cdot\sigma\sqrt{n\log n}\\
 & \qquad\overset{(\mathrm{ii})}{\leq}n^{3/2}\cdot\frac{\max_{i:i\neq l}|\lambda_{l}^{\star}-\lambda_{i}^{\star}|}{\Delta_{l}^{\star2}\wedge\lambda_{l}^{\star2}}\cdot\sigma\sqrt{\log n}\sum_{k:k\neq l}\frac{\big|\bm{a}^{\top}\bm{u}_{k}^{\star}\big|}{|\lambda_{l}^{\star}-\lambda_{k}^{\star}|}\\
 & \qquad\overset{(\mathrm{iii})}{\lesssim}n^{3/2}\cdot\bigg(\frac{\lambda_{\max}^{\star2}}{\Delta_{l}^{\star2}}+\kappa^{2}\bigg)\cdot\frac{1}{|\lambda_{l}^{\star}|}\cdot\sigma\sqrt{\log n}\sum_{k:k\neq l}\frac{\big|\bm{a}^{\top}\bm{u}_{k}^{\star}\big|}{|\lambda_{l}^{\star}-\lambda_{k}^{\star}|}\\
 & \qquad\overset{(\mathrm{iv})}{\lesssim}n^{3/2}\cdot\frac{\kappa^{2}\lambda_{\max}^{\star2}}{\Delta_{l}^{\star2}}\cdot\frac{1}{|\lambda_{l}^{\star}|}\cdot V.
\end{align*}
Here, (i) arises from (\ref{eq:lambdal-tilde-lambdak-LB}) in Lemma~\ref{lemma:lambda-S-minus-spectrum-1-1}
and the standard Gaussian concentration inequality; (ii) holds since
$\big\|\bm{u}_{k}^{\star\top}\bm{u}_{l}^{\star\perp}\big\|_{2}\leq\|\bm{u}_{k}^{\star}\|_{2}\|\bm{u}_{l}^{\star\perp}\|\leq1$;
(iii) is due to $\max_{i:i\neq l}|\lambda_{l}^{\star}-\lambda_{i}^{\star}|\lesssim\lambda_{\max}^{\star}$;
(iv) uses the definition of $V$ in (\ref{eq:claim:atukukt-lambda-inv-l2-norm}).
Moreover, it is easy to see that $\{\lambda\colon\lambda-\gamma(\lambda)\in\mathcal{B}_{\mathcal{E}_{\mathsf{MD}}}(\lambda_{l}^{\star})\}\subset\big[\lambda_{l}^{\star}-|\lambda_{l}^{\star}|/3,\,\lambda_{l}^{\star}+|\lambda_{l}^{\star}|/3\big]$.
As a consequence, we invoke Lemma \ref{lemma:eps-net} and the union
bound to find: with probability at least $1-O(n^{-10})$:
\begin{align*}
 & \Big|\sum_{k:k\ne l}\bm{a}^{\top}\bm{u}_{k}^{\star}\bm{u}_{k}^{\star(l)\top}\big(\lambda_{l}\bm{I}-\bm{M}^{(l)}\big)^{-1}(\bm{u}_{l}^{\star\perp})^{\top}\bm{H}\bm{u}_{l}^{\star}\Big|\lesssim\sigma\sqrt{r\log\bigg(\frac{n\kappa\lambda_{\max}}{\Delta_{l}^{\star}}\bigg)}\cdot V\\
 & \qquad\qquad\qquad\lesssim\sigma\sqrt{r\log\bigg(\frac{n\kappa\lambda_{\max}}{\Delta_{l}^{\star}}\bigg)}\sum_{k:k\neq l}\frac{\big|\bm{a}^{\top}\bm{u}_{k}^{\star}\big|}{|\lambda_{l}^{\star}-\lambda_{k}^{\star}|}+\frac{\sigma}{|\lambda_{l}^{\star}|}\sqrt{\log\bigg(\frac{n\kappa\lambda_{\max}}{\Delta_{l}^{\star}}\bigg)}
\end{align*}
as claimed.

The remainder of this section amounts to establishing (\ref{eq:claim:atukukt-lambda-inv-l2-norm}),
and we shall work under the event where Lemma~\ref{lemma:lambda-S-minus-spectrum-1-1}
holds, which happens with probability at least $1-O(n^{-10})$. Note
that for any $\lambda$ such that $\lambda-\gamma(\lambda)\in\mathcal{B}_{\mathcal{E}_{\mathsf{MD}}}(\lambda_{l}^{\star})$,
one can express
\begin{align}
\bigg\|\sum_{k:k\ne l}\bm{a}^{\top}\bm{u}_{k}^{\star}\bm{u}_{k}^{\star(l)\top}\big(\lambda\bm{I}-\bm{M}^{(l)}\big)^{-1}\bigg\|_{2} & =\bigg\|\sum_{k:k\ne l}\bm{a}^{\top}\bm{u}_{k}^{\star}\bm{u}_{k}^{\star(l)\top}\bm{U}^{(l)}\big(\lambda\bm{I}-\bm{\Lambda}^{(l)}\big)^{-1}\bm{U}^{(l)\top}\bigg\|_{2}\nonumber \\
 & =\bigg\|\sum_{k:k\ne l}\bm{a}^{\top}\bm{u}_{k}^{\star}\bm{u}_{k}^{\star(l)\top}\bm{U}^{(l)}\big(\lambda\bm{I}-\bm{\Lambda}^{(l)}\big)^{-1}\bigg\|_{2}\nonumber \\
 & =\sqrt{\sum_{1\leq i<n}\bigg(\frac{1}{\lambda-\lambda_{i}^{(l)}}\sum_{k:k\ne l}\bm{a}^{\top}\bm{u}_{k}^{\star}\bm{u}_{k}^{\star(l)\top}\bm{u}_{i}^{(l)}\bigg)^{2}}.\label{eq:claim:atukukt-lambda-inv-l2-norm-temp}
\end{align}
In what follows, we shall control the sum over $i<r$ and the sum
over $i\geq r$ separately.
\begin{itemize}
\item Let us consider the sum over $i\geq r$ first. According to Lemma
\ref{lemma:lambda-M-minus-spectrum}, we know that $|\lambda-\lambda_{i}^{(l)}|\gtrsim|\lambda_{l}^{\star}|$
for all $i\geq r$. This in turn yields
\begin{align}
\sqrt{\sum_{r\leq i\leq n-1}\bigg(\frac{1}{\lambda-\lambda_{i}^{(l)}}\sum_{k:k\ne l}\bm{a}^{\top}\bm{u}_{k}^{\star}\bm{u}_{k}^{\star(l)\top}\bm{u}_{i}^{(l)}\bigg)^{2}} & \lesssim\frac{1}{|\lambda_{l}^{\star}|}\sqrt{\sum_{r\leq i\leq n-1}\bigg(\sum_{k:k\ne l}\bm{a}^{\top}\bm{u}_{k}^{\star}\bm{u}_{k}^{\star(l)\top}\bm{u}_{i}^{(l)}\bigg)^{2}}\nonumber \\
 & \le\frac{1}{|\lambda_{l}^{\star}|}\Big\|\sum_{k:k\ne l}\bm{a}^{\top}\bm{u}_{k}^{\star}\bm{u}_{k}^{\star(l)\top}\bm{U}^{(l)}\Big\|_{2}\nonumber \\
 & \leq\frac{1}{|\lambda_{l}^{\star}|}\|\bm{a}\|_{2}\cdot\Big\|\sum_{k:k\ne l}\bm{u}_{k}^{\star}\bm{u}_{k}^{\star(l)\top}\Big\|\cdot\|\bm{U}^{(l)}\|=\frac{1}{|\lambda_{l}^{\star}|}.\label{eq:claim:atukukt-lambda-inv-l2-norm-sum1}
\end{align}
Here, we make use of the fact that $\|\bm{a}\|_{2}=1$, $\|\bm{U}^{(l)}\|=1$
as well as $\big\|\sum_{k:k\ne l}\bm{u}_{k}^{\star}\bm{u}_{k}^{\star(l)\top}\big\|=1$,
since both $\{\bm{u}_{k}^{\star}\}$ and $\{\bm{u}_{k}^{\star(l)}\}_{k:k\neq l}$
form orthonormal bases. 
\item We then move on to the sum over the range $1\leq i<r$. Given that
$\{\bm{u}_{i}^{(l)}\}_{i}$ are orthonormal, it is straightforward
to demonstrate that
\begin{align}
\sqrt{\sum_{1\leq i<r}\bigg(\frac{1}{\lambda-\lambda_{i}^{(l)}}\sum_{k:k\ne l}\bm{a}^{\top}\bm{u}_{k}^{\star}\bm{u}_{k}^{\star(l)\top}\bm{u}_{i}^{(l)}\bigg)^{2}} & =\bigg\|\sum_{1\leq i<r}\bigg(\frac{1}{\lambda-\lambda_{i}^{(l)}}\sum_{k:k\ne l}\bm{a}^{\top}\bm{u}_{k}^{\star}\bm{u}_{k}^{\star(l)\top}\bm{u}_{i}^{(l)}\bigg)\bm{u}_{i}^{(l)}\bigg\|_{2}\nonumber \\
 & =\bigg\|\sum_{k:k\ne l}\bm{a}^{\top}\bm{u}_{k}^{\star}\sum_{1\leq i<r}\frac{\bm{u}_{k}^{\star(l)\top}\bm{u}_{i}^{(l)}}{\lambda-\lambda_{i}^{(l)}}\bm{u}_{i}^{(l)}\bigg\|_{2}\nonumber \\
 & \le\sum_{k:k\ne l}\big|\bm{a}^{\top}\bm{u}_{k}^{\star}\big|\cdot\bigg\|\sum_{1\leq i<r}\frac{\bm{u}_{k}^{\star(l)\top}\bm{u}_{i}^{(l)}}{\lambda-\lambda_{i}^{(l)}}\bm{u}_{i}^{(l)}\bigg\|_{2}\nonumber \\
 & =\sum_{k:k\ne l}\big|\bm{a}^{\top}\bm{u}_{k}^{\star}\big|\sqrt{\sum_{1\leq i<r}\bigg(\frac{\bm{u}_{k}^{\star(l)\top}\bm{u}_{i}^{(l)}}{\lambda-\lambda_{i}^{(l)}}\bigg)^{2}}.\label{eq:claim:atukukt-lambda-inv-l2-norm-temp-1}
\end{align}
The preceding inequality motivates us to control the quantity $\sum_{1\leq i<r}\Big(\frac{\bm{u}_{k}^{\star(l)\top}\bm{u}_{i}^{(l)}}{\lambda-\lambda_{i}^{(l)}}\Big)^{2}$.
Towards this, let us decompose it as follows 
\begin{align*}
\sum_{1\leq i<r}\bigg(\frac{\bm{u}_{k}^{\star(l)\top}\bm{u}_{i}^{(l)}}{\lambda-\lambda_{i}^{(l)}}\bigg)^{2}=\sum_{i\in\mathcal{A}_{1}}\bigg(\frac{\bm{u}_{k}^{\star(l)\top}\bm{u}_{i}^{(l)}}{\lambda-\lambda_{i}^{(l)}}\bigg)^{2}+\sum_{i\in\mathcal{A}_{2}}\bigg(\frac{\bm{u}_{k}^{\star(l)\top}\bm{u}_{i}^{(l)}}{\lambda-\lambda_{i}^{(l)}}\bigg)^{2}.
\end{align*}
Here, the sets $\mathcal{A}_{1}$ and $\mathcal{A}_{2}$ are defined
respectively by
\begin{align*}
\mathcal{A}_{1} & :=\{1\leq i<r\mid\lambda_{i}^{(l)}-\gamma(\lambda_{i}^{(l)})\in\mathcal{B}_{\mathcal{E}_{k}}(\lambda_{k}^{\star})\},\\
\mathcal{A}_{2} & :=\{1\leq i<r\mid\lambda_{i}^{(l)}-\gamma(\lambda_{i}^{(l)})\notin\mathcal{B}_{\mathcal{E}_{k}}(\lambda_{k}^{\star})\},
\end{align*}
where $\mathcal{E}_{k}:=c\,|\lambda_{l}^{\star}-\lambda_{k}^{\star}|$
for some sufficiently small constant $c>0$. In the sequel, we shall
control these two sums separately.
\begin{itemize}
\item For each $i\in\mathcal{A}_{1}$, we claim that
\begin{equation}
|\lambda-\lambda_{i}^{(l)}|\ge\big|\lambda-f^{-1}(\lambda_{k}^{\star})\big|-\big|f^{-1}(\lambda_{k}^{\star})-\lambda_{i}^{(l)}\big|\ge\frac{1}{2}\,\big|\lambda-f^{-1}(\lambda_{k}^{\star})\big|\gtrsim|\lambda_{l}^{\star}-\lambda_{k}^{\star}|.\label{eq:lambda-lambdal-LB}
\end{equation}
To see this, arguing similarly as in the proof of Lemma~\ref{lemma:lambda-M-minus-spectrum},
we can use the Lipschitz property of $f$ (cf.~(\ref{eq:def:function-f}))
to obtain
\begin{align*}
\big|\lambda-f^{-1}(\lambda_{k}^{\star})\big| & \ge\frac{1}{2}\,\big|f(\lambda)-f\big(f^{-1}(\lambda_{k}^{\star})\big)\big|=\frac{1}{2}\,\big|(\lambda-\gamma(\lambda))-\lambda_{k}^{\star}\big|\\
 & \ge\frac{1}{2}\,|\lambda_{l}^{\star}-\lambda_{k}^{\star}|-\frac{1}{2}\,\big|\lambda-\gamma(\lambda)-\lambda_{l}^{\star}\big|\\
 & \ge\frac{1}{2}\,|\lambda_{l}^{\star}-\lambda_{k}^{\star}|-\frac{1}{2}\,\mathcal{E}_{\mathsf{MD}}\\
 & \gtrsim|\lambda_{l}^{\star}-\lambda_{k}^{\star}|.
\end{align*}
In a similar manner, we can also derive
\[
\big|f^{-1}(\lambda_{k}^{\star})-\lambda_{i}^{(l)}\big|\le2\,\big|f\big(f^{-1}(\lambda_{k}^{\star})\big)-f(\lambda_{i}^{(l)})\big|=2\,\big|\lambda_{k}^{\star}-\big(\lambda_{i}^{(l)}-\gamma(\lambda_{i}^{(l)})\big)\big|.
\]
Therefore, for any $\lambda_{i}^{(l)}$ such that $\lambda_{i}^{(l)}-\gamma(\lambda_{i}^{(l)})\in\mathcal{B}_{\mathcal{E}_{k}}(\lambda_{k}^{\star})$,
one has 
\[
\big|f^{-1}(\lambda_{k}^{\star})-\lambda_{i}^{(l)}\big|\le2\,\mathcal{E}_{k}+2\,\mathcal{E}_{\mathsf{MD}}\le2c\,|\lambda_{l}^{\star}-\lambda_{k}^{\star}|\le\frac{1}{2}\,\big|\lambda_{l}-f^{-1}(\lambda_{k}^{\star})\big|.
\]
 Then the claim is an immediate consequence of these two bounds. As
a result, we conclude that
\[
\sum_{i\in\mathcal{A}_{1}}\bigg(\frac{\bm{u}_{k}^{\star(l)\top}\bm{u}_{i}^{(l)}}{\lambda-\lambda_{i}^{(l)}}\bigg)^{2}\lesssim\frac{1}{(\lambda_{l}^{\star}-\lambda_{k}^{\star})^{2}}\sum_{i\in\mathcal{A}_{1}}\big(\bm{u}_{k}^{\star(l)\top}\bm{u}_{i}^{(l)}\big)^{2}\le\frac{1}{(\lambda_{l}^{\star}-\lambda_{k}^{\star})^{2}}.
\]
\item Turning to the set $\mathcal{A}_{2}$, we know from Lemma \ref{lemma:lambda-M-minus-spectrum}
that for any $i\in\mathcal{A}_{2}$, $|\lambda-\lambda_{i}^{(l)}|\gtrsim\min_{i:i\neq l}|\lambda_{l}^{\star}-\lambda_{i}^{\star}|$
and
\[
\Big\|\big(\lambda_{i}^{(l)}\bm{I}_{r-1}-\bm{\Lambda}^{\star(l)}-\gamma(\lambda_{i}^{(l)})\bm{I}_{r-1}\big)\bm{U}^{\star(l)\top}\bm{u}_{i,\parallel}^{(l)}\Big\|_{2}\lesssim\mathcal{E}_{\mathsf{MD}}=\sigma\sqrt{r}\log n.
\]
Meanwhile, since $\bm{u}_{i,\parallel}^{(l)}$ is the projection of
$\bm{u}_{i}^{(l)}$ onto the space of $\bm{U}^{\star(l)}$ followed
by normalization, one has 
\[
\big|\bm{u}_{k}^{\star(l)\top}\bm{u}_{i}^{(l)}\big|\le\big|\bm{u}_{k}^{\star(l)\top}\bm{u}_{i,\parallel}^{(l)}\big|,
\]
and therefore,
\begin{align*}
\Big\|\big(\lambda_{i}^{(l)}\bm{I}_{r-1}-\bm{\Lambda}^{\star(l)}-\gamma(\lambda_{i}^{(l)})\bm{I}_{r-1}\big)\bm{U}^{\star(l)\top}\bm{u}_{i,\parallel}^{(l)}\Big\|_{2} & \ge\big|\lambda_{i}^{(l)}-\lambda_{k}^{\star}-\gamma(\lambda_{i}^{(l)})\big|\cdot\big|\bm{u}_{k}^{\star(l)\top}\bm{u}_{i,\parallel}^{(l)}\big|\\
 & \ge\mathcal{E}_{k}\cdot\big|\bm{u}_{k}^{\star(l)\top}\bm{u}_{i,\parallel}^{(l)}\big|\\
 & \gtrsim\big|\lambda_{l}^{\star}-\lambda_{k}^{\star}\big|\cdot\big|\bm{u}_{k}^{\star(l)\top}\bm{u}_{i,\parallel}^{(l)}\big|.
\end{align*}
This in turn allows us to derive that
\begin{align*}
\bigg(\frac{\bm{u}_{k}^{\star(l)\top}\bm{u}_{i}^{(l)}}{\lambda-\lambda_{i}^{(l)}}\bigg)^{2} & \le\bigg(\frac{\bm{u}_{k}^{\star(l)\top}\bm{u}_{i,\parallel}^{(l)}}{\lambda-\lambda_{i}^{(l)}}\bigg)^{2}\le\bigg(\frac{\bm{u}_{k}^{\star(l)\top}\bm{u}_{i,\parallel}^{(l)}}{\big(\Delta_{l}^{\star}\big)^{2}}\bigg)^{2}\\
 & \lesssim\frac{\sigma^{2}r\log^{2}n}{|\lambda_{l}^{\star}-\lambda_{k}^{\star}|^{2}\big(\Delta_{l}^{\star}\big)^{2}}.
\end{align*}
\end{itemize}
Putting the above two sums together reveals that
\begin{align*}
\sum_{1\leq i<r}\bigg(\frac{\bm{u}_{k}^{\star(l)\top}\bm{u}_{i}^{(l)}}{\lambda-\lambda_{i}^{(l)}}\bigg)^{2} & \lesssim\frac{1}{(\lambda_{l}^{\star}-\lambda_{k}^{\star})^{2}}+\frac{\sigma^{2}r^{2}\log^{2}n}{|\lambda_{l}^{\star}-\lambda_{k}^{\star}|^{2}\big(\Delta_{l}^{\star}\big)^{2}}\lesssim\frac{r}{(\lambda_{l}^{\star}-\lambda_{k}^{\star})^{2}},
\end{align*}
where the last inequality follows since $\big(\Delta_{l}^{\star}\big)^{2}\gtrsim\sigma^{2}r\log^{2}n$.
Combining this inequality with (\ref{eq:claim:atukukt-lambda-inv-l2-norm-temp-1}),
one readily obtains
\begin{align}
\sqrt{\sum_{1\leq i<r}\bigg(\frac{1}{\lambda_{l}-\lambda_{i}^{(l)}}\sum_{k:k\ne l}\bm{a}^{\top}\bm{u}_{k}^{\star}\bm{u}_{k}^{\star(l)\top}\bm{u}_{i}^{(l)}\bigg)^{2}} & \leq\sum_{k:k\ne l}\big|\bm{a}^{\top}\bm{u}_{k}^{\star}\big|\sqrt{\sum_{1\leq i<r}\bigg(\frac{\bm{u}_{k}^{\star(l)\top}\bm{u}_{i}^{(l)}}{\lambda-\lambda_{i}^{(l)}}\bigg)^{2}}\nonumber \\
 & \lesssim\sum_{k:k\ne l}\frac{\big|\bm{a}^{\top}\bm{u}_{k}^{\star}\big|\sqrt{r}}{\left|\lambda_{l}^{\star}-\lambda_{k}^{\star}\right|}.\label{eq:claim:atukukt-lambda-inv-l2-norm-sum212}
\end{align}

\end{itemize}
Substituting the above two partial sums (\ref{eq:claim:atukukt-lambda-inv-l2-norm-sum1})
and (\ref{eq:claim:atukukt-lambda-inv-l2-norm-sum212}) into (\ref{eq:claim:atukukt-lambda-inv-l2-norm-temp}),
we conclude that: with probability exceeding $1-O(n^{-10})$, 
\[
\bigg\|\sum_{k:k\ne l}\bm{a}^{\top}\bm{u}_{k}^{\star}\bm{u}_{k}^{\star(l)\top}\big(\lambda_{l}\bm{I}_{n-1}-\bm{M}^{(l)}\big)^{-1}\bigg\|_{2}\lesssim\sum_{k:k\ne l}\frac{|\bm{a}^{\top}\bm{u}_{k}^{\star}|\sqrt{r}}{|\lambda_{l}^{\star}-\lambda_{k}^{\star}|}+\frac{1}{|\lambda_{l}^{\star}|}
\]
holds for any $\lambda$ such that $\lambda-\gamma(\lambda)\in\mathcal{B}_{\mathcal{E}_{\mathsf{MD}}}(\lambda_{l}^{\star})$,
as claimed in (\ref{eq:claim:atukukt-lambda-inv-l2-norm}). 

\subsection{Proof of Lemma \ref{lemma:a-top-P-U-perp-u-perp}}

\label{subsec:Proof-of-lemma:a-top-P-U-perp-u-perp}

Recall the definitions of $\bm{u}_{l}^{\star\perp}$, $\bm{u}_{l,\perp}$
and $\bm{U}^{\star(l)\perp}$ in (\ref{eq:def:u-l-star-perp}), (\ref{eq:decomposition-ul-123})
and (\ref{eq:def:U-Lambda-star-l}), respectively. Let us rewrite
\begin{align}
\big|\big\langle\bm{P}_{\bm{U}^{\star\perp}}\bm{a},\,\bm{P}_{\bm{U}^{\star\perp}}\bm{u}_{l}\big\rangle\big| & \overset{(\mathrm{i})}{=}\big|\big\langle(\bm{U}^{\star\perp})^{\top}\bm{a},\,(\bm{U}^{\star\perp})^{\top}\bm{u}_{l}\big\rangle\big|\nonumber \\
 & =\big|\big\langle(\bm{U}^{\star\perp})^{\top}\bm{a},\,(\bm{U}^{\star\perp})^{\top}(\bm{u}_{l}^{\star}\bm{u}_{l}^{\star\top}\bm{u}_{l}+\bm{P}_{\bm{u}_{l}^{\star\perp}}\bm{u}_{l})\big\rangle\big|\nonumber \\
 & \overset{(\mathrm{ii})}{=}\big|\big\langle(\bm{U}^{\star\perp})^{\top}\bm{a},\,(\bm{U}^{\star\perp})^{\top}\bm{P}_{\bm{u}_{l}^{\star\perp}}\bm{u}_{l})\big\rangle\big|\nonumber \\
 & \overset{(\mathrm{iii})}{=}\big|\big\langle(\bm{U}^{\star\perp})^{\top}\bm{a},\,(\bm{U}^{\star\perp})^{\top}\bm{u}_{l,\perp}\big\rangle\big|\cdot\big\|\bm{P}_{\bm{u}_{l}^{\star\perp}}\bm{u}_{l}\big\|_{2}\nonumber \\
 & \overset{(\mathrm{iv})}{=}\big|\big\langle(\bm{u}_{l}^{\star\perp}\bm{U}^{\star(l)\perp})^{\top}\bm{a},\,(\bm{u}_{l}^{\star\perp}\bm{U}^{\star(l)\perp})^{\top}\bm{u}_{l,\perp}\big\rangle\big|\cdot\big\|\bm{P}_{\bm{u}_{l}^{\star\perp}}\bm{u}_{l}\big\|_{2}\nonumber \\
 & =\big|\big\langle(\bm{U}^{\star(l)\perp})^{\top}\big((\bm{u}_{l}^{\star\perp})^{\top}\bm{a}\big),\,(\bm{U}^{\star(l)\perp})^{\top}\big((\bm{u}_{l}^{\star\perp})^{\top}\bm{u}_{l,\perp}\big)\big\rangle\big|\cdot\big\|\bm{P}_{\bm{u}_{l}^{\star\perp}}\bm{u}_{l}\big\|_{2}\nonumber \\
 & \overset{(\mathrm{v})}{=}\big|\big\langle\bm{P}_{\bm{U}^{\star(l)\perp}}\big((\bm{u}_{l}^{\star\perp})^{\top}\bm{a}\big),\,\bm{P}_{\bm{U}^{\star(l)\perp}}\big((\bm{u}_{l}^{\star\perp})^{\top}\bm{u}_{l,\perp}\big)\big\rangle\big|\cdot\big\|\bm{P}_{\bm{u}_{l}^{\star\perp}}\bm{u}_{l}\big\|_{2}\nonumber \\
 & \overset{(\mathrm{vi})}{=}\frac{1}{\big\|\widetilde{\bm{u}}_{l,\perp}\big\|_{2}}\big|\big\langle\bm{P}_{\bm{U}^{\star(l)\perp}}\big((\bm{u}_{l}^{\star\perp})^{\top}\bm{a}\big),\,\bm{P}_{\bm{U}^{\star(l)\perp}}\big((\bm{u}_{l}^{\star\perp})^{\top}\widetilde{\bm{u}}_{l,\perp}\big)\big\rangle\big|\cdot\big\|\bm{P}_{\bm{u}_{l}^{\star\perp}}\bm{u}_{l}\big\|_{2}.\label{eq:a-top-P-U-perp-u-perp-temp}
\end{align}
Here, (i) follows since $(\bm{U}^{\star\perp})^{\top}\bm{U}^{\star\perp}=\bm{I}_{n-r}$;
(ii) holds true since $(\bm{U}^{\star\perp})^{\top}\bm{u}_{l}^{\star}=0$;
(iii) holds due to the definition $\bm{u}_{l,\perp}\coloneqq(\bm{P}_{\bm{u}_{l}^{\star\perp}}\bm{u}_{l})/\|\bm{P}_{\bm{u}_{l}^{\star\perp}}\bm{u}_{l}\|_{2}$;
(iv) results from the fact $\bm{u}_{l}^{\star\perp}\bm{U}^{\star(l)\perp}=\bm{U}^{\star\perp}$;
(v) holds true since $(\bm{U}^{\star(l)\perp})^{\top}\bm{U}^{\star(l)\perp}=\bm{I}_{n-r}$;
(vi) arises from (\ref{eq:u-perp-rank1-1}) in Theorem~\ref{thm:master-thm-vector},
where we denote (i.e., $\bm{u}_{l,\perp}$ is the normalized version
of $\widetilde{\bm{u}}_{l,\perp}$)
\begin{equation}
\widetilde{\bm{u}}_{l,\perp}:=\bm{u}_{l}^{\star\perp}\big(\lambda_{l}\bm{I}_{n-1}-(\bm{u}_{l}^{\star\perp})^{\top}\bm{M}\bm{u}_{l}^{\star\perp}\big)^{-1}(\bm{u}_{l}^{\star\perp})^{\top}\bm{M}\bm{u}_{l}^{\star}.\label{eq:def:u-tilde-l-perp-md}
\end{equation}

Our proof strategy is to show that $\frac{\bm{P}_{\bm{U}^{\star(l)\perp}}\big((\bm{u}_{l}^{\star\perp})^{\top}\widetilde{\bm{u}}_{l,\perp}\big)}{\big\|\bm{P}_{\bm{U}^{\star(l)\perp}}\big((\bm{u}_{l}^{\star\perp})^{\top}\widetilde{\bm{u}}_{l,\perp}\big)\big\|_{2}}$
is a random vector uniformly distributed in the unit sphere of the
subspace $\bm{U}^{\star(l)\perp}$. If this claim were true, then
it would follow from standard measure concentration for uniform distributions
results \cite[Theorem 3.4.6]{vershynin2016high} that, with probability
at least $1-O(n^{-10})$,
\begin{align*}
\big|\big\langle\bm{P}_{\bm{U}^{\star(l)\perp}}\big((\bm{u}_{l}^{\star\perp})^{\top}\bm{a}\big),\,\bm{P}_{\bm{U}^{\star(l)\perp}}\big((\bm{u}_{l}^{\star\perp})^{\top}\widetilde{\bm{u}}_{l,\perp}\big)\big\rangle\big| & \lesssim\sqrt{\frac{\log n}{n-r}}\,\big\|\bm{P}_{\bm{U}^{\star(l)\perp}}\big((\bm{u}_{l}^{\star\perp})^{\top}\bm{a}\big)\big\|_{2}\big\|\bm{P}_{\bm{U}^{\star(l)\perp}}\big((\bm{u}_{l}^{\star\perp})^{\top}\widetilde{\bm{u}}_{l,\perp}\big)\big\|_{2}\\
 & \asymp\sqrt{\frac{\log n}{n}}\,\big\|\bm{P}_{\bm{U}^{\star\perp}}\bm{a}\big\|_{2}\big\|\bm{P}_{\bm{U}^{\star\perp}}(\bm{u}_{l}^{\star\perp})^{\top}\widetilde{\bm{u}}_{l,\perp}\big\|_{2},
\end{align*}
where we use $\bm{u}_{l}^{\star\perp}\bm{U}^{\star(l)\perp}=\bm{U}^{\star\perp}$
and the rank assumption $r\ll n/\log^{2}n$ in the last step. Combining
this with (\ref{eq:a-top-P-U-perp-u-perp-temp}), we arrive at the
advertised bound:
\begin{align*}
\big|\big\langle\bm{P}_{\bm{U}^{\star\perp}}\bm{a},\,\bm{P}_{\bm{U}^{\star\perp}}\bm{u}_{l}\big\rangle\big| & \lesssim\sqrt{\frac{\log n}{n}}\,\frac{\big\|\bm{P}_{\bm{U}^{\star\perp}}(\bm{u}_{l}^{\star\perp})^{\top}\widetilde{\bm{u}}_{l,\perp}\big\|_{2}}{\big\|\widetilde{\bm{u}}_{l,\perp}\big\|_{2}}\big\|\bm{P}_{\bm{U}^{\star\perp}}\bm{a}\big\|_{2}\big\|\bm{P}_{\bm{u}_{l}^{\star\perp}}\bm{u}_{l}\big\|_{2}\\
 & \leq\sqrt{\frac{\log n}{n}}\,\big\|\bm{P}_{\bm{U}^{\star\perp}}\bm{a}\big\|_{2}\big\|\bm{P}_{\bm{u}_{l}^{\star\perp}}\bm{u}_{l}\big\|_{2}.
\end{align*}

To justify the distributional property claimed above, we define --- for
an arbitrary rotation matrix $\bm{Q}\in\mathbb{R}^{(n-r)\times(n-r)}$
---  a new rotation matrix 
\[
\bm{R}=\bm{P}_{\bm{U}^{\star(l)}}+\bm{U}^{\star(l)\perp}\bm{Q}\,(\bm{U}^{\star(l)\perp})^{\top}\in\mathbb{R}^{(n-1)\times(n-1)};
\]
the matrix $\bm{R}$ rotates vectors in the subspace spanned by $\bm{U}^{\star(l)\perp}$
according to $\bm{Q}$, while preserving the part in the subspace
spanned by $\bm{U}^{\star(l)}$. We make note of two important ``rotational
invariance'' properties as follows.
\begin{itemize}
\item As shown in the proof of Lemma~\ref{lemma:lambda-M-minus-spectrum},
it is seen that
\begin{align*}
\bm{R}\,(\bm{u}_{l}^{\star\perp})^{\top}\bm{H}\bm{u}_{l}^{\star} & =\bm{U}^{\star(l)}(\bm{U}^{\star(l)})^{\top}(\bm{u}_{l}^{\star\perp})^{\top}\bm{H}\bm{u}_{l}^{\star}+\bm{U}^{\star(l)\perp}\bm{Q}\,(\bm{U}^{\star(l)\perp})^{\top}(\bm{u}_{l}^{\star\perp})^{\top}\bm{H}\bm{u}_{l}^{\star}\\
 & =(\bm{u}_{l}^{\star\perp})^{\top}\bm{U}_{\smallsetminus l}^{\star}\bm{U}_{\smallsetminus l}^{\star\top}\bm{H}\bm{u}_{l}^{\star}+(\bm{u}_{l}^{\star\perp})^{\top}\bm{U}^{\star\perp}\bm{Q}\,(\bm{U}^{\star\perp})^{\top}\bm{H}\bm{u}_{l}^{\star}\\
 & \overset{\mathrm{d}}{=}(\bm{u}_{l}^{\star\perp})^{\top}\bm{U}_{\smallsetminus l}^{\star}\bm{U}_{\smallsetminus l}^{\star\top}\bm{H}\bm{u}_{l}^{\star}+(\bm{u}_{l}^{\star\perp})^{\top}\bm{U}^{\star\perp}(\bm{U}^{\star\perp})^{\top}\bm{H}\bm{u}_{l}^{\star}\\
 & =(\bm{u}_{l}^{\star\perp})^{\top}\big(\bm{I}_{n}-\bm{u}_{l}^{\star}\bm{u}_{l}^{\star\top}\big)\bm{H}\bm{u}_{l}^{\star}\\
 & =(\bm{u}_{l}^{\star\perp})^{\top}\bm{H}\bm{u}_{l}^{\star}.
\end{align*}
Here, the second line arises from the definitions of $\bm{U}^{\star(l)}$
and $\bm{U}^{\star(l)\perp}$ in (\ref{eq:def:U-Lambda-star-l});
the third line follows because $\bm{Q}\,(\bm{U}^{\star\perp})^{\top}\bm{H}\bm{u}_{l}^{\star}\overset{\mathrm{d}}{=}(\bm{U}^{\star\perp})^{\top}\bm{H}\bm{u}_{l}^{\star}$;
the last line holds due to the fact $(\bm{u}_{l}^{\star\perp})^{\top}\bm{u}_{l}^{\star}=\bm{0}$. 
\item In a similar manner, we also know that
\begin{align*}
\bm{R}\,(\bm{u}_{l}^{\star\perp})^{\top}\bm{H}\bm{u}_{l}^{\star\perp}\bm{R}^{\top} & =(\bm{u}_{l}^{\star\perp})^{\top}\big(\bm{U}_{\smallsetminus l}^{\star}\bm{U}_{\smallsetminus l}^{\star\top}+\bm{U}^{\star\perp}\bm{Q}\,(\bm{U}^{\star\perp})^{\top}\big)\bm{H}\big(\bm{U}_{\smallsetminus l}^{\star}\bm{U}_{\smallsetminus l}^{\star\top}+\bm{U}^{\star\perp}\bm{Q}^{\top}\,(\bm{U}^{\star\perp})^{\top}\big)\bm{u}_{l}^{\star\perp}\\
 & =(\bm{u}_{l}^{\star\perp})^{\top}\big(\bm{U}^{\star}\bm{U}^{\star\top}+\bm{U}^{\star\perp}\bm{Q}\,(\bm{U}^{\star\perp})^{\top}\big)\bm{H}\big(\bm{U}^{\star}\bm{U}^{\star\top}+\bm{U}^{\star\perp}\bm{Q}^{\top}\,(\bm{U}^{\star\perp})^{\top}\big)\bm{u}_{l}^{\star\perp}\\
 & \overset{\mathrm{d}}{=}(\bm{u}_{l}^{\star\perp})^{\top}\bm{H}\bm{u}_{l}^{\star\perp},
\end{align*}
where the second line comes from $(\bm{U}^{\star}\bm{U}^{\star\top}-\bm{U}_{\smallsetminus l}^{\star}\bm{U}_{\smallsetminus l}^{\star\top})\bm{u}_{l}^{\star\perp}=\bm{u}_{l}^{\star}\bm{u}_{l}^{\star\top}\bm{u}_{l}^{\star\perp}=0$,
and the last line holds since $\bm{U}^{\star}\bm{U}^{\star\top}+\bm{U}^{\star\perp}\bm{Q}^{\top}\,(\bm{U}^{\star\perp})^{\top}$
is a rotation matrix. 
\end{itemize}
Using the statistical independence between these two parts, we reach
\begin{align*}
\bm{R}\,(\bm{u}_{l}^{\star\perp})^{\top}\widetilde{\bm{u}}_{l,\perp} & =\bm{R}\,\big(\lambda_{l}\bm{I}_{n-1}-(\bm{u}_{l}^{\star\perp})^{\top}\bm{M}\bm{u}_{l}^{\star\perp}\big)^{-1}(\bm{u}_{l}^{\star\perp})^{\top}\bm{M}\bm{u}_{l}^{\star}=\big(\lambda_{l}\bm{I}_{n-1}-\bm{R}\,(\bm{u}_{l}^{\star\perp})^{\top}\bm{M}\bm{u}_{l}^{\star\perp}\bm{R}^{\top}\big)^{-1}\bm{R}\,(\bm{u}_{l}^{\star\perp})^{\top}\bm{M}\bm{u}_{l}^{\star}\\
 & =\big(\lambda_{l}\bm{I}_{n-1}-\bm{R}\,(\bm{u}_{l}^{\star\perp})^{\top}\bm{H}\bm{u}_{l}^{\star\perp}\bm{R}^{\top}\big)^{-1}\bm{R}\,(\bm{u}_{l}^{\star\perp})^{\top}\bm{H}\bm{u}_{l}^{\star}\\
 & \overset{\mathrm{d}}{=}\big(\lambda_{l}\bm{I}_{n-1}-(\bm{u}_{l}^{\star\perp})^{\top}\bm{H}\bm{u}_{l}^{\star\perp}\big)^{-1}(\bm{u}_{l}^{\star\perp})^{\top}\bm{H}\bm{u}_{l}^{\star}\\
 & =\big(\lambda_{l}\bm{I}_{n-1}-(\bm{u}_{l}^{\star\perp})^{\top}\bm{M}\bm{u}_{l}^{\star\perp}\big)^{-1}(\bm{u}_{l}^{\star\perp})^{\top}\bm{M}\bm{u}_{l}^{\star}\\
 & =(\bm{u}_{l}^{\star\perp})^{\top}\widetilde{\bm{u}}_{l,\perp},
\end{align*}
where the last step replies on the definition of $\widetilde{\bm{u}}_{l,\perp}$
in (\ref{eq:def:u-tilde-l-perp-md}) and the fact $(\bm{u}_{l}^{\star\perp})^{\top}\bm{u}_{l}^{\star\perp}=\bm{I}_{n-1}$.
This enables us to conclude that $\frac{\bm{P}_{\bm{U}^{\star(l)\perp}}\big((\bm{u}_{l}^{\star\perp})^{\top}\widetilde{\bm{u}}_{l,\perp}\big)}{\big\|\bm{P}_{\bm{U}^{\star(l)\perp}}\big((\bm{u}_{l}^{\star\perp})^{\top}\widetilde{\bm{u}}_{l,\perp}\big)\big\|_{2}}$
is uniformly distributed in the unit sphere spanned by $\bm{U}^{\star(l)\perp}$.

\section{Proof of auxiliary lemmas in the analysis for Theorem \ref{thm:evector-pertur-sym-iid-pca}}

\label{sec:Proof-of-lemmas-pca}

\subsection{Proof of Lemma \ref{lemma:spectral-M-Sigma}}

\label{subsec:Proof-of-lemma:lemma:spectral-M-Sigma}

For any matrix $\bm{A}\in\mathbb{R}^{n\times n}$, we know from the
orthogonal invariance of the spectral norm that
\begin{align*}
\|\bm{A}\| & =\|[\bm{U}^{\star},\bm{U}^{\star\perp}]^{\top}\bm{A}[\bm{U}^{\star},\bm{U}^{\star\perp}]\|\\
 & =\bigg\|\begin{bmatrix}\bm{U}^{\star\top}\bm{A}\bm{U}^{\star} & \bm{U}^{\star\top}\bm{A}\bm{U}^{\star\perp}\\
(\bm{U}^{\star\perp})^{\top}\bm{A}\bm{U}^{\star} & (\bm{U}^{\star\perp})^{\top}\bm{A}\bm{U}^{\star\perp}
\end{bmatrix}\bigg\|\\
 & \leq\|\bm{U}^{\star\top}\bm{A}\bm{U}^{\star}\|+\|\bm{U}^{\star\top}\bm{A}\bm{U}^{\star\perp}\|+\|(\bm{U}^{\star\perp})^{\top}\bm{A}\bm{U}^{\star}\|+\|(\bm{U}^{\star\perp})^{\top}\bm{A}\bm{U}^{\star\perp}\|,
\end{align*}
where the last step holds due to the triangle inequality. As a result,
one can upper bound
\begin{align}
\Big\|\frac{1}{n}\bm{S}\bm{S}^{\top}-\bm{\Sigma}\Big\| & \le\Big\|\bm{U}^{\star\top}\Big(\frac{1}{n}\bm{S}\bm{S}^{\top}-\bm{\Sigma}\Big)\bm{U}^{\star}\Big\|+\Big\|(\bm{U}^{\star\perp})^{\top}\Big(\frac{1}{n}\bm{S}\bm{S}^{\top}-\bm{\Sigma}\Big)\bm{U}^{\star}\Big\|\nonumber \\
 & \quad+\Big\|\bm{U}^{\star\top}\Big(\frac{1}{n}\bm{S}\bm{S}^{\top}-\bm{\Sigma}\Big)\bm{U}^{\star\perp}\Big\|+\Big\|(\bm{U}^{\star\perp})^{\top}\Big(\frac{1}{n}\bm{S}\bm{S}^{\top}-\bm{\Sigma}\Big)\bm{U}^{\star\perp}\Big\|\nonumber \\
 & =\Big\|\frac{1}{n}\bm{S}_{\parallel}\bm{S}_{\parallel}^{\top}-\bm{\Lambda}\Big\|+2\,\Big\|\frac{1}{n}\bm{S}_{\perp}\bm{S}_{\parallel}^{\top}\Big\|+\Big\|\frac{1}{n}\bm{S}_{\perp}\bm{S}_{\perp}^{\top}-\sigma^{2}\bm{I}_{p-r}\Big\|\label{eq:cov-loss-op-norm-terms}
\end{align}
where we remind the readers of the notation $\bm{S}_{\parallel}:=\bm{U}^{\star\top}\bm{S}$,
$\bm{S}_{\perp}:=(\bm{U}^{\star\perp})^{\top}\bm{S}$ and $\bm{\Lambda}:=\bm{U}^{\star\top}\bm{\Sigma}\bm{U}^{\star}$
introduced in (\ref{eq:def-Lambda-pca}).

Before describing how to control these quantities, we pause to collect
a few results regarding a Gaussian random matrix $\bm{G}\in\mathbb{R}^{p\times n}$
consisting of i.i.d.~$\mathcal{N}(0,1)$ entries \cite[Theorem~4.6.1]{vershynin2016high}:
with probability at least $1-O(n^{-10})$,
\begin{align}
\|\bm{G}\| & \lesssim\sqrt{p}+\sqrt{n},\nonumber \\
\Big\|\frac{1}{n}\bm{G}\bm{G}^{\top}-\bm{I}_{p}\Big\| & \lesssim\sqrt{\frac{p}{n}}+\frac{p}{n}+\sqrt{\frac{\log n}{n}},\label{eq:ZZt-eig-val}\\
\Big\|\frac{1}{n}\bm{G}^{\top}\bm{G}-\frac{p}{n}\bm{I}_{n}\Big\| & \lesssim1+\sqrt{\frac{p}{n}}.\label{eq:ZtZ-eig-val}
\end{align}
With these bounds in place, we can start to bound the spectral norms
of the quantities in (\ref{eq:cov-loss-op-norm-terms}). Note that
the columns of $\bm{S}_{\parallel}$ (resp.~$\bm{S}_{\perp}$) are
i.i.d.~zero-mean Gaussian random vectors with covariance $\bm{\Lambda}$
(resp.~$\sigma^{2}\bm{I}_{p-r}$). Since we can rewrite $\bm{S}_{\parallel}=\bm{\Lambda}^{1/2}\bm{Z}$
with $\bm{Z}\in\mathbb{R}^{r\times n}$ being a Gaussian random matrix
with i.i.d.~$\mathcal{N}(0,1)$ entries, it immediately follows from
(\ref{eq:ZZt-eig-val}) that with probability more than $1-O(n^{-10})$,
\begin{align}
\Big\|\frac{1}{n}\bm{S}_{\perp}\bm{S}_{\perp}^{\top}-\sigma^{2}\bm{I}_{p-r}\Big\| & \lesssim\sigma^{2}\bigg(\sqrt{\frac{p}{n}}+\frac{p}{n}+\sqrt{\frac{\log n}{n}}\bigg)\qquad\text{and}\label{eq:S-perp-op-UB}
\end{align}
\begin{align}
\Big\|\frac{1}{n}\bm{S}_{\parallel}\bm{S}_{\parallel}^{\top}-\bm{\Lambda}\Big\| & =\Big\|\frac{1}{n}\bm{\Lambda}^{1/2}\bm{Z}\bm{Z}^{\top}\bm{\Lambda}^{1/2}-\bm{\Lambda}\Big\|\le\|\bm{\Lambda}\|\Big\|\frac{1}{n}\bm{Z}\bm{Z}^{\top}-\bm{I}_{r}\Big\|\nonumber \\
 & \lesssim(\lambda_{\max}^{\star}+\sigma^{2})\bigg(\sqrt{\frac{r}{n}}+\frac{r}{n}+\sqrt{\frac{\log n}{n}}\bigg)\nonumber \\
 & \lesssim(\lambda_{\max}^{\star}+\sigma^{2})\sqrt{\frac{r\log n}{n}},\label{eq:S-para-op-UB}
\end{align}
where the last step arises from the sample size assumption $n\geq r$.
As for $\bm{S}_{\perp}\bm{S}_{\parallel}^{\top}$, we can invoke Lemma~\ref{lemma:prod-Gaussian-concentration}
to show that: with probability at least $1-O(n^{-10})$,
\begin{align}
\Big\|\frac{1}{n}\bm{S}_{\perp}\bm{S}_{\parallel}^{\top}\Big\| & =\frac{1}{n}\big\|\bm{S}_{\perp}\bm{Z}^{\top}\bm{\Lambda}^{1/2}\big\|\leq\frac{1}{n}\big\|\bm{S}_{\perp}\bm{Z}^{\top}\big\|\,\big\|\bm{\Lambda}^{1/2}\big\|\lesssim\frac{\sigma}{n}\big(\sqrt{pn\log n}+\sqrt{pr}\log n\big)\cdot\sqrt{\lambda_{\max}^{\star}+\sigma^{2}}\nonumber \\
 & \lesssim\sqrt{(\lambda_{\max}^{\star}+\sigma^{2})\sigma^{2}\frac{p}{n}}\log n,\label{eq:S-para-perp-op-UB}
\end{align}
where the last step holds since $n\geq r$. Putting the bounds above
together immediately concludes the proof.

\subsection{Proof of Lemma \ref{lemma:lambda-S-minus-spectrum-1-1}}

\label{subsec:Proof-of-lemma:lambda-S-minus-spectrum-1-1}

The proof of this lemma is similar to that of Lemma \ref{lemma:lambda-M-minus-spectrum}.
By definition, one can compute
\begin{align}
\bm{\Sigma}_{l,\perp} & :=\frac{1}{n}\mathbb{E}\big[\bm{S}_{l,\perp}\bm{S}_{l,\perp}^{\top}\big]\label{eq:defn-Sigma-l-perp-defn}\\
 & =\frac{1}{n}(\bm{u}_{l}^{\star\perp})^{\top}\mathbb{E}\big[\bm{S}\bm{S}^{\top}\big]\bm{u}_{l}^{\star\perp}=(\bm{u}_{l}^{\star\perp})^{\top}\bm{\Sigma}\bm{u}_{l}^{\star\perp}\nonumber \\
 & =(\bm{u}_{l}^{\star\perp})^{\top}\bm{U}^{\star}\bm{\Lambda}^{\star}\bm{U}^{\star\top}\bm{u}_{l}^{\star\perp}+\sigma^{2}\bm{I}_{p-1}\nonumber \\
 & =\bm{U}^{\star(l)}\bm{\Lambda}^{\star(l)}\bm{U}^{\star(l)\top}+\sigma^{2}\bm{I}_{p-1},\nonumber 
\end{align}
where $\bm{U}^{\star(l)}$ and $\bm{\Lambda}^{\star(l)}$ have been
defined in (\ref{eq:def:U-Lambda-star-l-1-pca}). Our proof strategy
is to invoke Theorem~\ref{thm:master-theorem-general} by treating
$\frac{1}{n}\bm{S}_{l,\perp}\bm{S}_{l,\perp}^{\top}$ (resp.~$\bm{U}^{\star(l)}$)
as $\bm{M}$ (resp.~$\bm{Q}$).
\begin{itemize}
\item We shall start with the first claim. Let us define the following matrices
in $\mathbb{R}^{(r-1)\times(r-1)}$:
\begin{align*}
\bm{K}^{(l)}(\lambda) & :=\bm{U}^{\star(l)\top}\frac{1}{n}\bm{S}_{l,\perp}\bm{S}_{l,\perp}^{\top}\bm{U}^{\star(l)\perp}\Big(\lambda\bm{I}_{p-r}-(\bm{U}^{\star(l)\perp})^{\top}\frac{1}{n}\bm{S}_{l,\perp}\bm{S}_{l,\perp}^{\top}\bm{U}^{\star(l)\perp}\Big)^{-1}(\bm{U}^{\star(l)\perp})^{\top}\frac{1}{n}\bm{S}_{l,\perp}\bm{S}_{l,\perp}^{\top}\bm{U}^{\star(l)},\\
\bm{K}^{(l)\perp}(\lambda) & :=\mathbb{E}\Big[\bm{G}^{(l)}(\lambda)\mid(\bm{U}^{\star(l)\perp})^{\top}\bm{S}_{l,\perp}\bm{S}_{l,\perp}^{\top}\bm{U}^{\star(l)\perp}\Big].
\end{align*}
Recall the definitions of $\bm{K}(\lambda)$ (cf.~(\ref{eq:def:G-lambda-exp-pca}))
and $\bm{P}^{(l)}$ (cf.~(\ref{eq:def:P-l})), and notice that
\begin{align*}
\bm{u}_{l}^{\star\perp}\bm{U}^{\star(l)\perp} & =\bm{U}^{\star\perp},\\
(\bm{U}^{\star(l)\perp})^{\top}\bm{S}_{l,\perp} & =(\bm{U}^{\star(l)\perp})^{\top}(\bm{u}_{l}^{\star\perp})^{\top}\bm{S}=(\bm{U}^{\star\perp})^{\top}\bm{S}=\bm{S}_{\perp}.
\end{align*}
Straightforward calculation allows us to simplify the above expressions
as follows
\begin{align*}
\bm{K}^{(l)}(\lambda) & :=\frac{1}{n}\bm{P}^{(l)\top}\bm{S}_{\parallel}\cdot\frac{1}{n}\bm{S}_{\perp}^{\top}\Big(\lambda\bm{I}_{p-r}-\frac{1}{n}\bm{S}_{\perp}\bm{S}_{\perp}^{\top}\Big)^{-1}\bm{S}_{\perp}\cdot\bm{S}_{\parallel}^{\top}\bm{P}^{(l)}=\bm{P}^{(l)\top}\bm{K}(\lambda)\bm{P}^{(l)},\\
\bm{K}^{(l)\perp}(\lambda) & :=\mathbb{E}\big[\bm{G}^{(l)}(\lambda)\mid(\bm{U}^{\star(l)\perp})^{\top}\bm{S}_{l,\perp}\bm{S}_{l,\perp}^{\top}\bm{U}^{\star(l)\perp}\big]=\beta(\lambda)\bm{P}^{(l)\top}(\bm{\Lambda}^{\star}+\sigma^{2}\bm{I}_{r})\bm{P}^{(l)}=\beta(\lambda)(\bm{\Lambda}^{\star(l)}+\sigma^{2}\bm{I}_{r-1}).
\end{align*}
Theorem~\ref{thm:master-theorem-general} then tells us that
\begin{align*}
 & \big(\gamma_{i}^{(l)}\bm{I}_{r-1}-\bm{\Lambda}^{\star(l)}-\sigma^{2}\bm{I}_{r-1}-\bm{K}^{(l)\perp}(\gamma_{i}^{(l)})\big)\bm{U}^{\star(l)\top}\bm{u}_{i,\parallel}^{(l)}\\
 & \qquad=\Big(\frac{1}{n}\bm{P}^{(l)\top}\bm{S}_{\parallel}\bm{S}_{\parallel}^{\top}\bm{P}^{(l)}-\bm{\Lambda}^{\star(l)}-\sigma^{2}\bm{I}_{r-1}+\bm{K}(\gamma_{i}^{(l)})-\bm{K}^{\perp}(\gamma_{i}^{(l)})\Big)\bm{U}^{\star(l)\top}\bm{u}_{i,\parallel}^{(l)}.
\end{align*}
One can then adopt a similar argument as in the proof of Theorem \ref{thm:eigval-pertur-sym-iid-pca}
to demonstrate that: with probability at least $1-O(n^{-10})$,
\begin{align*}
 & \Big\|\Big(\gamma_{i}^{(l)}\bm{I}_{r-1}-\big(1+\beta(\gamma_{i}^{(l)})\big)\big(\bm{\Lambda}^{\star(l)}+\sigma^{2}\bm{I}_{r-1}\big)\Big)\bm{U}^{\star(l)\top}\bm{u}_{i,\parallel}^{(l)}\Big\|_{2}\\
 & \qquad\le\Big\|\frac{1}{n}\bm{P}^{(l)\top}\bm{S}_{\parallel}\bm{S}_{\parallel}^{\top}\bm{P}^{(l)}-\bm{\Lambda}^{\star(l)}-\sigma^{2}\bm{I}_{r-1}\Big\|+\sup_{\lambda:\lambda\in[2\lambda_{l}^{\star}/3,\,4\lambda_{l}^{\star}/3]}\Big\|\big(\bm{K}(\gamma_{i}^{(l)})-\bm{K}^{\perp}(\gamma_{i}^{(l)})\big)\bm{U}^{\star(l)\top}\bm{u}_{i,\parallel}^{(l)}\Big\|_{2}\\
 & \qquad\leq\Big\|\frac{1}{n}\bm{S}_{\parallel}\bm{S}_{\parallel}^{\top}-\bm{\Lambda}^{\star}-\sigma^{2}\bm{I}_{r}\Big\|+\sup_{\lambda:\lambda\in[2\lambda_{l}^{\star}/3,\,4\lambda_{l}^{\star}/3]}\Big\|\bm{K}(\gamma_{i}^{(l)})-\bm{K}^{\perp}(\gamma_{i}^{(l)})\Big\|\lesssim\mathcal{E}_{\mathsf{PCA}}
\end{align*}
and there exists some $k\neq l$ ($1\leq k\leq r$) obeying
\[
\bigg|\frac{\gamma_{i}^{(l)}}{1+\beta(\gamma_{i}^{(l)})}-\lambda_{k}^{\star}-\sigma^{2}\bigg|\lesssim\mathcal{E}_{\mathsf{PCA}}.
\]
\item Next, we turn to the second claim and we shall prove the upper and
lower bounds for $\gamma_{i}^{(l)}$ separately.
\begin{itemize}
\item For the upper bound, it suffices to upper bound $\gamma_{r}^{(l)}$
since $\{\gamma_{i}^{(l)}\}_{i}$ are defined in descending order.
In view of $(\bm{U}^{\star(l)\perp})^{\top}\bm{S}_{l,\perp}=\bm{S}_{\perp}$,
we can invoke Lemma \ref{lemma:eigval-interlacing} to see that
\[
\gamma_{r}^{(l)}\leq\lambda_{1}\Big(\frac{1}{n}\bm{S}_{\perp}\bm{S}_{\perp}^{\top}\Big).
\]
This suggests that we look at the spectrum of $\frac{1}{n}\bm{S}_{\perp}\bm{S}_{\perp}^{\top}$.
From (\ref{eq:ZZt-eig-val}) and (\ref{eq:ZtZ-eig-val}), we can obtain
\[
\Big|\lambda_{i}\Big(\frac{1}{n}\bm{S}_{\perp}\bm{S}_{\perp}^{\top}\Big)-\sigma^{2}\frac{n\vee(p-r)}{n}\Big|\lesssim\sigma^{2}\sqrt{\frac{p+\log n}{n}},\qquad1\leq i\leq(p-r)\wedge n,
\]
where we use the fact that the non-zero eigenvalues of $\frac{1}{n}\bm{S}_{\perp}\bm{S}_{\perp}^{\top}$
and $\frac{1}{n}\bm{S}_{\perp}^{\top}\bm{S}_{\perp}$ are identical.
Therefore, one has
\[
\gamma_{i}^{(l)}\le\gamma_{r}^{(l)}\le\sigma^{2}(n\vee p)/n+O(\sigma^{2}\sqrt{(p+\log n)/n})
\]
 for all $r\leq i\le n\wedge(p-1)$. 
\item Next, we move on to consider the lower bound. Observe that the matrix
$\bm{\Sigma}_{l,\perp}$ defined in (\ref{eq:defn-Sigma-l-perp-defn})
satisfies the following properties: (i) $\bm{\Sigma}_{l,\perp}\succeq\sigma^{2}\bm{I}_{p-1}$;
(ii) $\bm{\Sigma}_{l,\perp}^{-1/2}\bm{S}_{l,\perp}$ is a Gaussian
random matrix composed of i.i.d.~standard Gaussian entries. Then
we can lower bound the eigenvalues $\gamma_{i}^{(l)}$ for any $1\le i\leq(p-1)\wedge n$
as follows
\begin{align*}
\gamma_{i}^{(l)} & =\lambda_{i}\Big(\frac{1}{n}\bm{S}_{l,\perp}\bm{S}_{l,\perp}^{\top}\Big)\overset{(\mathrm{i})}{=}\lambda_{i}\Big(\frac{1}{n}\bm{S}_{l,\perp}^{\top}\bm{S}_{l,\perp}\Big)\\
 & =\lambda_{i}\Big(\frac{1}{n}\big(\bm{\Sigma}_{l,\perp}^{-1/2}\bm{S}_{l,\perp}\big)^{\top}\bm{\Sigma}_{l,\perp}\bm{\Sigma}_{l,\perp}^{-1/2}\bm{S}_{l,\perp}\Big)\\
 & \overset{(\mathrm{ii})}{\geq}\sigma^{2}\lambda_{i}\Big(\frac{1}{n}\big(\bm{\Sigma}_{l,\perp}^{-1/2}\bm{S}_{l,\perp}\big)^{\top}\bm{\Sigma}_{l,\perp}^{-1/2}\bm{S}_{l,\perp}\Big)\\
 & \overset{(\mathrm{iii})}{=}\sigma^{2}\lambda_{i}\Big(\frac{1}{n}\bm{\Sigma}_{l,\perp}^{-1/2}\bm{S}_{l,\perp}\big(\bm{\Sigma}_{l,\perp}^{-1/2}\bm{S}_{l,\perp}\big)^{\top}\Big),
\end{align*}
where (i) and (iii) hold because $i\leq(p-1)\wedge n$; (ii) follows
since for any matrix $\bm{A}$, one has $\bm{A}^{\top}(\bm{\Sigma}_{l,\perp}-\sigma^{2}\bm{I})\bm{A}\succeq0$
and hence $\lambda_{i}(\bm{A}^{\top}\bm{\Sigma}_{l,\perp}\bm{A})\ge\sigma^{2}\lambda_{i}(\bm{A}^{\top}\bm{A})$.
By invoking (\ref{eq:ZZt-eig-val}) and (\ref{eq:ZtZ-eig-val}) once
again, we arrive at
\[
\Big|\lambda_{i}\Big(\frac{1}{n}\bm{\Sigma}_{l,\perp}^{-1/2}\bm{S}_{l,\perp}\big(\bm{\Sigma}_{l,\perp}^{-1/2}\bm{S}_{l,\perp}\big)^{\top}\Big)-\sigma^{2}\frac{n\vee(p-1)}{n}\Big|\lesssim\sigma^{2}\sqrt{\frac{p+\log n}{n}},\qquad1\leq i\leq n\wedge(p-1).
\]
This immediately establishes the claimed lower bound.
\end{itemize}
\item The third claim is an immediate consequence of the fact that 
\[
\mathsf{rank}\Big(\frac{1}{n}\bm{S}_{l,\perp}\bm{S}_{l,\perp}^{\top}\Big)\le n\wedge(p-1).
\]
 
\item Finally, let us consider the last claim. In view of Theorem~\ref{thm:eigval-pertur-sym-iid-pca},
we have
\[
\frac{\lambda_{l}}{1+\beta(\lambda_{l})}\in\mathcal{B}_{\mathcal{E}_{\mathsf{PCA}}}(\lambda_{l}^{\star}+\sigma^{2}).
\]
In addition, the first claim asserts that for each $1\le i<r$, one
has
\[
\frac{\gamma_{i}^{(l)}}{1+\beta(\gamma_{i}^{(l)})}\in\mathcal{B}_{\mathcal{E}_{\mathsf{PCA}}}(\lambda_{k}^{\star}+\sigma^{2})
\]
 for some $k\ne l$. In view of the Lipschitz property of the function
$f(\lambda):=\frac{\lambda}{1+\beta(\lambda)}$ (so that $|f^{\prime}(\lambda)|\lesssim1$),
applying a similar argument as in the proof of Lemma~\ref{lemma:lambda-M-minus-spectrum}
(see Appendix~\ref{subsec:Proof-of-Lemma:lambda-M-minus-spectrum})
immediately allows us to establish the claim.
\end{itemize}

\subsection{Proof of Lemma \ref{lemma:lambda-M-l-inv-u-perp-M-u-l2-norm-1-pca}}

\label{subsec:Proof-of-lemma:lambda-M-l-inv-u-perp-M-u-l2-norm-1-pca-1}

Before continuing, we introduce several useful notation that will
be used throughout. 
\begin{itemize}
\item Let $\bm{U}^{(l)}\sqrt{\bm{\Gamma}^{(l)}}\bm{V}^{(l)\top}$ denote
the SVD of $\frac{1}{\sqrt{n}}\bm{S}_{l,\perp}$, where $\bm{\Gamma}^{(l)}$
is a diagonal matrix consisting of the singular values of interest.
Here, we recall that $\bm{S}_{l,\perp}$ has been defined in (\ref{eq:def:s-l-para-perp}).
\item Let $\bm{u}_{i}^{(l)}$ (resp.~$\bm{v}_{i}^{(l)}$) indicate the
$i$-th column of $\bm{U}^{(l)}$ (resp.~$\bm{V}^{(l)}$), and let
$\gamma_{i}^{(l)}$ represent the $i$-th diagonal entry of $\bm{\Gamma}^{(l)}$. 
\end{itemize}
In addition, we note that the vector $\bm{s}_{l,\parallel}$ (see
(\ref{eq:def:s-l-para-perp})) obeys
\[
\bm{s}_{l,\parallel}^{\top}\sim\mathcal{N}\big(\bm{0},(\bm{u}_{l}^{\star\top}\bm{\Sigma}\bm{u}_{l}^{\star})\bm{I}_{n}\big)=\mathcal{N}\big(\bm{0},(\lambda_{l}^{\star}+\sigma^{2})\bm{I}_{n}\big)
\]
and is independent of $\bm{S}_{l,\perp}$ (and thus $\bm{\Gamma}^{(l)}$
and $\bm{V}^{(l)}$). This implies that condition on $\bm{V}^{(l)}$,
one has
\begin{equation}
\bm{v}_{i}^{(l)\top}\bm{s}_{l,\parallel}^{\top}\overset{\mathrm{i.i.d.}}{\sim}\mathcal{N}\left(0,\lambda_{l}^{\star}+\sigma^{2}\right),\qquad1\leq i<p.\label{eq:independence-vil-sl-par}
\end{equation}
Moreover, by virtue of the rotational invariance of i.i.d.~Gaussian
random matrices, it is readily seen that $\bm{\Gamma}^{(l)}$ is independent
of $\bm{V}^{(l)}$. 

Now, we can begin to present the proof, towards which we start with
the following decomposition
\begin{align}
 & \Big\|\Big(\lambda_{l}\bm{I}_{p-1}-\frac{1}{n}\bm{S}_{l,\perp}\bm{S}_{l,\perp}^{\top}\Big)^{-1}\frac{1}{n}\bm{S}_{l,\perp}\bm{s}_{l,\parallel}^{\top}\Big\|_{2}^{2}\nonumber \\
 & \qquad=\frac{1}{n}\Big\|\big(\lambda_{l}\bm{I}_{p-1}-\bm{U}^{(l)}\bm{\Gamma}^{(l)}\bm{U}^{(l)\top}\big)^{-1}\bm{U}^{(l)}\sqrt{\bm{\Gamma}^{(l)}}\bm{V}^{(l)\top}\bm{s}_{l,\parallel}^{\top}\Big\|_{2}^{2}\nonumber \\
 & \qquad=\frac{1}{n}\Big\|\bm{U}^{(l)}\big(\lambda_{l}\bm{I}_{p-1}-\bm{\Gamma}^{(l)}\big)^{-1}\bm{U}^{(l)\top}\bm{U}^{(l)}\sqrt{\bm{\Gamma}^{(l)}}\bm{V}^{(l)\top}\bm{s}_{l,\parallel}^{\top}\Big\|_{2}^{2}\nonumber \\
 & \qquad=\frac{1}{n}\Big\|\big(\lambda_{l}\bm{I}_{p-1}-\bm{\Gamma}^{(l)}\big)^{-1}\sqrt{\bm{\Gamma}^{(l)}}\bm{V}^{(l)\top}\bm{s}_{l,\parallel}^{\top}\Big\|_{2}^{2}\nonumber \\
 & \qquad=\frac{1}{n}\sum_{1\leq i\le n\wedge(p-1)}\frac{\gamma_{i}^{(l)}}{(\lambda_{l}-\gamma_{i}^{(l)})^{2}}\big(\bm{v}_{i}^{(l)\top}\bm{s}_{l,\parallel}^{\top}\big)^{2}.\label{eq:cos0-pca}
\end{align}
In what follows, we shall control the sum over $i<r$ and the sum
over $i\geq r$ separately.

\paragraph{Controlling the sum over $i<r$.}

According to Lemma \ref{lemma:lambda-S-minus-spectrum-1-1}, one has
\[
\gamma_{i}^{(l)}\lesssim\lambda_{\max}^{\star}+\sigma^{2}\qquad\text{and}\qquad(\gamma_{i}^{(l)}-\lambda_{l})^{2}\gtrsim\min_{i:i\ne l}(\lambda_{l}^{\star}-\lambda_{i}^{\star})^{2}
\]
for all $1\leq i<r$. In addition, recall that $\bm{v}_{i}^{(l)\top}\bm{s}_{l,\parallel}^{\top}\,(1\leq i<p)$
are i.i.d.~zero-mean Gaussian random variables with variance $\lambda_{l}^{\star}+\sigma^{2}$
(see (\ref{eq:independence-vil-sl-par})). Invoking standard Gaussian
inequalities shows that with probability at least $1-O(n^{-10})$,
\begin{equation}
\max_{1\leq i<r}\big(\bm{v}_{i}^{(l)\top}\bm{s}_{l,\parallel}^{\top}\big)^{2}\lesssim(\lambda_{l}^{\star}+\sigma^{2})\log n.\label{eq:vl-slpara-UB}
\end{equation}
As a result, we obtain
\begin{equation}
\sum_{1\leq i<r}\frac{\gamma_{i}^{(l)}\big(\bm{v}_{i}^{(l)\top}\bm{s}_{l,\parallel}^{\top}\big)^{2}}{(\lambda_{l}-\gamma_{i}^{(l)})^{2}}\lesssim\frac{(\lambda_{\max}^{\star}+\sigma^{2})(\lambda_{l}^{\star}+\sigma^{2})\,r\log n}{\big(\Delta_{l}^{\star}\big)^{2}}\label{eq:lambda-M-r-sum1-bound-pca}
\end{equation}
with probability at least $1-O(n^{-10})$. 

\paragraph{Controlling the sum over $i\protect\geq r$.}

Now, let us control the sum over $i\geq r$, and we shall consider
the case with $n\geq p$ and the case with $n<p$ separately.
\begin{itemize}
\item \emph{Case I: $n\geq p$.} Note that $\bm{v}_{i}^{(l)\top}\bm{s}_{l,\parallel}^{\top}\overset{\mathrm{i.i.d.}}{\sim}\mathcal{N}(0,\lambda_{l}^{\star}+\sigma^{2})$.
This suggests that we decompose
\begin{align}
\sum_{r\leq i\le n\wedge(p-1)}\frac{\gamma_{i}^{(l)}(\bm{v}_{i}^{(l)\top}\bm{s}_{l,\parallel}^{\top})^{2}}{(\lambda_{l}-\gamma_{i}^{(l)})^{2}} & =\underbrace{\sum_{r\leq i\le n\wedge(p-1)}\frac{\gamma_{i}^{(l)}(\lambda_{l}^{\star}+\sigma^{2})}{(\lambda_{l}-\gamma_{i}^{(l)})^{2}}}_{=:\,\alpha_{1}}+\underbrace{\sum_{r\leq i\le n\wedge(p-1)}\frac{\gamma_{i}^{(l)}\big((\bm{v}_{i}^{(l)\top}\bm{s}_{l,\parallel}^{\top})^{2}-(\lambda_{l}^{\star}+\sigma^{2})\big)}{(\lambda_{l}-\gamma_{i}^{(l)})^{2}}}_{=:\,\alpha_{2}},\label{eq:lambda-M-r-n-value-bound-pca}
\end{align}
and control $\alpha_{1}$ as well as $\alpha_{2}$ individually.
\begin{itemize}
\item To begin with, let us consider $\alpha_{1}$, which requires estimating
$\lambda_{l}^{\star}+\sigma^{2}$ and $\sum_{r\leq i\le n\wedge(p-1)}\gamma_{i}^{(l)}/(\lambda_{l}-\gamma_{i}^{(l)})^{2}$.
This task is accomplished in Lemma~\ref{lemma:lambda-M-r-n-pca}
and Lemma~\ref{lemma:lambda-star-sigma-est} stated below. 

\begin{lemma}\label{lemma:lambda-M-r-n-pca}Instate the assumptions
of Theorem \ref{thm:evector-pertur-sym-iid-pca}. With probability
at least $1-O(n^{-10})$, we have
\begin{equation}
\bigg|\underbrace{\sum_{i\ge r}\frac{\gamma_{i}^{(l)}}{(\lambda_{l}-\gamma_{i}^{(l)})^{2}}-\sum_{i>r}\frac{\lambda_{i}}{(\lambda_{l}-\lambda_{i})^{2}}}_{=:\,\epsilon_{1}}\bigg|\lesssim\frac{\sigma^{2}}{\lambda_{l}^{\star2}}\Big(1+\frac{p}{n}\Big).\label{eq:lambda-M-r-n-value-pca}
\end{equation}
and
\begin{equation}
\bigg|\sum_{i>r}\frac{\lambda_{i}}{(\lambda_{l}-\lambda_{i})^{2}}\bigg|\vee\bigg|\sum_{i\ge r}\frac{\gamma_{i}^{(l)}}{(\lambda_{l}-\gamma_{i}^{(l)})^{2}}\bigg|\lesssim\frac{\sigma^{2}p}{\lambda_{l}^{\star2}},\label{eq:lambda-M-r-n-bound-pca}
\end{equation}
\end{lemma} \begin{proof}See Appendix \ref{subsec:Proof-of-lemma:lambda-M-r-n-pca}.
\end{proof}

\begin{lemma}\label{lemma:lambda-star-sigma-est}Instate the assumptions
of Theorem \ref{thm:evector-pertur-sym-iid-pca}. With probability
at least $1-O(n^{-10})$, one has
\begin{align}
\bigg|\underbrace{\frac{\lambda_{l}}{1+\frac{1}{n}\sum_{i>r}\frac{\lambda_{i}}{\lambda_{l}-\lambda_{i}}}-(\lambda_{l}^{\star}+\sigma^{2})}_{=:\,\epsilon_{2}}\bigg| & \lesssim\underbrace{(\lambda_{\max}^{\star}+\sigma^{2})\sqrt{\frac{r\log n}{n}}}_{=\,\mathcal{E}_{\mathsf{PCA}}}.\label{claim:lambda-sigma-est}
\end{align}
and
\begin{equation}
\frac{\lambda_{l}}{1+\frac{1}{n}\sum_{i>r}\frac{\lambda_{i}}{\lambda_{l}-\lambda_{i}}}\asymp\lambda_{l}^{\star}.\label{claim:lambda-sigma-est-bound}
\end{equation}

\end{lemma} \begin{proof}See Appendix \ref{subsec:Proof-of-lemma:lambda-star-sigma-est}.
\end{proof}

With these two lemmas in place, we are ready to control $\alpha_{1}$.
According to (\ref{eq:lambda-M-r-n-bound-pca}), we have
\begin{equation}
\alpha_{1}\lesssim\frac{(\lambda_{l}^{\star}+\sigma^{2})\sigma^{2}p}{\lambda_{l}^{\star2}}.\label{eq:PCA-alpha1-n-large}
\end{equation}
In addition, recall the definition of $c_{l}$ (cf.~(\ref{lemma:lambda-M-r-n-pca})).
We can upper bound
\begin{align}
|\alpha_{1}-c_{l}\cdot n| & =\bigg|(\lambda_{l}^{\star}+\sigma^{2})\sum_{i\geq r}\frac{\gamma_{i}^{(l)}}{(\lambda_{l}-\gamma_{i}^{(l)})^{2}}-c_{l}\cdot n\bigg|\nonumber \\
 & =\Bigg|(\lambda_{l}^{\star}+\sigma^{2})\bigg(\sum_{i>r}\frac{\lambda_{i}}{(\lambda_{l}-\lambda_{i})^{2}}+\epsilon_{1}\bigg)-(\lambda_{l}^{\star}+\sigma^{2}+\epsilon_{2})\sum_{i>r}\frac{\lambda_{i}}{(\lambda_{l}-\lambda_{i})^{2}}\Bigg|\nonumber \\
 & =\bigg|\epsilon_{1}(\lambda_{l}^{\star}+\sigma^{2})-\epsilon_{2}\sum_{i>r}\frac{\lambda_{i}}{(\lambda_{l}-\lambda_{i})^{2}}\bigg|\nonumber \\
 & \overset{(\mathrm{i})}{\lesssim}(\lambda_{l}^{\star}+\sigma^{2})|\epsilon_{1}|+\frac{\sigma^{2}p}{\lambda_{l}^{\star2}}|\epsilon_{2}|\nonumber \\
 & \overset{(\mathrm{ii})}{\lesssim}(\lambda_{l}^{\star}+\sigma^{2})\cdot\frac{\sigma^{2}}{\lambda_{l}^{\star2}}\Big(1+\frac{p}{n}\Big)+\frac{\sigma^{2}p}{\lambda_{l}^{\star2}}\cdot(\lambda_{\max}^{\star}+\sigma^{2})\sqrt{\frac{r\log n}{n}}\nonumber \\
 & \asymp\frac{(\lambda_{l}^{\star}+\sigma^{2})\sigma^{2}}{\lambda_{l}^{\star2}}+\frac{(\lambda_{\max}^{\star}+\sigma^{2})\sigma^{2}p}{\lambda_{l}^{\star2}}\sqrt{\frac{r\log n}{n}}.\label{eq:PCA-alpha1-bl-n-large}
\end{align}
Here, (i) makes use of (\ref{eq:lambda-M-r-n-bound-pca}) and (\ref{eq:lambda-M-r-n-value-pca});
(ii) relies on (\ref{eq:lambda-M-r-n-value-pca}) and (\ref{claim:lambda-sigma-est}).
\item Next, we move on to look at $\alpha_{2}$. Observe that $\big\{(\bm{v}_{i}^{(l)\top}\bm{s}_{l,\parallel}^{\top})^{2}-(\lambda_{l}^{\star}+\sigma^{2})\big\}_{i\geq r}$
is a sequence of zero-mean sub-exponential random variables, which
is independent of $\bm{\Gamma}^{(l)}$ but depends on $\lambda_{l}$.
Hence, we shall apply the epsilon-net argument (cf.~Lemma~\ref{lemma:eps-net})
to bound $\alpha_{2}$. To do so, let us first verify the conditions
required therein. With probability exceeding $1-O(n^{-20})$, one
has
\begin{align*}
 & \Bigg|\frac{\mathrm{d}}{\mathrm{d}\lambda}\sum_{r\leq i\le n\wedge(p-1)}\frac{\gamma_{i}^{(l)}}{(\lambda-\gamma_{i}^{(l)})^{2}}\big((\bm{v}_{i}^{(l)\top}\bm{s}_{l,\parallel}^{\top})^{2}-(\lambda_{l}^{\star}+\sigma^{2})\big)\Bigg|\\
 & \qquad=\Bigg|\sum_{r\leq i\le n\wedge(p-1)}\frac{\gamma_{i}^{(l)}}{(\lambda-\gamma_{i}^{(l)})^{3}}\big((\bm{v}_{i}^{(l)\top}\bm{s}_{l,\parallel}^{\top})^{2}-(\lambda_{l}^{\star}+\sigma^{2})\big)\Bigg|\\
 & \qquad\leq(p\wedge n)\cdot\max_{r\leq i\leq n\wedge(p-1)}\frac{\gamma_{i}^{(l)}}{(\lambda-\gamma_{i}^{(l)})^{3}}\cdot\max_{r\leq i\leq n\wedge(p-1)}\big|(\bm{v}_{i}^{(l)\top}\bm{s}_{l,\parallel}^{\top})^{2}-(\lambda_{l}^{\star}+\sigma^{2})\big|\\
 & \qquad\lesssim(p\wedge n)\cdot\frac{\sigma^{2}(p\vee n)}{\lambda_{l}^{\star3}n}\cdot(\lambda_{l}^{\star}+\sigma^{2})\log n\\
 & \qquad=\frac{(\lambda_{l}^{\star}+\sigma^{2})\sigma^{2}p\log n}{\lambda_{l}^{\star3}}
\end{align*}
for all $\lambda$ with $\lambda/\big(1+\beta(\lambda)\big)\in\mathcal{B}_{\mathcal{E}_{\mathsf{PCA}}}(\lambda_{l}^{\star}+\sigma^{2})$.
In addition, it is seen that
\begin{align*}
\big\|(\bm{v}_{i}^{(l)\top}\bm{s}_{l,\parallel}^{\top})^{2}-(\lambda_{l}^{\star}+\sigma^{2})\big\|_{\psi_{1}} & \lesssim\lambda_{l}^{\star}+\sigma^{2},\\
\mathbb{E}\Big[\big((\bm{v}_{i}^{(l)\top}\bm{s}_{l,\parallel}^{\top})^{2}-(\lambda_{l}^{\star}+\sigma^{2})\big)^{2}\Big] & \lesssim(\lambda_{l}^{\star}+\sigma^{2})^{2},
\end{align*}
where $\|\cdot\|_{\psi_{1}}$ denotes the sub-exponential norm. Invoke
the matrix Bernstein inequality \cite[Corollary 2.1]{Koltchinskii2011oracle}
to show that: for any fixed $\lambda$ obeying $\lambda/\big(1+\beta(\lambda)\big)\in\mathcal{B}_{\mathcal{E}_{\mathsf{PCA}}}(\lambda_{l}^{\star}+\sigma^{2})$,
one has
\begin{align*}
 & \Bigg|\sum_{r\leq i\text{\ensuremath{\le}}n\wedge(p-1)}\frac{\gamma_{i}^{(l)}}{(\lambda-\gamma_{i}^{(l)})^{2}}\big((\bm{v}_{i}^{(l)\top}\bm{s}_{l,\parallel}^{\top})^{2}-(\lambda_{l}^{\star}+\sigma^{2})\big)\Bigg|\\
 & \qquad\lesssim\max_{r\leq i\text{\ensuremath{\le}}n\wedge(p-1)}\frac{\gamma_{i}^{(l)}}{(\lambda-\gamma_{i}^{(l)})^{2}}\cdot(\lambda_{l}^{\star}+\sigma^{2})(\log n+\sqrt{p\wedge n}\log n)\\
 & \qquad\lesssim\frac{(\lambda_{l}^{\star}+\sigma^{2})\sigma^{2}}{\lambda_{l}^{\star2}}\frac{p\vee n}{n}\sqrt{p\wedge n}\log n\\
 & \qquad\lesssim\frac{(\lambda_{l}^{\star}+\sigma^{2})\sigma^{2}p}{\lambda_{l}^{\star2}}\frac{\log n}{\sqrt{p\wedge n}}
\end{align*}
with probability at least $1-O(n^{-10})$, where the last inequalities
hold since $|\lambda-\gamma_{i}^{(l)}|\gtrsim\lambda_{l}^{\star}$
and $\gamma_{i}^{(l)}\lesssim\sigma^{2}(1+p/n)$. In addition, we
make the observation that
\[
\{\lambda\colon\lambda/\big(1+\beta(\lambda)\big)\in\mathcal{B}_{\mathcal{E}_{\mathsf{PCA}}}(\lambda_{l}^{\star}+\sigma^{2})\}\subseteq[2\lambda_{l}^{\star}/3,\,4\lambda_{l}^{\star}/3].
\]
With these in place, one can readily invoke Lemma~\ref{lemma:eps-net}
to derive
\begin{equation}
\alpha_{2}\lesssim\frac{(\lambda_{l}^{\star}+\sigma^{2})\sigma^{2}p}{\lambda_{l}^{\star2}}\frac{\log n}{\sqrt{p\wedge n}}\label{eq:PCA-alpha2-n-large}
\end{equation}
with probability at least $1-O(n^{-10})$. 
\item Combining (\ref{eq:PCA-alpha1-n-large}) and (\ref{eq:PCA-alpha2-n-large}),
we conclude
\begin{align}
\sum_{r\leq i\le n\wedge(p-1)}\frac{\gamma_{i}^{(l)}(\bm{v}_{i}^{(l)\top}\bm{s}_{l,\parallel}^{\top})^{2}}{n\,(\lambda_{l}-\gamma_{i}^{(l)})^{2}} & \leq\frac{|\alpha_{1}|+|\alpha_{2}|}{n}\lesssim\frac{(\lambda_{l}^{\star}+\sigma^{2})\sigma^{2}p}{\lambda_{l}^{\star2}n}+\frac{(\lambda_{l}^{\star}+\sigma^{2})\sigma^{2}p}{\lambda_{l}^{\star2}n}\frac{\log n}{\sqrt{p\wedge n}}\nonumber \\
 & \lesssim\frac{(\lambda_{l}^{\star}+\sigma^{2})\sigma^{2}p\log n}{\lambda_{l}^{\star2}n}.\label{eq:lambda-M-r-sum2-bound-pca}
\end{align}
\end{itemize}
\begin{quotation}
Moreover, putting (\ref{eq:PCA-alpha1-bl-n-large}) and (\ref{eq:PCA-alpha2-n-large})
together gives
\begin{align}
\Bigg|\sum_{i\geq r}\frac{\gamma_{i}^{(l)}(\bm{v}_{i}^{(l)\top}\bm{s}_{l,\parallel}^{\top})^{2}}{n\,(\lambda_{l}-\gamma_{i}^{(l)})^{2}}-c_{l}\Bigg| & \le\frac{|\alpha_{1}-c_{l}\cdot n|+|\alpha_{2}|}{n}\nonumber \\
 & \lesssim\frac{(\lambda_{l}^{\star}+\sigma^{2})\sigma^{2}}{\lambda_{l}^{\star2}n}+\frac{(\lambda_{\max}^{\star}+\sigma^{2})\sigma^{2}p}{\lambda_{l}^{\star2}n}\sqrt{\frac{r\log n}{n}}+\frac{(\lambda_{l}^{\star}+\sigma^{2})\sigma^{2}p}{\lambda_{l}^{\star2}n}\frac{\log n}{\sqrt{p\wedge n}}\nonumber \\
 & \asymp\frac{\sigma^{2}p}{\lambda_{l}^{\star2}n}\bigg((\lambda_{\max}^{\star}+\sigma^{2})\sqrt{\frac{r\log n}{n}}+(\lambda_{l}^{\star}+\sigma^{2})\frac{\log n}{\sqrt{p\wedge n}}\bigg),\label{eq:cos_r_n-pca}
\end{align}
where $c_{l}$ is defined in (\ref{eq:def:bl-pca}).
\end{quotation}
\item \emph{Case II: $n<p$.} As it turns out, the above analysis for (\ref{eq:cos_r_n-pca})
is not tight when it comes to the case $n<p$. To remedy the issue,
we provide a more precise estimate for terms (\ref{eq:lambda-M-r-n-value-bound-pca})
in the following lemma.

\begin{lemma}\label{lemma:lambda-M-r-n-pca-precise}Instate the assumptions
of Theorem \ref{thm:evector-pertur-sym-iid-pca}. Suppose that $p>n$,
then the following holds with probability at least $1-O(n^{-10})$:
\begin{align}
\sum_{i\geq r}\frac{\gamma_{i}^{(l)}(\bm{v}_{i}^{(l)\top}\bm{s}_{l,\parallel}^{\top})^{2}}{(\lambda_{l}-\gamma_{i}^{(l)})^{2}} & =\underbrace{\frac{\sigma^{2}p}{\lambda_{l}-\sigma^{2}p/n}+\frac{\lambda_{l}}{\lambda_{l}-\sigma^{2}p/n}\ \frac{\lambda_{l}}{1+\frac{1}{n}\sum_{i>r}\frac{\lambda_{i}}{\lambda_{l}-\lambda_{i}}}\sum_{r<i\le n}\frac{\lambda_{i}-\sigma^{2}p/n}{(\lambda_{l}-\lambda_{i})^{2}}}_{=c_{l}\cdot n}\nonumber \\
 & \quad+O\Big(\frac{\sigma^{2}pr\log n}{\min_{i:i\ne l}\left|\lambda_{l}^{\star}-\lambda_{i}^{\star}\right|n}+\frac{\sigma^{2}\kappa\sqrt{pr\log n}}{\lambda_{l}^{\star}}\Big).\label{eq:lambda-M-r-n-value-bound-pca-p}
\end{align}
\end{lemma}\begin{proof}See Appendix \ref{subsec:Proof-of-lemma:lambda-M-r-n-pca-precise}.
\end{proof}

Here, the quantity $c_{l}$ introduced above is precisely the one
defined in (\ref{eq:def:bl-pca}). Consequently, we arrive at
\begin{align}
\Bigg|\sum_{i\geq r}\frac{\gamma_{i}^{(l)}(\bm{v}_{i}^{(l)\top}\bm{s}_{l,\parallel}^{\top})^{2}}{n\,(\lambda_{l}-\gamma_{i}^{(l)})^{2}}-c_{l}\Bigg| & \lesssim\frac{\sigma^{2}pr\log n}{\min_{i:i\ne l}\left|\lambda_{l}^{\star}-\lambda_{i}^{\star}\right|n}+\frac{\sigma^{2}\kappa\sqrt{pr\log n}}{\lambda_{l}^{\star}n}.\label{eq:cos_r_n-pca-p}
\end{align}

\end{itemize}

\paragraph{Combining two sums. }

Substituting (\ref{eq:lambda-M-r-sum1-bound-pca}) and (\ref{eq:lambda-M-r-sum2-bound-pca})
(which holds universally for any $n$) into (\ref{eq:cos0-pca}),
we reach the first claim (\ref{eq:claim:lambda-M-l-inv-u-perp-M-u-l2-norm-1-pca-2}):
\[
\Big\|\Big(\lambda_{l}\bm{I}_{p-1}-\frac{1}{n}\bm{S}_{l,\perp}\bm{S}_{l,\perp}^{\top}\Big)^{-1}\frac{1}{n}\bm{S}_{l,\perp}\bm{s}_{l,\parallel}^{\top}\Big\|_{2}^{2}\lesssim\frac{(\lambda_{\max}^{\star}+\sigma^{2})(\lambda_{l}^{\star}+\sigma^{2})\,r\log n}{\min_{i:i\neq l}\left|\lambda_{l}^{\star}-\lambda_{i}^{\star}\right|^{2}n}+\frac{(\lambda_{l}^{\star}+\sigma^{2})\sigma^{2}p\log^{2}n}{\lambda_{l}^{\star2}n}\ll1,
\]
where the last step holds due to the assumptions (\ref{eq:noise-condition-iid-pca})
and (\ref{eq:eigengap-condition-iid-pca}).

We now turn attention to the estimation error. Regarding the case
with $n\geq p$, we can combine (\ref{eq:lambda-M-r-sum1-bound-pca})
and (\ref{eq:cos_r_n-pca}) to conclude that
\begin{align*}
 & \bigg|\Big\|\Big(\lambda_{l}\bm{I}_{p-1}-\frac{1}{n}\bm{S}_{l,\perp}\bm{S}_{l,\perp}^{\top}\Big)^{-1}\frac{1}{n}\bm{S}_{l,\perp}\bm{s}_{l,\parallel}^{\top}\Big\|_{2}^{2}-c_{l}\bigg|\\
 & \qquad\lesssim\frac{(\lambda_{\max}^{\star}+\sigma^{2})(\lambda_{l}^{\star}+\sigma^{2})\,r\log n}{\min_{i:i\neq l}|\lambda_{l}^{\star}-\lambda_{i}^{\star}|^{2}n}+\frac{\sigma^{2}p}{\lambda_{l}^{\star2}n}\bigg((\lambda_{l}^{\star}+\sigma^{2})\frac{\log^{2}n}{\sqrt{p\wedge n}}+(\lambda_{\max}^{\star}+\sigma^{2})\sqrt{\frac{r\log n}{n}}\bigg).
\end{align*}
As for the case with $n<p$, substituting (\ref{eq:lambda-M-r-sum1-bound-pca})
and (\ref{eq:cos_r_n-pca-p}) into (\ref{eq:cos0-pca}) yields
\begin{align*}
\bigg|\Big\|\Big(\lambda_{l}\bm{I}_{p-1}-\frac{1}{n}\bm{S}_{l,\perp}\bm{S}_{l,\perp}^{\top}\Big)^{-1}\frac{1}{n}\bm{S}_{l,\perp}\bm{s}_{l,\parallel}^{\top}\Big\|_{2}^{2}-c_{l}\bigg| & \lesssim\frac{(\lambda_{\max}^{\star}+\sigma^{2})(\lambda_{l}^{\star}+\sigma^{2})r\log n}{\min_{i:i\ne l}\left|\lambda_{l}^{\star}-\lambda_{i}^{\star}\right|^{2}n}+\frac{\sigma^{2}\kappa\sqrt{pr\log n}}{\lambda_{l}^{\star}n}+\frac{\sigma^{2}pr\log n}{\min_{i:i\ne l}\left|\lambda_{l}^{\star}-\lambda_{i}^{\star}\right|n^{2}}\\
 & \asymp\frac{(\lambda_{\max}^{\star}+\sigma^{2})(\lambda_{l}^{\star}+\sigma^{2})r\log n}{\min_{i:i\ne l}\left|\lambda_{l}^{\star}-\lambda_{i}^{\star}\right|^{2}n}+\frac{\sigma^{2}\kappa\sqrt{pr\log n}}{\lambda_{l}^{\star}n},
\end{align*}
where in the last line we use the conditions $\min_{i:i\ne l}\left|\lambda_{l}^{\star}-\lambda_{i}^{\star}\right|\lesssim\lambda_{\max}^{\star}$
and $\sigma^{2}p/n\ll\lambda_{l}^{\star}$ (according to the noise
assumption (\ref{eq:noise-condition-iid-pca})).

\subsubsection{Proof of Lemma \ref{lemma:lambda-M-r-n-pca}}

\label{subsec:Proof-of-lemma:lambda-M-r-n-pca}

To begin with, let us consider (\ref{eq:lambda-M-r-n-value-pca}).
By construction, we have $\bm{S}_{l,\perp}\bm{S}_{l,\perp}^{\top}=(\bm{u}_{l}^{\star\perp})^{\top}\bm{S}\bm{S}^{\top}\bm{u}_{l}^{\star\perp}$,
and it follows from Lemma \ref{lemma:eigval-interlacing} that $\lambda_{i+1}\le\gamma_{i}^{(l)}\le\lambda_{i}$
for all $1\leq i<p$. Simple calculation yields
\[
\frac{\gamma_{i}^{(l)}}{(\lambda_{l}-\gamma_{i}^{(l)})^{2}}-\frac{\lambda_{i}}{(\lambda_{l}-\lambda_{i})^{2}}=\frac{(\gamma_{i}^{(l)}-\lambda_{i})(\lambda_{l}^{2}-\lambda_{i}\gamma_{i}^{(l)})}{(\lambda_{l}-\gamma_{i}^{(l)})^{2}(\lambda_{l}-\lambda_{i})^{2}},
\]
and consequently
\[
\frac{\lambda_{i+1}}{(\lambda_{l}-\lambda_{i+1})^{2}}\leq\frac{\gamma_{i}^{(l)}}{(\lambda_{l}-\gamma_{i}^{(l)})^{2}}\leq\frac{\lambda_{i}}{(\lambda_{l}-\lambda_{i})^{2}},\qquad i\geq r.
\]
We can then invoke a similar argument used in the proof of Lemma~\ref{lemma:lambda-M-r-n}
to bound
\[
\sum_{r<i\le n}\frac{\lambda_{i}}{(\lambda_{l}-\lambda_{i})^{2}}\le\sum_{r\le i\le n}\frac{\gamma_{i}^{(l)}}{(\lambda_{l}-\gamma_{i}^{(l)})^{2}}\le\sum_{r<i\le n}\frac{\lambda_{i}}{(\lambda_{l}-\lambda_{i})^{2}}+\frac{\gamma_{r}^{(l)}}{(\lambda_{l}-\gamma_{r}^{(l)})^{2}}.
\]
Hence, the conclusion immediately follows since $\gamma_{r}^{(l)}/(\lambda_{l}-\gamma_{r}^{(l)})^{2}\lesssim\sigma^{2}(1+p/n)/\lambda_{l}^{\star2}$
by Lemma \ref{lemma:lambda-S-minus-spectrum-1-1}.

We proceed to consider (\ref{eq:lambda-M-r-n-bound-pca}). According
to Lemma \ref{lemma:lambda-S-minus-spectrum-1-1}, the following holds
for eigenvalues $\{\gamma_{i}^{(l)}\}_{i\geq r}$: (i) $|\lambda_{l}-\gamma_{i}^{(l)}|\gtrsim\lambda_{l}^{\star}$
for all $i\ge r$; (ii) $|\gamma_{i}^{(l)}|\lesssim\sigma^{2}(p\vee n)/n$
for all $r\le i\le n\wedge(p-1)$; (iii) $\gamma_{i}^{(l)}=0$ for
$n\wedge(p-1)<i<p$. Therefore, we can upper bound
\[
\bigg|\sum_{i:\,i\geq r}\frac{\gamma_{i}^{(l)}}{(\lambda_{l}-\gamma_{i}^{(l)})^{2}}\bigg|\lesssim(p\wedge n)\frac{\sigma^{2}(p\vee n)}{\lambda_{l}^{\star2}n}=\frac{\sigma^{2}p}{\lambda_{l}^{\star2}},
\]
and the upper bound for $\sum_{i>r}\lambda_{i}/(\lambda_{l}-\lambda_{i})^{2}$
immediately follows from the triangle inequality.

\subsubsection{Proof of Lemma \ref{lemma:lambda-star-sigma-est}}

\label{subsec:Proof-of-lemma:lambda-star-sigma-est}

By the definition of $\beta(\cdot)$ in (\ref{eq:definition-gamma-lambda-iid-pca}),
we can express
\begin{align*}
\beta(\lambda_{l}) & =\frac{1}{n}\sum_{1\le i\le p-r}\frac{\lambda_{i}(\frac{1}{n}\bm{S}_{\perp}\bm{S}_{\perp}^{\top})}{\lambda_{l}-\lambda_{i}(\frac{1}{n}\bm{S}_{\perp}\bm{S}_{\perp}^{\top})}.
\end{align*}
From Lemma \ref{lemma:eigval-interlacing}, we know that $\lambda_{i+r}\le\lambda_{i}(\frac{1}{n}\bm{S}_{\perp}\bm{S}_{\perp}^{\top})\le\lambda_{i}$
for each $1\leq i\leq p-r$, and thus
\[
\frac{\lambda_{i+r}}{\lambda_{l}-\lambda_{i+r}}\le\frac{\lambda_{i}(\frac{1}{n}\bm{S}_{\perp}\bm{S}_{\perp}^{\top})}{\lambda_{l}-\lambda_{i}(\frac{1}{n}\bm{S}_{\perp}\bm{S}_{\perp}^{\top})}\le\frac{\lambda_{i}}{\lambda_{l}-\lambda_{i}},\qquad i>r.
\]
Hence, we have
\begin{equation}
0\le\beta(\lambda_{l})-\frac{1}{n}\sum_{r<i\leq p}\frac{\lambda_{i}}{\lambda_{l}-\lambda_{i}}\le\frac{1}{n}\sum_{1\le i\le r}\frac{\lambda_{i}(\frac{1}{n}\bm{S}_{\perp}\bm{S}_{\perp}^{\top})}{\lambda_{l}-\lambda_{i}(\frac{1}{n}\bm{S}_{\perp}\bm{S}_{\perp}^{\top})}.\label{eq:beta-lambda-p>n-temp}
\end{equation}

As shown in ((\ref{eq:eigenvalue-perturbation-bound-iid-pca})), the
eigenvalue $\lambda_{l}$ satisfies $\lambda_{l}/\big(1+\beta(\lambda_{l})\big)=\lambda_{l}^{\star}+\sigma^{2}+O(\mathcal{E}_{\mathsf{PCA}})$
where $\mathcal{E}_{\mathsf{PCA}}$ is defined in (\ref{eq:def-E-PCA}).
In particular, we note that for the case with $n<p$, the assumption
(\ref{eq:noise-condition-iid-pca}) guarantees that
\begin{equation}
\sigma^{2}=o(\lambda_{\min}^{\star})\qquad\text{and}\qquad\mathcal{E}_{\mathsf{PCA}}\coloneqq(\lambda_{\max}^{\star}+\sigma^{2})\sqrt{\frac{r}{n}}\log n=o(\lambda_{\min}^{\star}).\label{eq:sigma-Epca-UB-p>n}
\end{equation}
Combined with (\ref{eq:beta-UB}), this also implies that
\begin{equation}
\lambda_{l}\asymp\lambda_{l}^{\star}.\label{eq:lambda-lambda-star-p>n}
\end{equation}

With these estimates in place, one can use (\ref{eq:beta-lambda-p>n-temp})
and the high-probability bound $\|\frac{1}{n}\bm{S}_{\perp}\bm{S}_{\perp}^{\top}\|\lesssim\sigma^{2}(1+p/n)\ll\lambda_{l}^{\star}$
in (\ref{eq:S-perp-op-UB}) to derive
\begin{equation}
\bigg|\frac{1}{n}\sum_{r<i\leq p}\frac{\lambda_{i}}{\lambda_{l}-\lambda_{i}}-\beta(\lambda_{l})\bigg|\leq\frac{r}{n}\cdot\max_{1\leq i\leq r}\bigg|\frac{\lambda_{i}(\frac{1}{n}\bm{S}_{\perp}\bm{S}_{\perp}^{\top})}{\lambda_{l}-\lambda_{i}(\frac{1}{n}\bm{S}_{\perp}\bm{S}_{\perp}^{\top})}\bigg|\lesssim\frac{r}{n}\cdot\frac{\sigma^{2}}{\lambda_{l}^{\star}}\bigg(1+\frac{p}{n}\bigg)=o\Big(\frac{r}{n}\Big)\label{eq:beta-lambda-p>n}
\end{equation}
where the last step holds due the noise condition (\ref{eq:noise-condition-iid-pca}).
Meanwhile, we can combine (\ref{eq:beta-UB}) with (\ref{eq:beta-lambda-p>n})
to find
\begin{equation}
\bigg|\frac{1}{n}\sum_{r<i\leq p}\frac{\lambda_{i}}{\lambda_{l}-\lambda_{i}}\bigg|\ll1\label{eq:de-biased-para-UB-p>n}
\end{equation}
as long as $n\gg r$. Plugging this into ((\ref{eq:eigenvalue-perturbation-bound-iid-pca}))
reveals that
\begin{align}
\lambda_{l}^{\star}+\sigma^{2} & =\frac{\lambda_{l}}{1+\beta(\lambda_{l})}+O(\mathcal{E}_{\mathsf{PCA}})\nonumber \\
 & \overset{(\mathrm{i})}{=}\frac{\lambda_{l}}{1+\frac{1}{n}\sum_{r<i\leq p}\frac{\lambda_{i}}{\lambda_{l}-\lambda_{i}}+o(\frac{r}{n})}+O(\mathcal{E}_{\mathsf{PCA}})\nonumber \\
 & =\frac{\lambda_{l}}{1+\frac{1}{n}\sum_{r<i\leq p}\frac{\lambda_{i}}{\lambda_{l}-\lambda_{i}}}+o\Big(\frac{\lambda_{l}^{\star}r}{n}\Big)+O(\mathcal{E}_{\mathsf{PCA}})\nonumber \\
 & =\frac{\lambda_{l}}{1+\frac{1}{n}\sum_{r<i\leq p}\frac{\lambda_{i}}{\lambda_{l}-\lambda_{i}}}+O(\mathcal{E}_{\mathsf{PCA}}),\label{eq:lambda-sigma-est}
\end{align}
where (i) holds due to (\ref{eq:beta-lambda-p>n}); (ii) follows from
(\ref{eq:lambda-lambda-star-p>n}) and (\ref{eq:de-biased-para-UB-p>n});
(iii) holds as long as $r\ll n$. This completes the proof for (\ref{claim:lambda-sigma-est}). 

In addition, the claim (\ref{claim:lambda-sigma-est-bound}) is an
immediate consequence of (\ref{eq:sigma-Epca-UB-p>n}) and (\ref{eq:de-biased-para-UB-p>n}).

\subsubsection{Proof of Lemma \ref{lemma:lambda-M-r-n-pca-precise}}

\label{subsec:Proof-of-lemma:lambda-M-r-n-pca-precise}

In view of (\ref{eq:lambdal-tilde-lambdak-pca-UB-i>r}) and the fact
that $\bm{v}_{i}^{(l)\top}\bm{s}_{l,\parallel}^{\top}\overset{\mathrm{i.i.d.}}{\sim}\mathcal{N}(0,\lambda_{l}^{\star}+\sigma^{2})$
(see (\ref{eq:independence-vil-sl-par})), we are motivated to first
decompose
\begin{align}
\sum_{r\le i\le n}\frac{\gamma_{i}^{(l)}(\bm{v}_{i}^{(l)\top}\bm{s}_{l,\parallel}^{\top})^{2}}{(\lambda_{l}-\gamma_{i}^{(l)})^{2}} & =\underbrace{\sum_{r\le i\le n}\frac{\sigma^{2}p/n}{(\lambda_{l}-\sigma^{2}p/n)^{2}}(\bm{v}_{i}^{(l)\top}\bm{s}_{l,\parallel}^{\top})^{2}}_{=:\,\alpha_{1}}\nonumber \\
 & \quad+\underbrace{\sum_{r\le i\le n}\Big(\frac{\gamma_{i}^{(l)}}{(\lambda_{l}-\gamma_{i}^{(l)})^{2}}-\frac{\sigma^{2}p/n}{(\lambda_{l}-\sigma^{2}p/n)^{2}}\Big)(\lambda_{l}^{\star}+\sigma^{2})}_{=:\,\alpha_{2}}\nonumber \\
 & \quad+\underbrace{\sum_{r\le i\le n}\Big(\frac{\gamma_{i}^{(l)}}{(\lambda_{l}-\gamma_{i}^{(l)})^{2}}-\frac{\sigma^{2}p/n}{(\lambda_{l}-\sigma^{2}p/n)^{2}}\Big)\big\{(\bm{v}_{i}^{(l)\top}\bm{s}_{l,\parallel}^{\top})^{2}-(\lambda_{l}^{\star}+\sigma^{2})\big\}}_{=:\,\alpha_{3}}.\label{eq:lambda-M-r-n-value-pca-p-decomp}
\end{align}
In what follows, we shall control $\alpha_{1}$, $\alpha_{2}$ and
$\alpha_{3}$ separately in a reverse order.

\paragraph{Controlling $\alpha_{3}$. }

We intend to apply Lemma~\ref{lemma:eps-net} in Section~\ref{sec:Auxiliary-lemma}
to control $\alpha_{3}$. Before proceeding, we pause to make a few
observations. It is straightforward to compute that $f^{\prime}(x)=\frac{k}{(\lambda-x)^{k+1}}$
for the function $f(x):=\frac{1}{(\lambda-x)^{k}}$ and $g^{\prime}(x)=\frac{\lambda+(k-1)x}{(\lambda-x)^{k+1}}$
for the function $g(x):=\frac{x}{(\lambda-x)^{k}}$. Since $|f(x)-f(y)|\leq\left\{ \sup_{z}|f'(z)|\right\} |x-y|$
for any function $f(\cdot)$, one can demonstrate that: for all $\lambda$
satisfying $\lambda/\big(1+\beta(\lambda)\big)\in\mathcal{B}_{\mathcal{E}_{\mathsf{PCA}}}(\lambda_{l}^{\star}+\sigma^{2})$,
we claim that the following holds \begin{subequations}
\begin{align}
\max_{1\leq i\leq n}\bigg|\frac{1}{\lambda-\gamma_{i}^{(l)}}-\frac{1}{\lambda-\sigma^{2}p/n}\bigg| & \lesssim\left\{ \max_{\gamma:\,|\gamma-\sigma^{2}p/n|\lesssim\sigma^{2}\sqrt{p/n}}\frac{1}{(\lambda-\gamma)^{2}}\right\} \cdot\max_{1\leq i\leq n}|\gamma_{i}^{(l)}-\sigma^{2}p/n|\nonumber \\
 & \lesssim\frac{\sigma^{2}}{\lambda_{l}^{\star2}}\sqrt{\frac{p}{n}},\label{eq:gamma-lambda-diff-UB}
\end{align}
\begin{align}
\max_{1\leq i\leq n}\bigg|\frac{1}{(\lambda-\gamma_{i}^{(l)})^{2}}-\frac{1}{(\lambda-\sigma^{2}p/n)^{2}}\bigg| & \lesssim\left\{ \max_{\gamma:\,|\gamma-\sigma^{2}p/n|\lesssim\sigma^{2}\sqrt{p/n}}\frac{1}{(\lambda-\gamma)^{3}}\right\} \cdot\max_{1\leq i\leq n}|\gamma_{i}^{(l)}-\sigma^{2}p/n|\nonumber \\
 & \lesssim\frac{\sigma^{2}}{\lambda_{l}^{\star3}}\sqrt{\frac{p}{n}},\label{eq:gamma-lambda-diff-two-UB}
\end{align}
\begin{align}
\max_{1\leq i\leq n}\bigg|\frac{\gamma_{i}^{(l)}}{(\lambda-\gamma_{i}^{(l)})^{2}}-\frac{\sigma^{2}p/n}{(\lambda-\sigma^{2}p/n)^{2}}\bigg| & \lesssim\left\{ \max_{\gamma:\,|\gamma-\sigma^{2}p/n|\lesssim\sigma^{2}\sqrt{p/n}}\frac{|\lambda+\gamma|}{|\lambda-\gamma|^{3}}\right\} \cdot\max_{1\leq i\leq n}\big|\gamma_{i}^{(l)}-\sigma^{2}p/n\big|\nonumber \\
 & \lesssim\frac{\lambda_{l}^{\star}}{\lambda_{l}^{\star3}}\cdot\sigma^{2}\sqrt{\frac{p}{n}}=\frac{\sigma^{2}}{\lambda_{l}^{\star2}}\sqrt{\frac{p}{n}},\label{eq:gamma-lambda-diff-sq-UB}
\end{align}
and
\begin{align}
\max_{1\leq i\leq n}\bigg|\frac{\gamma_{i}^{(l)}}{(\lambda-\gamma_{i}^{(l)})^{3}}-\frac{\sigma^{2}p/n}{(\lambda-\sigma^{2}p/n)^{3}}\bigg| & \lesssim\left\{ \max_{\gamma:\,|\gamma-\sigma^{2}p/n|\lesssim\sigma^{2}\sqrt{p/n}}\frac{|\lambda+2\gamma|}{(\lambda-\gamma)^{4}}\right\} \cdot\max_{1\leq i\leq n}\big|\gamma_{i}^{(l)}-\sigma^{2}p/n\big|\nonumber \\
 & \lesssim\frac{\lambda_{l}^{\star}}{\lambda_{l}^{\star4}}\cdot\sigma^{2}\sqrt{\frac{p}{n}}=\frac{\sigma^{2}}{\lambda_{l}^{\star3}}\sqrt{\frac{p}{n}}.\label{eq:gamma-lambda-diff-cub-UB}
\end{align}
\end{subequations}To justify the inequalities above, we have taken
advantage of the following conditions:
\begin{itemize}
\item It is seen from (\ref{eq:lambdal-tilde-lambdak-pca-UB-i>r}) in Lemma
\ref{lemma:lambda-S-minus-spectrum-1-1} that $|\gamma_{i}^{(l)}-\sigma^{2}p/n|\lesssim\sigma^{2}\sqrt{p/n}$.
\item We have used the condition that for any $\lambda,\gamma$ satisfying
$\lambda/\big(1+\beta(\lambda)\big)\in\mathcal{B}_{\mathcal{E}_{\mathsf{PCA}}}(\lambda_{l}^{\star}+\sigma^{2})$
and $|\gamma-\sigma^{2}p/n|\lesssim\sigma^{2}\sqrt{p/n}$, one has
\[
|\lambda-\gamma|\gtrsim\lambda_{l}^{\star};
\]
this can be established via almost the same argument for justifying
(\ref{eq:lambda-lambdal-pca-LB}) in Lemma~\ref{lemma:lambda-S-minus-spectrum-1-1}
(which we omit here for brevity).
\end{itemize}
With the preceding upper bounds in place, one can begin to verify
the conditions required to invoke Lemma~\ref{lemma:eps-net}. First,
one can derive: with probability at least $1-O(n^{-20})$, for all
$\lambda$ satisfying $\lambda/\big(1+\beta(\lambda)\big)\in\mathcal{B}_{\mathcal{E}_{\mathsf{PCA}}}(\lambda_{l}^{\star}+\sigma^{2})$,
\begin{align*}
 & \Bigg|\frac{\mathrm{d}}{\mathrm{d}\lambda}\sum_{r\le i\le n}\Big(\frac{\gamma_{i}^{(l)}}{(\lambda-\gamma_{i}^{(l)})^{2}}-\frac{\sigma^{2}p/n}{(\lambda-\sigma^{2}p/n)^{2}}\Big)\big((\bm{v}_{i}^{(l)\top}\bm{s}_{l,\parallel}^{\top})^{2}-(\lambda_{l}^{\star}+\sigma^{2})\big)\Bigg|\\
 & \qquad=\Bigg|\sum_{r\le i\le n}\Big(\frac{\gamma_{i}^{(l)}}{(\lambda-\gamma_{i}^{(l)})^{3}}-\frac{\sigma^{2}p/n}{(\lambda-\sigma^{2}p/n)^{3}}\Big)\big((\bm{v}_{i}^{(l)\top}\bm{s}_{l,\parallel}^{\top})^{2}-(\lambda_{l}^{\star}+\sigma^{2})\big)\Bigg|\\
 & \qquad\leq n\cdot\max_{1\leq i\leq n}\bigg|\frac{\gamma_{i}^{(l)}}{(\lambda-\gamma_{i}^{(l)})^{3}}-\frac{\sigma^{2}p/n}{(\lambda-\sigma^{2}p/n)^{3}}\bigg|\cdot\max_{1\leq i\leq n}\big|(\bm{v}_{i}^{(l)\top}\bm{s}_{l,\parallel}^{\top})^{2}-(\lambda_{l}^{\star}+\sigma^{2})\big|\\
 & \qquad\overset{}{\lesssim}n\cdot\frac{\sigma^{2}}{\lambda_{l}^{\star3}}\sqrt{\frac{p}{n}}\cdot(\lambda_{l}^{\star}+\sigma^{2})^{2}\log n\asymp\frac{\sigma^{2}\sqrt{pn}\log n}{\lambda_{l}^{\star}},
\end{align*}
where the last line holds due to the upper bound (\ref{eq:gamma-lambda-diff-sq-UB}),
the fact $\sigma^{2}\leq\sigma^{2}p/n\ll\lambda_{l}^{\star}$ (from
the noise assumption (\ref{eq:noise-condition-iid-pca})), as well
as the the standard Gaussian concentration inequality. In addition,
one can then apply the matrix Bernstein inequality \cite[Corollary 2.1]{Koltchinskii2011oracle}
to conclude: for any fixed $\lambda$ satisfying $\lambda/\big(1+\beta(\lambda)\big)\in\mathcal{B}_{\mathcal{E}_{\mathsf{PCA}}}(\lambda_{l}^{\star}+\sigma^{2})$,
with probability at least $1-O(n^{-10})$,
\begin{align*}
 & \Bigg|\sum_{r\le i\le n}\Big(\frac{\gamma_{i}^{(l)}}{(\lambda-\gamma_{i}^{(l)})^{2}}-\frac{\sigma^{2}p/n}{(\lambda-\sigma^{2}p/n)^{2}}\Big)\big((\bm{v}_{i}^{(l)\top}\bm{s}_{l,\parallel}^{\top})^{2}-(\lambda_{l}^{\star}+\sigma^{2})\big)\Bigg|\\
 & \qquad\lesssim\max_{1\leq i\leq n}\bigg|\frac{\gamma_{i}^{(l)}}{(\lambda-\gamma_{i}^{(l)})^{2}}-\frac{\sigma^{2}p/n}{(\lambda-\sigma^{2}p/n)^{2}}\bigg|\cdot(\lambda_{l}^{\star}+\sigma^{2})\cdot(\log^{2}n+\sqrt{n\log n})\\
 & \qquad\lesssim\frac{\sigma^{2}}{\lambda_{l}^{\star2}}\sqrt{\frac{p}{n}}\cdot\lambda_{l}^{\star}(\log^{2}n+\sqrt{n\log n})\asymp\frac{\sigma^{2}\sqrt{p\log n}}{\lambda_{l}^{\star}},
\end{align*}
where the second line arises from the the matrix Bernstein inequality,
and the last line follows from (\ref{eq:gamma-lambda-diff-cub-UB})
and the facts $\lambda_{l}^{\star}+\sigma^{2}\asymp\lambda_{l}^{\star}$
(given the assumption that $\sigma^{2}\leq\sigma^{2}p/n\ll\lambda_{l}^{\star}$).
Taking this together with the fact 
\[
\big\{\lambda\colon\lambda/\big(1+\beta(\lambda)\big)\in\mathcal{B}_{\mathcal{E}_{\mathsf{PCA}}}(\lambda_{l}^{\star}+\sigma^{2})\big\}\subseteq[2\lambda_{l}^{\star}/3,\,4\lambda_{l}^{\star}/3],
\]
we can apply Lemma~\ref{lemma:eps-net} to show that
\begin{equation}
|\alpha_{3}|\lesssim\frac{\sigma^{2}}{\lambda_{l}^{\star}}\sqrt{p\log n}\label{eq:PCA-alpha3}
\end{equation}
 with probability exceeding $1-O(n^{-10})$. 

\paragraph{Controlling $\alpha_{2}$. }

With regards to $\alpha_{2}$, we claim that the following upper bound
holds, whose proof is deferred to the end of this section.
\begin{equation}
\left|\sum_{r<i\le n}\Big(\frac{\lambda_{i}}{(\lambda_{l}-\lambda_{i})^{2}}-\frac{\sigma^{2}p/n}{(\lambda_{l}-\sigma^{2}p/n)^{2}}\Big)-\sum_{r\le i\le n}\Big(\frac{\gamma_{i}^{(l)}}{(\lambda_{l}-\gamma_{i}^{(l)})^{2}}-\frac{\sigma^{2}p/n}{(\lambda_{l}-\sigma^{2}p/n)^{2}}\Big)\right|\lesssim\frac{\sigma^{2}}{\lambda_{l}^{\star2}}\sqrt{\frac{p}{n}}.\label{eq:lambda-M-r-n-value-pca-p}
\end{equation}
Combing this with $\mathcal{E}_{\mathsf{PCA}}$ defined in (\ref{eq:def-E-PCA}),
we arrive at
\begin{align}
\alpha_{2} & =\sum_{r\le i\le n}\Big(\frac{\gamma_{i}^{(l)}}{(\lambda_{l}-\gamma_{i}^{(l)})^{2}}-\frac{\sigma^{2}p/n}{(\lambda_{l}-\sigma^{2}p/n)^{2}}\Big)(\lambda_{l}^{\star}+\sigma^{2})\nonumber \\
 & \overset{(\mathrm{i})}{=}\sum_{r\le i\le n}\Big(\frac{\gamma_{i}^{(l)}}{(\lambda_{l}-\gamma_{i}^{(l)})^{2}}-\frac{\sigma^{2}p/n}{(\lambda_{l}-\sigma^{2}p/n)^{2}}\Big)\bigg(\frac{\lambda_{l}}{1+\frac{1}{n}\sum_{r<i\leq p}\frac{\lambda_{i}}{\lambda_{l}-\lambda_{i}}}+O(\mathcal{E}_{\mathsf{PCA}})\bigg)\nonumber \\
 & \leq\sum_{r\le i\le n}\Big(\frac{\gamma_{i}^{(l)}}{(\lambda_{l}-\gamma_{i}^{(l)})^{2}}-\frac{\sigma^{2}p/n}{(\lambda_{l}-\sigma^{2}p/n)^{2}}\Big)\frac{\lambda_{l}}{1+\frac{1}{n}\sum_{r<i\leq p}\frac{\lambda_{i}}{\lambda_{l}-\lambda_{i}}}\nonumber \\
 & \quad+n\cdot\max_{1\leq i\leq n}\bigg|\frac{\gamma_{i}^{(l)}}{(\lambda-\gamma_{i}^{(l)})^{2}}-\frac{\sigma^{2}p/n}{(\lambda-\sigma^{2}p/n)^{2}}\bigg|\cdot O(\mathcal{E}_{\mathsf{PCA}})\nonumber \\
 & \overset{(\mathrm{ii})}{=}\Bigg(\sum_{r<i\le n}\Big(\frac{\lambda_{i}}{(\lambda_{l}-\lambda_{i})^{2}}-\frac{\sigma^{2}p/n}{(\lambda_{l}-\sigma^{2}p/n)^{2}}\Big)+O\bigg(\frac{\sigma^{2}}{\lambda_{l}^{\star2}}\sqrt{\frac{p}{n}}\bigg)\Bigg)\frac{\lambda_{l}}{1+\frac{1}{n}\sum_{r<i\leq p}\frac{\lambda_{i}}{\lambda_{l}-\lambda_{i}}}\nonumber \\
 & \quad+n\cdot\max_{1\leq i\leq n}\bigg|\frac{\gamma_{i}^{(l)}}{(\lambda-\gamma_{i}^{(l)})^{2}}-\frac{\sigma^{2}p/n}{(\lambda-\sigma^{2}p/n)^{2}}\bigg|\cdot O(\mathcal{E}_{\mathsf{PCA}})\nonumber \\
 & =\sum_{r\le i\le n}\Big(\frac{\lambda_{i}}{(\lambda_{l}-\lambda_{i})^{2}}-\frac{\sigma^{2}p/n}{(\lambda_{l}-\sigma^{2}p/n)^{2}}\Big)\frac{\lambda_{l}}{1+\frac{1}{n}\sum_{r<i\leq p}\frac{\lambda_{i}}{\lambda_{l}-\lambda_{i}}}\nonumber \\
 & \quad+O\bigg(\frac{\sigma^{2}}{\lambda_{l}^{\star2}}\sqrt{\frac{p}{n}}\bigg)\cdot\frac{\lambda_{l}}{1+\frac{1}{n}\sum_{r<i\leq p}\frac{\lambda_{i}}{\lambda_{l}-\lambda_{i}}}+n\cdot\max_{1\leq i\leq n}\bigg|\frac{\gamma_{i}^{(l)}}{(\lambda-\gamma_{i}^{(l)})^{2}}-\frac{\sigma^{2}p/n}{(\lambda-\sigma^{2}p/n)^{2}}\bigg|\cdot O(\mathcal{E}_{\mathsf{PCA}})\nonumber \\
 & \overset{(\mathrm{iii})}{=}\sum_{r\le i\le n}\Big(\frac{\lambda_{i}}{(\lambda_{l}-\lambda_{i})^{2}}-\frac{\sigma^{2}p/n}{(\lambda_{l}-\sigma^{2}p/n)^{2}}\Big)\frac{\lambda_{l}}{1+\frac{1}{n}\sum_{r<i\leq p}\frac{\lambda_{i}}{\lambda_{l}-\lambda_{i}}}\nonumber \\
 & \quad+O\bigg(\frac{\sigma^{2}}{\lambda_{l}^{\star2}}\sqrt{\frac{p}{n}}\cdot\lambda_{l}^{\star}\bigg)+O\bigg(\frac{\sigma^{2}}{\lambda_{l}^{\star2}}\sqrt{\frac{p}{n}}\cdot n\cdot(\lambda_{\max}^{\star}+\sigma^{2})\sqrt{\frac{r\log n}{n}}\bigg)\nonumber \\
 & \overset{(\mathrm{iv})}{=}\sum_{r\le i\le n}\Big(\frac{\lambda_{i}}{(\lambda_{l}-\lambda_{i})^{2}}-\frac{\sigma^{2}p/n}{(\lambda_{l}-\sigma^{2}p/n)^{2}}\Big)\frac{\lambda_{l}}{1+\frac{1}{n}\sum_{r<i\leq p}\frac{\lambda_{i}}{\lambda_{l}-\lambda_{i}}}+O\Big(\frac{\sigma^{2}}{\lambda_{l}^{\star}}\kappa\sqrt{pr\log n}\Big).\label{eq:PCA-alpha2}
\end{align}
Here, (i) arises from (\ref{claim:lambda-sigma-est}); (ii) is due
to the claim (\ref{eq:lambda-M-r-n-value-pca-p}); (iii) follows from
(\ref{claim:lambda-sigma-est-bound}), (\ref{eq:gamma-lambda-diff-sq-UB})
and the definition of $\mathcal{E}_{\mathsf{PCA}}$; (iv) holds true
under the condition $\sigma^{2}\ll\lambda_{\max}^{\star}$ (see (\ref{eq:sigma-Epca-UB-p>n})).

\paragraph{Controlling $\alpha_{1}$. }

Regarding $\alpha_{1}$, the key step lies in controlling $\sum_{r\le i\le n}(\bm{v}_{i}^{(l)\top}\bm{s}_{l,\parallel}^{\top})^{2}$.
Recall that $\bm{U}^{(l)}\sqrt{\bm{\Gamma}^{(l)}}\bm{V}^{(l)\top}$
is the SVD of $\frac{1}{\sqrt{n}}\bm{S}_{l,\perp}$ with $\bm{V}^{(l)}\coloneqq[\bm{v}_{1}^{(l)},\cdots,\bm{v}_{n}^{(l)}]\in\mathbb{R}^{n\times n}$.
Towards this, we invoke Theorem \ref{thm:master-thm-vector} to derive
the following identity:
\begin{align*}
\lambda_{l} & \overset{(\mathrm{i})}{=}\frac{1}{n}\bm{u}_{l}^{\star\top}\bm{S}\bm{S}^{\top}\bm{u}_{l}^{\star}+\frac{1}{n}\bm{u}_{l}^{\star\top}\bm{S}\bm{S}^{\top}\bm{u}_{l}^{\star\perp}\Big(\lambda_{l}\bm{I}_{p-1}-\frac{1}{n}(\bm{u}_{l}^{\star\perp})^{\top}\bm{S}\bm{S}^{\top}\bm{u}_{l}^{\star\perp}\Big)^{-1}\frac{1}{n}(\bm{u}_{l}^{\star\perp})^{\top}\bm{S}\bm{S}^{\top}\bm{u}_{l}^{\star\top}\\
 & \overset{(\mathrm{ii})}{=}\frac{1}{n}\|\bm{s}_{l,\parallel}^{\top}\|_{2}^{2}+\frac{1}{n}\bm{s}_{l,\parallel}^{\top}\bm{S}_{l,\perp}^{\top}\Big(\lambda_{l}\bm{I}_{p-1}-\frac{1}{n}\bm{S}_{l,\perp}\bm{S}_{l,\perp}^{\top}\Big)^{-1}\frac{1}{n}\bm{S}_{l,\perp}\bm{s}_{l,\parallel}^{\top}\\
 & \overset{(\mathrm{iii})}{=}\frac{1}{n}\sum_{1\leq i\le n}(\bm{v}_{i}^{(l)\top}\bm{s}_{l,\parallel}^{\top})^{2}+\frac{1}{n}\sum_{1\leq i\le n}\frac{\gamma_{i}^{(l)}}{\lambda_{l}-\gamma_{i}^{(l)}}(\bm{v}_{i}^{(l)\top}\bm{s}_{l,\parallel}^{\top})^{2}\\
 & =\frac{1}{n}\sum_{1\leq i\le n}(\bm{v}_{i}^{(l)\top}\bm{s}_{l,\parallel}^{\top})^{2}\bigg(1+\frac{\gamma_{i}^{(l)}}{\lambda_{l}-\gamma_{i}^{(l)}}\bigg)\\
 & =\frac{\lambda_{l}}{n}\sum_{1\leq i\le n}\frac{(\bm{v}_{i}^{(l)\top}\bm{s}_{l,\parallel}^{\top})^{2}}{\lambda_{l}-\gamma_{i}^{(l)}},
\end{align*}
where (i) arises from (\ref{eq:lambda-rank1}); (ii) relies on the
definitions of $\bm{s}_{l,\parallel}$ and $\bm{S}_{l,\perp}$ in
(\ref{claim:lambda-sigma-est}); (iii) follows since $\{\bm{v}_{i}^{(l)}\}_{1\le i\le n}$
forms a set of orthonormal bases in $\mathbb{R}^{n}$. Rearranging
terms, we are left with
\begin{align*}
n & =\sum_{1\leq i\leq n}\frac{(\bm{v}_{i}^{(l)\top}\bm{s}_{l,\parallel}^{\top})^{2}}{\lambda_{l}-\gamma_{i}^{(l)}}=\sum_{1\leq i<r}\frac{(\bm{v}_{i}^{(l)\top}\bm{s}_{l,\parallel}^{\top})^{2}}{\lambda_{l}-\gamma_{i}^{(l)}}+\sum_{r\leq i\le n}\frac{(\bm{v}_{i}^{(l)\top}\bm{s}_{l,\parallel}^{\top})^{2}}{\lambda_{l}-\gamma_{i}^{(l)}}\\
 & =\sum_{1\leq i<r}\frac{(\bm{v}_{i}^{(l)\top}\bm{s}_{l,\parallel}^{\top})^{2}}{\lambda_{l}-\gamma_{i}^{(l)}}+\sum_{r\le i\le n}\frac{(\bm{v}_{i}^{(l)\top}\bm{s}_{l,\parallel}^{\top})^{2}}{\lambda_{l}-\sigma^{2}p/n}+\sum_{r\le i\le n}\Big(\frac{1}{\lambda_{l}-\gamma_{i}^{(l)}}-\frac{1}{\lambda_{l}-\sigma^{2}p/n}\Big)(\bm{v}_{i}^{(l)\top}\bm{s}_{l,\parallel}^{\top})^{2}.
\end{align*}
As a result, we obtain the following decomposition:
\begin{align}
\sum_{r\le i\le n}\frac{(\bm{v}_{i}^{(l)\top}\bm{s}_{l,\parallel}^{\top})^{2}}{\lambda_{l}-\sigma^{2}p/n} & =n-\underbrace{\sum_{1\leq i<r}\frac{(\bm{v}_{i}^{(l)\top}\bm{s}_{l,\parallel}^{\top})^{2}}{\lambda_{l}-\gamma_{i}^{(l)}}}_{=:\,\varphi_{1}}-\underbrace{(\lambda_{l}^{\star}+\sigma^{2})\sum_{r\le i\le n}\Big(\frac{1}{\lambda_{l}-\gamma_{i}^{(l)}}-\frac{1}{\lambda_{l}-\sigma^{2}p/n}\Big)}_{=:\,\varphi_{2}}\nonumber \\
 & \quad-\underbrace{\sum_{r\le i\le n}\Big(\frac{1}{\lambda_{l}-\gamma_{i}^{(l)}}-\frac{1}{\lambda_{l}-\sigma^{2}p/n}\Big)\big((\bm{v}_{i}^{(l)\top}\bm{s}_{l,\parallel}^{\top})^{2}-(\lambda_{l}^{\star}+\sigma^{2})\big)}_{=:\,\varphi_{3}}.\label{eq:PCA-alpha1-decomp}
\end{align}
In what follows, we shall control $\varphi_{1}$, $\varphi_{2}$ and
$\varphi_{3}$ separately.
\begin{itemize}
\item We start with $\varphi_{1}$. Given that $\bm{v}_{i}^{(l)\top}\bm{s}_{l,\parallel}^{\top}\overset{\mathrm{i.i.d.}}{\sim}\mathcal{N}(0,\lambda_{l}^{\star}+\sigma^{2})$,
we can develop an upper bound as follows: with probability at least
$1-O(n^{-10})$,
\begin{equation}
|\varphi_{1}|\overset{(\mathrm{i})}{\leq}\frac{\sum_{1\leq i<r}(\bm{v}_{i}^{(l)\top}\bm{s}_{l,\parallel}^{\top})^{2}}{\min_{i:i\ne l}|\lambda_{l}^{\star}-\lambda_{i}^{\star}|}\overset{(\mathrm{ii})}{\lesssim}\frac{(\lambda_{l}^{\star}+\sigma^{2})r\log n}{\min_{i:i\ne l}|\lambda_{l}^{\star}-\lambda_{i}^{\star}|}\overset{(\mathrm{iii})}{\asymp}\frac{\lambda_{l}^{\star}r\log n}{\min_{i:i\ne l}|\lambda_{l}^{\star}-\lambda_{i}^{\star}|}.\label{eq:PCA-phi1}
\end{equation}
Here, (i) utilizes the condition $|\lambda_{l}-\gamma_{i}^{(l)}|\gtrsim\min_{i:i\ne l}\left|\lambda_{l}^{\star}-\lambda_{i}^{\star}\right|$
(in view of Lemma \ref{lemma:lambda-S-minus-spectrum-1-1}), (ii)
holds due to (\ref{eq:vl-slpara-UB}), whereas (iii) arises from the
noise assumption $\sigma^{2}\lesssim\lambda_{l}^{\star}$.
\item As for $\varphi_{2}$, recall from Lemma \ref{lemma:eigval-interlacing}
that $\lambda_{i+1}\le\gamma_{i}^{(l)}\le\lambda_{i}$ for all $1\leq i<p$.
This in turn leads to
\begin{align*}
\sum_{r\le i\le n}\Big(\frac{1}{\lambda_{l}-\gamma_{i}^{(l)}}-\frac{1}{\lambda_{l}-\sigma^{2}p/n}\Big)-\sum_{r<i\le n}\Big(\frac{1}{\lambda_{l}-\lambda_{i}}-\frac{1}{\lambda_{l}-\sigma^{2}p/n}\Big) & \geq\frac{1}{\lambda_{l}-\gamma_{n}^{(l)}}-\frac{1}{\lambda_{l}-\sigma^{2}p/n};\\
\sum_{r\le i\le n}\Big(\frac{1}{\lambda_{l}-\gamma_{i}^{(l)}}-\frac{1}{\lambda_{l}-\sigma^{2}p/n}\Big)-\sum_{r<i\le n}\Big(\frac{1}{\lambda_{l}-\lambda_{i}}-\frac{1}{\lambda_{l}-\sigma^{2}p/n}\Big) & \leq\frac{1}{\lambda_{l}-\gamma_{r}^{(l)}}-\frac{1}{\lambda_{l}-\sigma^{2}p/n}.
\end{align*}
This taken collectively with (\ref{eq:gamma-lambda-diff-UB}) yields
\begin{equation}
\Bigg|\sum_{r\le i\le n}\Big(\frac{1}{\lambda_{l}-\gamma_{i}^{(l)}}-\frac{1}{\lambda_{l}-\sigma^{2}p/n}\Big)-\sum_{r<i\le n}\Big(\frac{1}{\lambda_{l}-\lambda_{i}}-\frac{1}{\lambda_{l}-\sigma^{2}p/n}\Big)\Bigg|\lesssim\frac{\sigma^{2}}{\lambda_{l}^{\star2}}\sqrt{\frac{p}{n}}.\label{eq:PCA-phi2-ineq1}
\end{equation}
In addition, it is also seen from (\ref{eq:gamma-lambda-diff-UB})
that 
\begin{align}
 & \Bigg|\sum_{r\le i\le n}\Big(\frac{1}{\lambda_{l}-\gamma_{i}^{(l)}}-\frac{1}{\lambda_{l}-\sigma^{2}p/n}\Big)\Bigg|\vee\Bigg|\sum_{r<i\le n}\Big(\frac{1}{\lambda_{l}-\lambda_{i}}-\frac{1}{\lambda_{l}-\sigma^{2}p/n}\Big)\Bigg|\lesssim\frac{\sigma^{2}}{\lambda_{l}^{\star2}}\sqrt{pn}.\label{eq:PCA-phi2-ineq2}
\end{align}
Therefore, we can obtain
\begin{align*}
 & \Bigg|(\lambda_{l}^{\star}+\sigma^{2})\sum_{r\le i\le n}\Big(\frac{1}{\lambda_{l}-\gamma_{i}^{(l)}}-\frac{1}{\lambda_{l}-\sigma^{2}p/n}\Big)-\frac{\lambda_{l}}{1+\frac{1}{n}\sum_{r<i\leq p}\frac{\lambda_{i}}{\lambda_{l}-\lambda_{i}}}\sum_{r<i\le n}\Big(\frac{1}{\lambda_{l}-\lambda_{i}}-\frac{1}{\lambda_{l}-\sigma^{2}p/n}\Big)\Bigg|\\
 & \qquad\leq\Bigg|(\lambda_{l}^{\star}+\sigma^{2})-\frac{\lambda_{l}}{1+\frac{1}{n}\sum_{r<i\leq p}\frac{\lambda_{i}}{\lambda_{l}-\lambda_{i}}}\Bigg|\cdot\Bigg|\sum_{r\le i\le n}\Big(\frac{1}{\lambda_{l}-\gamma_{i}^{(l)}}-\frac{1}{\lambda_{l}-\sigma^{2}p/n}\Big)\Bigg|\\
 & \qquad\quad+\frac{\lambda_{l}}{1+\frac{1}{n}\sum_{r<i\leq p}\frac{\lambda_{i}}{\lambda_{l}-\lambda_{i}}}\cdot\Bigg|\sum_{r\le i\le n}\Big(\frac{1}{\lambda_{l}-\gamma_{i}^{(l)}}-\frac{1}{\lambda_{l}-\sigma^{2}p/n}\Big)-\sum_{r<i\le n}\Big(\frac{1}{\lambda_{l}-\lambda_{i}}-\frac{1}{\lambda_{l}-\sigma^{2}p/n}\Big)\Bigg|\\
 & \qquad\lesssim\lambda_{\max}^{\star}\sqrt{\frac{r\log n}{n}}\cdot\frac{\sigma^{2}}{\lambda_{l}^{\star2}}\sqrt{pn}+\lambda_{l}^{\star}\cdot\frac{\sigma^{2}}{\lambda_{l}^{\star2}}\sqrt{\frac{p}{n}}\asymp\frac{\sigma^{2}}{\lambda_{l}^{\star}}\kappa\sqrt{pr\log n},
\end{align*}
where the last step uses (\ref{claim:lambda-sigma-est}), (\ref{claim:lambda-sigma-est-bound}),
(\ref{eq:PCA-phi2-ineq1}) and (\ref{eq:PCA-phi2-ineq2}). This reveals
that
\begin{equation}
\varphi_{2}=\frac{\lambda_{l}}{1+\frac{1}{n}\sum_{r<i\leq p}\frac{\lambda_{i}}{\lambda_{l}-\lambda_{i}}}\sum_{r<i\le n}\Big(\frac{1}{\lambda_{l}-\lambda_{i}}-\frac{1}{\lambda_{l}-\sigma^{2}p/n}\Big)+O\Big(\frac{\sigma^{2}}{\lambda_{l}^{\star}}\kappa\sqrt{pr\log n}\Big).\label{eq:PCA-phi2}
\end{equation}
\item Turning to $\varphi_{3}$, we shall apply Lemma~\ref{lemma:eps-net}
to bound it. Similar to the analysis above for bounding $\alpha_{3}$,
one can check that the following holds with probability at least $1-O(n^{-10})$:
for all $\lambda$ satisfying $\lambda/\big(1+\beta(\lambda)\big)\in\mathcal{B}_{\mathcal{E}_{\mathsf{PCA}}}(\lambda_{l}^{\star}+\sigma^{2})$,
\begin{align*}
 & \Bigg|\frac{\mathrm{d}}{\mathrm{d}\lambda}\sum_{r\le i\le n}\Big(\frac{1}{\lambda-\gamma_{i}^{(l)}}-\frac{1}{\lambda-\sigma^{2}p/n}\Big)\big((\bm{v}_{i}^{(l)\top}\bm{s}_{l,\parallel}^{\top})^{2}-(\lambda_{l}^{\star}+\sigma^{2})\big)\Bigg|\\
 & \qquad=\Bigg|\sum_{r\le i\le n}\Big(\frac{1}{(\lambda-\gamma_{i}^{(l)})^{2}}-\frac{1}{(\lambda-\sigma^{2}p/n)^{2}}\Big)\big((\bm{v}_{i}^{(l)\top}\bm{s}_{l,\parallel}^{\top})^{2}-(\lambda_{l}^{\star}+\sigma^{2})\big)\Bigg|\\
 & \qquad\lesssim n\cdot\max_{r\le i\le n}\Bigg|\frac{1}{(\lambda-\gamma_{i}^{(l)})^{2}}-\frac{1}{(\lambda-\sigma^{2}p/n)^{2}}\Bigg|\cdot\max_{r\leq i\leq n}\big|(\bm{v}_{i}^{(l)\top}\bm{s}_{l,\parallel}^{\top})^{2}-(\lambda_{l}^{\star}+\sigma^{2})\big|\\
 & \qquad\lesssim n\cdot\frac{\sigma^{2}}{\lambda_{l}^{\star3}}\sqrt{\frac{p}{n}}\cdot(\lambda_{l}^{\star}+\sigma^{2})\log n\asymp\frac{\sigma^{2}}{\lambda_{l}^{\star2}}\sqrt{pn}\log n,
\end{align*}
where the last line comes from (\ref{eq:gamma-lambda-diff-two-UB}).
In addition, for any fixed $\lambda$ such that $\lambda/\big(1+\beta(\lambda)\big)\in\mathcal{B}_{\mathcal{E}_{\mathsf{PCA}}}(\lambda_{l}^{\star}+\sigma^{2})$,
we can use the matrix Bernstein inequality \cite[Corollary 2.1]{Koltchinskii2011oracle}
to demonstrate that: with probability at least $1-O(n^{-10})$,
\begin{align*}
 & \Bigg|\sum_{r\leq i\le n}\Big(\frac{1}{\lambda-\gamma_{i}^{(l)}}-\frac{1}{\lambda-\sigma^{2}p/n}\Big)\big((\bm{v}_{i}^{(l)\top}\bm{s}_{l,\parallel}^{\top})^{2}-(\lambda_{l}^{\star}+\sigma^{2})\big)\Bigg|\\
 & \qquad\lesssim\max_{r\le i\le n}\Big|\frac{1}{\lambda-\gamma_{i}^{(l)}}-\frac{1}{\lambda-\sigma^{2}p/n}\Big|\cdot(\lambda_{l}^{\star}+\sigma^{2})(\log^{2}n+\sqrt{n\log n})\\
 & \qquad\lesssim\frac{\sigma^{2}}{\lambda_{l}^{\star}}\sqrt{p\log n},
\end{align*}
where the last line comes from (\ref{eq:gamma-lambda-diff-UB}). With
these in place, we invoke Lemma~\ref{lemma:eps-net} to conclude
that
\begin{equation}
|\varphi_{3}|\lesssim\frac{\sigma^{2}}{\lambda_{l}^{\star}}\sqrt{p\log n}\label{eq:PCA-phi3}
\end{equation}
with probability at least $1-O(n^{-10})$. 
\item Substituting (\ref{eq:PCA-phi1}), (\ref{eq:PCA-phi2}) and (\ref{eq:PCA-phi3})
into (\ref{eq:PCA-alpha1-decomp}) reveals that: with probability
exceeding $1-O(n^{-10})$,
\begin{align}
\sum_{r\le i\le n}\frac{(\bm{v}_{i}^{(l)\top}\bm{s}_{l,\parallel}^{\top})^{2}}{\lambda_{l}-\sigma^{2}p/n} & =n-\frac{\lambda_{l}}{1+\frac{1}{n}\sum_{r<i\leq n}\frac{\lambda_{i}}{\lambda_{l}-\lambda_{i}}}\sum_{r<i\le n}\Big(\frac{1}{\lambda_{l}-\lambda_{i}}-\frac{1}{\lambda_{l}-\sigma^{2}p/n}\Big)\nonumber \\
 & \qquad+O\Big(\frac{\sigma^{2}\kappa\sqrt{pr\log n}}{\lambda_{l}^{\star}}+\frac{\lambda_{l}^{\star}r\log n}{\min_{i:i\ne l}\left|\lambda_{l}^{\star}-\lambda_{i}^{\star}\right|}\Big).
\end{align}
As a consequence, we arrive at
\begin{align}
\alpha_{1} & =\frac{\sigma^{2}p/n}{\lambda_{l}-\sigma^{2}p/n}\sum_{r\le i\le n}\frac{(\bm{v}_{i}^{(l)\top}\bm{s}_{l,\parallel}^{\top})^{2}}{\lambda_{l}-\sigma^{2}p/n}\nonumber \\
 & =\frac{\sigma^{2}p/n}{\lambda_{l}-\sigma^{2}p/n}\bigg(n-\frac{\lambda_{l}}{1+\frac{1}{n}\sum_{r<i\leq n}\frac{\lambda_{i}}{\lambda_{l}-\lambda_{i}}}\sum_{r<i\le n}\Big(\frac{1}{\lambda_{l}-\lambda_{i}}-\frac{1}{\lambda_{l}-\sigma^{2}p/n}\Big)\bigg)\nonumber \\
 & \qquad+o\Big(\frac{\sigma^{2}}{\lambda_{l}^{\star}}\kappa\sqrt{pr\log n}\Big)+O\Big(\frac{\sigma^{2}pr\log n}{\min_{i:i\ne l}\left|\lambda_{l}^{\star}-\lambda_{i}^{\star}\right|n}\Big),\label{eq:PCA-alpha1}
\end{align}
where we have made use of the bound $\lambda_{l}-\sigma^{2}p/n\gtrsim\lambda_{l}^{\star}$
and $\sigma^{2}p/n=o(\lambda_{l}^{\star}).$ 
\end{itemize}

\paragraph{Combining the bounds on $\alpha_{1}$, $\alpha_{2}$ and $\alpha_{3}$.}

Putting (\ref{eq:lambda-M-r-n-value-pca-p-decomp}), (\ref{eq:PCA-alpha3}),
(\ref{eq:PCA-alpha2}) and (\ref{eq:PCA-alpha1}) together, we conclude
\begin{align*}
\sum_{r\le i\le n}\frac{\gamma_{i}^{(l)}(\bm{v}_{i}^{(l)\top}\bm{s}_{l,\parallel}^{\top})^{2}}{(\lambda_{l}-\gamma_{i}^{(l)})^{2}} & =\alpha_{1}+\alpha_{2}+\alpha_{3}\\
 & =\frac{\sigma^{2}p/n}{\lambda_{l}-\sigma^{2}p/n}\bigg(n-\frac{\lambda_{l}}{1+\frac{1}{n}\sum_{r<i\leq n}\frac{\lambda_{i}}{\lambda_{l}-\lambda_{i}}}\sum_{r<i\le n}\Big(\frac{1}{\lambda_{l}-\lambda_{i}}-\frac{1}{\lambda_{l}-\sigma^{2}p/n}\Big)\bigg)\\
 & \quad\quad+\frac{\lambda_{l}}{1+\frac{1}{n}\sum_{r<i\leq n}\frac{\lambda_{i}}{\lambda_{l}-\lambda_{i}}}\sum_{r\le i\le n}\Big(\frac{\lambda_{i}}{(\lambda_{l}-\lambda_{i})^{2}}-\frac{\sigma^{2}p/n}{(\lambda_{l}-\sigma^{2}p/n)^{2}}\Big)\\
 & \quad\quad+O\Big(\frac{\sigma^{2}}{\lambda_{l}^{\star}}\kappa\sqrt{pr\log n}+\frac{\sigma^{2}pr\log n}{\min_{i:i\ne l}\left|\lambda_{l}^{\star}-\lambda_{i}^{\star}\right|n}\Big)\\
 & =\frac{\sigma^{2}p}{\lambda_{l}-\sigma^{2}p/n}+\frac{\lambda_{l}}{1+\frac{1}{n}\sum_{r<i\leq n}\frac{\lambda_{i}}{\lambda_{l}-\lambda_{i}}}\sum_{r\le i\le n}\Big(\frac{\lambda_{i}}{(\lambda_{l}-\lambda_{i})^{2}}-\frac{\sigma^{2}p/n}{(\lambda_{l}-\lambda_{i})(\lambda_{l}-\sigma^{2}p/n)}\Big)\\
 & \quad+O\Big(\frac{\sigma^{2}}{\lambda_{l}^{\star}}\kappa\sqrt{pr\log n}+\frac{\sigma^{2}pr\log n}{\min_{i:i\ne l}\left|\lambda_{l}^{\star}-\lambda_{i}^{\star}\right|n}\Big)
\end{align*}
as claimed.

\paragraph{Proof of the inequality (\ref{eq:lambda-M-r-n-value-pca-p}).}

Given that $\lambda_{i+1}\le\gamma_{i}^{(l)}\le\lambda_{i}$ for all
$1\leq i<p$, we can bound
\[
\sum_{r<i\le n}\frac{\lambda_{i}}{(\lambda_{l}-\lambda_{i})^{2}}+\frac{\gamma_{n}^{(l)}}{(\lambda_{l}-\gamma_{n}^{(l)})^{2}}\le\sum_{r\le i\le n}\frac{\gamma_{i}^{(l)}}{(\lambda_{l}-\gamma_{i}^{(l)})^{2}}\le\sum_{r<i\le n}\frac{\lambda_{i}}{(\lambda_{l}-\lambda_{i})^{2}}+\frac{\gamma_{r}^{(l)}}{(\lambda_{l}-\gamma_{r}^{(l)})^{2}}.
\]
By subtracting $\sum_{r\le i\le n}\frac{\sigma^{2}p/n}{(\lambda_{l}-\sigma^{2}p/n)^{2}}$
from both sides and rearranging terms, we have
\begin{align*}
 & \Bigg|\sum_{r\le i\le n}\Big(\frac{\gamma_{i}^{(l)}}{(\lambda_{l}-\gamma_{i}^{(l)})^{2}}-\frac{\sigma^{2}p/n}{(\lambda_{l}-\sigma^{2}p/n)^{2}}\Big)-\sum_{r<i\le n}\Big(\frac{\lambda_{i}}{(\lambda_{l}-\lambda_{i})^{2}}-\frac{\sigma^{2}p/n}{(\lambda_{l}-\sigma^{2}p/n)^{2}}\Big)\Bigg|\\
 & \qquad\le\Bigg|\frac{\gamma_{n}^{(l)}}{(\lambda_{l}-\gamma_{n}^{(l)})^{2}}-\frac{\sigma^{2}p/n}{(\lambda_{l}-\sigma^{2}p/n)^{2}}\Bigg|\vee\Bigg|\frac{\gamma_{r}^{(l)}}{(\lambda_{l}-\gamma_{r}^{(l)})^{2}}-\frac{\sigma^{2}p/n}{(\lambda_{l}-\sigma^{2}p/n)^{2}}\Bigg|.
\end{align*}
In view of the basic property $|f(x)-f(y)|\leq\left\{ \sup_{z}|f'(z)|\right\} |x-y|$,
we can upper bound
\begin{align*}
\Bigg|\frac{\gamma_{i}^{(l)}}{(\lambda_{l}-\gamma_{i}^{(l)})^{2}}-\frac{\sigma^{2}p/n}{(\lambda_{l}-\sigma^{2}p/n)^{2}}\Bigg| & \le\max_{\gamma:\,|\gamma-\sigma^{2}p/n|\lesssim\sigma^{2}\sqrt{p/n}}\bigg|\frac{\lambda_{l}+\gamma}{(\lambda_{l}-\gamma)^{3}}\bigg|\cdot\big|\gamma_{i}^{(l)}-\sigma^{2}p/n\big|\\
 & \lesssim\frac{\lambda_{l}^{\star}}{\lambda_{l}^{\star3}}\sigma^{2}\sqrt{\frac{p}{n}}=\frac{\sigma^{2}}{\lambda_{l}^{\star2}}\sqrt{\frac{p}{n}}
\end{align*}
for any $r\le i\le n$. Here, the last line holds because (i) $|\lambda_{l}-\lambda_{l}^{\star}|\vee\sigma^{2}(p/n+\sqrt{p/n})\ll\lambda_{l}^{\star}$
holds due to the assumption (\ref{eq:noise-condition-iid-pca}), and
hence $|\lambda_{l}-\gamma|\gtrsim\lambda_{l}^{\star}$ and $\lambda_{l}+\gamma\lesssim\lambda_{l}^{\star}$;
(ii) $|\gamma_{i}^{(l)}-\sigma^{2}p/n|\lesssim\sigma^{2}\sqrt{p/n}$
holds according to Lemma \ref{lemma:lambda-S-minus-spectrum-1-1}.
This finishes the proof for the inequality (\ref{eq:lambda-M-r-n-value-pca-p}).

\subsection{Proof of Lemma \ref{lemma:a-top-P-Uk-1-pca}}

\label{subsec:Proof-of-lemma:a-top-P-Uk-1-pca}

Our proof strategy is to utilize the Gaussian concentration inequality
and the epsilon-net argument. 

To apply Lemma~\ref{lemma:eps-net}, we shall first check its conditions.
To begin with, we claim that the following holds with probability
at least $1-O(n^{-20})$:
\begin{align}
V & :=\sup_{\lambda:\frac{\lambda}{1+\beta(\lambda)}\in\mathcal{B}_{\mathcal{E}_{\mathsf{PCA}}}(\lambda_{l}^{\star}+\sigma^{2})}\bigg\|\sum_{k:k\ne l}\bm{a}^{\top}\bm{u}_{k}^{\star}\bm{u}_{k}^{\star(l)\top}\Big(\lambda\bm{I}_{p-1}-\frac{1}{n}\bm{S}_{l,\perp}\bm{S}_{l,\perp}^{\top}\Big)^{-1}\frac{1}{n}\bm{S}_{l,\perp}\bigg\|_{2}\nonumber \\
 & \,\lesssim\sum_{k:k\ne l}\frac{\left|\bm{a}^{\top}\bm{u}_{k}^{\star}\right|}{\left|\lambda_{l}^{\star}-\lambda_{k}^{\star}\right|}\sqrt{\frac{\left(\lambda_{\max}^{\star}+\sigma^{2}\right)(\kappa^{2}+r)}{n}}.\label{eq:a-top-P-Uk-1-pca-UB}
\end{align}
Consequently, we can apply the Gaussian concentration inequality to
show that: for any fixed $\lambda$ such that $\lambda/(1+\beta(\lambda))\in\mathcal{B}_{\mathcal{E}_{\mathsf{PCA}}}(\lambda_{l}^{\star}+\sigma^{2})$,
one has
\begin{align*}
 & \Big|\sum_{k:k\ne l}\bm{a}^{\top}\bm{u}_{k}^{\star}\bm{u}_{k}^{\star(l)\top}\Big(\lambda\bm{I}_{p-1}-\frac{1}{n}\bm{S}_{l,\perp}\bm{S}_{l,\perp}^{\top}\Big)^{-1}\frac{1}{n}\bm{S}_{l,\perp}\cdot\bm{s}_{l,\parallel}^{\top}\Big|\lesssim\sqrt{(\lambda_{l}^{\star}+\sigma^{2})\log\bigg(\frac{n\kappa\lambda_{\max}}{\Delta_{l}^{\star}}\bigg)}\cdot V
\end{align*}
with probability exceeding $1-O\big(\kappa^{-10}(\lambda_{\max}/\Delta_{l}^{\star})^{-20}n^{-20}\big)$. 

In addition, one can derive: for all $\lambda$ such that $\lambda/(1+\beta(\lambda))\in\mathcal{B}_{\mathcal{E}_{\mathsf{PCA}}}(\lambda_{l}^{\star}+\sigma^{2})$, 

\begin{align*}
 & \bigg|\frac{\mathrm{d}}{\mathrm{d}\lambda}\sum_{k:k\ne l}\bm{a}^{\top}\bm{u}_{k}^{\star}\cdot\bm{u}_{k}^{\star\top}\bm{u}_{l}^{\star\perp}\Big(\lambda\bm{I}_{p-1}-\frac{1}{n}\bm{S}_{l,\perp}\bm{S}_{l,\perp}^{\top}\Big)^{-1}\frac{1}{n}\bm{S}_{l,\perp}\cdot\bm{s}_{l,\parallel}^{\top}\bigg|\\
 & \qquad=\bigg|\sum_{k:k\ne l}\bm{a}^{\top}\bm{u}_{k}^{\star}\cdot\bm{u}_{k}^{\star\top}\bm{u}_{l}^{\star\perp}\Big(\lambda\bm{I}_{p-1}-\frac{1}{n}\bm{S}_{l,\perp}\bm{S}_{l,\perp}^{\top}\Big)^{-2}\frac{1}{n}\bm{S}_{l,\perp}\cdot\bm{s}_{l,\parallel}^{\top}\bigg|\\
 & \qquad\overset{}{\leq}n\cdot\max_{1\leq i<n}\frac{1}{(\lambda-\gamma_{i}^{(l)})^{2}}\cdot\bigg\|\sum_{k:k\ne l}\bm{a}^{\top}\bm{u}_{k}^{\star}\bm{u}_{k}^{\star\top}\bm{u}_{l}^{\star\perp}\bigg\|_{2}\cdot\big\|\bm{s}_{l,\parallel}\big\|_{2}\\
 & \qquad\overset{(\mathrm{i})}{\leq}n\cdot\frac{1}{\min_{i:i\neq l}\left|\lambda_{l}^{\star}-\lambda_{i}^{\star}\right|^{2}\wedge\lambda_{l}^{\star2}}\cdot\sum_{k:k\ne l}|\bm{a}^{\top}\bm{u}_{k}^{\star}|\big\|\bm{u}_{k}^{\star\top}\bm{u}_{l}^{\star\perp}\big\|_{2}\cdot(\lambda_{l}^{\star}+\sigma^{2})\sqrt{n\log n}\\
 & \qquad\overset{(\mathrm{ii})}{\lesssim}n^{3/2}\cdot\frac{\max_{i:i\neq l}\left|\lambda_{l}^{\star}-\lambda_{i}^{\star}\right|}{\min_{i:i\neq l}\left|\lambda_{l}^{\star}-\lambda_{i}^{\star}\right|^{2}\wedge\lambda_{l}^{\star2}}\cdot(\lambda_{l}^{\star}+\sigma^{2})\sqrt{\log n}\cdot\sum_{k:k\ne l}\frac{|\bm{a}^{\top}\bm{u}_{k}^{\star}|}{\left|\lambda_{l}^{\star}-\lambda_{k}^{\star}\right|}\\
 & \qquad\overset{(\mathrm{iii})}{\lesssim}n^{3/2}\cdot\frac{\kappa^{2}\lambda_{\max}^{\star2}}{\Delta_{l}^{\star2}}\cdot\frac{1}{\lambda_{l}^{\star}}\cdot V
\end{align*}
holds with probability at least $1-O(n^{-20})$. Here, (i) uses Lemma
\ref{lemma:lambda-S-minus-spectrum-1-1} and the high-probability
fact that $\|\bm{s}_{l,\parallel}\|_{2}\lesssim(\lambda_{l}^{\star}+\sigma^{2})\sqrt{n\log n}$,
(ii) holds since $\big\|\bm{u}_{k}^{\star\top}\bm{u}_{l}^{\star\perp}\big\|_{2}\leq\|\bm{u}_{k}^{\star}\|_{2}\|\bm{u}_{l}^{\star\perp}\|\leq1$,
whereas (iii) arises from the definition of $V$ in (\ref{eq:a-top-P-Uk-1-pca-UB}).

Combining the above two bounds, we are ready to invoke Lemma~ \ref{lemma:eps-net}
and the union bound to arrive at the advertised bound
\begin{align*}
 & \bigg\|\sum_{k:k\ne l}\bm{a}^{\top}\bm{u}_{k}^{\star}\bm{u}_{k}^{\star(l)\top}\Big(\lambda\bm{I}_{p-1}-\frac{1}{n}\bm{S}_{l,\perp}\bm{S}_{l,\perp}^{\top}\Big)^{-1}\frac{1}{n}\bm{S}_{l,\perp}\bm{s}_{l,\parallel}^{\top}\bigg\|_{2}\lesssim\sqrt{\left(\lambda_{l}^{\star}+\sigma^{2}\right)\log\bigg(\frac{n\kappa\lambda_{\max}}{\Delta_{l}^{\star}}\bigg)}\cdot V\\
 & \qquad\qquad\lesssim\sum_{k:k\ne l}\frac{\left|\bm{a}^{\top}\bm{u}_{k}^{\star}\right|}{\left|\lambda_{l}^{\star}-\lambda_{k}^{\star}\right|\sqrt{n}}\sqrt{(\lambda_{l}^{\star}+\sigma^{2})(\lambda_{\max}^{\star}+\sigma^{2})(\kappa^{2}+r)\log\bigg(\frac{n\kappa\lambda_{\max}}{\Delta_{l}^{\star}}\bigg)}
\end{align*}
with probability at least $1-O(n^{-10})$. 

Therefore, the remainder of the proof amounts to establishing (\ref{eq:a-top-P-Uk-1-pca-UB}).
Let us work under the event where the claims in Lemma \ref{lemma:lambda-S-minus-spectrum-1-1}
holds, which holds with probability exceeding $1-O(n^{-10})$. Recall
the SVD of $\frac{1}{\sqrt{n}}\bm{S}_{l,\perp}=\bm{U}^{(l)}\sqrt{\bm{\Gamma}^{(l)}}\bm{V}^{(l)\top}$.
Similar to (\ref{eq:claim:atukukt-lambda-inv-l2-norm-temp}), any
$\lambda$ such that $\lambda/(1+\beta(\lambda))\in\mathcal{B}_{\mathcal{E}_{\mathsf{PCA}}}(\lambda_{l}^{\star}+\sigma^{2})$,
one can rewrite
\begin{align*}
 & \bigg\|\sum_{k:k\ne l}\bm{a}^{\top}\bm{u}_{k}^{\star}\bm{u}_{k}^{\star(l)\top}\Big(\lambda\bm{I}_{p-1}-\frac{1}{n}\bm{S}_{l,\perp}\bm{S}_{l,\perp}^{\top}\Big)^{-1}\frac{1}{\sqrt{n}}\bm{S}_{l,\perp}\bigg\|_{2}\\
 & \qquad=\bigg\|\sum_{k:k\ne l}\bm{a}^{\top}\bm{u}_{k}^{\star}\bm{u}_{k}^{\star(l)\top}\bm{U}^{(l)}\big(\lambda\bm{I}_{p-1}-\bm{\Gamma}^{(l)}\big)^{-1}\bm{U}^{(l)\top}\bm{U}^{(l)}\sqrt{\bm{\Gamma}^{(l)}}\bm{V}^{(l)\top}\bigg\|_{2}\\
 & \qquad=\bigg\|\sum_{k:k\ne l}\bm{a}^{\top}\bm{u}_{k}^{\star}\bm{u}_{k}^{\star(l)\top}\bm{U}^{(l)}\big(\lambda\bm{I}_{p-1}-\bm{\Gamma}^{(l)}\big)^{-1}\sqrt{\bm{\Gamma}^{(l)}}\bigg\|_{2}\\
 & \qquad=\sqrt{\sum_{1\leq i\le n\wedge(p-1)}\frac{\gamma_{i}^{(l)}}{(\lambda-\gamma_{i}^{(l)})^{2}}\bigg(\sum_{k:k\ne l}\bm{a}^{\top}\bm{u}_{k}^{\star}\bm{u}_{k}^{\star(l)\top}\bm{u}_{i}^{(l)}\bigg)^{2}}.
\end{align*}

\begin{itemize}
\item With regards to the sum over the range $i\geq r$, it is seen from
Lemma \ref{lemma:lambda-S-minus-spectrum-1-1} and the assumption
(\ref{eq:noise-condition-iid-pca}) that for all $i\geq r$, $\gamma_{i}^{(l)}\lesssim\sigma^{2}(1+p/n)\lesssim\lambda_{\max}^{\star}+\sigma^{2}$
and $|\lambda_{l}-\gamma_{i}^{(l)}|\gtrsim\lambda_{l}^{\star}\geq\left|\lambda_{l}^{\star}-\lambda_{k}^{\star}\right|/\kappa$
for any $k\neq i$. This enables us to derive
\begin{align*}
 & \sqrt{\sum_{r\leq i\le n\wedge(p-1)}\frac{\gamma_{i}^{(l)}}{(\lambda_{l}-\gamma_{i}^{(l)})^{2}}\Big(\sum_{k:k\ne l}\bm{a}^{\top}\bm{u}_{k}^{\star}\bm{u}_{k}^{\star(l)\top}\bm{u}_{i}^{(l)}\Big)^{2}}\\
 & \qquad\leq\sqrt{(\lambda_{\max}^{\star}+\sigma^{2})\sum_{r\leq i\le n\wedge(p-1)}\Big(\sum_{k:k\ne l}\frac{\kappa}{\left|\lambda_{l}^{\star}-\lambda_{k}^{\star}\right|}\bm{a}^{\top}\bm{u}_{k}^{\star}\bm{u}_{k}^{\star(l)\top}\bm{u}_{i}^{(l)}\Big)^{2}}\\
 & \qquad\leq\sqrt{\lambda_{\max}^{\star}+\sigma^{2}}\,\bigg\|\sum_{k:k\ne l}\frac{\kappa}{\left|\lambda_{l}^{\star}-\lambda_{k}^{\star}\right|}\bm{a}^{\top}\bm{u}_{k}^{\star}\bm{u}_{k}^{\star(l)\top}\bm{U}^{(l)}\bigg\|\\
 & \qquad\leq\sqrt{(\lambda_{\max}^{\star}+\sigma^{2})\kappa^{2}}\,\sum_{k:k\ne l}\frac{\left|\bm{a}^{\top}\bm{u}_{k}^{\star}\right|}{\left|\lambda_{l}^{\star}-\lambda_{k}^{\star}\right|}\big\|\bm{u}_{k}^{\star(l)\top}\bm{U}^{(l)}\big\|_{2}\\
 & \qquad\leq\sqrt{(\lambda_{\max}^{\star}+\sigma^{2})\kappa^{2}}\:\sum_{k:k\ne l}\frac{\left|\bm{a}^{\top}\bm{u}_{k}^{\star}\right|}{\left|\lambda_{l}^{\star}-\lambda_{k}^{\star}\right|},
\end{align*}
where the last line holds since $\big\|\bm{u}_{k}^{\star(l)\top}\bm{U}^{(l)}\big\|_{2}\leq\big\|\bm{u}_{k}^{\star(l)}\big\|_{2}\big\|\bm{U}^{(l)}\big\|\leq1$.
\item Turning to the sum over the range $i<r$, we can control
\begin{align*}
\sqrt{\sum_{1\leq i<r}\frac{\gamma_{i}^{(l)}}{(\lambda_{l}-\gamma_{i}^{(l)})^{2}}\Big(\sum_{k:k\ne l}\bm{a}^{\top}\bm{u}_{k}^{\star}\bm{u}_{k}^{\star(l)\top}\bm{u}_{i}^{(l)}\Big)^{2}} & \text{\ensuremath{=\Bigg\|\sum_{1\leq i<r}\bigg(\frac{(\gamma_{i}^{(l)})^{1/2}}{\lambda_{l}-\gamma_{i}^{(l)}}\sum_{k:k\ne l}\bm{a}^{\top}\bm{u}_{k}^{\star}\bm{u}_{k}^{\star(l)\top}\bm{u}_{i}^{(l)}\bigg)\bm{u}_{i}^{(l)}}\ensuremath{\Bigg\|_{2}}}\\
 & =\Bigg\|\sum_{k:k\ne l}\bm{a}^{\top}\bm{u}_{k}^{\star}\sum_{1\leq i<r}\frac{(\gamma_{i}^{(l)})^{1/2}\bm{u}_{k}^{\star(l)\top}\bm{u}_{i}^{(l)}}{\lambda_{l}-\gamma_{i}^{(l)}}\bm{u}_{i}^{(l)}\Bigg\|_{2}\\
 & \le\sum_{k:k\ne l}|\bm{a}^{\top}\bm{u}_{k}^{\star}|\bigg\|\sum_{1\leq i<r}\frac{(\gamma_{i}^{(l)})^{1/2}\bm{u}_{k}^{\star(l)\top}\bm{u}_{i}^{(l)}}{\lambda_{l}-\gamma_{i}^{(l)}}\bm{u}_{i}^{(l)}\bigg\|_{2}\\
 & =\sum_{k:k\ne l}|\bm{a}^{\top}\bm{u}_{k}^{\star}|\sqrt{\sum_{1\leq i<r}\frac{\gamma_{i}^{(l)}(\bm{u}_{k}^{\star(l)\top}\bm{u}_{i}^{(l)})^{2}}{(\lambda_{l}-\gamma_{i}^{(l)})^{2}}}.
\end{align*}
This leads us to control $\sum_{1\leq i\leq r-1}\gamma_{i}^{(l)}(\bm{u}_{k}^{\star(l)\top}\bm{u}_{i}^{(l)})^{2}/(\lambda_{l}-\gamma_{i}^{(l)})^{2}$
for each $k\neq l$, which can be decomposed as follows
\begin{align*}
\sum_{1\leq i<r}\frac{\gamma_{i}^{(l)}(\bm{u}_{k}^{\star(l)\top}\bm{u}_{i}^{(l)})^{2}}{(\lambda_{l}-\gamma_{i}^{(l)})^{2}}=\sum_{i\in\mathcal{C}_{1}}\frac{\gamma_{i}^{(l)}(\bm{u}_{k}^{\star(l)\top}\bm{u}_{i}^{(l)})^{2}}{(\lambda_{l}-\gamma_{i}^{(l)})^{2}}+\sum_{i\in\mathcal{C}_{2}}\frac{\gamma_{i}^{(l)}(\bm{u}_{k}^{\star(l)\top}\bm{u}_{i}^{(l)})^{2}}{(\lambda_{l}-\gamma_{i}^{(l)})^{2}}.
\end{align*}
Here, the sets $\mathcal{C}_{1}$ and $\mathcal{C}_{2}$ are defined
respectively as follows 
\begin{align*}
\mathcal{C}_{1} & :=\{1\leq i<r\mid\gamma_{i}^{(l)}/(1+\gamma(\gamma_{i}^{(l)}))\in\mathcal{B}_{\mathcal{E}_{k}}(\lambda_{k}^{\star})\},\\
\mathcal{C}_{2} & :=\{1\leq i<r\mid\gamma_{i}^{(l)}/(1+\gamma(\gamma_{i}^{(l)}))\notin\mathcal{B}_{\mathcal{E}_{k}}(\lambda_{k}^{\star})\},
\end{align*}
where we take $\mathcal{E}_{k}:=c\left|\lambda_{l}^{\star}-\lambda_{k}^{\star}\right|$
for some sufficiently small constant $c>0$. In the sequel, we shall
control the above two sums separately.
\begin{itemize}
\item With respect to the sum over $\mathcal{C}_{1}$, one can apply a similar
argument for (\ref{eq:lambda-lambdal-LB}) to show $|\lambda_{l}-\gamma_{i}^{(l)}|\gtrsim\left|\lambda_{l}^{\star}-\lambda_{k}^{\star}\right|$
for $i\in\mathcal{C}_{1}$. This enables us to bound
\[
\sum_{i\in\mathcal{C}_{1}}\frac{\gamma_{i}^{(l)}\big(\bm{u}_{k}^{\star(l)\top}\bm{u}_{i}^{(l)}\big)^{2}}{(\lambda_{l}-\gamma_{i}^{(l)})^{2}}\lesssim\frac{\lambda_{\max}^{\star}+\sigma^{2}}{\left|\lambda_{l}^{\star}-\lambda_{k}^{\star}\right|^{2}}\sum_{i\in\mathcal{C}_{1}}\big(\bm{u}_{k}^{\star(l)\top}\bm{u}_{i}^{(l)}\big)^{2}\leq\frac{\lambda_{\max}^{\star}+\sigma^{2}}{\left|\lambda_{l}^{\star}-\lambda_{k}^{\star}\right|^{2}}\big\|\bm{u}_{k}^{\star(l)}\big\|_{2}^{2}\big\|\bm{U}^{(l)}\big\|^{2}\leq\frac{\lambda_{\max}^{\star}+\sigma^{2}}{\left|\lambda_{l}^{\star}-\lambda_{k}^{\star}\right|^{2}}.
\]
\item Next, we move on to look at the sum over $\mathcal{C}_{2}$. According
to Lemma \ref{lemma:lambda-S-minus-spectrum-1-1}, we have
\begin{align*}
\mathcal{E}_{\mathsf{PCA}} & \gtrsim\Big\|\big(\gamma_{i}^{(l)}\bm{I}_{r-1}-(1+\beta(\gamma_{i}^{(l)}))\bm{\Lambda}^{(l)}\big)\bm{U}^{\star(l)\top}\bm{u}_{i,\parallel}^{(l)}\Big\|_{2}\\
 & \geq\big|\gamma_{i}^{(l)}-(1+\beta(\gamma_{i}^{(l)}))\lambda_{k}^{(l)}\big|\cdot\big|\bm{u}_{k}^{\star(l)\top}\bm{u}_{i,\parallel}^{(l)}\big|\\
 & \gtrsim\mathcal{E}_{k}\cdot\big|\bm{u}_{k}^{\star(l)\top}\bm{u}_{i,\parallel}^{(l)}\big|\\
 & \gtrsim|\lambda_{l}^{\star}-\lambda_{k}^{\star}|\cdot\big|\bm{u}_{k}^{\star(l)\top}\bm{u}_{i}^{(l)}\big|,
\end{align*}
where we use the fact that $\big|\bm{u}_{k}^{\star(l)\top}\bm{u}_{i}^{(l)}\big|\le\big|\bm{u}_{k}^{\star(l)\top}\bm{u}_{i,\parallel}^{(l)}\big|$
and $\big|\gamma_{i}^{(l)}-(1+\beta(\gamma_{i}^{(l)}))\lambda_{k}^{(l)}\big|\gtrsim\mathcal{E}_{k}$
for all $i\in\mathcal{C}_{2}$. Therefore, we arrive at the upper
bound
\[
\frac{\gamma_{i}^{(l)}(\bm{u}_{k}^{\star(l)\top}\bm{u}_{i}^{(l)})^{2}}{(\lambda_{l}-\gamma_{i}^{(l)})^{2}}\lesssim\frac{(\lambda_{\max}^{\star}+\sigma^{2})\mathcal{E}_{\mathsf{PCA}}^{2}}{\left|\lambda_{l}^{\star}-\lambda_{k}^{\star}\right|^{2}\min_{i:i\neq l}|\lambda_{l}^{\star}-\lambda_{i}^{\star}|^{2}}\lesssim\frac{\lambda_{\max}^{\star}+\sigma^{2}}{\left|\lambda_{l}^{\star}-\lambda_{k}^{\star}\right|^{2}},\qquad i\in\mathcal{C}_{2},
\]
where we invoke the condition $\min_{i:i\neq l}|\lambda_{l}^{\star}-\lambda_{i}^{\star}|\gtrsim\mathcal{E}_{\mathsf{PCA}}$.
Taking these two bounds collectively, we reach
\begin{align*}
\sum_{1\leq i<r}\frac{\gamma_{i}^{(l)}\big(\bm{u}_{k}^{\star(l)\top}\bm{u}_{i}^{(l)}\big)^{2}}{(\lambda_{l}-\gamma_{i}^{(l)})^{2}}\lesssim\frac{\lambda_{\max}^{\star}+\sigma^{2}}{\left|\lambda_{l}^{\star}-\lambda_{k}^{\star}\right|^{2}}+\frac{\left(\lambda_{\max}^{\star}+\sigma^{2}\right)r}{\left|\lambda_{l}^{\star}-\lambda_{k}^{\star}\right|^{2}}\asymp\frac{\left(\lambda_{\max}^{\star}+\sigma^{2}\right)r}{\left|\lambda_{l}^{\star}-\lambda_{k}^{\star}\right|^{2}},
\end{align*}
and hence
\[
\sqrt{\sum_{1\leq i<r}\frac{\gamma_{i}^{(l)}}{(\lambda_{l}-\gamma_{i}^{(l)})^{2}}\Big(\sum_{k:k\ne l}\bm{a}^{\top}\bm{u}_{k}^{\star}\bm{u}_{k}^{\star(l)\top}\bm{u}_{i}^{(l)}\Big)^{2}}\lesssim\sqrt{\left(\lambda_{\max}^{\star}+\sigma^{2}\right)r}\sum_{k:k\ne l}\frac{|\bm{a}^{\top}\bm{u}_{k}^{\star}|}{\left|\lambda_{l}^{\star}-\lambda_{k}^{\star}\right|}.
\]
\end{itemize}
\item Combining the preceding two bounds, we finish the proof for (\ref{eq:a-top-P-Uk-1-pca-UB}).
\end{itemize}

\section{Proof for minimax lower bounds (Theorem \ref{thm:minimax-evector-pertur-sym-iid-pca})}

\label{sec:Proof-for-minimax}

Fix an arbitrary $1\leq l\leq r$ and an arbitrary $k\neq l$ and
$1\leq k\leq r$. In what follows, we intend to prove the following
two claims:
\begin{align}
\inf_{u_{\bm{a},l}}\sup_{\bm{\Sigma}\in\mathcal{M}_{1}(\bm{\Sigma}^{\star})}\mathbb{E}\Big[\min\big|u_{\bm{a},l}\pm\bm{a}^{\top}\bm{u}_{l}(\bm{\Sigma})\big|\Big] & \gtrsim\frac{(\lambda_{k}^{\star}+\sigma^{2})(\lambda_{l}^{\star}+\sigma^{2})}{|\lambda_{l}^{\star}-\lambda_{k}^{\star}|^{2}\,n}\big|\bm{a}^{\top}\bm{u}_{l}^{\star}\big|+\frac{\sqrt{(\lambda_{k}^{\star}+\sigma^{2})(\lambda_{l}^{\star}+\sigma^{2})}}{|\lambda_{l}^{\star}-\lambda_{k}^{\star}|\sqrt{n}}\big|\bm{a}^{\top}\bm{u}_{k}^{\star}\big|,\label{eq:lower-bound-claim1}\\
\inf_{u_{\bm{a},l}}\sup_{\bm{\Sigma}\in\mathcal{M}_{2}(\bm{\Sigma}^{\star})}\mathbb{E}\Big[\min\big|u_{\bm{a},l}\pm\bm{a}^{\top}\bm{u}_{l}(\bm{\Sigma})\big|\Big] & \gtrsim\sqrt{\frac{(\lambda_{l}^{\star}+\sigma^{2})\sigma^{2}}{\lambda_{l}^{\star2}n}}\|\bm{P}_{\bm{U}^{\star\perp}}\bm{a}\|_{2},\label{eq:lower-bound-claim2}
\end{align}
where the infimum is over all estimators, and $\mathcal{M}_{1}(\bm{\Sigma}^{\star})$
and $\mathcal{M}_{2}(\bm{\Sigma}^{\star})$ are defined right before
the statement of Theorem \ref{thm:minimax-evector-pertur-sym-iid-pca}.
It is self-evident that Theorem \ref{thm:minimax-evector-pertur-sym-iid-pca}
follows from these two claims by taking the maximum over all $k\neq l$.

\subsection{Proof of the lower bound (\ref{eq:lower-bound-claim1})}

\paragraph{Step 1: constructing a collection of hypotheses.} Let
us consider the following two hypotheses:
\begin{align*}
\mathcal{H}_{0} & \,:\,\bm{s}_{i}\overset{\mathrm{i.i.d.}}{\sim}\mathcal{N}(\bm{0},\bm{\Sigma}^{\star}+\sigma^{2}\bm{I}_{p}),\quad1\leq i\leq n;\\
\mathcal{H}_{k} & \,:\,\bm{s}_{i}\overset{\mathrm{i.i.d.}}{\sim}\mathcal{N}(\bm{0},\bm{\Sigma}_{k}+\sigma^{2}\bm{I}_{p}),\quad1\leq i\leq n.
\end{align*}
Here, the covariance matrix $\bm{\Sigma}_{k}$ is defined as follows:
\[
\bm{\Sigma}_{k}\coloneqq\lambda_{l}^{\star}\bm{u}_{l}\bm{u}_{l}^{\top}+\lambda_{k}^{\star}\bm{u}_{k}\bm{u}_{k}^{\top}+\sum_{i:\,i\ne k,l,\,1\leq i\leq r}\lambda_{i}^{\star}\bm{u}_{i}^{\star}\bm{u}_{i}^{\star\top}.
\]
where $\bm{u}_{l}$ and $\bm{u}_{k}$ are defined as
\begin{equation}
[\bm{u}_{l},\,\bm{u}_{k}]\coloneqq[\bm{u}_{l}^{\star},\,\bm{u}_{k}^{\star}]\begin{bmatrix}\cos\theta_{n} & -\sin\theta_{n}\\
\sin\theta_{n} & \cos\theta_{n}
\end{bmatrix}\label{eq:def:ul-uk}
\end{equation}
for some $\theta_{n}\in[-\pi/2,\pi/2]$ to be specified later. Straightforward
calculation yields
\[
\bm{u}_{l}\bm{u}_{l}^{\top}+\bm{u}_{k}\bm{u}_{k}^{\top}=\bm{u}_{l}^{\star}\bm{u}_{l}^{\star\top}+\bm{u}_{k}^{\star}\bm{u}_{k}^{\star\top}.
\]
This identity further leads to
\begin{align*}
\bm{\Sigma}_{k}-\bm{\Sigma}^{\star} & =\lambda_{l}^{\star}\bm{u}_{l}\bm{u}_{l}^{\top}+\lambda_{k}^{\star}\bm{u}_{k}\bm{u}_{k}^{\top}-(\lambda_{l}^{\star}\bm{u}_{l}^{\star}\bm{u}_{l}^{\star\top}+\lambda_{k}^{\star}\bm{u}_{k}^{\star}\bm{u}_{k}^{\star\top})\\
 & =\lambda_{l}^{\star}(\bm{u}_{l}\bm{u}_{l}^{\top}+\bm{u}_{k}\bm{u}_{k}^{\top})+(\lambda_{k}^{\star}-\lambda_{l}^{\star})\bm{u}_{k}\bm{u}_{k}^{\top}-\big(\lambda_{l}^{\star}(\bm{u}_{l}^{\star}\bm{u}_{l}^{\star\top}+\bm{u}_{k}^{\star}\bm{u}_{k}^{\star\top})+(\lambda_{k}^{\star}-\lambda_{l}^{\star})\bm{u}_{k}^{\star}\bm{u}_{k}^{\star\top}\big)\\
 & =(\lambda_{k}^{\star}-\lambda_{l}^{\star})(\bm{u}_{k}\bm{u}_{k}^{\top}-\bm{u}_{k}^{\star}\bm{u}_{k}^{\star\top}).
\end{align*}
In addition, it is also seen that
\begin{align}
\|\bm{\Sigma}_{k}-\bm{\Sigma}^{\star}\|_{\mathrm{F}} & =|\lambda_{k}^{\star}-\lambda_{l}^{\star}|\cdot\|\bm{u}_{k}\bm{u}_{k}^{\top}-\bm{u}_{k}^{\star}\bm{u}_{k}^{\star\top}\|_{\mathrm{F}}\leq|\lambda_{k}^{\star}-\lambda_{l}^{\star}|\cdot\big(\|\bm{u}_{k}(\bm{u}_{k}-\bm{u}_{k}^{\star})^{\top}\|_{\mathrm{F}}+\|(\bm{u}_{k}-\bm{u}_{k}^{\star})\bm{u}_{k}^{\star\top}\|_{\mathrm{F}}\big)\nonumber \\
 & =|\lambda_{k}^{\star}-\lambda_{l}^{\star}|\cdot\big(\|\bm{u}_{k}\|_{2}\|\bm{u}_{k}-\bm{u}_{k}^{\star}\|_{2}+\|\bm{u}_{k}-\bm{u}_{k}^{\star}\|_{2}\|\bm{u}_{k}^{\star}\|_{2}\big)\nonumber \\
 & =2\,|\lambda_{k}^{\star}-\lambda_{l}^{\star}|\cdot\|\bm{u}_{k}-\bm{u}_{k}^{\star}\|_{2}\nonumber \\
 & \overset{(\mathrm{i})}{=}2\,|\lambda_{k}^{\star}-\lambda_{l}^{\star}|\cdot\|-\bm{u}_{l}^{\star}\sin\theta_{n}+\bm{u}_{k}^{\star}\cos\theta_{n}-\bm{u}_{k}^{\star}\|_{2}\nonumber \\
 & \leq2\,|\lambda_{k}^{\star}-\lambda_{l}^{\star}|\cdot\big(\sin\theta_{n}+2\sin^{2}(\theta_{n}/2)\big)\nonumber \\
 & \overset{(\mathrm{ii})}{\leq}4\,|\lambda_{k}^{\star}-\lambda_{l}^{\star}|\cdot|\theta_{n}|,\label{eq:Sigma-k-Sigma-star-dist}
\end{align}
where (i) arises from the definition of $\bm{u}_{l}$ in (\ref{eq:def:ul-uk});
(ii) holds since $\sin\theta\leq|\theta|$. 

In what follows, we denote by $\mathbb{P}^{0}$ and $\mathbb{P}^{k}$
the distribution of $\bm{S}$ under the hypothesis $\mathcal{H}_{0}$
and $\mathcal{H}_{k}$, respectively, and let $\mathbb{P}_{i}^{0}$
and $\mathbb{P}_{i}^{k}$ denote the distribution of $\bm{s}_{i}$
($i$-th column of $\bm{S}$) under $\mathcal{H}_{0}$ and $\mathcal{H}_{k}$,
respectively. 

\paragraph{Step 2: bounding the KL divergence between hypotheses.}
Recall the elementary fact that the KL divergence of multivariate
Gaussians is given by \citep{kullback1952application}
\[
\mathsf{KL}\big(\mathcal{N}(\bm{0},\bm{\Sigma}_{1})\parallel\mathcal{N}(\bm{0},\bm{\Sigma}_{0})\big)=\frac{1}{2}\bigg(\mathsf{tr}\big(\bm{\Sigma}_{0}^{-1}\bm{\Sigma}_{1}\big)-p+\log\frac{|\bm{\Sigma}_{0}|}{|\bm{\Sigma}_{1}|}\bigg).
\]
Since the KL divergence is additive over independent distributions
\citep{tsybakov2009introduction}, one has
\begin{align}
\mathsf{KL}\big(\mathbb{P}^{k}\parallel\mathbb{P}^{0}\big) & =\sum_{i=1}^{n}\mathsf{KL}\big(\mathbb{P}_{i}^{k}\parallel\mathbb{P}_{i}^{0}\big)=\frac{1}{2}\sum_{i=1}^{n}\big(\mathsf{tr}\big((\bm{\Sigma}^{\star}+\sigma^{2}\bm{I}_{p})^{-1}(\bm{\Sigma}_{k}+\sigma^{2}\bm{I}_{p})\big)-p\big).\label{eq:KL-sum-equation}
\end{align}
This suggests that we need to compute $\mathsf{tr}\big((\bm{\Sigma}^{\star}+\sigma^{2}\bm{I}_{p})^{-1}(\bm{\Sigma}_{k}+\sigma^{2}\bm{I}_{p})\big)$.
By construction in (\ref{eq:def:ul-uk}), we know that $\bm{u}_{l}$
and $\bm{u}_{k}$ span the same subspace as $\bm{u}_{l}^{\star}$
and $\bm{u}_{k}^{\star}$, and are orthogonal to $\{\bm{u}_{i}^{\star}\}_{i:i\neq k,l}$.
Denote by $\bm{U}^{\star\perp}\in\mathbb{R}^{p\times(p-r)}$ the matrix
whose columns form an orthonormal basis of the complement to the subspace
spanned by $\bm{U}^{\star}$. One can then derive
\begin{align*}
(\bm{\Sigma}^{\star}+\sigma^{2}\bm{I}_{p})^{-1}(\bm{\Sigma}_{k}+\sigma^{2}\bm{I}_{p}) & =\bigg(\sum_{1\leq i\leq r}\frac{1}{\lambda_{i}^{\star}+\sigma^{2}}\bm{u}_{i}^{\star}\bm{u}_{i}^{\star\top}+\frac{1}{\sigma^{2}}\bm{U}^{\star\perp}(\bm{U}^{\star\perp})^{\top}\bigg)\\
 & \qquad\cdot\bigg((\lambda_{l}^{\star}+\sigma^{2})\bm{u}_{l}\bm{u}_{l}^{\top}+(\lambda_{k}^{\star}+\sigma^{2})\bm{u}_{k}\bm{u}_{k}^{\top}+\sum_{i:\,i\ne k,l,\,1\leq i\leq r}(\lambda_{i}^{\star}+\sigma^{2})\bm{u}_{i}^{\star}\bm{u}_{i}^{\star\top}+\sigma^{2}\bm{U}^{\star\perp}(\bm{U}^{\star\perp})^{\top}\bigg)\\
 & =\Big(\frac{1}{\lambda_{l}^{\star}+\sigma^{2}}\bm{u}_{l}^{\star}\bm{u}_{l}^{\star\top}+\frac{1}{\lambda_{k}^{\star}+\sigma^{2}}\bm{u}_{k}^{\star}\bm{u}_{k}^{\star\top}\Big)\Big((\lambda_{l}^{\star}+\sigma^{2})\bm{u}_{l}\bm{u}_{l}^{\top}+(\lambda_{k}^{\star}+\sigma^{2})\bm{u}_{k}\bm{u}_{k}^{\top}\Big)\\
 & \quad+\sum_{i:\,i\ne k,l,\,1\leq i\leq r}\bm{u}_{i}^{\star}\bm{u}_{i}^{\star\top}+\bm{U}^{\star\perp}(\bm{U}^{\star\perp})^{\top}.
\end{align*}
As a result, we find
\begin{align*}
 & \mathsf{tr}\big((\bm{\Sigma}^{\star}+\sigma^{2}\bm{I}_{p})^{-1}(\bm{\Sigma}_{k}+\sigma^{2}\bm{I}_{p})\big)\\
 & \qquad\overset{(\mathrm{i})}{=}\mathsf{tr}\bigg(\Big(\frac{1}{\lambda_{l}^{\star}+\sigma^{2}}\bm{u}_{l}^{\star}\bm{u}_{l}^{\star\top}+\frac{1}{\lambda_{k}^{\star}+\sigma^{2}}\bm{u}_{k}^{\star}\bm{u}_{k}^{\star\top}\Big)\Big((\lambda_{l}^{\star}+\sigma^{2})\bm{u}_{l}\bm{u}_{l}^{\top}+(\lambda_{k}^{\star}+\sigma^{2})\bm{u}_{k}\bm{u}_{k}^{\top}\Big)\bigg)+p-2\\
 & \qquad=|\bm{u}_{l}^{\star\top}\bm{u}_{l}|^{2}+\frac{\lambda_{k}^{\star}+\sigma^{2}}{\lambda_{l}^{\star}+\sigma^{2}}|\bm{u}_{l}^{\star\top}\bm{u}_{k}|^{2}+\frac{\lambda_{l}^{\star}+\sigma^{2}}{\lambda_{k}^{\star}+\sigma^{2}}|\bm{u}_{k}^{\star\top}\bm{u}_{l}|^{2}+|\bm{u}_{k}^{\star\top}\bm{u}_{k}|^{2}+p-2\\
 & \qquad\overset{(\mathrm{ii})}{=}\cos^{2}\theta_{n}+\frac{\lambda_{k}^{\star}+\sigma^{2}}{\lambda_{l}^{\star}+\sigma^{2}}\sin^{2}\theta_{n}+\frac{\lambda_{l}^{\star}+\sigma^{2}}{\lambda_{k}^{\star}+\sigma^{2}}\sin^{2}\theta_{n}+\cos^{2}\theta_{n}+p-2\\
 & \qquad=\frac{(\lambda_{l}^{\star}+\sigma^{2})^{2}+(\lambda_{k}^{\star}+\sigma^{2})^{2}}{(\lambda_{l}^{\star}+\sigma^{2})(\lambda_{k}^{\star}+\sigma^{2})}\sin^{2}\theta_{n}-2\sin^{2}\theta_{n}+p\\
 & \qquad=\frac{(\lambda_{l}^{\star}-\lambda_{k}^{\star})^{2}}{(\lambda_{l}^{\star}+\sigma^{2})(\lambda_{k}^{\star}+\sigma^{2})}\sin^{2}\theta_{n}+p.
\end{align*}
Here, (i) holds since $\mathsf{tr}(\bm{u}_{i}^{\star}\bm{u}_{i}^{\star\top})=1$
and $\mathsf{tr}\big(\bm{U}^{\star\perp}(\bm{U}^{\star\perp})^{\top}\big)=\mathsf{tr}\big((\bm{U}^{\star\perp})^{\top}\bm{U}^{\star\perp}\big)=\mathsf{tr}(\bm{I}_{p-r})=p-r$;
(ii) follows from the following observations:
\begin{align*}
\bm{u}_{l}^{\star\top}\bm{u}_{l} & =\bm{u}_{l}^{\star\top}\bm{u}_{l}^{\star}\cos\theta_{n}+\bm{u}_{l}^{\star\top}\bm{u}_{k}^{\star}\sin\theta_{n}=\cos\theta_{n};\\
\bm{u}_{k}^{\star\top}\bm{u}_{l} & =\bm{u}_{k}^{\star\top}\bm{u}_{l}^{\star}\cos\theta_{n}+\bm{u}_{k}^{\star\top}\bm{u}_{k}^{\star}\sin\theta_{n}=\sin\theta_{n};\\
\bm{u}_{l}^{\star\top}\bm{u}_{k} & =-\bm{u}_{l}^{\star\top}\bm{u}_{l}^{\star}\sin\theta_{n}+\bm{u}_{l}^{\star\top}\bm{u}_{k}^{\star}\cos\theta_{n}=-\sin\theta_{n};\\
\bm{u}_{k}^{\star\top}\bm{u}_{k} & =-\bm{u}_{k}^{\star\top}\bm{u}_{l}^{\star}\sin\theta_{n}+\bm{u}_{k}^{\star\top}\bm{u}_{k}^{\star}\cos\theta_{n}=\cos\theta_{n};
\end{align*}
where we have used the construction (\ref{eq:def:ul-uk}) and the
fact that $\bm{u}_{l}^{\star\top}\bm{u}_{k}^{\star}=0$. Therefore,
combining the above identities allows us to conclude that
\begin{equation}
\mathsf{KL}\big(\mathbb{P}^{k}\parallel\mathbb{P}^{0}\big)=\frac{n(\lambda_{l}^{\star}-\lambda_{k}^{\star})^{2}}{2(\lambda_{l}^{\star}+\sigma^{2})(\lambda_{k}^{\star}+\sigma^{2})}\sin^{2}\theta_{n}.\label{eq:kl-expression-1}
\end{equation}

\paragraph{Step 3: invoking Fano's inequality.} Suppose that we
choose $\theta_{n}$
\begin{equation}
|\theta_{n}|=c_{n}\sqrt{\frac{(\lambda_{l}^{\star}+\sigma^{2})(\lambda_{k}^{\star}+\sigma^{2})}{(\lambda_{l}^{\star}-\lambda_{k}^{\star})^{2}n}}\label{eq:lower-bound-theta-value}
\end{equation}
where $c_{n}\asymp1$ is a sequence that depends on $n$ and obeys
$c_{n}\in\{1/64,\,1/16,\,1/4\}$ (which we shall discuss momentarily).
Then we can see from (\ref{eq:Sigma-k-Sigma-star-dist}) that
\[
\|\bm{\Sigma}_{k}-\bm{\Sigma}^{\star}\|_{\mathrm{F}}\leq\sqrt{\frac{(\lambda_{l}^{\star}+\sigma^{2})(\lambda_{k}^{\star}+\sigma^{2})}{n}}.
\]
In other words, $\bm{\Sigma}_{k}\in\mathcal{M}_{1}(\bm{\Sigma}^{\star})$.
Moreover, plugging the value (\ref{eq:lower-bound-theta-value}) of
$\theta_{n}$ into (\ref{eq:kl-expression-1}) and using the facts
$|\sin\theta|\leq|\theta|$ as well as $\max_{n}c_{n}=1/4$ yields
\[
\mathrm{\mathsf{KL}}(\mathbb{P}^{k}\parallel\mathbb{P}^{0})\le1/16.
\]
It then follows from Fano's inequality \cite[Theorem 2]{tsybakov2009introduction}
that
\[
p_{e,k}:=\inf_{\psi}\max\Big\{\mathbb{P}\{\psi\;{\rm rejects}\;\mathcal{H}_{0}\mid\mathcal{H}_{0}\},\,\mathbb{P}\{\psi\;{\rm rejects}\;\mathcal{H}_{k}\mid\mathcal{H}_{k}\}\Big\}\ge1/5,
\]
where the infimum is taken over all tests. One can then apply the
standard reduction scheme in \cite[Chapter 2.2]{tsybakov2009introduction}
to show that

\[
\inf_{u_{\bm{a},l}}\sup_{\bm{\Sigma}\in\mathcal{M}_{1}(\bm{\Sigma}^{\star})}\mathbb{E}\Big[\min\big|u_{\bm{a},l}\pm\bm{a}^{\top}\bm{u}_{l}(\bm{\Sigma})\big|\Big]\gtrsim p_{e,k}\min\big|\bm{a}^{\top}\bm{u}_{l}\pm\bm{a}^{\top}\bm{u}_{l}^{\star}\big|\gtrsim\min\big|\bm{a}^{\top}\bm{u}_{l}\pm\bm{a}^{\top}\bm{u}_{l}^{\star}\big|.
\]
Observe that once we prove
\begin{equation}
\min\big|\bm{a}^{\top}\bm{u}_{l}\pm\bm{a}^{\top}\bm{u}_{l}^{\star}\big|\geq\frac{1}{8\pi^{2}}\big(\theta_{n}^{2}\cdot\big|\bm{a}^{\top}\bm{u}_{l}^{\star}\big|+|\theta_{n}|\cdot\big|\bm{a}^{\top}\bm{u}_{k}^{\star}\big|\big),\label{claim:min-atu}
\end{equation}
then (\ref{eq:lower-bound-theta-value}) would immediately lead to
the advertised bound
\begin{align*}
\inf_{u_{\bm{a},l}}\sup_{\bm{\Sigma}\in\mathcal{M}_{1}(\bm{\Sigma}^{\star})}\mathbb{E}\Big[\min\big|u_{\bm{a},l}\pm\bm{a}^{\top}\bm{u}_{l}(\bm{\Sigma})\big|\Big] & \gtrsim c_{n}\frac{(\lambda_{k}^{\star}+\sigma^{2})(\lambda_{l}^{\star}+\sigma^{2})\big|\bm{a}^{\top}\bm{u}_{l}^{\star}\big|}{(\lambda_{l}^{\star}-\lambda_{k}^{\star})^{2}n}+c_{n}\frac{\sqrt{(\lambda_{k}^{\star}+\sigma^{2})(\lambda_{l}^{\star}+\sigma^{2})}\big|\bm{a}^{\top}\bm{u}_{k}^{\star}\big|}{|\lambda_{l}^{\star}-\lambda_{k}^{\star}|\sqrt{n}}\\
 & \gtrsim\frac{(\lambda_{k}^{\star}+\sigma^{2})(\lambda_{l}^{\star}+\sigma^{2})\big|\bm{a}^{\top}\bm{u}_{l}^{\star}\big|}{(\lambda_{l}^{\star}-\lambda_{k}^{\star})^{2}n}+\frac{\sqrt{(\lambda_{k}^{\star}+\sigma^{2})(\lambda_{l}^{\star}+\sigma^{2})}\big|\bm{a}^{\top}\bm{u}_{k}^{\star}\big|}{|\lambda_{l}^{\star}-\lambda_{k}^{\star}|\sqrt{n}}
\end{align*}
where the last step holds since $\min_{n}c_{n}=1/64$. As a consequence,
the remainder of the proof amounts to establishing the claim (\ref{claim:min-atu}).
In view of (\ref{eq:def:ul-uk}), we know that
\begin{align}
\big|\bm{a}^{\top}\bm{u}_{l}^{\star}-\bm{a}^{\top}\bm{u}_{l}\big| & =\big|\bm{a}^{\top}\bm{u}_{l}^{\star}-\bm{a}^{\top}(\bm{u}_{l}^{\star}\cos\theta_{n}+\bm{u}_{k}^{\star}\sin\theta_{n})\big|\nonumber \\
 & =\big|\bm{a}^{\top}\bm{u}_{l}^{\star}(1-\cos\theta_{n})-\bm{a}^{\top}\bm{u}_{k}^{\star}\sin\theta_{n}\big|\nonumber \\
 & =\big|2\bm{a}^{\top}\bm{u}_{l}^{\star}\sin^{2}(\theta_{n}/2)-\bm{a}^{\top}\bm{u}_{k}^{\star}\sin\theta_{n}\big|\nonumber \\
 & \overset{(\mathrm{i})}{=}\big|2\bm{a}^{\top}\bm{u}_{l}^{\star}\sin^{2}(\theta_{n}/2)\big|+\big|\bm{a}^{\top}\bm{u}_{k}^{\star}\sin\theta_{n}\big|\nonumber \\
 & \overset{(\mathrm{ii})}{\geq}\frac{2}{\pi^{2}}\big(\theta_{n}^{2}\cdot|\bm{a}^{\top}\bm{u}_{l}^{\star}|+|\theta_{n}|\cdot\big|\bm{a}^{\top}\bm{u}_{k}^{\star}\big|\big)\label{eq:atustar-atu-temp}
\end{align}
where (i) holds true as long as we choose $\mathsf{sign}(\theta_{n})=-\mathsf{sign}(\bm{a}^{\top}\bm{u}_{l}^{\star}/\bm{a}^{\top}\bm{u}_{k}^{\star})$;
(ii) relies on the fact $|\sin\theta|\geq\frac{2}{\pi}|\theta|$ for
$\theta\in[-\frac{\pi}{2},\frac{\pi}{2}]$. In addition, we can derive
\begin{align}
\big|\bm{a}^{\top}\bm{u}_{l}^{\star}+\bm{a}^{\top}\bm{u}_{l}\big| & =\big|\bm{a}^{\top}\bm{u}_{l}^{\star}+\bm{a}^{\top}(\bm{u}_{l}^{\star}\cos\theta_{n}+\bm{u}_{k}^{\star}\sin\theta_{n})\big|\nonumber \\
 & =\big|\bm{a}^{\top}\bm{u}_{l}^{\star}(1+\cos\theta_{n})+\bm{a}^{\top}\bm{u}_{k}^{\star}\sin\theta_{n}\big|\nonumber \\
 & =\big|2\bm{a}^{\top}\bm{u}_{l}^{\star}\cos^{2}(\theta_{n}/2)+\bm{a}^{\top}\bm{u}_{k}^{\star}\sin(\theta_{n}/2)\cos(\theta_{n}/2)\big|\nonumber \\
 & =\cos(\theta_{n}/2)\big|2\bm{a}^{\top}\bm{u}_{l}^{\star}\cos(\theta_{n}/2)+\bm{a}^{\top}\bm{u}_{k}^{\star}\sin(\theta_{n}/2)\big|\nonumber \\
 & \overset{(\mathrm{i})}{\geq}\frac{1}{2}\,\big|\bm{a}^{\top}\bm{u}_{l}^{\star}\cos(\theta_{n}/2)+\bm{a}^{\top}\bm{u}_{k}^{\star}\sin(\theta_{n}/2)\big|\nonumber \\
 & \overset{(\mathrm{ii})}{=}\frac{1}{2}\,\Big|\sqrt{(\bm{a}^{\top}\bm{u}_{l}^{\star})^{2}+(\bm{a}^{\top}\bm{u}_{k}^{\star})^{2}}\sin(\theta_{n}/2+\omega_{k})\Big|\nonumber \\
 & \geq\frac{1}{4}\,\big(|\bm{a}^{\top}\bm{u}_{l}^{\star}|\sin|\theta_{n}/2+\omega_{k}|+|\bm{a}^{\top}\bm{u}_{k}^{\star}|\sin|\theta_{n}/2+\omega_{k}|\big)\label{eq:atustar+atu-temp}
\end{align}
where (i) holds due to $|\theta_{n}|\leq1/4$ by the choice of $c_{n}$
in (\ref{eq:lower-bound-theta-value}) and the sample size condition
(\ref{eq:minimax-sample-size}); $\omega_{k}\in[-\frac{\pi}{2},\frac{\pi}{2}]$
in (ii) is defined such that $\tan\omega_{k}=\bm{a}^{\top}\bm{u}_{l}^{\star}/\bm{a}^{\top}\bm{u}_{k}^{\star}$.
In particular, recall that the sign of $\theta_{n}$ is chosen such
that $\mathsf{sign}(\theta_{n})=-\mathsf{sign}(\bm{a}^{\top}\bm{u}_{l}^{\star}/\bm{a}^{\top}\bm{u}_{k}^{\star})$,
one has $\mathsf{sign}(\theta_{n})=-\mathsf{sign}(\omega_{k})$. Next,
our goal is to show if we choose $c_{n}\in\{1/64,\,1/16,\,1/4\}$
of $\theta_{n}$ in (\ref{eq:lower-bound-theta-value}) suitably,
one has
\begin{equation}
\sin|\theta_{n}/2+\omega_{k}|\geq\frac{1}{2\pi}|\theta_{n}|.\label{eq:minimax-lb-sin-claim}
\end{equation}
To this end, for each $n$, we choose $c_{n}$ of $\theta_{n}$ in
(\ref{eq:lower-bound-theta-value}) to be $c_{n}=1/16$ temporarily,
and consider the following three scenarios:
\begin{itemize}
\item If $|\theta_{n}|/2\geq2\,|\omega_{k}|$, then one has $\pi/2\geq|\theta_{n}/2+\omega_{k}|\geq|\theta_{n}|/2-|\omega_{k}|\geq|\theta_{n}/4|$
where the first inequality holds since the signs of $\theta_{n}$
and $\omega_{k}$ are different. Combined with the inequality $|\sin\theta|\geq\frac{2}{\pi}|\theta|$
for $\theta\in[-\frac{\pi}{2},\frac{\pi}{2}]$, this leads to $\sin|\theta_{n}/2+\omega_{k}|\geq\sin|\theta_{n}/4|\geq\frac{1}{2\pi}|\theta_{n}|;$
\item If $|\theta_{n}|/2|\leq|\omega_{k}|/2$, then we know $\pi/2\geq|\theta_{n}/2+\omega_{k}|\geq|\omega_{k}|-|\theta_{n}|/2\geq|\omega_{k}/2|\geq|\theta_{n}/2|$.
This implies that $\sin|\theta_{n}/2+\omega_{k}|\geq\sin|\theta_{n}/2|\geq\frac{1}{\pi}|\theta_{n}|$.
\item Otherwise, (i.e.~$|\omega_{k}|/2<|\theta_{n}|/2<2\,|\omega_{k}|$),
one can adjust $c_{n}$ to be either $1/4$ or $1/64$ (namely, increasing
it or decreasing it by $4$ times). After doing so, it is easily seen
that $\theta_{n}$ must satisfy one of the two conditions above, thereby
guaranteeing that $\sin|\theta_{n}/2+\omega_{k}|\geq\frac{1}{2\pi}|\theta_{n}|$.
\end{itemize}
This completes the proof for the claim (\ref{eq:minimax-lb-sin-claim}).
Combining (\ref{eq:minimax-lb-sin-claim}) with (\ref{eq:atustar+atu-temp}),
we arrive at
\[
\big|\bm{a}^{\top}\bm{u}_{l}^{\star}+\bm{a}^{\top}\bm{u}_{l}\big|\geq\frac{1}{8\pi}\big(|\theta_{n}|\cdot|\bm{a}^{\top}\bm{u}_{l}^{\star}|+|\theta_{n}|\cdot|\bm{a}^{\top}\bm{u}_{k}^{\star}|\big)\geq\frac{1}{8\pi^{2}}\big(\theta_{n}^{2}\cdot|\bm{a}^{\top}\bm{u}_{l}^{\star}|+|\theta_{n}|\cdot|\bm{a}^{\top}\bm{u}_{k}^{\star}|\big),
\]
where the last step holds since $|\theta_{n}|\leq1$. Combining this
with (\ref{eq:atustar-atu-temp}) finishes the proof of the claim
(\ref{claim:min-atu}).

\subsection{Proof of the lower bound (\ref{eq:lower-bound-claim2})}

\paragraph{Step 1: constructing a collection of hypotheses.} Consider
the following hypotheses regarding the eigen-decomposition of the
covariance matrix:
\begin{align*}
\mathcal{H}_{0} & \,:\,\bm{s}_{i}\overset{\mathrm{i.i.d.}}{\sim}\mathcal{N}(\bm{0},\bm{\Sigma}^{\star}+\sigma^{2}\bm{I}_{p}),\quad1\leq i\leq n;\\
\mathcal{H}_{1} & \,:\,\bm{s}_{i}\overset{\mathrm{i.i.d.}}{\sim}\mathcal{N}(\bm{0},\widetilde{\bm{\Sigma}}+\sigma^{2}\bm{I}_{p}),\quad1\leq i\leq n.
\end{align*}
Here, the covariance matrix $\widetilde{\bm{\Sigma}}$ is defined
to be
\[
\widetilde{\bm{\Sigma}}\coloneqq\lambda_{l}^{\star}\widetilde{\bm{u}}_{l}\widetilde{\bm{u}}_{l}^{\top}+\sum_{i:i\neq l}\lambda_{i}^{\star}\bm{u}_{i}^{\star}\bm{u}_{i}^{\star\top},
\]
where $\widetilde{\bm{u}}_{l}$ is defined as
\begin{align*}
\widetilde{\bm{u}}_{l} & :=\frac{\bm{u}_{l}^{\star}+\delta_{n}\bm{a}_{\perp}}{\|\bm{u}_{l}^{\star}+\delta_{n}\bm{a}_{\perp}\|_{2}}=\frac{\bm{u}_{l}^{\star}+\delta_{n}\bm{a}_{\perp}}{\sqrt{1+\delta_{n}^{2}}}\qquad\text{with}\quad\bm{a}_{\perp}:=\frac{\bm{P}_{\bm{U}^{\star\perp}}\bm{a}}{\|\bm{P}_{\bm{U}^{\star\perp}}\bm{a}\|_{2}}
\end{align*}
for some $0<\delta_{n}<1$ to be specified later. We note that $\bm{u}_{i}^{\star\top}\bm{a}_{\perp}=0$
for all $1\leq i\leq r$ and $\bm{u}_{i}^{\star\top}\widetilde{\bm{u}}_{l}=0$
for all $i\neq l$. As can be straightforwardly verified, one has
\begin{equation}
\|\widetilde{\bm{u}}-\bm{u}_{l}^{\star}\|_{2}\leq\Big(1-\frac{1}{\sqrt{1+\delta_{n}^{2}}}\Big)\|\bm{u}_{l}^{\star}\|_{2}+\frac{\delta_{n}}{\sqrt{1+\delta_{n}^{2}}}\|\bm{a}_{\perp}\|_{2}=\frac{\sqrt{1+\delta_{n}^{2}}-1+\delta_{n}}{\sqrt{1+\delta_{n}^{2}}}\leq2\delta_{n},\label{eq:minimax-u-dist}
\end{equation}
where the last step holds since $\sqrt{1+\delta_{n}^{2}}\leq1+\delta_{n}$
for $\delta_{n}>0$.

In the sequel, we denote by $\mathbb{P}^{0}$ and $\mathbb{P}^{1}$
the distribution of $\bm{S}$ under the hypothesis $\mathcal{H}_{0}$
and $\mathcal{H}_{1}$, respectively. We also let $\mathbb{P}_{i}^{0}$
and $\mathbb{P}_{i}^{1}$ denote the distribution of $\bm{s}_{i}$
($i$-th column of $\bm{S}$) under $\mathcal{H}_{0}$ and $\mathcal{H}_{1}$,
respectively.

\paragraph{Step 2: bounding the KL divergence between hypotheses.}Let
us define vector $\widehat{\bm{u}}_{l}$ as
\[
\widehat{\bm{u}}_{l}:=\frac{\bm{u}_{l}^{\star}-\frac{1}{\delta_{n}}\bm{a}_{\perp}}{\|\bm{u}_{l}^{\star}-\frac{1}{\delta_{n}}\bm{a}_{\perp}\|_{2}}=\frac{\bm{u}_{l}^{\star}-\frac{1}{\delta_{n}}\bm{a}_{\perp}}{\sqrt{1+\frac{1}{\delta_{n}^{2}}}}.
\]
where the last step holds since $\bm{u}_{l}^{\star}$ is orthogonal
to $\bm{a}_{\perp}$. Note that $\widehat{\bm{u}}_{l}$ is a unit
vector orthogonal to the subspace spanned by $\widetilde{\bm{u}}_{l}$
and $\{\bm{u}_{i}^{\star}\}_{i\neq l}$, namely, $\widehat{\bm{u}}_{l}^{\top}\widetilde{\bm{u}}_{l}=0$
and $\bm{u}_{i}^{\star\top}\widehat{\bm{u}}_{l}=0$ for all $i\neq l$.
Similar to the proof for the claim (\ref{eq:lower-bound-claim1}),
one can derive
\begin{align*}
\mathsf{tr}\big((\bm{\Sigma}^{\star}+\sigma^{2}\bm{I}_{p})^{-1}(\widetilde{\bm{\Sigma}}+\sigma^{2}\bm{I}_{p})\big) & =\mathsf{tr}\bigg(\Big(\frac{1}{\lambda_{l}^{\star}+\sigma^{2}}\bm{u}_{l}^{\star}\bm{u}_{l}^{\star\top}+\frac{1}{\sigma^{2}}\bm{a}_{\perp}\bm{a}_{\perp}^{\top}\Big)\Big((\lambda_{l}^{\star}+\sigma^{2})\widetilde{\bm{u}}_{l}\widetilde{\bm{u}}_{l}^{\top}+\sigma^{2}\widehat{\bm{u}}_{l}\widehat{\bm{u}}_{l}^{\top}\Big)\bigg)+p-2\\
 & =(\widetilde{\bm{u}}_{l}^{\top}\bm{u}_{l}^{\star})^{2}+\frac{\sigma^{2}}{\lambda_{l}^{\star}+\sigma^{2}}(\widehat{\bm{u}}_{l}^{\top}\bm{u}_{l}^{\star})^{2}+\frac{\lambda_{l}^{\star}+\sigma^{2}}{\sigma^{2}}(\widetilde{\bm{u}}_{l}^{\top}\bm{a}_{\perp})^{2}+(\widehat{\bm{u}}_{l}^{\top}\bm{a}_{\perp})^{2}+p-2\\
 & =\frac{1}{1+\delta_{n}^{2}}\bigg(1+\frac{\lambda_{l}^{\star}+\sigma^{2}}{\sigma^{2}}\delta_{n}^{2}\bigg)+\frac{1}{1+\delta_{n}^{2}}\bigg(\frac{\sigma^{2}}{\lambda_{l}^{\star}+\sigma^{2}}\delta_{n}^{2}+1\bigg)+p-2\\
 & =\frac{\lambda_{l}^{\star2}}{(\lambda_{l}^{\star}+\sigma^{2})\sigma^{2}}\frac{\delta_{n}^{2}}{1+\delta_{n}^{2}}+p,
\end{align*}
where the second step is due to $\bm{u}_{i}^{\star\top}\bm{a}_{\perp}=0$
and the third line follows from the following facts:
\begin{align*}
\widetilde{\bm{u}}_{l}^{\top}\bm{u}_{l}^{\star} & =\frac{\bm{u}_{l}^{\star\top}\bm{u}_{l}^{\star}+\delta_{n}\bm{a}_{\perp}^{\top}\bm{u}_{l}^{\star}}{\|\bm{u}_{l}^{\star}+\delta_{n}\bm{a}_{\perp}\|_{2}}=\frac{1}{\sqrt{1+\delta_{n}^{2}}};\\
\widetilde{\bm{u}}_{l}^{\top}\bm{a}_{\perp} & =\frac{\bm{u}_{l}^{\star\top}\bm{a}_{\perp}+\delta_{n}\bm{a}_{\perp}^{\top}\bm{a}_{\perp}}{\|\bm{u}_{l}^{\star}+\delta_{n}\bm{a}_{\perp}\|_{2}}=\frac{\delta_{n}}{\sqrt{1+\delta_{n}^{2}}};\\
\widehat{\bm{u}}_{l}^{\top}\bm{u}_{l}^{\star} & =\frac{\bm{u}_{l}^{\star\top}\bm{u}_{l}^{\star}-\frac{1}{\delta_{n}}\bm{a}_{\perp}^{\top}\bm{u}_{l}^{\star}}{\|\bm{u}_{l}^{\star}-\frac{1}{\delta_{n}}\bm{a}_{\perp}\|_{2}}=\frac{\delta_{n}}{\sqrt{1+\delta_{n}^{2}}};\\
\widehat{\bm{u}}_{l}^{\top}\bm{a}_{\perp} & =\frac{\bm{u}_{l}^{\star\top}\bm{a}_{\perp}-\frac{1}{\delta_{n}}\bm{a}_{\perp}^{\top}\bm{a}_{\perp}}{\|\bm{u}_{l}^{\star}-\frac{1}{\delta_{n}}\bm{a}_{\perp}\|_{2}}=-\frac{1}{\sqrt{1+\delta_{n}^{2}}}.
\end{align*}
As a consequence, we can upper bound the KL divergence as follows
\begin{align*}
\mathsf{KL}\big(\mathbb{P}^{1}\parallel\mathbb{P}^{0}\big) & =\sum_{i=1}^{n}\mathsf{KL}\big(\mathbb{P}_{i}^{1}\parallel\mathbb{P}_{i}^{0}\big)=\frac{1}{2}\sum_{i=1}^{n}\big(\mathsf{tr}\big((\bm{\Sigma}^{\star}+\sigma^{2}\bm{I}_{p})^{-1}(\widetilde{\bm{\Sigma}}+\sigma^{2}\bm{I}_{p})\big)-p\big)\\
 & =\frac{n\lambda_{l}^{\star2}}{2(\lambda_{l}^{\star}+\sigma^{2})\sigma^{2}}\frac{\delta_{n}^{2}}{1+\delta_{n}^{2}}\leq\frac{\delta_{n}^{2}n\lambda_{l}^{\star2}}{2(\lambda_{l}^{\star}+\sigma^{2})\sigma^{2}}.
\end{align*}

\paragraph{Step 3: invoking Fano's inequality.} 

From the preceding upper bound on the KL divergence, it is easy to
see that $\mathrm{\mathsf{KL}}(\mathbb{P}^{1}\parallel\mathbb{P}^{0})\le1/16$
if we choose
\begin{equation}
\delta_{n}=c_{n}\sqrt{\frac{(\lambda_{l}^{\star}+\sigma^{2})\sigma^{2}}{\lambda_{l}^{\star2}n}}\leq1,\label{eq:minimax-lb2-delta-value}
\end{equation}
where $c_{n}\asymp1$ obeys $c_{n}\in\{1/64,\,1/16,\,1/4\}$ and the
last step holds due to the assumption (\ref{eq:minimax-sample-size}).
It follows from Fano's inequality \cite[Theorem 2]{tsybakov2009introduction}
that
\[
p_{e}:=\inf_{\psi}\max\big\{\mathbb{P}\{\psi\;{\rm rejects}\;\mathcal{H}_{0}\mid\mathcal{H}_{0}\},\,\mathbb{P}\{\psi\;{\rm rejects}\;\mathcal{H}_{1}\mid\mathcal{H}_{1}\}\big\}\ge1/5,
\]
where the infimum is taken over all tests. Further, we know from (\ref{eq:minimax-u-dist}),
(\ref{eq:minimax-lb2-delta-value}) and $c_{n}\leq1/4$ that
\[
\|\widetilde{\bm{u}}-\bm{u}_{l}^{\star}\|_{2}\leq\sqrt{\frac{(\lambda_{l}^{\star}+\sigma^{2})\sigma^{2}}{\lambda_{l}^{\star2}n}},
\]
namely, $\bm{\Sigma}_{1}\in\mathcal{M}_{2}(\bm{\Sigma}^{\star})$.

Next, let us continue to control $\min\big|\bm{a}^{\top}\widetilde{\bm{u}}_{l}\pm\bm{a}^{\top}\bm{u}_{l}^{\star}\big|$.
Our goal is to show
\[
\min\big|\bm{a}^{\top}\widetilde{\bm{u}}_{l}\pm\bm{a}^{\top}\bm{u}_{l}^{\star}\big|\gtrsim\delta_{n}\|\bm{P}_{\bm{U}^{\star\perp}}\bm{a}\|_{2},
\]
and we shall use the same argument as for (\ref{claim:min-atu}) to
prove it. Towards this, let us first consider $\big|\bm{a}^{\top}\widetilde{\bm{u}}_{l}-\bm{a}^{\top}\widetilde{\bm{u}}_{l}\big|$.
By construction, one can derive
\begin{align}
\bm{a}^{\top}\widetilde{\bm{u}}_{l}-\bm{a}^{\top}\bm{u}_{l}^{\star} & =\frac{\bm{a}^{\top}\bm{u}_{l}^{\star}+\delta_{n}\bm{a}^{\top}\bm{a}_{\perp}}{\sqrt{1+\delta_{n}^{2}}}-\bm{a}^{\top}\bm{u}_{l}^{\star}=\underbrace{\frac{\delta_{n}}{\sqrt{1+\delta_{n}^{2}}}\|\bm{P}_{\bm{U}^{\star\perp}}\bm{a}\|_{2}}_{=:\,\eta_{1}}-\underbrace{\bigg(1-\frac{1}{\sqrt{1+\delta_{n}^{2}}}\bigg)\bm{a}^{\top}\bm{u}_{l}^{\star}}_{=:\,\eta_{2}}\label{eq:at-utilde-at-ustar-dist}
\end{align}
where the last step holds because
\begin{align*}
\bm{a}^{\top}\bm{a}_{\perp} & =\bm{a}^{\top}\bm{P}_{\bm{U}^{\star\perp}}\bm{a}/\|\bm{P}_{\bm{U}^{\star\perp}}\bm{a}\|_{2}=(\bm{P}_{\bm{U}^{\star\perp}}\bm{a})^{\top}\bm{P}_{\bm{U}^{\star\perp}}\bm{a}/\|\bm{P}_{\bm{U}^{\star\perp}}\bm{a}\|_{2}=\|\bm{P}_{\bm{U}^{\star\perp}}\bm{a}\|_{2}.
\end{align*}
Moreover, it is straightforward to verify that
\begin{equation}
\frac{1}{4}\delta_{n}^{2}\leq1-\frac{1}{\sqrt{1+\delta_{n}^{2}}}\leq\sqrt{1+\delta_{n}^{2}}-1\leq\frac{1}{2}\delta_{n}^{2}\label{eq:eta_2-bound}
\end{equation}
for $0<\delta_{n}<1$. With these basic facts in place, let us first
choose the pre-factor $c_{n}\asymp1$ in \ref{eq:minimax-lb2-delta-value}
to be $c_{n}=1/16$ for the moment, and compare the two terms on the
right-hand side of (\ref{eq:at-utilde-at-ustar-dist}).
\begin{itemize}
\item If $|\eta_{1}|\geq2\,|\eta_{2}|$, then one has 
\[
\big|\bm{a}^{\top}\widetilde{\bm{u}}_{l}-\bm{a}^{\top}\bm{u}_{l}^{\star}\big|\geq|\eta_{1}|-|\eta_{2}|\geq\frac{|\eta_{1}|}{2}\geq\frac{\delta_{n}}{4}\|\bm{P}_{\bm{U}^{\star\perp}}\bm{a}\|_{2},
\]
where we have used the fact that $\delta_{n}\leq1$.
\item If $|\eta_{1}|\leq\,|\eta_{2}|/2$, then we know that
\[
\big|\bm{a}^{\top}\widetilde{\bm{u}}_{l}-\bm{a}^{\top}\bm{u}_{l}^{\star}\big|\geq|\eta_{2}|-|\eta_{1}|\geq|\eta_{1}|\geq\frac{\delta_{n}}{2}\|\bm{P}_{\bm{U}^{\star\perp}}\bm{a}\|_{2}
\]
as long as $\delta_{n}\leq1$.
\item Otherwise, consider the case where $|\eta_{2}|/2<|\eta_{1}|<2\,|\eta_{2}|$.
In this case, we can adjust the pre-factor $c_{n}$ to be $1/4$.
By doing so, $|\eta_{1}|$ increases by at most $4$ times, while
$|\eta_{2}|$ increases by at least $8$ times (according to (\ref{eq:eta_2-bound})).
As a result, the new values of $\eta_{1}$ and $\eta_{2}$ satisfy
$|\eta_{1}|\leq\,|\eta_{2}|/2$, thus belonging to the second case
discussed above and hence $\big|\bm{a}^{\top}\widetilde{\bm{u}}_{l}-\bm{a}^{\top}\bm{u}_{l}^{\star}\big|\geq\frac{\delta_{n}}{4}\,\|\bm{P}_{\bm{U}^{\star\perp}}\bm{a}\|_{2}.$
Clearly, we can also adjust $c_{n}$ to be $1/64$ so as to meet the
condition of the first case discussed above.
\end{itemize}
To sum up, the above analysis reveals that: by properly choosing the
constants $\{c_{n}\}$ in (\ref{eq:minimax-lb2-delta-value}), one
can guarantee that
\[
\big|\bm{a}^{\top}\widetilde{\bm{u}}_{l}-\bm{a}^{\top}\bm{u}_{l}^{\star}\big|\geq\frac{\delta_{n}}{4}\|\bm{P}_{\bm{U}^{\star\perp}}\bm{a}\|_{2}.
\]
Similarly, we can also derive
\begin{align*}
\big|\bm{a}^{\top}\widetilde{\bm{u}}_{l}+\bm{a}^{\top}\bm{u}_{l}^{\star}\big| & =\bigg|\frac{\bm{a}^{\top}\bm{u}_{l}^{\star}+\delta_{n}\bm{a}^{\top}\bm{a}_{\perp}}{\sqrt{1+\delta_{n}^{2}}}+\bm{a}^{\top}\bm{u}_{l}^{\star}\bigg|=\bigg|\frac{\delta_{n}}{\sqrt{1+\delta_{n}^{2}}}\|\bm{P}_{\bm{U}^{\star\perp}}\bm{a}\|_{2}+\bigg(1+\frac{1}{\sqrt{1+\delta_{n}^{2}}}\bigg)\bm{a}^{\top}\bm{u}_{l}^{\star}\bigg|\\
 & \geq\bigg|\frac{\delta_{n}}{\sqrt{1+\delta_{n}^{2}}}\|\bm{P}_{\bm{U}^{\star\perp}}\bm{a}\|_{2}-\bigg(1+\frac{1}{\sqrt{1+\delta_{n}^{2}}}\Big)\big|\bm{a}^{\top}\bm{u}_{l}^{\star}\big|\bigg|\\
 & \gtrsim\frac{\delta_{n}}{\sqrt{1+\delta_{n}^{2}}}\,\|\bm{P}_{\bm{U}^{\star\perp}}\bm{a}\|_{2}
\end{align*}
Taking these two relations collectively yields the advertised bound:
\[
\min\big|\bm{a}^{\top}\bm{u}_{l}\pm\bm{a}^{\top}\bm{u}_{l}^{\star}\big|\gtrsim\delta_{n}\,\|\bm{P}_{\bm{U}^{\star\perp}}\bm{a}\|_{2}.
\]

As a consequence, one can readily apply the standard reduction scheme
in \cite[Chapter 2.2]{tsybakov2009introduction} again to arrive that

\begin{align*}
\inf_{u_{\bm{a},l}}\sup_{\bm{\Sigma}\in\mathcal{M}_{2}(\bm{\Sigma}^{\star})}\mathbb{E}\Big[\min\big|u_{\bm{a},l}\pm\bm{a}^{\top}\bm{u}_{l}(\bm{\Sigma})\big|\Big] & \gtrsim p_{e}\min\big|\bm{a}^{\top}\widetilde{\bm{u}}_{l}\pm\bm{a}^{\top}\bm{u}_{l}^{\star}\big|\gtrsim\min\big|\bm{a}^{\top}\widetilde{\bm{u}}_{l}\pm\bm{a}^{\top}\bm{u}_{l}^{\star}\big|\\
 & \gtrsim c_{n}\sqrt{\frac{(\lambda_{l}^{\star}+\sigma^{2})\sigma^{2}}{\lambda_{l}^{\star2}n}}\|\bm{P}_{\bm{U}^{\star\perp}}\bm{a}\|_{2}\gtrsim\sqrt{\frac{(\lambda_{l}^{\star}+\sigma^{2})\sigma^{2}}{\lambda_{l}^{\star2}n}}\|\bm{P}_{\bm{U}^{\star\perp}}\bm{a}\|_{2}
\end{align*}
where the last step holds since $\min_{n}c_{n}=1/64$.

\section{Proof for the lower bound of the plug-in estimator (Theorem \ref{thm:lower-bound-plugin})}

\label{sec:Proof-for-lower-bound-plugin}

Evidently, it is sufficient to establish the lower bound for the rank-$1$ case (i.e.~$r=1$), 
which forms the content of this section.
To begin with, let us decompose the leading eigenvector $\bm{u}$ of $\bm{M}$ as follows:
\begin{align*}
\bm{u} = \bm{u}^{\star}\cos\theta + \bm{u}_{\perp}\sin\theta,
	\qquad \text{with }\theta \in \big[0, \pi/2 \big],
\end{align*}
where as before, $\bm{u}_{\perp}$ denotes some unit vector perpendicular to $\bm{u}^{\star}$.
Denote by $s \coloneqq \mathrm{sign} (\bm{a}^{\top}\bm{u}_{\perp})$ the sign of $\bm{a}^{\top}\bm{u}_{\perp}$. 
Armed with these, one can express the plug-in estimator $u_{\bm{a}}^{\mathsf{plugin}}$ as
\begin{align*}
u_{\bm{a}}^{\mathsf{plugin}} = \bm{a}^{\top}\bm{u} = \bm{a}^{\top}\bm{u}^{\star}\cos\theta + \bm{a}^{\top}\bm{u}_{\perp}\sin\theta = \bm{a}^{\top}\bm{u}^{\star}\cos\theta + s \, |\bm{a}^{\top}\bm{u}_{\perp}|\sin\theta, 
\end{align*}
which in turn leads to
\begin{align*}
\mathsf{dist}\left(u_{\bm{a}}^{\mathsf{plugin}},\bm{a}^{\top}\bm{u}^{\star}\right) &= \min\left|\bm{a}^{\top}\bm{u} \pm \bm{a}^{\top}\bm{u}^{\star}\right| 
= \left| |\bm{a}^{\top}\bm{u}^{\star}| - |\bm{a}^{\top}\bm{u} |\right| \\
&= \Big| | \bm{a}^{\top}\bm{u}^{\star} | - \big|\bm{a}^{\top}\bm{u}^{\star}\cos\theta + s\,|\bm{a}^{\top}\bm{u}_{\perp}|\sin\theta  \big| \Big|.
\end{align*}

In view of the analysis in Appendix~\ref{subsec:Proof-of-lemma:a-top-P-U-perp-u-perp}, one can see that: conditioned on $\theta$, $\bm{u}_{\perp}$ is uniformly distributed over the unit sphere when restricted to the subspace spanned by the columns of $\bm{u}^{\star \perp}$.
Consequently, we have
\begin{align*}
\mathbb{P}\big\{ s = 1 \,\mid\, |\bm{a}^{\top}\bm{u}_{\perp}|, \theta\big\} = \mathbb{P}\big \{s = -1 \,\mid\, |\bm{a}^{\top}\bm{u}_{\perp}|, \theta\big\} = \frac{1}{2}.
\end{align*}
This implies the independence between $s$ and $|\bm{a}^{\top}\bm{u}_{\perp}|$ as well as $\theta$, which further reveals that
\begin{align*}
s = \argmax_{s' = \pm 1} \Big| | \bm{a}^{\top}\bm{u}^{\star}| - \big| \bm{a}^{\top}\bm{u}^{\star}\cos\theta + s' \, |\bm{a}^{\top}\bm{u}_{\perp}|\sin\theta \big| \Big| 
\end{align*}
with probability $1/2$. 
Therefore, it is seen that
\begin{align*}
\mathsf{dist}\left(u_{\bm{a}}^{\mathsf{plugin}},\bm{a}^{\top}\bm{u}^{\star}\right) = \max_{s' = \pm 1}\Big| | \bm{a}^{\top}\bm{u}^{\star} | - \big|\bm{a}^{\top}\bm{u}^{\star}\cos\theta + s' \, |\bm{a}^{\top}\bm{u}_{\perp}|\sin\theta \big| \Big| 
\end{align*}
with probability $1/2$. 
In the following, we seek to lower bound $\max_{s' = \pm 1}\big| | \bm{a}^{\top}\bm{u}^{\star} | - |\bm{a}^{\top}\bm{u}^{\star}\cos\theta + s' \, |\bm{a}^{\top}\bm{u}_{\perp}|\sin\theta | \big|$, 
dividing into two cases based on the relative size of  $\bm{a}^{\top}\bm{u}^{\star}\cos\theta $ compared to $ |\bm{a}^{\top}\bm{u}_{\perp}|\sin\theta$. Without loss of generality, let us assume that $\bm{a}^{\top}\bm{u}^{\star} \geq 0$ in the sequel, and recall that $\theta \in [0,\pi/2]$.

\begin{itemize}
	\item
In the case where $\bm{a}^{\top}\bm{u}^{\star}\cos\theta \geq |\bm{a}^{\top}\bm{u}_{\perp}|\sin\theta$, one can demonstrate that
\begin{align*}
\max_{s' = \pm 1} \Big| \bm{a}^{\top}\bm{u}^{\star} - \big|\bm{a}^{\top}\bm{u}^{\star}\cos\theta + s'\,|\bm{a}^{\top}\bm{u}_{\perp}|\sin\theta \big| \Big| &=  \bm{a}^{\top}\bm{u}^{\star}(1-\cos\theta)  + |\bm{a}^{\top}\bm{u}_{\perp}|\sin\theta  \\
	&\geq (1-\cos\theta)|\bm{a}^{\top}\bm{u}^{\star}| \geq \frac{1-\cos^2\theta }{2}  |\bm{a}^{\top}\bm{u}^{\star}| \\
& \gtrsim \frac{\sigma^{2}n}{\lambda_{1}^{\star2}} \left|\bm{a}^{\top}\bm{u}^{\star}\right| 
\end{align*}
with high probability, where the last step results from \eqref{eq:cos-bound}.

\item
On the other hand, if instead $\bm{a}^{\top}\bm{u}^{\star}\cos\theta < |\bm{a}^{\top}\bm{u}_{\perp}|\sin\theta$, then one can deduce that
\begin{align*}
\max_{s' = \pm 1} \Big| \bm{a}^{\top}\bm{u}^{\star} - \big|\bm{a}^{\top}\bm{u}^{\star}\cos\theta + s' \, \big|\bm{a}^{\top}\bm{u}_{\perp}|\sin\theta \big| \Big| 
	&= \max_{s' = \pm 1} \Big| \bm{a}^{\top}\bm{u}^{\star} - \big| s'\bm{a}^{\top}\bm{u}^{\star}\cos\theta + |\bm{a}^{\top}\bm{u}_{\perp}|\sin\theta \big | \Big| \\
&\overset{\mathrm{(i)}}{=} \max_{s' = \pm 1} \Big| \bm{a}^{\top}\bm{u}^{\star}  - |\bm{a}^{\top}\bm{u}_{\perp}|\sin\theta - s'\bm{a}^{\top}\bm{u}^{\star}\cos\theta \Big| \\
&\overset{\mathrm{(ii)}}{\geq} \left|\bm{a}^{\top}\bm{u}^{\star} \right| \cos\theta  \overset{\mathrm{(iii)}}{\asymp} \left|\bm{a}^{\top}\bm{u}^{\star}\right| \\
&  \overset{\mathrm{(iv)}}{\gtrsim}  \frac{\sigma^{2}n}{\lambda_{1}^{\star2}} \left|\bm{a}^{\top}\bm{u}^{\star}\right|.
\end{align*}
Here, (i) is valid since $\bm{a}^{\top}\bm{u}^{\star}\cos\theta < |\bm{a}^{\top}\bm{u}_{\perp}|\sin\theta$; (ii) holds given that $\max |a\pm b|\geq |b|$ for any $a,b$; (iii) follows since, according to \eqref{eq:cos-bound}, $\cos \theta \asymp 1$ holds with high probability; 
	and	(iv) arises from the assumption (\ref{eq:eigengap-condition-iid}).

\end{itemize}
Combining the preceding two cases, we can readily conclude that
\begin{align*}
\mathsf{dist}\left(u_{\bm{a}}^{\mathsf{plugin}},\bm{a}^{\top}\bm{u}^{\star}\right) \gtrsim \frac{\sigma^{2}n}{\lambda_{1}^{\star2}} \left|\bm{a}^{\top}\bm{u}^{\star}\right|
\end{align*}
with probability at least $1/3$.

\section{Technical lemmas\label{sec:Auxiliary-lemma}}

This section collects a few technical lemmas that prove useful in
the analysis of our main results. In what follows, we shall start
by stating the precise statements of these lemmas, followed by the
proofs for each of them. 

\begin{lemma}\label{lemma:covariance-Gaussian-concentration}Let
$\{\bm{h}_{i}\}_{i=1}^{n}$ be a sequence of independent zero-mean
Gaussian random vectors in $\mathbb{R}^{r}$ with covariance matrix
$\sigma^{2}\bm{I}_{r}$, and let $\bm{a}=[a_{i}]_{1\leq i\leq n}\in\mathbb{R}^{n}$
be a fixed vector. Then with probability at least $1-O\left(n^{-10}\right)$,
one has
\begin{align}
\Big\|\sum_{1\leq i\leq n}a_{i}\big(\bm{h}_{i}\bm{h}_{i}^{\top}-\sigma^{2}\bm{I}_{r}\big)\Big\| & \leq C_{1}\sigma^{2}\big(\|\bm{a}\|_{2}\sqrt{r\log n}+\|\bm{a}\|_{\infty}(r\log n+\log^{2}n)\big)\nonumber \\
 & \leq C_{2}\sigma^{2}\|\bm{a}\|_{\infty}\big(\sqrt{rn\log n}+r\log n\big)\label{eq:lem-covariance-Gaussian}
\end{align}
for some sufficiently large constants $C_{1},C_{2}>0$. Here, $\|\bm{a}\|_{\infty}:=\max_{1\leq i\leq n}|a_{i}|$.\end{lemma}

\begin{lemma}\label{lemma:prod-Gaussian-concentration}Let $\{\bm{h}_{i}\}_{i=1}^{n}$
and $\{\bm{g}_{i}\}_{i=1}^{n}$ be two independent sequences of standard
Gaussian random vectors in $\mathbb{R}^{r}$ and $\mathbb{R}^{p}$,
respectively. Then with probability at least $1-O\left(n^{-10}\right)$,
the following holds:
\begin{align*}
\Big\|\sum_{1\leq i\leq n}\bm{h}_{i}\bm{g}_{i}^{\top}\Big\| & \leq C_{3}\big(\sqrt{pn\log n}+\sqrt{pr}\log n\big)
\end{align*}
where $C_{3}>0$ is some sufficiently large constant. \end{lemma}

\begin{lemma}\label{lemma:eps-net}Let $\{X_{i}\}_{i=1}^{n}$ be
a sequence of independent random variables in $\mathbb{R}$, and let
$\mathcal{I}$ be an interval in $\mathbb{R}$. Consider a collection
of functions $\{f_{i}\}_{i=1}^{n}$ from $\mathbb{R}\times\mathcal{I}$
to $\mathbb{R}$, and we suppose that 
\begin{enumerate}
\item for any fixed $\lambda\in\mathcal{I}$, with probability at least
$1-\delta_{1}$,
\begin{align*}
\bigg|\sum_{1\leq i\leq n}f_{i}(X_{i},\lambda)\bigg| & \le\frac{\varepsilon}{2};
\end{align*}
\item with probability at least $1-\delta_{2}$, 
\[
\sup_{\lambda\in\mathcal{I}}\bigg|\frac{\mathrm{d}}{\mathrm{d}\lambda}\sum_{1\leq i\leq n}f_{i}(X_{i},\lambda)\bigg|\le L.
\]
\end{enumerate}
Then with probability exceeding $1-\frac{8L|\mathcal{I}|}{\varepsilon}\delta_{1}-\delta_{2}$,
one has
\[
\sup_{\lambda\in\mathcal{I}}\bigg|\sum_{1\leq i\leq n}f_{i}(X_{i},\lambda)\bigg|\le\varepsilon.
\]
\end{lemma}

Next, we record an eigenvalue interlacing lemma, which has been documented
in \citet[Corollary 4.3.37]{horn2012matrix}.

\begin{lemma}[Poincar\'e separation theorem]\label{lemma:eigval-interlacing}Let
$\bm{M}$ be a symmetric matrix in $\mathbb{R}^{n\times n}$ and $\bm{U}$
be an orthonormal matrix in $\mathbb{R}^{n\times r}$ satisfying $\bm{U}^{\top}\bm{U}=\bm{I}_{r}$.
Then one has
\[
\lambda_{n-r+i}(\bm{M})\leq\lambda_{i}(\bm{U}^{\top}\bm{M}\bm{U})\leq\lambda_{i}(\bm{M}),\qquad1\leq i\leq r,
\]
where $\lambda_{i}(\bm{A})$ denote the $i$-th largest eigenvalue
of matrix $\bm{A}$.\end{lemma}

\subsection{Proof of Lemma \ref{lemma:covariance-Gaussian-concentration}}

As can be easily seen, the second inequality in (\ref{eq:lem-covariance-Gaussian})
follows immediately from the elementary bound $\|\bm{a}\|_{2}\leq\|\bm{a}\|_{\infty}\sqrt{n}$.
Hence, the proof boils down to justifying the first inequality in
(\ref{eq:lem-covariance-Gaussian}).

We shall invoke the truncated matrix Bernstein inequality \citep[Proposition A.7]{hopkins2016fast}
to control the spectral norm of $\sum_{i}a_{i}\big(\bm{h}_{i}\bm{h}_{i}^{\top}-\sigma^{2}\bm{I}_{r}\big)$,
which is a sum of independent zero-mean random matrices. To do so,
we need to bound three quantities: (1) the covariance of the sum $\sum_{i}a_{i}\big(\bm{h}_{i}\bm{h}_{i}^{\top}-\sigma^{2}\bm{I}_{r}\big)$,
(2) a high-probability upper bound $L$ on $\max_{i}\big\| a_{i}\big(\bm{h}_{i}\bm{h}_{i}^{\top}-\sigma^{2}\bm{I}_{r}\big)\big\|$,
(3) the expectation of the truncated summand $\max_{i}\mathbb{E}\big[\|a_{i}\big(\bm{h}_{i}\bm{h}_{i}^{\top}-\sigma^{2}\bm{I}_{r}\big)\|\ind\{\|a_{i}\big(\bm{h}_{i}\bm{h}_{i}^{\top}-\sigma^{2}\bm{I}_{r}\big)\|\geq L\}\big]$.
We shall look at each of them separately. 
\begin{enumerate}
\item Straightforward computation gives
\begin{equation}
\bm{\Sigma}:=\sum_{i=1}^{n}a_{i}^{2}\mathbb{E}\big[(\bm{h}_{i}\bm{h}_{i}^{\top}-\sigma^{2}\bm{I}_{r})^{2}\big]=(r+1)\sigma^{4}\sum_{i=1}^{n}a_{i}^{2}\bm{I}_{r}=(r+1)\sigma^{4}\|\bm{a}\|_{2}^{2}\bm{I}_{r}.\label{eq:Gaussian-covariance-concentration-cov}
\end{equation}
\item We now turn to bounding the spectral norm of each summand $a_{i}\big(\bm{h}_{i}\bm{h}_{i}^{\top}-\sigma^{2}\bm{I}_{r}\big)$,
which clearly satisfies
\[
\|a_{i}(\bm{h}_{i}\bm{h}_{i}^{\top}-\sigma^{2}\bm{I}_{r})\|\leq|a_{i}|\cdot(\|\bm{h}_{i}\|_{2}^{2}+\sigma^{2}).
\]
By virtue of the Gaussian concentration inequality \citep[Proposition 1.1]{hsu2012tail},
we obtain
\begin{equation}
\mathbb{P}\big\{\|\bm{h}_{i}\|_{2}^{2}-\sigma^{2}r\geq t\big\}\leq\exp\left(-\frac{1}{16}\min\Big\{\frac{t^{2}}{r\sigma^{4}},\,\frac{t}{\sigma^{2}}\Big\}\right).\label{eq:gaussian-vector-ineq}
\end{equation}
In particular, this implies that with probability at least $1-O\left(n^{-20}\right)$,
one has
\begin{equation}
\|\bm{h}_{i}\|_{2}^{2}\lesssim\sigma^{2}\big(r+\log n\big).\label{eq:gaussian-vector-norm}
\end{equation}
In what follows, we shall set
\begin{equation}
L:=C\sigma^{2}\big(r+\log n\big)\label{eq:Gaussian-covariance-concentration-L}
\end{equation}
 for some sufficiently large constant $C>0$. 
\item We then look the truncated mean. To this end, we observe that
\begin{align*}
\mathbb{E}\big[\|\bm{h}_{i}\|_{2}^{2}\ind\{\|\bm{h}_{i}\|_{2}^{2}\geq L\}\big] & \leq L\mathbb{P}\big\{\|\bm{h}_{i}\|_{2}^{2}\geq L\big\}+\int_{L}^{\infty}\mathbb{P}\big\{\|\bm{h}_{i}\|_{2}^{2}\geq t\big\}\,\mathrm{d}t\\
 & \leq O\left(n^{-20}\right)L+\int_{L}^{\infty}\mathbb{P}\big\{\|\bm{h}_{i}\|_{2}^{2}\geq t\big\}\,\mathrm{d}t.
\end{align*}
For any $t\geq L/2$, it is seen that $\min\left\{ t^{2}/(r\sigma^{4}),\,t/\sigma^{2}\right\} \geq t/\sigma^{2}$,
and hence
\begin{align*}
\int_{L}^{\infty}\mathbb{P}\big\{\|\bm{h}_{i}\|_{2}^{2}\geq t\big\}\,\mathrm{d}t & \leq\int_{L/2}^{\infty}\mathbb{P}\big\{\|\bm{h}_{i}\|_{2}^{2}-\sigma^{2}r\geq t\big\}\,\mathrm{d}t\leq\int_{L/2}^{\infty}\exp\left(-\frac{t}{16\sigma^{2}}\right)\,\mathrm{d}t\\
 & \lesssim\sigma^{2}\exp\left(-\frac{C(r+\log n)}{32}\right)\lesssim\frac{L}{n^{2}},
\end{align*}
provided that $C>0$ is sufficiently large. As a result, taking this
together with $\|\bm{h}_{i}\bm{h}_{i}^{\top}-\sigma^{2}\bm{I}_{r}\|\leq\|\bm{h}_{i}\|_{2}^{2}+\sigma^{2}$,
we arrive at
\begin{align}
R & \coloneqq\mathbb{E}\big[\|\bm{h}_{i}\bm{h}_{i}^{\top}-\sigma^{2}\bm{I}_{r}\|\ind\{\|\bm{h}_{i}\bm{h}_{i}^{\top}-\sigma^{2}\bm{I}_{r}\|\geq L\}\big]\leq\mathbb{E}\big[\big(\|\bm{h}_{i}\|_{2}^{2}+\sigma^{2}\big)\ind\{\|\bm{h}_{i}\|_{2}^{2}+\sigma^{2}\geq L\}\big]\lesssim\frac{L}{n^{2}}.\label{eq:Gaussian-covariance-concentration-R}
\end{align}
\end{enumerate}
With the preceding bounds in place, we can invoke the truncated matrix
Bernstein inequality \citep[Proposition A.7]{hopkins2016fast} to obtain
that: with probability at least $1-O\left(n^{-11}\right)$,
\begin{align*}
\Big\|\sum_{i}a_{i}\big(\bm{h}_{i}\bm{h}_{i}^{\top}-\sigma^{2}\bm{I}_{r}\big)\Big\| & \lesssim\sqrt{\|\bm{\Sigma}\|\log n}+\|\bm{a}\|_{\infty}(nR+L\log n)\\
 & \overset{(\mathrm{i})}{\asymp}\sqrt{\|\bm{\Sigma}\|\log n}+\|\bm{a}\|_{\infty}L\log n\\
 & \overset{(\mathrm{ii})}{\lesssim}\sqrt{\sum_{i=1}^{n}a_{i}^{2}(r+1)\sigma^{4}\log n}+\|\bm{a}\|_{\infty}\sigma^{2}(r\log n+\log^{2}n)\\
 & \overset{}{\lesssim}\sigma^{2}\big(\|\bm{a}\|_{2}\sqrt{r\log n}+\|\bm{a}\|_{\infty}(r\log n+\log^{2}n)\big),
\end{align*}
where (i) arises from (\ref{eq:Gaussian-covariance-concentration-R}),
and (ii) relies on (\ref{eq:Gaussian-covariance-concentration-cov})
and (\ref{eq:Gaussian-covariance-concentration-L}). This completes
the proof.

\subsection{Proof of Lemma \ref{lemma:prod-Gaussian-concentration}}

The proof strategy here is almost identical to that for Lemma \ref{lemma:covariance-Gaussian-concentration}
--- we shall apply the truncated matrix Bernstein inequality \citep[Proposition A.7]{hopkins2016fast}
to upper bound the spectral norm of $\sum_{1\leq i\leq n}\bm{h}_{i}\bm{g}_{i}^{\top}$,
which is a sum of independent zero-mean random matrices. Towards this,
we start by estimating several key quantities.
\begin{itemize}
\item In view of the independence between $\{\bm{h}_{i}\}_{i}$ and $\{\bm{g}_{i}\}_{i}$,
the covariance matrices can be computed as
\begin{align}
\bm{\Sigma}_{1} & :=\sum_{i=1}^{n}\mathbb{E}\big[\bm{h}_{i}\bm{g}_{i}^{\top}\bm{g}_{i}\bm{h}_{i}^{\top}\big]=\sum_{i=1}^{n}\mathbb{E}\big[\|\bm{g}_{i}\|_{2}^{2}\big]\mathbb{E}\big[\bm{h}_{i}\bm{h}_{i}^{\top}\big]=np\bm{I}_{r};\label{eq:prod-Gaussian-concentration-cov1}\\
\bm{\Sigma}_{2} & :=\sum_{i=1}^{n}\mathbb{E}\big[\bm{g}_{i}\bm{h}_{i}^{\top}\bm{h}_{i}\bm{g}_{i}^{\top}\big]=\sum_{i=1}^{n}\mathbb{E}\big[\|\bm{h}_{i}\|_{2}^{2}\big]\mathbb{E}\big[\bm{g}_{i}\bm{g}_{i}^{\top}\big]=nr\bm{I}_{p}.\label{eq:prod-Gaussian-concentration-cov2}
\end{align}
\item As for the spectral norm of each summand $\bm{h}_{i}\bm{g}_{i}^{\top}$,
we know from (\ref{eq:gaussian-vector-norm}) that with probability
at least $1-O\left(n^{-20}\right)$,
\[
\|\bm{h}_{i}\bm{g}_{i}^{\top}\|=\|\bm{h}_{i}\|_{2}\|\bm{g}_{i}\|_{2}\lesssim\sqrt{(r+\log n)(p+\log n)}\asymp\sqrt{pr}+\sqrt{p\log n}+\log n.
\]
Therefore, this suggests that we define
\begin{equation}
L:=C\big(\sqrt{pr}+\sqrt{p\log n}+\log n\big)\label{eq:prod-Gaussian-concentration-L}
\end{equation}
 for some sufficiently large constant $C>0$.
\item Next, we turn to the truncated mean. Observe that
\begin{align*}
\mathbb{E}\big[\|\bm{h}_{i}\bm{g}_{i}^{\top}\|\ind\{\|\bm{h}_{i}\bm{g}_{i}^{\top}\|\geq L\}\big] & \leq L\mathbb{P}\big\{\|\bm{h}_{i}\bm{g}_{i}^{\top}\|\geq L\big\}+\int_{L}^{\infty}\mathbb{P}\big\{\|\bm{h}_{i}\bm{g}_{i}^{\top}\|\geq t\big\}\,\mathrm{d}t\\
 & \leq O\left(n^{-20}\right)L+\int_{L}^{\infty}\mathbb{P}\big\{\|\bm{h}_{i}\bm{g}_{i}^{\top}\|\geq t\big\}\,\mathrm{d}t\\
 & \leq O\left(n^{-20}\right)L+\int_{L}^{\infty}\mathbb{P}\big\{\|\bm{h}_{i}\|_{2}^{2}\geq t\big\}\,\mathrm{d}t+\int_{L}^{\infty}\mathbb{P}\big\{\|\bm{g}_{i}\|_{2}^{2}\geq t\big\}\,\mathrm{d}t,
\end{align*}
where the last holds arises from the following bound due to the union
bound:
\[
\mathbb{P}\big\{\|\bm{h}_{i}\bm{g}_{i}^{\top}\|\geq t\big\}=\mathbb{P}\big\{\|\bm{h}_{i}\|_{2}^{2}\|\bm{g}_{i}\|_{2}^{2}\geq t^{2}\big\}\leq\mathbb{P}\big\{\|\bm{h}_{i}\|_{2}^{2}\geq t\big\}+\mathbb{P}\big\{\|\bm{g}_{i}\|_{2}^{2}\geq t\big\}.
\]
In addition, since $\min\left\{ t^{2}/r,\,t\right\} \geq t$ for all
$t\geq L/2\geq2r$, we can use (\ref{eq:gaussian-vector-ineq}) to
bound
\begin{align*}
\int_{L}^{\infty}\mathbb{P}\big\{\|\bm{h}_{i}\|_{2}^{2}\geq t\big\}\,\mathrm{d}t & \leq\int_{L}^{\infty}\mathbb{P}\Big\{\|\bm{h}_{i}\|_{2}^{2}-r\geq\frac{t}{2}\Big\}\,\mathrm{d}t\lesssim\int_{L/2}^{\infty}\mathbb{P}\big\{\|\bm{h}_{i}\|_{2}^{2}-r\geq t\big\}\,\mathrm{d}t\\
 & \leq\int_{L/2}^{\infty}\exp\left(-\frac{1}{16}\min\Big\{\frac{t^{2}}{r},\,t\Big\}\right)\,\mathrm{d}t\leq\int_{L/2}^{\infty}\exp\left(-\frac{t}{16}\right)\,\mathrm{d}t\\
 & \lesssim\exp\left(-\frac{C}{32}\sqrt{pr+p\log n+\log^{2}n}\right)\lesssim\frac{L}{n^{2}}.
\end{align*}
Clearly, the same bound also holds for $\int_{L}^{\infty}\mathbb{P}\big\{\|\bm{g}_{i}\|_{2}^{2}\geq t\big\}\,\mathrm{d}t$.
Therefore, combining these estimates yields
\begin{equation}
R:=\mathbb{E}\big[\|\bm{h}_{i}\bm{g}_{i}^{\top}\|\ind\{\|\bm{h}_{i}\bm{g}_{i}^{\top}\|\geq L\}\big]\lesssim\frac{L}{n^{2}}.\label{eq:prod-Gaussian-concentration-R}
\end{equation}
\end{itemize}
With these parameters in place, one can apply the truncated matrix
Bernstein inequality \citep[Proposition A.7]{hopkins2016fast} to demonstrate
that: with probability at least $1-O\left(n^{-10}\right)$,
\begin{align*}
\Big\|\sum_{1\leq i\leq n}\bm{h}_{i}\bm{g}_{i}^{\top}\Big\| & \lesssim\sqrt{(\|\bm{\Sigma}_{1}\|+\|\bm{\Sigma}_{2}\|)\log n}+nR+L\log n\\
 & \overset{(\mathrm{i})}{\asymp}\sqrt{\|\bm{\Sigma}_{1}\|\log n}+L\log n\\
 & \overset{(\mathrm{ii})}{\lesssim}\sqrt{pn\log n}+\sqrt{pr}\log n+\sqrt{p\log^{3}n}+\log^{2}n\\
 & \asymp\sqrt{pn\log n}+\sqrt{pr}\log n.
\end{align*}
Here, (i) uses (\ref{eq:prod-Gaussian-concentration-cov1}), (\ref{eq:prod-Gaussian-concentration-cov1})
and (\ref{eq:prod-Gaussian-concentration-R}); (ii) arises from (\ref{eq:prod-Gaussian-concentration-L}).
The proof is thus complete.

\subsection{Proof of Lemma \ref{lemma:eps-net}}

Let $\mathcal{N}$ be a $\frac{\varepsilon}{2L}$-covering of $\mathcal{I}$
with cardinality $|\mathcal{N}|\le\frac{4L|\mathcal{I}|}{\varepsilon}$,
and let $\mathcal{E}$ denote an event such that
\begin{align*}
\sup_{\lambda\in\mathcal{N}}\bigg|\sum_{1\leq i\leq n}f_{i}(X_{i},\lambda)\bigg| & \le\frac{\varepsilon}{2},\\
\sup_{\lambda\in\mathcal{I}}\bigg|\frac{\mathrm{d}}{\mathrm{d}\lambda}\sum_{1\leq i\leq n}f_{i}(X_{i},\lambda)\bigg| & \le L,
\end{align*}
which holds with probability at least $1-\frac{8L|\mathcal{I}|}{\varepsilon}\delta_{1}-\delta_{2}$
(according to the assumptions and the union bound). 

For any $\lambda\in\mathcal{I}$, let $\hat{\lambda}\in\mathcal{N}$
such that $|\lambda-\hat{\lambda}|\le\frac{\varepsilon}{2L}$. One
can easily check that on the event $\mathcal{E}$, one has
\begin{align*}
\sup_{\lambda\in\mathcal{I}}\bigg|\sum_{1\leq i\leq n}f_{i}(X_{i},\lambda)\bigg| & =\sup_{\lambda\in\mathcal{I}}\bigg|\sum_{1\leq i\leq n}\big(f_{i}(X_{i},\lambda)-f_{i}(X_{i},\hat{\lambda})+f_{i}(X_{i},\hat{\lambda})\big)\bigg|\\
 & \le\sup_{\lambda\in\mathcal{I}}\bigg|\frac{\mathrm{d}}{\mathrm{d}\lambda}\sum_{1\leq i\leq n}f_{i}(X_{i},\lambda)\bigg|\cdot|\lambda-\hat{\lambda}|+\sup_{\hat{\lambda}\in\mathcal{N}}\bigg|\sum_{1\leq i\leq n}f_{i}(X_{i},\hat{\lambda})\bigg|\\
 & \leq L\cdot\frac{\varepsilon}{2L}+\sup_{\hat{\lambda}\in\mathcal{N}}\bigg|\sum_{1\leq i\leq n}f_{i}(X_{i},\hat{\lambda})\bigg|\\
 & \le\frac{\varepsilon}{2}+\frac{\varepsilon}{2}=\varepsilon,
\end{align*}
thus concluding the proof.

\bibliographystyle{abbrvnat}
\bibliography{bibfile_eiglinearform}

\end{document}